\documentclass[preprint,hidelinks,onefignum,onetabnum]{siamart220329}

\usepackage{lipsum}
\usepackage{amsfonts}
\usepackage{graphicx}
\usepackage{epstopdf}
\usepackage{algorithmic}
\usepackage{hyperref}       % hyperlinks
\usepackage{url}            % simple URL typesetting
\usepackage{booktabs}       % professional-quality tables
\usepackage{amsfonts}       % blackboard math symbols
\usepackage{nicefrac}       % compact symbols for 1/2, etc.
\usepackage{microtype}      % microtypography
\usepackage{xcolor}         % colors
\usepackage{array}
\usepackage{caption}
\usepackage{amsfonts}       % blackboard math symbols
\usepackage{amssymb}
\usepackage{comment}
\usepackage{tcolorbox}
\usepackage{empheq}
\usepackage{algorithmic}
\usepackage{caption}
\usepackage{subcaption}
\usepackage{dsfont}
\usepackage{enumitem}
\usepackage{placeins}
\usepackage{multirow}
\usepackage{makecell}
\usepackage{xcolor,colortbl}
\usepackage{stmaryrd}
\usepackage{cleveref}
\usepackage[title]{appendix}

\usepackage{todonotes,tikz}

\DeclareMathOperator{\Tr}{Tr}

\newcommand{\mbf}[1]{\mathbf{#1}}

\newcommand{\rank}[0]{\mathtt{rank}}
\newcommand{\prank}[0]{\mathtt{rankproj}}

\newcommand{\mc}[1]{\mathcal{#1}}

\newcommand{\ip}[2]{\langle #1,#2 \rangle}
\newcommand{\vertiii}[1]{{\left\vert\kern-0.25ex\left\vert\kern-0.25ex\left\vert #1 
		\right\vert\kern-0.25ex\right\vert\kern-0.25ex\right\vert}}

\DeclareMathOperator*{\argmin}{arg\,min}
\newcommand{\supp}[0]{\mathtt{supp}}

\newcommand\mat[1]{\mathbf{#1}}
\newcommand\seq[2]{{#1}_{#2}}
\newcommand\matseq[2]{\seq{\mathbf{#1}}{#2}}
\newcommand\setDBfactor[1]{\Sigma^{#1}}
\newcommand\DBsupport[1]{\mat{S}_{#1}}
\newcommand\setButterfly[1]{\cB^{#1}}
\newcommand\setblocklowrank[1]{\mathcal{A}^{#1}}
\DeclareMathAlphabet\mathbfcal{OMS}{cmsy}{b}{n}
\newcommand\rankonesuppseq[1]{\bU_{#1}}

\newcommand\tleft{\mathrm{left}}
\newcommand\tright{\mathrm{right}}

\newcommand{\colrange}{\mathrm{colspan}}
\newcommand{\rowrange}{\mathrm{rowspan}}

\usepackage{mathtools} % for 'bsmallmatrix' environment

\newcommand*{\Z}{\mathbb{Z}}
\newcommand*{\RR}{\mathbb{R}}

\newcommand*{\CC}{\mathbb{C}}

\newcommand*{\NN}{\mathbb{N}}

\newcommand*{\bA}{\mathbf{A}}
\newcommand*{\bB}{\mathbf{B}}
\newcommand*{\bC}{\mathbf{C}}
\newcommand*{\bD}{\mathbf{D}}
\newcommand*{\bE}{\mathbf{E}}
\newcommand*{\bF}{\mathbf{F}}

\newcommand*{\bH}{\mathbf{H}}
\newcommand*{\bL}{\mathbf{L}}
\newcommand*{\bI}{\mathbf{I}}
\newcommand*{\bK}{\mathbf{K}}
\newcommand*{\bN}{\mathbf{N}}
\newcommand*{\bM}{\mathbf{M}}
\newcommand*{\bP}{\mathbf{P}}
\newcommand*{\bQ}{\mathbf{Q}}
\newcommand*{\bR}{\mathbf{R}}
\newcommand*{\bS}{\mathbf{S}}

\newcommand*{\bU}{\mathbf{U}}
\newcommand*{\bV}{\mathbf{V}}
\newcommand*{\bW}{\mathbf{W}}
\newcommand*{\bX}{\mathbf{X}}
\newcommand*{\bY}{\mathbf{Y}}
\newcommand*{\bZ}{\mathbf{Z}}

\newcommand*{\cB}{\mathcal{B}}

\newcommand*{\cM}{\mathcal{M}}
\newcommand*{\cO}{\mathcal{O}}
\newcommand*{\cP}{\mathcal{P}}

\newcommand*{\cT}{\mathcal{T}}

\newcommand\pattern{{\boldsymbol{\pi}}}
\newcommand\arch{{\boldsymbol{\beta}}}
\newcommand\factarch{\Sigma}

\newcommand\paramfour{{(a,b,c,d)}}

\newcommand\intset[1]{\llbracket #1 \rrbracket}

\newcommand{\ttheta}{{\Tilde{\pattern}}}
\newcommand\col[1]{{:, {#1}}}
\newcommand\row[1]{{{#1}, :}}

\newcommand\one[1]{\mbf{1}_{#1}}

\newcommand\treelevel[2]{{#1}(#2)}

\newcommand\rankprojerrsq[2]{\varepsilon_{#1}^2(#2)}

\newcommand\matindex[3]{{#1}[#2,#3]}
\newcommand\transpose[1]{{#1}^\top}

\newcommand{\ksf}{Kronecker-sparse factor}
\newcommand{\ksfs}{Kronecker-sparse factors}

\newcommand\splitarch[2]{{#1}_{#2}}

\newcommand\partitionrow{\cP^{\mathrm{row}}}
\newcommand\partitioncol{\cP^{\mathrm{col}}}

\newcommand\treerow{T^{\textrm{row}}}
\newcommand\treecol{T^{\textrm{col}}}

\newcommand\ER[1]{\textcolor{blue}{#1}}
\newcommand\TL[1]{\textcolor{red}{#1}}

\newcommand\HL[1]{\textcolor{blue}{#1}}

\newcommand\revision[1]{\textcolor{black}{#1}}

\newcolumntype{R}{>{$}r<{$}} %
\newcolumntype{V}[1]{>{[\;}*{#1}{R@{\;\;}}R<{\;]}} %
\ifpdf
  \DeclareGraphicsExtensions{.eps,.pdf,.png,.jpg}
\else
  \DeclareGraphicsExtensions{.eps}
\fi

\newsiamremark{remark}{Remark}
\newsiamremark{hypothesis}{Hypothesis}
\crefname{hypothesis}{Hypothesis}{Hypotheses}
\newsiamthm{claim}{Claim}
\crefname{claim}{Claim}{Claim}
\newsiamthm{example}{Example}
\crefname{example}{Example}{Example}
\newsiamthm{prop}{Proposition}
\crefname{prop}{Proposition}{Proposition}
\newsiamthm{conjecture}{Conjecture}
\crefname{conjecture}{Conjecture}{Conjecture}
\newsiamthm{assume}{Assumption}
\crefname{assume}{Assumption}{Assumption}

\makeatletter 
\@mparswitchfalse%
\makeatother
\normalmarginpar

\headers{Butterfly factorization with error guarantees
}{Q.-T. Le, L. Zheng, E. Riccietti, R. Gribonval}

\title{Butterfly factorization with error guarantees\thanks{Submitted to the editors DATE. \textsuperscript{\textdagger}Equal contribution, the first two co-authors are listed alphabetically.
\textsuperscript{1}{ENS de Lyon, CNRS, Inria, Université Claude Bernard Lyon 1, LIP, UMR 5668, 69342, Lyon cedex 07, France};
\textsuperscript{2}{valeo.ai, Paris, France};
\textsuperscript{3}{Inria, CNRS, ENS de Lyon, Université Claude Bernard Lyon 1, LIP, UMR 5668, 69342, Lyon cedex 07, France (quoc-tung.le@tse-fr.eu, leonzheng314@gmail.com, elisa.riccietti@ens-lyon.fr, remi.gribonval@inria.fr). This work has been finalised while Tung Le was a postdoctoral researcher at Toulouse School of Economics, France and L\'{e}on Zheng a researcher at Huawei research center, Paris, France.}
\funding{This project was supported in part by the AllegroAssai ANR project ANR-19-CHIA-0009, by the
CIFRE fellowship N°2020/1643, and by the SHARP ANR project ANR-23-PEIA-0008 funded in the context of the France 2030 program. Tung Le was supported by AI Interdisciplinary Institute ANITI funding, through the French "Investments for the Future – PIA3" program under the grant agreement ANR-19-PI3A0004, and Air Force Office of Scientific Research, Air Force Material Command, USAF, under grant numbers FA8655-22-1-7012.}}}

\author{Quoc-Tung Le\textsuperscript{\textdagger,1}, Léon Zheng\textsuperscript{\textdagger,1,2}, Elisa Riccietti \textsuperscript{1}, Rémi Gribonval \textsuperscript{3}}

\usepackage{amsopn}

\ifpdf
\hypersetup{
  pdftitle={Butterfly factorization with error guarantees},
  pdfauthor={Q.-T. Le, L. Zheng, E. Riccietti, R. Gribonval}
}
\fi

\begin{document}
\maketitle
\begin{abstract}
   %Finding a fast algorithm for a certain linear operator usually amounts to approximate the corresponding matrix by a product of several sparse factors. 
    In this paper, we investigate the butterfly factorization problem, i.e., the problem of approximating a matrix by a product of sparse and structured factors. %called the butterfly factors. 
 We propose a new formal mathematical description of such factors, that encompasses many different variations of butterfly factorization with different choices of the prescribed sparsity patterns. Among these 
 \revision{choices}
%  supports, 
 we identify those that ensure that the factorization problem admits an optimum, thanks to a new property called ``chainability". 
     For those supports we propose a new butterfly
     %\todo{Léon: par la suite tout les ``butterfly" en couleur sont issus du changement de ``hierarchical" en ``butterfly"} 
     algorithm that yields an approximate solution to the butterfly factorization problem and that is supported by stronger theoretical guarantees than existing factorization methods.
     %for butterfly factorization.
    % guarantees on the control of the approximation error, that is, 
    Specifically, we show that the ratio of the approximation error by the minimum value is bounded by a constant, independent of the target matrix.
   % To the best of our knowledge, this type of guarantee for the 
    % (the?) 
  %  butterfly factorization problem is new in the literature.  
    % This contrasts with first-order optimization methods that lack guarantees of success, or existing butterfly algorithms that only possess guarantees in the exact case. 
    % The proposed algorithm performs successive factorization into two factors via low-rank approximation of specific submatrices, until the desired number of factors is obtained.   
  %  This analysis is based on our newly proposed notion of chainability that unifies several existing variations of butterfly factorization in the literature, which differ by their choice of the prescribed sparsity patterns. 
    % Our analysis covers many different variations of butterfly factorization with different choices of the prescribed sparsity patterns, as long as they satisfy a condition of so-called chainability. 
\end{abstract}

\section{Introduction}
\label{sec:intro}
%\ERc{I would change the structure of the introduction, I put what I have in mind in the old ex article (which needs to be fixed if we decide to keep this new version of the intro)}
%\mRG{J'aime bcp la proposition d'Elisa, je l'ai donc implémentée ici}
%Si cela vous convient je vous laisse l'implémenter dans ce document, et je relirai en détail apres.}
% Contexte: accélérer la multiplication matricielle
Algorithms for the rapid evaluation of linear operators are important tools in many domains like scientific computing, signal processing, and machine learning. In such applications, where a very large number of parameters is involved, the \emph{direct} computation of the matrix-vector multiplication hardly scales due to its quadratic complexity in the matrix size. Many existing works therefore rely on analytical or algebraic assumptions on the considered matrix to approximate the evaluation of matrix-vector multiplication with a subquadratic complexity. Examples of such structures include low-rank matrices, hierarchical matrices \cite{hackbusch2015hierarchical}, fast multipole methods \cite{engheta1992fast}, etc.

% Il faut un modèle: on s'intéresse au modèle butterfly. Expliquer en quoi ça consiste.
% Pourquoi butterfly? car il a été montré qu'une certaine propriété de low-rank implique la matrice a une factorisation creuse. Exemple en analyse numérique
Among these different structures, 
% recent
previous
%\mRG{l'idee est ancienne, le renouveau d'interet en ML recent} 
work has identified another class of matrices that can be compressed for accelerating matrix multiplication. It is the class of so-called \emph{butterfly} matrices \cite{michielssen1996multilevel,o2007new,candes2009fast}, and includes many matrices appearing in scientific computing problems, like kernel matrices associated with special function transforms \cite{o2007new,ying2009sparse} or Fourier integral operators \cite{candes2009fast,demanet2012fast,li2015multiscale}. Such matrices satisfy a certain low-rank property, named the \emph{complementary low-rank} property \cite{li2015butterfly}: it has been shown that if specific submatrices of a target matrix $\mathbf{A}$ of size $n \times n$ are numerically low-rank, then $\mathbf{A}$ can be compressed by successive hierarchical low-rank approximations of these submatrices, 
%in the sense that 
and that as a result
it can be approximated by a sparse factorization 
\begin{equation*}
    \hat{\mathbf{A}} = \matseq{X}{1} \ldots \matseq{X}{L}
\end{equation*}
with $L = \mathcal{O}(\log n)$
% \LZc{decide notation for number of factors} 
factors $\mathbf{X}_\ell$ having at most $\mathcal{O}(n)$ nonzero entries for each  $\ell \in \intset{L} := \{1, \ldots, L \}$.
This sparse factorization, called in general \emph{butterfly factorization}, would then yield a fast algorithm for the approximate evaluation of the matrix-vector multiplication by $\mathbf{A}$, in $\mathcal{O}(n \log n)$ complexity. 
% Since the discrete Fourier transform (DFT) matrix satisfy the complementary low-rank property, the butterfly factorization is often viewed as a generalization of the fast Fourier transform algorithm.

% \paragraph{Butterfly factorization}
An alternative definition of the butterfly factorization refers to a sparse matrix factorization with \emph{specific} constraints on the sparse factors. According to \cite{dao2019learning,dao2020kaleidoscope,le2022fastlearning,zheng2023efficient,dao2022monarch,lin2021deformable}, a matrix $\mathbf{A}$ admits a certain butterfly factorization if, up to some row and column permutations, it can be factorized into a certain number of factors $\matseq{X}{1}, \ldots, \matseq{X}{L}$ for a prescribed number $L \geq 2$, such that each factor $\matseq{X}{\ell}$ for $\ell \in \intset{L}$ satisfies a so-called \emph{fixed-support} constraint, i.e., 
the support of $\matseq{X}{\ell}$, denoted $\supp(\matseq{X}{\ell})$, is included in the support of a prescribed binary matrix $\matseq{S}{\ell}$.
% \LZc{explain support/binary matrices notation} 
The different existing butterfly factorizations only vary by their number of factors $L$, and their choice of binary matrices $\matseq{S}{1}, \ldots, \matseq{S}{L}$. Let us give some examples of such factorizations.
\begin{enumerate}
% [leftmargin=*]
    \item \textbf{Square dyadic butterfly factorization} \cite{dao2019learning,dao2020kaleidoscope,le2022fastlearning,zheng2023efficient}. It is defined for matrices of size $n \times n$ where $n$ is a power of two. The number of factors is $L := \log_2 n$. For $\ell \in \intset{L}$, the factor $\matseq{X}{\ell}$ is of size $n \times n$, and satisfies the support constraint $\supp(\matseq{X}{\ell}) \subseteq \supp(\matseq{S}{\ell})$, where
    \begin{equation*}
        \forall \ell \in \intset{L}, \quad \matseq{S}{\ell} := \mat{I}_{2^{\ell - 1}} \otimes \mat{1}_{2 \times 2} \otimes \mat{I}_{n / 2^\ell}. 
    \end{equation*}
    % \LZc{illustration?}
    Here, $\mat{I}_n$ denotes the identity matrix of size $n$, $\mat{1}_{p \times q}$ denotes the matrix of size $p \times q$ full of ones, and $\otimes$ denotes the Kronecker product. 
    This butterfly factorization appears in many structured linear maps commonly used in machine learning and signal processing, like the Hadamard matrix, or the discrete Fourier transform (DFT) matrix (up to the bit-reversal permutation of column indices). Other structured matrices like circulant matrix, Toeplitz matrix or Fastfood transform \cite{yang2015deep} can be written as a product of matrices admitting such a butterfly factorization, up to matrix transposition \cite{dao2020kaleidoscope}. 
    %\mRG{Privilegier un historique chronologique, DFT et approximation d'operateurs pseudo-différentiels viennent d'abord je crois \LZ{c'est fait}} 
    This factorization is also used to design structured random orthogonal matrices \cite{parker1995random}, and for quadrature rules on the hypersphere \cite{munkhoeva2018quadrature}.
    % This representation was originally used for designing structured random orthogonal matrices \cite{parker1995random}, and later used for quadrature rules on the hypersphere \cite{munkhoeva2018quadrature}. In recent machine learning applications, this parameterization has been used to replace hand-crafted structure in speech processing models or channel shuffling in certain convolutional neural networks, or to learn latent permutation.

    \item \textbf{Monarch factorization} \cite{dao2022monarch}. A Monarch factorization parameterized by two integers $p$, $q$ decomposes a matrix $\mathbf{A}$ of size $m \times n$ into $L := 2$ factors $\matseq{X}{1}$, $\matseq{X}{2}$ such that $\supp(\matseq{X}{\ell}) \subseteq \supp(\matseq{S}{\ell})$ for $\ell = 1, 2$ where
    \begin{equation*}
        \matseq{S}{1} := \mat{1}_{p \times q} \otimes \mat{I}_{\frac{m}{p}}, \quad \matseq{S}{2} := \mat{I}_q \otimes  \mat{1}_{\frac{m}{p} \times \frac{n}{q}}.
    \end{equation*}
    % \begin{equation*}
    %     \begin{cases}
    %         \matseq{S}{1} := \mat{1}_{\frac{m}{b} \times \frac{n}{b}} \otimes \mat{I}_b, \quad \matseq{S}{2} := \mat{I}_b \otimes  \mat{1}_{\frac{n}{b} \times \frac{n}{b}} & \text{if } n < m \\
    %         \matseq{S}{1} := \mat{1}_{\frac{m}{b} \times \frac{m}{b}} \otimes \mat{I}_b, \quad \matseq{S}{2} := \mat{I}_b \otimes  \mat{1}_{\frac{m}{b} \times \frac{n}{b}} & \text{otherwise}
    %     \end{cases}.
    % \end{equation*}
    Here, we assume that $p$, $q$ divides $m$, $n$ respectively.
    The DFT matrix of size $n \times n$ admits such a factorization for $p = q$, up to a column permutation. Indeed, according to the Cooley-Tukey algorithm, computing the discrete Fourier transform of size $n$ is equivalent to performing $p$ discrete Fourier transforms of size $n / p$ first, and then $n / p$ discrete Fourier transforms of size $p$, see, e.g., equations (14) and (21) in \cite{duhamel1990fast}.
    % . This translates to a Monarch factorization of the DFT matrix of size $n$ with parameter $p = q$ up to columns permutation
    % In the case of $m = n$, the Monarch factorization parameterized by $p = q$ corresponds to the Cooley-Tukey factorization of the DFT matrix

    % This factorization can be seen as an extension of the . 
    % \footnote{This definition of Monarch factorization does not correspond to the definition given in the original paper \cite{dao2022monarch}, but it corresponds to the definition given in the code repository used for the experiments of the paper \cite{dao2022monarch}.} 
    % \LZc{Peut être qu'il vaut mieux introduire la version du papier car cela permet de justifier pourquoi on s'intéresse à cette structure. C'est en fait un square butterfly relâché, donc on hérite de l'expressivité des square butterfly.}
    % When this parameterization is used for certain weight matrices in the transformer architecture for vision or language tasks, \cite{dao2022monarch} claims that the obtained sparse network yields  similar performance to the original architecture, while having a lower time complexity for training and inference, since they can implement a fast matrix multiplication algorithm for a matrix admitting a Monarch factorization.

    % \item \LZc{Parler de block butterfly? ce serait bien que ça rentre dans notre cadre. mais c'est de la forme $\mat{I}_a \otimes \mat{1}_{b \times c} \otimes \mat{I}_d \otimes \mat{1}_r$}

    \item \textbf{Deformable butterfly factorization} \cite{lin2021deformable}. Previous conventional butterfly factorizations can be generalized as follows. Given an integer $L \geq 2$, a matrix $\mat{A}$ admits a deformable butterfly factorization parameterized by a list of tuples $(p_\ell, q_\ell, r_\ell, s_\ell, t_\ell)_{\ell=1}^L$ if $\mat{A} = \matseq{X}{1} \ldots \matseq{X}{L}$ where each factor $\matseq{X}{\ell}$ for $\ell \in \intset{L}$ is of size $p_\ell \times q_\ell$ and has a support included in $\supp(\mat{S}_\ell)$, defined as:
    \begin{equation*}
        \forall \ell \in \intset{L}, \quad \mat{S}_{\ell} := \mat{I}_{\frac{p_\ell}{r_{\ell} t_{\ell}}} \otimes \mat{1}_{r_\ell \times s_\ell} \otimes \mat{I}_{t_\ell}.
    \end{equation*}
    Here, it is assumed that $\frac{p_\ell}{r_\ell t_\ell} = \frac{q_\ell}{s_\ell t_\ell}$ is an integer, for each $\ell \in \intset{L}$.
\end{enumerate}

% The definitions above are however restricted to some specific matrix sizes and specific number of factors. Hence, an extended definition of butterfly factorization was proposed in \cite{lin2021deformable}. 

% and one goal of this paper is to formalize mathematically this extension. 
%In general, we remark in the
In all these examples the fixed-support constraint on each %butterfly 
factor $\mat{X}$ takes the form $\supp(\mat{X}) \subseteq \supp(\mat{I}_a \otimes \mat{1}_{b \times c} \otimes \mat{I}_d)$ for some  integer parameters $(a, b, c, d)$. \Cref{fig:DBfactorillu}
%\todo{Léon: est-ce que la personne qui a fait la figure 1 pourra nous aider à changer deformable \ksf{}en Kronecker-sparse factor? \TL{Done}}
 illustrates the sparsity pattern $\mat{S}_\pattern \coloneqq \mat{I}_a \otimes \mat{1}_{b \times c} \otimes \mat{I}_d$ of a %butterfly 
factor associated with the tuple
% parameter 
$\pattern = (a,b,c,d)$, % is illustrated in \Cref{fig:DBfactorillu}.
that we call a {\em pattern}. 
One of the main benefits of choosing such fixed-support constraints \emph{instead of an arbitrary sparse support} is its \emph{block structure} that enables efficient implementation on specific hardware like Intelligence 
% \mRG{Intelligent?}
Processing Unit (IPU) \cite{shekofteh2023reducing} or GPU \cite{dao2019learning, dao2022monarch,gonon2024inference}, with practical speed-up for matrix multiplication. 
%\mRG{TODO Tung/ Léon: ajouter ref biblio vers votre preprint GPU soumis à NeurIPS} 

Such butterfly factorizations have been used in some machine learning applications.
In line with recent works \cite{dao2019learning,dao2020kaleidoscope,dao2022monarch,lin2021deformable}, the parameterization 
\eqref{eq:butterfly-factorization-intro} can be used to construct a
generic representation for structured matrices that is not only expressive, 
%enough to capture common structured linear maps (Hadamard, DFT, circulant matrices, etc.),
but also
%while being 
differentiable and thus compatible with %in order to use this parameterization in a 
machine learning
pipelines involving gradient-based optimization of parameters given training samples. %where parameters of a model are optimized by gradient-based methods during a training phase. % in a learning task. 
The expected benefits then range from a more compressed storage and better generalization properties (thanks to the reduced number of parameters) to possibly faster implementations.
% Typically, parameterizing certain weight matrices in a neural network using %a deformable\LZc{remove deformable?} 
% butterfly factors with some architeture $\arch$ introduces a certain bias in the model, and if the bias is
% well-adapted to the given learning task, the obtained model could lead
% to good \mRG{définir "good" ? l'effet régularisant peut-il meme améliorer la perf ?} performance compared to the original model, while having a reduced number of
% parameters, and eventually a reduced computational cost, provided that a fast algorithm
% associated with the obtained sparse factorization can be implemented efficiently. 
For instance:
\begin{itemize}
    \item The square dyadic butterfly factorization was used to replace hand-crafted structures in speech processing models or channel shuffling in certain convolutional neural networks, or to learn a latent permutation \cite{dao2020kaleidoscope}.

    \item The Monarch parameterization \cite{dao2022monarch} of %was used to replace 
    certain weight matrices in transformers for vision or language tasks led to speed-ups of training and inference time.

    \item Certain choices of deformable butterfly %factorization
    parameterizations %in 
    \cite{lin2021deformable} of %is used to replace 
    kernel weights in convolutional layers, for vision tasks, led to similar performance as the original convolutional neural network with fewer parameters.

\end{itemize}

\begin{figure}[t]
	\centering
	\includegraphics[width=\textwidth]{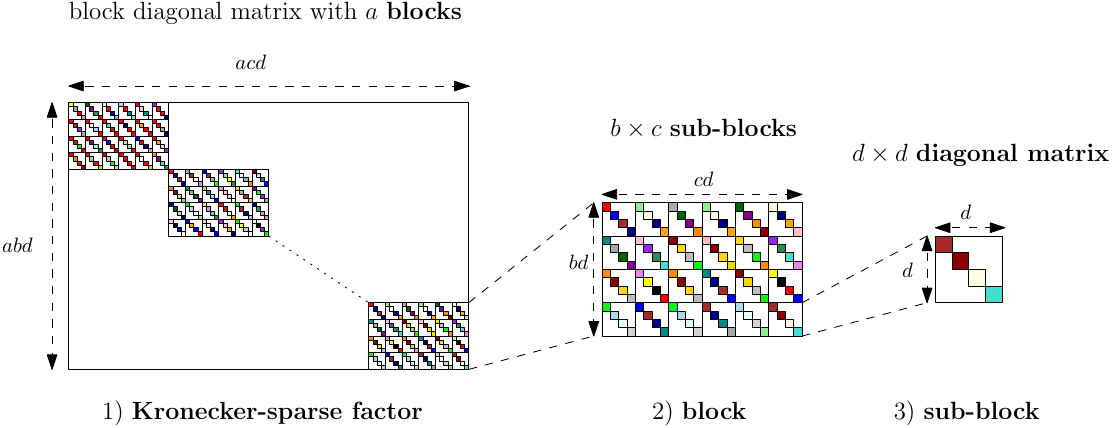}
	\caption{Illustration of the support of a 
 % deformable \ksf{}of parameters $(a,b,c,d)$. 
%butterfly 
factor with pattern $\pattern=(a,b,c,d)$.
 The colored squares indicate the indices belonging to the support. The sub-figures (1), (2), (3) illustrate respectively the concepts of factor, block and sub-block.} %\LZc{It is a detail but maybe it is better to avoid enumerating the subcaption by a), b), c) because a, b, c, d are parameters for the DB parameter. Use instead 1) 2) 3) or i) ii) iii)?}
%\RG{Suggestion: mettre des couleurs {\em variees} a la place de chaque carré jaune élémentaire pour eviter de suggerer des coefficients constants}
\label{fig:DBfactorillu}
\end{figure}

% \begin{figure}
% \caption{example of random matrix I found on the net}
% \begin{tikzpicture}[inner sep=0pt, minimum size=1cm]
%     \foreach \y in {1, ..., 9} {
%         % add vertical shift and dots
%         \ifnum\y=2\relax
%           \tikzset{yshift=-1cm}
%         \else
%           \ifnum\y=1\relax
%             \path node at (5, \y+1) {\huge$\vdots$};
%           \fi
%         \fi
        
%         \foreach \x in {1, ..., 9} {
%             % add horizontal shift and dots
%             \ifnum\x<8\relax
%             \else
%               \tikzset{xshift=1cm*(\x-7)}
%               \ifnum\y=1\relax
%               \path node at (\x-1, 5) {\huge\ldots}
%                     node at (\x-1, 1) {\huge\ldots};
%               \fi
%             \fi
            
%             % get random color
%             \pgfmathparse{int(150*rnd-75)} % [-75, 75]
%             % color range: blue!75!white .. blue!100 .. blue!75!black
%             \ifnum\pgfmathresult>0\relax
%               \colorlet{MyColor}{black!\pgfmathresult!yellow}
%             \else
%               \edef\pgfmathresult{-\pgfmathresult}
%               \colorlet{MyColor}{white!\pgfmathresult!yellow}
%             \fi
            
%             % draw filled square
%             \node[fill=MyColor] at (\x,\y) {};
%         }
%     }
% \end{tikzpicture}
% \end{figure}

\subsection{Problem formulation and contributions}

This paper focuses on the problem of approximating a target matrix $\mat{A}$ by a product of 
%deformable\LZc{remove deformable? c'est assez clair que c'est dans le framework deformable quand on dit que c'est associé à une architecture $\arch$} 
%butterfly 
structured sparse
factors associated with a given {\em architecture} $\arch = (\pattern_\ell)_{\ell=1}^L$: %\mRG{commenter $\inf$ vs $\min$}
% \begin{equation}
% \label{eq:butterfly-approximation-pb}
%     \inf_{(\matseq{X}{\ell})_{\ell=1}^L \in \setDBfactor{\arch}} \| \mat{A} -  \matseq{X}{1} \ldots \matseq{X}{L} \|_F,
% \end{equation}
\begin{equation}
\label{eq:butterfly-approximation-pb}
%    E^\arch(\bA):= \inf_{(\matseq{X}{\ell})_{\ell=1}^L \in \setDBfactor{\arch}} \| \mat{A} - \matseq{X}{1} \ldots \matseq{X}{L} \|_F = \inf_{\bB \in \setButterfly{\arch}} \|\bA - \bB\|_F^2,\tag{BF}
    E^\arch(\bA):= \inf_{(\matseq{X}{\ell})_{\ell=1}^L} \| \mat{A} - \matseq{X}{1} \ldots \matseq{X}{L} \|_F^2 = \inf_{\bB} \|\bA - \bB\|_F^2,
    % \tag{BF}
\end{equation}
where $\bB$ is a butterfly matrix (cf.~\eqref{eq:butterfly-factorization-intro} below and \Cref{def:dbarchitecture}), 
%\ER{maybe put a remark here to comment on the def of butterfly matrix }% and 
each $\matseq{X}{\ell}$ is a %butterfly 
factor with sparsity pattern prescribed by $\pattern_\ell$, and
 $\| \cdot \|_F$ is the Frobenius norm. We will call these factors ``\ksf", due to the Kronecker-structure of their sparsity pattern. %their block structure}.
%  ``deformable butterfly factors" because the description of the prescribed supports above is equivalent to that of \cite{lin2021deformable}.
 Several methods have been proposed to address this butterfly factorization problem, but we argue that they either lack guarantees of success, or only have partial guarantees. We fix this issue here by introducing a new butterfly algorithm endowed with theoretical guarantees.

\begin{table}[t]
% \footnotesize
\caption{Existing architectures $\arch$ in the literature. ${(\star)}$ 
Note that \cite{lin2021deformable} did not explicitly state constraints on $(a_\ell, b_\ell, c_\ell, d_\ell)_{\ell=1}^L$ for deformable butterfly factorization, because they use an alternative description of the sparsity patterns of the factors. See \Cref{example:ex_chainable} for more details.}
%\RG{peut-on enlever les couleurs? oui} 
%In fact, they only consider variants of butterfly factorization corresponding to chainable architectures (cf.~\Cref{def:chainableseqDBPs} below) with  $\mathbf{r}(\arch) = (1, \ldots, 1)$.
% Note that \cite{lin2021deformable} only considers chainable architectures $\arch$ such that $\mathbf{r}(\arch) = (1,\ldots,1)$ (cf.~\Cref{def:chainableseqDBPs}). 
% \RG{Attention vers quoi pointe $\dagger$? Et quelles contraintes sur $a_\ell, b_\ell, c_\ell, d_\ell$ pour \cite{lin2021deformable}?} \LZ{vers la ligne deformable butterfly, et les contraintes sur $a_\ell, b_\ell, c_\ell, d_\ell$ ne sont pas explicites dans le papier (juste évoqué intuitivement) mais il y a la chainabilité avec $\mathbf{r}(\arch) = (1, \ldots, 1)$}
%\ER{cardinality not been introduced, is this a problem?}
 
\label{tab:butterfly-parametrization}
\begin{center}
    % \scriptsize
    \resizebox{\linewidth}{!}{
  \begin{tabular}{ccccc} 
    \toprule
   \bf Architectures &  $L$ & \bf Size 
   & \bf Values of $\arch = (\pattern_\ell)_{\ell=1}^L$ & \bf Chainable? \\ \midrule
    Low-rank matrix & 2 & $m \times n$ & {$(1, m, r, 1)$, $(1, r, n, 1)$}& Yes\\
    \midrule
    Monarch \cite{dao2022monarch} &  2 & $m \times n$ & {$(1, p, q, m/p)$, $(q, m/p, n/q, 1)$} & Yes\\ 
    \midrule
    Square dyadic butterfly \cite{dao2019learning} &  any
    & $2^{L} \times 2^{L}$ & $(2^{\ell-1}, 2, 2, 2^{L - \ell})_{\ell = 1}^{L}$ & Yes \\
    \midrule
    ($\star$) Deformable butterfly \cite{lin2021deformable} &  any %$L$
    & $m \times n$ & $(a_\ell, b_\ell, c_\ell, d_\ell)_{\ell=1}^L$ & Yes \\
    \midrule
    Kaleidoscope \cite{dao2020kaleidoscope} & even & $2^{L/2} \times 2^{L/2}$ &  $\pattern_\ell = \begin{cases}
        (2^{\ell-1}, 2, 2, 2^{L/2 - \ell}) & \text{if } \ell \leq L/2\\
        (2^{L - \ell}, 2, 2, 2^{\ell - L/2 - 1}) & \text{if } \ell > L/2\\
    \end{cases}$ & No \\   \bottomrule
  \end{tabular}
  }
\end{center}
\label{tab:existingparameterization}
\end{table}

More precisely, the main contributions of this paper are:

\begin{enumerate}
\item To introduce, via the definition of a \ksf{}, a formal mathematical description of the 
``deformable butterfly factors" 
of 
\cite{lin2021deformable}. While we owe \cite{lin2021deformable} the original idea of extending previous butterfly factorizations, the mathematical formulation of the prescribed supports as  Kronecker products is a novelty that allows a theoretical study of the corresponding butterfly factorization, as done in this paper. Moreover, our parameterization uses 4 parameters and removes the redundancy in the original 5-parameter description of deformable butterfly factors of \cite{lin2021deformable}. {\Cref{tab:butterfly-parametrization} summarizes the main characteristics of existing butterfly architectures covered by our framework.}
    \item To define  the {\em chainability} of an architecture $\arch$ (\Cref{def:chainableseqDBPs}), which is basically a ``stability" property that ensures that a product of \ksfs{} is still a \ksf{}. We prove that Problem \eqref{eq:butterfly-approximation-pb} admits an optimum when $\arch$ is chainable (\Cref{cor:optimum-exists-in-bf}). 
    \item To characterize \emph{analytically} the set of butterfly matrices with architecture $\arch$, 
\begin{equation}
\label{eq:butterfly-factorization-intro}
    \setButterfly{\arch} := \left \{ \matseq{X}{1} \ldots \matseq{X}{L} \, | \,
    \supp(\matseq{X}{\ell}) \subseteq \supp(\mat{S}_{\pattern_\ell})
    \,\ \forall \ell \in \intset{L} \right \},
\end{equation}
    for a chainable $\arch$, in terms of low-rank properties of certain submatrices of $\mat{A}$ (\Cref{cor:characterizationofDBmatrix}) that are equivalent to a {generalization} of the complementary low-rank property (\Cref{def:complementary-low-rank,cor:equivalence-clr}).
    \item {To define the {\em redundancy} of a chainable architecture (\Cref{def:redundant}). Intuitively,  %if
    a chainable architecture $\arch$ is redundant {if} one can replace it with a ``compressed'' (non-redundant) one $\arch'$ such that $\setButterfly{\arch} = \setButterfly{\arch'}$ (\Cref{prop:procedure-for-redundant-arch}). Thus, from the perspective of accelerating linear operators, redundant architectures have no practical interest.
    % With regards to Problem \eqref{eq:butterfly-approximation-pb}, it \LZ{is sufficient} to consider only non-redundant architectures since the instances corresponding to $\arch$ and $\arch'$ are equivalent, \LZ{as} they share the same optimal values given the same factorized matrix $\bA$.}   \LZc{A-t-on besoin de la dernière phrase? on dit de toute manière dans la prochaine contribution que les garanties sont automatique pour les chaînes redondantes.
    }
    \item To propose a new butterfly algorithm (\Cref{algo:modifedbutterflyalgo}) able to provide an approximate solution to Problem \eqref{eq:butterfly-approximation-pb} %\mRG{Utiliser des vrais numéros pour toutes les équations, pas \eqref{eq:butterfly-approximation-pb} ni \eqref{eq:error-bound-intro} par exemple \LZ{c'est à dire qu'il faut par exemple qqchose comme (1), (2), etc.?}}
        {\em for {non-redundant} chainable architectures}. {Compared to previous similar algorithms, this algorithm introduces a new orthogonalization step that is key to obtain approximation guarantees.} {The algorithm can be easily extended to redundant chainable architectures, with the same theoretical guarantee {(see %\Cref{algo:algoforanychainable}
        \Cref{rem:redundantcase})}. }
        %\ERc{mention here the orthogonalization rather than below?}
    \item %To define the notion of {\em non-redundant} (chainable) architecture $\arch$ and to prove that if $\arch$ is non-redundant then \Cref{algo:modifedbutterflyalgo} outputs $(\seq{\hat{\mat{X}}}{\ell})_{\ell=1}^L \in \setDBfactor{\arch}$ such that 
    %, \ldots, \hat{\matseq{X}{L}}$ associated with $\arch$ 
    To prove that, {for a chainable $\arch$,} \Cref{algo:modifedbutterflyalgo} outputs {butterfly factors} $(\seq{\hat{\mat{X}}}{\ell})_{\ell=1}^L$ % \in \setDBfactor{\arch}$  
    such that 
    \begin{equation}
        \label{eq:error-bound-intro}
        \| \mathbf{A} - \seq{\hat{\mat{X}}}{1} \ldots \seq{\hat{\mat{X}}}{L} \|_F \leq C_\arch \cdot \inf_{{(\mat{X}_{\ell})_{\ell=1}^L }
        %\in \setDBfactor{\arch}
        } \| \mat{A} -  \mat{X}_{1} \ldots \mat{X}_{L} \|_F,
        % \tag{QO}
    \end{equation}
    where $C_\arch \geq 1$ depends \emph{only} on $\arch$ 
   (\Cref{cor:mainresults}), {see \Cref{tab:constant} for examples}. {To the best of our knowledge, this is the first time such a bound is established for a butterfly approximation algorithm.}
    %\ERc{say here this is the first one rather than below?}

    % \LZc{modifier cette phrase}
    % \LZ{We show that this low-rank property is precisely what is known in the literature as the complementary low-rank property.}\RG{Etait-ce deja (partiellement) connu ?}

\end{enumerate}

%\mLZ{discussion avec Tung: il y a moyen d'enlever l'hypothèse non-redundant même pour le 4ème point.}\mRG{A discuter en effet alors.}

\subsection{Outline}
\Cref{section:related-work} discusses related work.
{\Cref{section:preliminaries} introduces some preliminaries on two-factor matrix factorization with fixed-support constraints. {This is also where we setup our general notations.}
% , as they  are required to describe the hierarchical algorithm for butterfly factorization. 
\Cref{sec:DB-factorization} formalizes the definition of deformable
% \LZc{je garderai ici deformable} 
butterfly factorization associated with $\arch$, and introduces the \emph{chainability} and \emph{non-redundancy} conditions for an architecture $\arch$, that will be at the core of the proof of error guarantees on {our proposed} %the 
butterfly algorithm.
\Cref{sec:hierarchical} extends %the
{an existing} hierarchical algorithm,
%for  
% deformable\LZc{remove deformable?} 
{currently expressed only for dyadic}
butterfly factorization, to any chainable $\arch$. 
For non-redundant chainable $\arch$, \Cref{section:normalizedbutterflyfactorization} introduces novel orthonormalization operations in the {proposed} butterfly algorithm.
 % presents the variation of the hierarchical algorithm with , and 
This allows us to establish in \Cref{sec:error} our main results on the control of the approximation error and the low-rank characterization of butterfly matrices associated with chainable $\arch$.
% the proposed hierarchical algorithm, endowed with provable guarantees on the control of approximation error. 
% \Cref{sec:complexity} discusses complexity of the proposed algorithm.
\Cref{sec:experimentbutterfly} proposes some numerical experiments about the proposed butterfly algorithm. 
% \LZc{compléter avec la section expérimentale}. 
The most technical proofs are deferred to the appendices.}
% \LZc{à changer peut être en fonction de la réorganisation Il n'y aura pas de section 9 sur la nécessité chainabilité, il faudra reformuler une version courte en conclusion.}

\begin{table}[t]
	% \footnotesize
	\caption{The approximation ratio $C_\arch$ (see \Cref{eq:error-bound-intro}) of \Cref{algo:modifedbutterflyalgo} with a selection of chainable architectures $\arch$ from \Cref{tab:existingparameterization}. 
%	\RG{peut-on enlever les couleurs? Oui} 
	}
	\begin{center}
    \resizebox{\linewidth}{!}{
		\begin{tabular}{ccccc} \toprule
			\bf Parameterization & \bf Size & \bf $L=|\arch|$& \bf $C_\arch$ in \eqref{eq:mainresults} - \Cref{cor:mainresults} & \bf $C_\arch$ in \eqref{eq:betterbound} - \Cref{cor:mainresults} \\ \midrule
			Low rank matrix & $m \times n$ & 2 & $1$ & $1$\\
			\midrule
			Monarch \cite{dao2022monarch} & $m \times n$ & 2 & $1$ & $1$\\ 
			\midrule
			\makecell{Square dyadic\\ butterfly \cite{dao2019learning}} & $n \times n$ & $\log n$ & $\log n - 1$ & $\sqrt{\log n - 1}$ \\
			\midrule
			\makecell{\revision{Chainable}\\\revision{deformable butterfly}\\$(a_\ell, b_\ell, c_\ell, d_\ell)_{\ell=1}^L$} & \makecell{$m \times n$,\\$m=a_1b_1d_1$,\\$n=a_Lc_Ld_L$} & $L$ & ${\max(L,\,2)} - 1$ & $\sqrt{{\max(L,\,2)} - 1}$\\
			\bottomrule
		\end{tabular}
  }
      \label{tab:constant}
	\end{center}
\end{table}

\section{Related work} 
%\todo{"butterfly factor" $\to$ "Kronecker-sparse factors" TODO Léon via macro. Faire une passe à la fin en control F}
\label{section:related-work}
% \mRG{Pourrait-on mettre des maintenant les contributions ?}
Several methods have been proposed to address the butterfly factorization problem \eqref{eq:butterfly-approximation-pb}, but we argue that they either lack guarantees of success, or only have partial guarantees.

{\bf \em First-order methods.} Optimization methods based on gradient descent \cite{dao2019learning} or alternating least squares \cite{lin2021deformable} are not suitable for Problem \eqref{eq:butterfly-approximation-pb} and lack guarantees of success, because of the non-convexity of the objective function. In fact, the problem of approximating a given matrix by the product of factors with fixed-support constraints, as it is the case for \eqref{eq:butterfly-approximation-pb}, is generally NP-hard and might even lead to numerical instability even for $L = 2$ factors, as shown in \cite{le2022spurious}.
In contrast, we show that the minimum of \eqref{eq:butterfly-approximation-pb} always exists for chainable $\arch$.
%\LZ{In this paper, we show that problem \eqref{eq:butterfly-approximation-pb} admits an optimum under the chainability condition on the architecture $\arch$.}
%\ERc{I would remove this sentence}
% Nevertheless, we show in this work that problem \eqref{eq:butterfly-approximation-pb} \ERc{careful we refer to the problem below  }becomes tractable \mRG{"tractable" est vague. Preciser aussi que $\inf$ est un $\min$?} under some appropriate conditions on $\arch$.

%\mRG{Et nous via la chainabilite on garanti existence optimum}

{\bf \em Hierarchical approach for butterfly factorization.} 
For the specific choice of $\arch$ corresponding to a square dyadic butterfly factorization, i.e., with the architecture $\arch = (2^{\ell-1}, 2, 2, 2^{L-\ell})_{\ell=1}^L$, 
% there exists an efficient hierarchical algorithm that gives a good \ERc{define good?} solution to problem \eqref{eq:butterfly-approximation-pb}, endowed with exact recovery guarantees \cite{le2022fastlearning,zheng2023efficient}. 
there exists an efficient hierarchical algorithm for Problem \eqref{eq:butterfly-approximation-pb}, endowed with {\em exact} recovery guarantees \cite{le2022fastlearning,zheng2023efficient}.
The hierarchical approach performs successive two-factor matrix factorizations, until the desired number of factors $L$ is obtained. In the case of square dyadic butterfly factorization, it is shown that each two-factor matrix factorization in the hierarchical procedure can be solved optimally by computing the best rank-one approximation of some specific submatrices \cite{le2022spurious}, which leads to an overall $\mathcal{O}(n^2)$ complexity for approximating a matrix of size $n \times n$ with the hierarchical procedure. In fact, the hierarchical algorithm in \cite{le2022fastlearning, zheng2023efficient}  can be seen as a variation of previous butterfly algorithms \cite{li2015butterfly}, with the novelty that it works for \emph{any} hierarchical order under which successive two-layers matrix factorizations are performed, whereas existing butterfly algorithms \cite{michielssen1996multilevel,o2007new,candes2009fast} were only focusing on some \emph{specific} hierarchical orders \cite{liu2021butterfly}. However, the question of controlling the approximation error of the algorithm was left as an open question in \cite{zheng2023efficient}. Moreover, it was %is 
not known in the literature if the hierarchical algorithm \cite{zheng2023efficient,le2022fastlearning} could be extended to architectures $\arch$ beyond the square dyadic one. Both questions are answered positively here.
%proposed for square butterfly factorization can be extended to other choices of $\arch$. 

{\bf \em Butterfly algorithms and the complementary low-rank property.}
%\mRG{Donner ici un premier aperçu des liens connus entre butterfly et complementary low-rank, de leurs limites, de ce qu'on completera, et pointer vers où sera le paragraphe de positionnement détaillé}
Butterfly algorithms \cite{michielssen1994multilevel,michielssen1996multilevel,o2007new,candes2009fast,li2015butterfly,li2015multiscale,liu2021butterfly} 
% \LZc{add more citations?} 
% are a family of algorithms that 
% construct a data-sparse representation of a target matrix $\mat{A}$ that satisfies the low-rank complementary property \cite{li2015butterfly}, in order to enable .
look for an approximation of a target matrix $\mathbf{A}$ by a sparse factorization $\hat{\mathbf{A}} = \matseq{X}{1} \ldots \matseq{X}{L}$, assuming that $\mat{A}$ satisfies the so-called \emph{complementary low-rank property}, formally introduced in \cite{li2015butterfly}. This low-rank property assumes that the rank of certain submatrices of $\mat{A}$ restricted to some specific blocks is numerically low and that these blocks satisfy some conditions described by a hierarchical partitioning of the row and column indices, using the notion of \emph{cluster tree} \cite{hackbusch2015hierarchical}. Then, the butterfly algorithm leverages this low-rank property to approximate the target matrix by a data-sparse representation, by performing successive low-rank approximation of specific submatrices. The literature in numerical analysis describes many linear operators associated with matrices satisfying the complementary low-rank property, such as kernel matrices encountered in electromagnetic or acoustic scattering problems \cite{michielssen1994multilevel,michielssen1996multilevel,guo2017butterfly}, special function transforms \cite{oneil2010203}, spherical harmonic transforms \cite{tygert2010fast} or Fourier integral operators \cite{candes2009fast,ying2009sparse,li2015multiscale,li2017interpolative,li2018multidimensional}.

The formal definition of the complementary low-rank property currently given in the literature only considers cluster trees that are dyadic \cite{li2015butterfly} or quadtrees \cite{li2018multidimensional}. In this work, we give a more general definition of the complementary low-rank property that considers arbitrary cluster trees. To the best of our knowledge, this allows us to give the first formal characterization of the set of matrices admitting a (deformable) butterfly factorization associated with an architecture $\arch$, as defined in \eqref{eq:butterfly-factorization-intro}, using this extended definition of the complementary low-rank property. 
% To the best of our knowledge, such a characterization is new in the literature.
% it was previously known that compressing a target matrix $\mat{A}$ yields a sparse matrix factorization of $\mat{A}$ \cite{li2015butterfly}, but a formal description of the supports of the sparse factors was lacking. Our characterization in the general framework of (deformable) butterfly factorization associated with an architecture $\arch$ as in \eqref{eq:butterfly-factorization-intro} shows that the support 
In particular, this shows that the definition in \eqref{eq:butterfly-factorization-intro} is more general than the previous definitions of the complementary low-rank property that were restricted to dyadic trees or quadtrees \cite{li2015butterfly,li2018multidimensional}.
% \todo{Léon: I need to check whether this paragraph is consistent with the changes in the main text and appendices concerning the CLR property}

% To the best of our knowledge, this is the first work that formally relates the definition of the butterfly factorization in the sense of \eqref{eq:butterfly-factorization-intro} using fixed-support constraints described by structured sparsity patterns as in \eqref{eq:sparsity-pattern-intro} to the complementary low-rank property. 
% Moreover, it was previously known that compressing a target matrix $\mat{A}$ yields a sparse matrix factorization of $\mat{A}$ \cite{li2015butterfly}, but a precise description of the supports of the obtained sparse factors from the butterfly algorithm \cite{li2015butterfly} was lacking. In contrast, the characterization given in this work allows for a precise description of the sparsity patterns of the sparse factors obtained 

% this characterization allows for a precise description of the support of the sparse factors that can be obtained 

% In particular, this equivalence completes the gap 
% Previous work on butterfly algorithms \cite{li2015butterfly} showed that compressing a target matrix $\mat{A}$ , but a formal description of the supports of the sparse factors was lacking.

% \LZc{New version, on going} 

% \LZc{Renommer le paragraphe ci-dessous ``Existing error bound for butterfly algorihms''}
{\bf \em Existing error bounds for butterfly algorithms.}
Several existing
%Previous 
butterfly algorithms \cite{o2007new,candes2009fast,li2015butterfly,li2015multiscale} are guaranteed to provide an approximation error $\| \mathbf{A} - \hat{\mathbf{A}} \|_F$ equal to zero, when $\mathbf{A}$ satisfies \emph{exactly} the complementary low-rank property, i.e., the best low-rank approximation errors of the submatrices described by the property are \emph{exactly} zero \cite{o2007new,li2015butterfly}. However, when these submatrices are only approximately low-rank (with a positive best low-rank approximation error), existing butterfly factorization algorithms {\em are not} guaranteed to provide an approximation $\hat{\mathbf{A}}$ with the \emph{best} approximation error. 
% This leads to the question of controlling the approximation error $\| \mathbf{A} - \hat{\mathbf{A}} \|_F$, in order to provide theoretical guarantees to these algorithms.
To the best of our knowledge, the %The 
only existing error bound in the literature  is based on a butterfly algorithm that performs successive low-rank approximation of blocks $\mathbf{M}$ \cite{liu2021butterfly}. However, this bound does not compare   the approximation error $\| \mathbf{A} - \hat{\mathbf{A}} \|_F$ to the \emph{best} approximation error, that is, the minimal error $\| \mathbf{A} - \mathbf{A}^* \|_F$ with $\mathbf{A}^*$ satisfying \emph{exactly} the complementary low-rank property. Moreover, in contrast to our algorithm, the algorithm proposed in \cite{liu2021butterfly}   is not designed for butterfly factorization problems with a \emph{fixed} architecture. In \cite{liu2021butterfly} the architecture is the result of the stopping criterion that imposes a given accuracy on the low-rank approximations of the blocks.  We discuss this further in   \Cref{sec:comparison}.  
% \LZc{ce paragraphe est un copier coller de l'ancienne section contributions} To our knowledge, this is the first error bound for butterfly factorization that compares the approximation error to the minimal approximation error. 
% \RG{Et nous \ldots}
In this paper, we thus propose %, to the best of our knowledge, 
the first error bound for butterfly factorization that compares the approximation error to the minimal approximation error, cf.~\eqref{eq:error-bound-intro}.

\section{Preliminaries} % on two-factor fixed-support matrix factorization} 
\label{section:preliminaries}

% \mRG{two-factors $\to$ two-factor (grammaire anglaise); fixed-support (utilisé comme un adjectif ici) \TL{Done}}
{Following the hierarchical approach \cite{le2016flexible,le2022fastlearning,zheng2023butterfly}, our analysis of the butterfly factorization problem \eqref{eq:butterfly-approximation-pb} with \emph{multiple} factors in general ($L \geq 2$) relies on the analysis of the simplest setting with only $L=2$ factors. This setting is studied in \cite{le2022spurious} and after setting up our general notations we recall some important results that will be used in the rest of the paper.}

\subsection{Notations} \label{sec:notations}
%First, we setup our notations.
%{\bf \em Notations.} 
The set $\intset{a,b}$ is the set of integers $\{a, a+1, \ldots, b \}$ for $a \leq b$, and
$\intset{a} := \intset{1, a}$. The notation $a \mid b$ means that $a$ divides $b$. 
$A \times B$ is the Cartesian product of two sets $A$ and $B$. $|A|$ is the cardinal of a set $A$.
By abuse of notation, for any matrix $\mat{X}$ and any binary matrix $\mat{S}$, the support constraint $\supp(\mat{X}) \subseteq \supp(\mat{S})$ is simply written as $\supp(\mat{X}) \subseteq \mat{S}$. 
$\mat{X}[i, :]$ and $\mat{X}[:, j]$ are the $i$-th row and the $j$-th column of $\mat{X}$, respectively. $\mat{X}[i,j]$ is the entry of $\mat{X}$ at the $i$-th row and $j$-th column. $\mat{X}[I, :]$ and $\mat{X}[:, J]$ are the submatrices of $\mat{X}$ restricted to a subset of row indices $I$ and a subset of column indices $J$, respectively. $\mat{X}[I, J]$ is the submatrix of $\mat{X}$ restricted to both $I$ and $J$. $\transpose{\mat{X}}$ is the transposed matrix of $\mat{X}$. $\mbf{0}_{m \times n}$ 
%\todo{peut-on enlever les couleurs sur cette page? Oui, mettre $\epsilon_q$.} 
(resp.~$\mbf{1}_{m \times n}$) is the $m \times n$ matrix full of zeros (resp.~of ones). The indicator (column) vector of a subset $R \subseteq \intset{m}$ is denoted $\mbf{1}_R$. The rank of a matrix $\mat{M}$ is denoted $\rank(\mat{M})$.
Finally, for any matrix $\mat{X}$, we denote $\rankprojerrsq{r}{\mat{X}} := \min_{\mat{Y}:\  \rank(\mat{Y}) \leq r} \|\mat{X}-\mat{Y}\|_F^2$, and $\prank_r(\mat{X})$ is defined as the collection of all $\mat{Y}$ of rank at most $r$ achieving the minimum. All these matrices have the same Frobenius norm, denoted by $\|\prank_r(\mat{X})\|_F$.
%\todo{Léon: isn't better to introduce these notations at the right place? I am afraid that the reader will forget these notations by the time we really use them (check this after I finish proof reading the manuscript)}
%\todo{Notations to be reminded the first time they are actually used}

% Since our problem of interest \eqref{eq:butterfly-approximation-pb} is actually a matrix factorization problem with fixed-support constraints on multiple factors, it is natural to {first} consider the simplest setting with only two factors, which is studied in \cite{le2022spurious}. Therefore we introduce some important results {from} \cite{le2022spurious} that are required for the description and analysis of our proposed hierarchical algorithm for butterfly factorization
%   (cf.~\Cref{algo:modifedbutterflyalgo}).
\subsection{Two-factor, fixed-support matrix factorization}
{Given two binary matrices $\mbf{L}, \mbf{R}$, }the problem of \emph{fixed-support matrix factorization} (FSMF) with two factors is formulated as:
\begin{equation}
	\label{eq:FSMF}
	% \tag{FSMF}
	\underset{(\bX, \bY)}{\inf} \quad \| {\bA} - \bX \bY\|_F^2, 
	\; \text{with } \supp(\bX) \subseteq \bL, \supp(\bY) \subseteq \bR.
\end{equation}
{While Problem \eqref{eq:FSMF} is NP-hard\footnote{{and  does not always admit an optimum: the infimum may not be achieved \cite[Remark A.1]{le2022spurious}.}} in general \cite[Theorem 2.4]{le2022spurious}, it becomes tractable under certain conditions on $(\bL,\bR)$}. 
To describe one of these conditions, we recall the following definitions.

\begin{definition}[Rank-one contribution supports \cite{le2022spurious,zheng2023efficient}]
	\label{def:rankonesupportcontribution}
	The \emph{rank-one contribution supports} of two binary matrices $\mbf{L} \in \{0,1\}^{m \times r}, \mbf{R} \in \{0,1\}^{r \times n}$  is the tuple {$\varphi(\mbf{L}, \mbf{R})$} of $r$ binary matrices defined by: 
	\begin{equation*}
		\varphi(\mbf{L}, \mbf{R}) := (\rankonesuppseq{i})_{i=1}^r, \quad \text{ where }\quad \bU_i:= \mbf{L}[\col{i}]\mbf{R}[\row{i}] \in \{0,1\}^{m \times n}.
	\end{equation*}
\end{definition}
\Cref{fig:rankonecontribution} illustrates the notion of \emph{rank-one contribution supports} in  \Cref{def:rankonesupportcontribution}.
\begin{remark}
    % Intuitively, 
    The binary matrix $\mbf{L}[\col{i}]\mbf{R}[\row{i}]$ for $i \in \intset{r}$ encodes the \emph{support constraint} of  $\bX[\col{i}]\bY[\row{i}]$ for each $(\mat{X}, \mat{Y})$ such that $\supp(\mat{X}) \subseteq \mat{L}$, $\supp(\mat{Y}) \subseteq \mat{R}$.
\end{remark}
\begin{figure}[h]
    \centering
    \includegraphics[width=\textwidth]{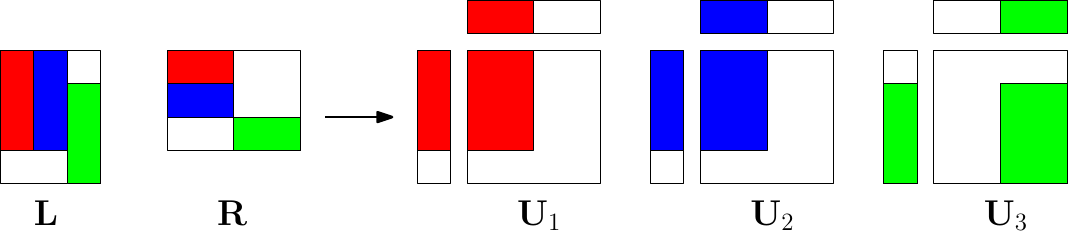}
    \caption{An example of support constraints $(\bL,\bR)$ and the supports of the corresponding rank-one contributions. Colored parts indicate indices inside the support constraints $\bL, \bR$ and $\bU_i$ for $i \in \intset{3}$. $\{1,2\}$ and $\{3\}$ are the two equivalence classes (\Cref{def:classequivalence}).}
    \label{fig:rankonecontribution}
\end{figure}
% For any $(\bX, \bY)$ such that $\supp(\bX) \subseteq \bL, \supp(\bY) \subseteq \bR$, we have:
% \begin{equation}
% 	\bX\bY = \sum_{i = 1}^n \bX[\col{i}]\bY[\row{i}]
% \end{equation}
% Thus, the binary matrix $\mbf{L}[\col{i}]\mbf{R}[\row{i}]$ encodes the support constraint on $\bX[\col{i}]\bY[\row{i}]$ - the $i$th rank-one contribution of $\bX\bY$. Explicitly, $\mbf{L}[\col{i}]\mbf{R}[\row{i}]$ is the binary matrix representation of $\supp(\mbf{L}[\col{i}]) \times \supp(\mbf{R}[\row{i}]) \subseteq \intset{m} \times \intset{n}$ ($m,n$ as in \Cref{def:rankonesupportcontribution}).
% The following presents the conditions where \eqref{eq:FSMF} is tractable.

The rank-one supports $( \rankonesuppseq{i} )_{i=1}^r$ defines an equivalence relation and its induced equivalence classes on the set of indices $\intset{r}$, as illustrated in \Cref{fig:rankonecontribution}.

\begin{definition}[Equivalence classes of rank-one supports, representative rank-one supports \cite{le2022spurious}] 
\label{def:classequivalence}
Given $\mbf{L} \in \{0,1\}^{m \times r}, \mbf{R} \in \{0,1\}^{r \times n}$, denoting $(\rankonesuppseq{i})_{i=1}^r {=} \varphi(\mbf{L}, \mbf{R})$, define an equivalence relation on the index set $\intset{r}$ of the rows of $\bL$ / columns of $\bR$ as: 
\begin{equation*}
    i \sim j \iff \rankonesuppseq{i} = \rankonesuppseq{j}.
\end{equation*}
This yields a partition of the index set $\intset{r}$ into equivalence classes, denoted $\mathcal{P}(\mat{L}, \mat{R})$. 
For each $P \in \mathcal{P}(\mat{L}, \mat{R})$, denote $\rankonesuppseq{P}$ a
representative rank-one support, $R_P \subseteq \intset{m}$ and $C_P \subseteq \intset{n}$ the supports of rows and columns in $\rankonesuppseq{P}$, respectively, i.e., $\supp(\mat{U}_P) = R_P \times C_P$, and denote $|P|$ the cardinal of the equivalence class $P$.
\end{definition}

% As an example for \Cref{def:classequivalence}, in \Cref{fig:rankonecontribution}, there are two equivalence classes: $\{1,2\}$ and $\{3\}$ since $\bU_1 = \bU_2$ and $\bU_2 \neq \bU_3$. In the following, 
We now recall a sufficient condition on the binary support matrices $(\bL,\bR)$ for which corresponding instances of Problem \eqref{eq:FSMF} 
% are polynomially tractable and 
can be solved in polynomial time via \Cref{algo:algorithm1}.

% . The polynomial algorithm for these tractable instances is presented in \Cref{algo:algorithm1}.  %\ERc{introduce the algorithm?}
\begin{theorem}[Tractable support constraints of Problem \eqref{eq:FSMF} {\cite[Theorem 3.3]{le2022spurious}}]
	\label{theorem:tractablefsmf}
	%\todo{peut-on enlever les couleurs sur cette page? Oui} 
If all components $\mat{U}_i$ %elements 
 of $\varphi(\bL,\bR)$ are pairwise disjoint or identical, then \Cref{algo:algorithm1} yields an optimal solution of Problem \eqref{eq:FSMF}, % \RG{\sout{in polynomial time}}. 
    and the infimum of Problem \eqref{eq:FSMF} is :
	\begin{equation}
    \label{eq:explicit-formula-fsmf}
		\inf_{\supp(\mat{X}) \subseteq \mat{L}, \supp(\mat{Y}) \subseteq \mat{R}} 
		\| \mat{A} - \mat{X} \mat{Y} \|_F^2  
		= \sum_{P \in \cP(\bL,\bR)}
		\min_{\bB, \rank(\bB) \leq |P|}\|\bA[R_P,C_P] - \bB\|_F^2 + c,
        %\|\bA\|_F^2 - \sum_{P \in \mathcal{P}(\mat{L}, \mat{R})} \sum_{j = 1}^{[P|} \sigma_j^2 (\bA[R_P, C_P]).
	\end{equation}
    where\footnote{Note that $\bL\bR$ is a product of two binary matrices.} $c := \displaystyle\sum_{(i,j) \notin \supp(\bL\bR)} \bA[i,j]^2$ is a constant depending only on $(\bA,\bL,\bR)$.
\end{theorem}

Equation \eqref{eq:explicit-formula-fsmf} was not proved in \cite{le2022spurious}, so we provide a complete proof of \Cref{theorem:tractablefsmf} in \Cref{appendix:proofpreliminaries}. 
The main idea is the following.

% A complete proof of \Cref{theorem:tractablefsmf} can be found in \Cref{appendix:proofpreliminaries}. We re-do the proof %in
% {from} \cite{le2022spurious} 
% % for two reasons: {{a)} 
% to make this work self-contained, and to prove the explicit formula for the infimum value of instances of \eqref{eq:FSMF}, which is not available in \cite{le2022spurious}. 

\begin{proof}[Sketch of proof]
    For $(\bX,\bY)$ such that $\supp(\bX) \subseteq \bL$ and $\supp(\bY) \subseteq \bR$:
    \begin{equation}
        \label{eq:lossdecomposition}
        \|\bA - \bX\bY\|_F^2 = \sum_{P \in \cP(\bL,\bR)}\|\bA[R_P,C_P] - \bX[R_P,P]\bY[P,C_P]\|_F^2 + c,
    \end{equation}
    because the fact that the components $\mat{U}_i$ of $\varphi(\mat{L}, \mat{R})$ are pairwise disjoint or identical implies that the blocks of indices $R_P \times C_P$ are pairwise disjoint. %, by assumption. 
    Thus, minimizing the left-hand-side is equivalent to minimizing each summand in the {right-hand side}, which is equivalent to finding the best rank-$|P|$ approximation of the matrix $\bA[R_P,C_P]$ for each $P \in \cP(\bL,\bR)$.
\end{proof}
% The proof's main idea is the following:
% if {the} conditions of \Cref{theorem:tractablefsmf} are satisfied, then for each $(\bX,\bY)$ such that $\supp(\bX) \subseteq \bL$ and $\supp(\bY) \subseteq \bR$, the function $\|\bA - \bX\bY\|_F^2$ can be decomposed {using a decomposition of matrix $\mat{A}$ into disjoint blocks} as:
% \begin{equation}
%     \label{eq:lossdecomposition}
%     \|\bA - \bX\bY\|_F^2 = \sum_{P \in \cP(\bL,\bR)}\|\bA[R_P,C_P] - \bX[R_P,P]\bY[P,C_P]\|_F^2 + c,
% \end{equation}
% where $c$ is defined as in \Cref{theorem:tractablefsmf}. 
% Thus, the algorithm yielding the optimal solutions $(\bX,\bY)$ consists in finding the pairs $(\bX[R_P,P],\bY[P,C_P])$ for each $P \in \cP(\bL,\bR)$ that minimize each summand in the {right-hand side} (RHS) of \eqref{eq:lossdecomposition};
% % . Note that 
% {and} minimizing each term in the RHS is equivalent to finding the best rank-$|P|$ approximation of the matrix $\bA[R_P,C_P]$, {which is the goal of} line~\ref{line:svd} of \Cref{algo:algorithm1}. {This can be achieved for instance} via truncated singular value decomposition (SVD). 

\begin{remark}
    \label{rem:nonuniquesolution}
     Best low-rank approximation in line~\ref{line:svd} of \Cref{algo:algorithm1} can be computed via truncated singular value decomposition (SVD).
    {Note that the definition of $\hat{\mat{H}}, \hat{\mat{K}}$ in this line is not unique, because, for instance, the product $\hat{\mat{H}} \hat{\mat{K}}$ is invariant to some rescaling of columns and rows.}
\end{remark}
% That also explains the infimum value in \Cref{theorem:tractablefsmf}.

\begin{algorithm}[t]
	\centering
	\caption{Two-factor fixed-support matrix factorization} 
	\label{algo:algorithm1}
	\begin{algorithmic}[1]
        \REQUIRE $\bA \in \CC^{m \times n}$, $\bL \in \{0, 1 \}^{m \times r}$, $\bR  \in \{0, 1 \}^{r \times n}$
        \ENSURE $(\bX,\bY)$ such that $\supp(\mat{X}) \subseteq \mat{L}$, $\supp(\mat{Y}) \subseteq \mat{R}$
		\STATE $(\mat{X}, \mat{Y}) \gets (\mbf{0}_{m \times r}, \mat{0}_{r \times n})$
		% \STATE $\bY \gets \mbf{0}$.
        \FOR {$P \in \mathcal{P}(\mat{L}, \mat{R})$ (cf.~\Cref{def:classequivalence})}
		      \STATE {$(\bX[R_P,P], \bY[P,C_P]) \gets$ $(\hat{\bH}, \hat{\bK}) \in \underset{\substack{\mat{H} \in \CC^{|R_P| \times |P|} \\ \mat{K} \in \CC^{|P| \times |C_P|}}}{\arg\min} \| \bA[R_P,C_P] - \mat{H} \mat{K} \|_F$}{\label{line:svd}}
		\ENDFOR
		\RETURN $(\bX,\bY)$
	\end{algorithmic}
\end{algorithm}

% \mLZ{I moved the complexity discussion to appendix}

\section{Deformable butterfly factorization}
\label{sec:DB-factorization}

% \RG{Un point de vocabulaire: que diriez-vous de la terminologie suivante ?
% Call $\boldsymbol{\pi} = (a,b,c,d)$ a {\em (deformable butterfly) factor architecture}, and define a $\boldsymbol{\pi}$-factor (au lieu de $\theta$-DB factor) to be a matrix satisfying the corresponding support constraint \ldots Call $\boldsymbol{\beta} = (\boldsymbol{\pi}_\ell)_{\ell=1}^L$ a (deformable butterfly) architecture.
% }
% 
% 
This section presents a mathematical formulation of the deformable butterfly {factorization} \cite{lin2021deformable} associated with a sequence of patterns $\arch := (\pattern_\ell)_{\ell=1}^L$ called an architecture.
% , a definition that encompasses other existing butterfly factorizations as discussed in \Cref{sec:intro}. 
% \mLZ{I commented the sentences below for the sake of conciseness}
% It is parameterized by . 
We then introduce the notions of \emph{chainability} and \emph{non-redundancy} of an architecture, that are crucial conditions for constructing a butterfly algorithm for Problem \eqref{eq:butterfly-approximation-pb} with error guarantees.

% This section presents a concise \mRG{est-ce seulement plus concis ou aussi plus général?\TL{seulement plus concis. Notre notion de chainabilité est plus général que le deformable butterfly original.}} mathematical formulation of the deformable butterfly {factorization} \cite{lin2021deformable}, a definition that encompasses other existing butterfly factorizations such as {the} square dyadic butterfly factorization \cite{dao2019learning,dao2020kaleidoscope,le2022fastlearning,zheng2023efficient}, {or} {the} Monarch factorization \cite{dao2022monarch}. 

% {After defining the allowed patterns for each factor, we}
% introduce the notion of \emph{chainability} {of such patterns} (cf.~\Cref{def:chainableDBfour,def:chainableseqDBPs}), which will be shown to be a sufficient condition that allows {problem} \eqref{eq:butterfly-approximation-pb} to admit an approximate algorithm, {i.e., an algorithm yielding factors granting an approximation error comparable to the minimal approximation error}, cf.~\Cref{eq:approximationratio}. %\ERc{ref to an eq.? }

\subsection{A mathematical formulation for \ksfs{}}

% As mentioned in \Cref{sec:intro}, existing 
Many butterfly factorizations \cite{dao2019learning,dao2020kaleidoscope,le2022fastlearning,zheng2023efficient,dao2022monarch,lin2021deformable} take the form $\mat{A} = \matseq{X}{1} \ldots \matseq{X}{L}$ with 
% $L \geq 2$ is a given number of factors, and 
% each factor $\matseq{X}{\ell}$ for $\ell \in \intset{L}$ satisfies the fixed-support constraint 
$\supp(\matseq{X}{\ell}) \subseteq \mat{I}_{a_\ell} \otimes \mat{1}_{b_\ell \times c_\ell} \otimes \mat{I}_{d_\ell}$ for $\ell \in \intset{L}$, for some parameters $(a_\ell, b_\ell, c_\ell, d_\ell)_{\ell=1}^L$, cf.~\Cref{sec:intro}. We therefore introduce the following definition.

\begin{definition}[\ksfs{} and their sparsity patterns]
	\label{def:dbfactorfour}
For $a, b, c, d \in \NN$, a \emph{\ksf{}} 
{of pattern} $\boldsymbol{\pi} := \paramfour$ (or {$\boldsymbol{\pi}$-factor}) is a matrix in $\RR^{m \times n}$ or $\CC^{m \times n}$,
%matrix of size $m \times n$ 
where $m: = abd$, $n: = acd$, such that
% \mLZ{je remplace factor architecture par pattern. J'enlève aussi deformable.}
%\begin{enumerate}
%    \item its size is $m \times n$ where $m: = abd, n: = acd$;
%    \item
    its support is included in $\DBsupport{\pattern} := \mbf{I}_{a} \otimes \mbf{1}_{b \times c} \otimes \mbf{I}_{d} \in \{0,1\}^{m \times n}$.
%\end{enumerate}
The tuple $\boldsymbol{\pi}$ will be called an \emph{elementary 
% deformable butterfly
Kronecker-sparse pattern}, or simply a \emph{pattern}. The set of all $\boldsymbol{\pi}$-factors is denoted by $\factarch^{\pattern}$. 
%\RG{Such matrices are of size $m \times n$ where $m: = abd, n: = acd$.}
\end{definition} 

% \begin{remark}
%     \Cref{def:DBmatrix} formalizes the description of the supports of a deformable \ksf{} \cite{lin2021deformable}. Indeed, the original paper \cite{lin2021deformable} did not formalize mathematically the description of these supports as a Kronecker product, which is the purpose of \Cref{def:DBmatrix} here. This is one of our novelties and it allows us to manipulate the operations involved with DB factors much easier. Moreover, in the original paper \cite{lin2021deformable}, the support constraint on the butterfly factors is parameterized with 5 parameters, with 1 redundant parameter, while \Cref{def:DBmatrix} only uses 4 parameters.    
% \end{remark}

\Cref{fig:DBfactorillu} illustrates the support of a $\pattern$-factor, for a given pattern $\pattern = (a, b, c, d)$. A $\pattern$-factor matrix is block diagonal with $a$ blocks in total. By definition, each block in the diagonal has support included in $\mbf{1}_{b \times c} \otimes \bI_d$. 
\revision{%Thus, e
Each block is a \emph{block matrix} partitioned into $b \times c$ \emph{sub-blocks}, and every %one of these 
\emph{sub-block} is a diagonal matrix of fixed dimensions $d \times d$.}
% Thus, each block is a \emph{block matrix} \revision{with $b$ sub-blocks per row and $c$ sub-blocks per column,}
% of size $b \times c$, 
% where each \emph{sub-block} is a diagonal matrix of size $d \times d$. 
% Readers might want to view \Cref{fig:DBfactorillu} for a better illustration of our description. 
%\TLc{Move this to appendix} For the remainder of the paper, the term \emph{block} refers to the large block matrix on the diagonal, while the term \emph{sub-block} refers to the small diagonal blocks inside each block.
% \LZc{Est-ce qu'on parle de block et sub-block plus bas? Sinon on peut enlever, la figure suffit.}

\begin{example}
\label{example:typicalDBfactors}
    The following matrices are $\pattern$-factors for certain choices of $\pattern$.
	\begin{enumerate}
		\item {\bf Dense matrix}: Any matrix of size $m \times n$ is a $(1,m,n,1)$-factor.
		\item {\bf Diagonal matrix}:  Any diagonal matrix of size $m \times m$ is either a $(m,1,1,1)$-factor or $(1, 1, 1, m)$-factor.
	\end{enumerate}
    
    \begin{enumerate}
        \setcounter{enumi}{2}
        \item {\bf Factors in a square dyadic butterfly factorization} \cite{dao2019learning,dao2020kaleidoscope,le2022fastlearning,zheng2023efficient}: the pattern of the $\ell$-th factor is $\pattern_\ell = (2^{\ell - 1}, 2, 2, 2^{L - \ell})$ for $\ell \in \intset{L}$.
        \item {\bf Factors in a Monarch factorization} \cite{dao2022monarch}: the patterns of the two factors are $\pattern_1 = (1, p, q, m/p)$, $\pattern_2 = (q, m/p, n/q, 1)$ for
        %some integers
        {$p,q$ such that} $p \mid m$ and $r \mid n$.
    \end{enumerate}
\end{example}

\begin{lemma}[{Sparsity level} of {a $\boldsymbol{\pi}$-factor}]
	\label{lemma:numberofparameters}
	For ${\boldsymbol{\pi}} = (a,b,c,d)$, the number of nonzero entries of a {$\boldsymbol{\pi}$}-factor of size $m \times n$ is at most $\|\pattern\|_0:= abcd = mc = nb$.
\end{lemma}

\begin{remark}
\label{rem:abuse-notation}
    \revision{With an abuse of notation we use the shorthand $\|\pattern\|_0 := \| \text{vec}(\mathbf{S}_\pattern) \|_0$, where $\text{vec}(\cdot)$ is the vectorization operator than concatenates the columns of a matrix into a vector, and $\| \cdot \|_0$ is the $\ell_0$-norm of a vector (the number of its nonzero entries).}
\end{remark}

\begin{proof}
    {The cardinal of $\supp(\mbf{I}_{a} \otimes \mbf{1}_{b \times c} \otimes \mbf{I}_{d})$} is $abcd = mc = nb = \frac{mn}{ad}$.
\end{proof}

{A $\pattern$-factor is sparse if it has few nonzero entries compared to its size, i.e., if $\|\boldsymbol{\pi}\|_0 \ll mn$, or equivalently\footnote{\revision{The number of nonzero entries in a $\pattern$-Kronecker sparse factor of size $m \times n$ with $\pattern = (a,b,c,d)$ is at most $abcd = mn / ad$, because $m=abd$ and $n=acd$. Therefore, the sparsity level is $\frac{mn / ad}{mn} = \frac{1}{ad}$.}} if} $ad \gg \mathcal{O}(1)$. Given a number of factors $L \geq 1$, a sequence of patterns $\arch := (\pattern_\ell)_{\ell=1}^L$ parameterizes the set 
\begin{equation}
    \label{eq:tuple-of-butterfly-factors}
    \setDBfactor{\arch} := \setDBfactor{\pattern_1} \times \ldots \times \setDBfactor{\pattern_L}
\end{equation}
of $L$-tuples of $\pattern_\ell$-factors, $\ell = 1, \ldots, L$.
%  \ERc{do we really need this definition?}\RG{Does not seem necessary, if needed the shorthand can be locally introduced in the few places where it would help ?} \LZc{c'est quand même utile, comme dans la phrase avant \Cref{lem:suppdbfactorprod}: a sequence of matrices $(\matseq{\mat{X}}{\ell})_{\ell=1}^L$ associated with a chainable $\arch$. Ou alors on peut dire: $(\matseq{\mat{X}}{\ell})_{\ell=1}^L \in \setDBfactor{\arch}$ with chainable $\arch$ et on enlève la définition suivante.}
% \begin{definition}
% \label{def:associatingseq}
%     Given a sequence of patterns $\arch = (\pattern_\ell)_{\ell=1}^L$, we say that a sequence of matrices $(\matseq{X}{\ell})_{\ell=1}^L$ is \emph{associated} to $\arch$ if $(\matseq{X}{\ell})_{\ell=1}^L \in \setDBfactor{\arch}$, i.e., $\matseq{X}{\ell}$ is a $\pattern_\ell$-factor for each $\ell \in \intset{L}$.
% \end{definition}
% \mLZ{J'ai enlevé la définition de ``$(\matseq{\mat{X}}{\ell})_{\ell=1}^L$ associated with $\arch$'', car on peut simplement écrire par la suite $(\matseq{\mat{X}}{\ell})_{\ell=1}^L \in \setDBfactor{\arch}$}
Since we are interested in matrix products $\matseq{X}{1} \ldots \matseq{X}{L}$ for $(\matseq{X}{\ell})_{\ell=1}^L \in \setDBfactor{\arch}$, we will only consider {sequences of patterns} $\arch$ such that the size of $\matseq{X}{\ell} \in \setDBfactor{\pattern_\ell}$ and $\matseq{X}{\ell + 1} \in \setDBfactor{\pattern_{\ell+1}}$ are compatible for computing the matrix product $\matseq{X}{\ell} \matseq{X}{\ell + 1}$, for each $\ell \in \intset{L - 1}$. In other words, we require that the sequence of patterns $\arch$ satisfies:
\begin{equation}
	\label{eq:condition-size-chainable}
	\forall \ell \in \intset{L-1}, \quad \underbrace{a_\ell c_\ell d_\ell}_{n_\ell} = \underbrace{a_{\ell + 1} b_{\ell + 1} d_{\ell + 1}}_{m_{\ell + 1}}.
\end{equation}
Therefore, under assumption \eqref{eq:condition-size-chainable}, a sequence $\arch$ can describe a factorization of the type $\mat{A} = \matseq{X}{1} \ldots \matseq{X}{L}$ such that $(\matseq{X}{\ell})_{\ell=1}^L \in \setDBfactor{\arch}$. We introduce the following terminology for such a sequence.

% \begin{definition}[{(Deformable) butterfly architecture}]
\begin{definition}[Butterfly architecture {and butterfly matrices}]
\label{def:dbarchitecture}
% A sequence $\boldsymbol{\beta} := (\boldsymbol{\pi}_\ell)_{\ell=1}^L$ of \ksf{}architectures $\boldsymbol{\pi}_\ell$ satisfying \eqref{eq:condition-size-chainable} 
A sequence of patterns $\arch := (\boldsymbol{\pi}_\ell)_{\ell=1}^L$ 
is called a {\em (deformable) butterfly architecture, or simply an \emph{architecture}, when it satisfies  \eqref{eq:condition-size-chainable}.}
By analogy with deep networks, the number of factors is called the depth of the chain and denoted by  $|\arch|:= L$ 
% \mRG{A propopager}
and, using the notation
 $\|\pattern\|_0$ from \Cref{lemma:numberofparameters}, the number of parameters is denoted by
\begin{equation*}
    \|\arch\|_0 := \sum_{\ell = 1}^L \|\pattern_\ell\|_0.
\end{equation*}
\begin{remark}
    \revision{As in \Cref{rem:abuse-notation}, the shorthand $\|\arch\|_0$ is an abuse of notation.}
    %for $\sum_{\ell=1}^L \| \text{vec} (\mathbf{S}_\pattern) \|_0$.}
\end{remark}

% \mLZ{une architecture n'est pas juste une séquence de patterns, il faut que les tailles soient compatibles. En effet c'est logique si on pense au graphe d'un réseau de neurones multicouches.}
\end{definition}
For any  architecture $\arch$, $\setButterfly{\arch}$ is the set of \emph{(deformable) butterfly matrices} associated with $\arch$, as defined in \eqref{eq:butterfly-factorization-intro}. We also say that any {$\mat{A}  \in \setButterfly{\arch}$} admits an \emph{exact (deformable) butterfly factorization} associated with the architecture $\arch$.
\Cref{tab:butterfly-parametrization} describes existing architectures fitting our framework.

The rest of this section introduces two important properties of an architecture $\arch$:
\begin{itemize}
    \item \emph{Chainability} will be shown (\Cref{cor:optimum-exists-in-bf})
    to ensure the existence of an optimum in \eqref{eq:butterfly-approximation-pb}, so that we can replace ``$\inf$''  by ``$\min$'' in \eqref{eq:butterfly-approximation-pb}. 
    We also show that, for any chainable architecture, one can exploit
    a hierarchical algorithm (\Cref{algo:recursivehierarchicalalgo}) that
    extends an algorithm from \cite{le2022fastlearning,zheng2023butterfly} to compute 
    an approximate solution to Problem \eqref{eq:butterfly-approximation-pb}.
    % for any chainable architecture.
    \item \emph{Non-redundancy} is an additional property satisfied by some
    chainable architectures $\arch$, that allows us to insert orthonormalization steps 
    in the hierarchical algorithm, in order to control the approximation error for Problem \eqref{eq:butterfly-approximation-pb} in the sense of \eqref{eq:error-bound-intro}.
    Non-redundancy %seems as 
    {plays the role of} an intermediate tool to design and analyze our algorithms. However, it should not be treated as an additional hypothesis, because we do propose a factorization method (cf.~%\Cref{algo:algoforanychainable}), 
    \Cref{rem:redundantcase}),
    endowed with error guarantees, for {\em any chainable architecture, whether  redundant or not}.
\end{itemize}
Both conditions are first defined for the most basic architectures $\arch$
of depth $| \arch | = 2$, before being generalized to 
architectures $\arch$ of arbitrary depth $L \geq 2$.

\subsection{Chainability}

% One main contribution of this paper is to exhibit a (sufficient) general condition on the architecture 
% % \mRG{to be propagated in the rest of the paper \TL{Fait}} 
% $\arch$ under which the corresponding butterfly factorization problem {\eqref{eq:butterfly-approximation-pb}} admits a hierarchical algorithm endowed with exact recovery guarantees when $\mat{A} \in \mathcal{B}^\arch$. As a novelty, we will also prove that the proposed hierarchical algorithm is also endowed with guarantees on the control of {the} approximation error, in the sense that it yields factors $(\matseq{X}{\ell})_{\ell=1}^L \in \setDBfactor{\arch}$ satisfying:
% \begin{equation}
%     \label{eq:approximationratio}
%     \| \mat{A} - \matseq{X}{1} \ldots \matseq{X}{L} \|_F \leq C_\arch E^\arch(\mbf{A}),
% \end{equation}
% where $C_\arch$ is a constant only depending on {the architecture} $\arch$. {Specifically, in our analysis, the constant $C_\arch$ depends only on $L = |\arch|$ (cf.~\Cref{def:dbarchitecture}).
%\mLZ{compression (ancien texte en commentaire latex)}
% \mRG{A-t-on formellement introduit la notation $|\arch|$ auparavant ? A faire dans une def existante ?\TL{Fait}} 
% the number of factors of $\arch$}.

We start by defining this condition {in} the case of architectures of depth $L=2$. This definition is primarily introduced to ensure a key ``stability'' property given next in \Cref{prop:stability}. \revision{This property is a closedness condition with respect to matrix multiplication, which guarantees that the product of two Kronecker-sparse factors is also a Kronecker-sparse factor. Such a property will have many nice consequences for our analysis. We also introduce the operator $*$ between patterns $(\pattern_1, \pattern_2)$, to describe the sparsity pattern of the product between a $\pattern_1$-Kronecker sparse factor and a $\pattern_2$-Kronecker sparse factor.}
\begin{definition}[Chainable pair of patterns, operator {$*$} on patterns]
    \label{def:chainableDBfour}
 Two patterns 
 $\pattern_1 := (a_1, b_1, c_1, d_1)$ and $\pattern_2:= (a_2, b_2, c_2, d_2)$ are
 \emph{chainable} if:
	\begin{enumerate}
        \item 
        %There exists (a unique) $q \in \NN$ such that 
        $\frac{a_1c_1}{a_2} = \frac{b_2d_2}{d_1}$ and this quantity\footnote{As we will see, it plays the role of a rank, hence the choice of $r$ to denote it.}, denoted $r(\pattern_1,\pattern_2)$, is an integer;
        %\todo{Une suggestion pénible: partout $q$ jouera le rôle d'un rang. Utiliser $r$ à la place? Oui, décider quand le faire en une passe méthodique. "As we will see [here] this plays the role of a rank"}
		\item $a_1 \mid a_2$ and $d_2 \mid d_1$.
	\end{enumerate}
We also say that the pair $(\pattern_1,\pattern_2)$ is chainable.
Observe that we always have $r(\pattern_1,\pattern_2) = c_1(a_1/a_2) = b_2(d_2/d_1) \leq \min(b_2,c_1)$, and $a_1c_1d_1=a_2b_2d_2$.
    We 
    %denote $r(\pattern_1, \pattern_2) = q$ and 
    define the operator $*$ on the set of chainable pairs of patterns as follows: if $(\pattern_1,\pattern_2)$ is chainable, then 
    \begin{equation}
    \label{eq:operatorDBPfour}
        \pattern_1 * \pattern_2 := \left(a_1, \frac{b_1d_1}{d_2}, \frac{a_2c_2}{a_1}, d_2\right) \in \mathbb{N}^4.
    \end{equation}
\end{definition}
Note that even though the definition of $r(\pattern_1,\pattern_2)$ involves the quotient $a_1/a_2$ (and $d_2/d_1$), assumption 2 in \Cref{def:chainableDBfour} is indeed that $a_1$ divides $a_2$ (and $d_2$ divides $d_1$).
\begin{remark}
    The order $(\pattern_1, \pattern_2)$ in the definition matters, {\em i.e.},  this property is not symmetric: the chainability of $(\pattern_1, \pattern_2)$ does not imply that of $(\pattern_2, \pattern_1)$. Moreover, by the first condition of \Cref{def:chainableDBfour}, a chainable pair is indeed an architecture in the sense of \Cref{def:dbarchitecture}.
\end{remark}

% \begin{remark}
    % The equality $a_1c_1/a_2 = b_2d_2/{d_1} \in \NN$ in the first condition implies $a_1c_1d_1 = a_2b_2d_2$, which is the condition \eqref{eq:condition-size-chainable}. Therefore, a chainable pair of patterns $(\pattern_1, \pattern_2)$ is an architecture in the sense of \Cref{def:dbarchitecture}.
    % Therefore, we do not need to add explicitly \eqref{eq:condition-size-chainable} to guarantee the well-definedness of the matrix multiplication operator. 
% \end{remark}

\Cref{def:chainableDBfour} comes with the following two key results.
%propositions.\todo{Léon: there is only \Cref{prop:stability}, where is the second one? Should we make \Cref{lem:associativity} a proposition?}

\begin{proposition}
\label{prop:stability}
    If $(\pattern_1, \pattern_2)$ is chainable, then:
    \begin{equation}
    \label{eq:DB-support}
        \DBsupport{\pattern_1} \DBsupport{\pattern_2} = r(\pattern_1,\pattern_2) \DBsupport{\pattern_1 * \pattern_2}.
    \end{equation}
    % where $\pattern = (a_1, \frac{b_1d_1}{d_2}, \frac{a_2c_2}{a_1}, d_2)$ is a pattern ($\bS_\pattern$ is defined as in \Cref{def:dbfactorfour}). 
\end{proposition}
The proof is deferred to \Cref{app:stability-chainable-pair}. 
%\begin{remark}
    The equality \eqref{eq:DB-support} was proved in \cite[Lemma 3.4]{zheng2023efficient} for the choice $\pattern_1 = (2^{\ell - 1}, 2, 2, 2^{L - \ell})$ and $\pattern_2 = (2^{\ell}, 2, 2, 2^{L - \ell - 1})$, for any integer $L \geq 2$ and $\ell \in \intset{L-1}$. \Cref{prop:stability} extends \eqref{eq:DB-support} to \emph{all} chainable pairs $(\pattern_1, \pattern_2)$.
%\end{remark}

Chainability and \Cref{def:dbfactorfour} imply that
$\forall (\matseq{X}{1}, \matseq{X}{2}) \in \setDBfactor{\pattern_1} \times \setDBfactor{\pattern_2}, \, \matseq{X}{1} \matseq{X}{2} \in \setDBfactor{(\pattern_1 * \pattern_2)}$, i.e., a product of \ksfs{} with chainable patterns $(\pattern_1,\pattern_2)$ is still a \ksf{}, with pattern $\pattern_1 * \pattern_2$. Moreover, the matrix supports corresponding to a pair of chainable patterns also satisfy useful many interesting properties related to \cref{def:classequivalence} and \cref{theorem:tractablefsmf}, as shown in the following result proved in \cref{app:prodDBBis}:

%We can characterize the set $\{\bX_1\bX_2 \, | \,  \bX_i \in \setDBfactor{\pattern_i}, i \in \intset{2} \}$ more precisely as follows. The proof is deferred to \Cref{app:prodDBBis}.

\begin{lemma}
	\label{lem:prodDBBis}
 If $\arch := (\pattern_1, \pattern_2)$ is chainable then (with the notations of \Cref{def:classequivalence}) for each $P \in \cP(\bS_{\pattern_1}, \bS_{\pattern_2})$ we have
 \begin{enumerate}
 \item $R_P = \supp(\bS_{\pattern_1}[\col{i}])$ and $C_P = \supp(\bS_{\pattern_2}[\row{i}])$ for every $i \in P$.
    \item  The sets $R_P \times C_P$, $P \in \cP(\bS_{\pattern_1}, \bS_{\pattern_2})$ are pairwise disjoint.
    \item $|P| = r(\pattern_1,\pattern_2)$, $|R_P| = b_1$ and $|C_P| = c_2$ (with $\pattern_i = (a_i,b_i,c_i,d_i)$). 
\item $\supp(\bS_{\pattern_1 * \pattern_2}) = \supp(\bS_{\pattern_1}\bS_{\pattern_2})= \cup_{P \in \cP} R_P \times C_P$.
\end{enumerate}
%\todo{Need to complete the proof with the additional claims. Would be better to move the (4.5) claim to a corollary of Lemma 5.1 \TL{Done}}
%  Moreover, with $\setblocklowrank{\arch}$ the set of matrices of the size of matrices in $\setButterfly{\arch}$ such that $\rank\big(\bA[R_P,C_P]\big) \leq r(\pattern_1,\pattern_2)$ for each $P \in \cP(\bS_{\pattern_1}, \bS_{\pattern_2})$, we have  
%     \begin{equation}
%     \label{eq:prodDB}
%         \setButterfly{\arch} = \setDBfactor{\pattern_1 * \pattern_2} \cap 
%         \setblocklowrank{\arch}.
%     \end{equation}
\end{lemma}
\begin{lemma}[Associativity of $*$]
	\label{lem:associativity}
	If $(\pattern_1, \pattern_2)$ and $(\pattern_2, \pattern_3)$ are chainable, then
    \begin{enumerate}
        \item $(\pattern_1, \pattern_2 * \pattern_3)$ and $(\pattern_1 * \pattern_2, \pattern_3)$ are chainable;
        \item $r(\pattern_1, \pattern_2*\pattern_3)=r(\pattern_1,\pattern_2)$ and $r(\pattern_1 * \pattern_2, \pattern_3)=r(\pattern_2,\pattern_3)$; 
        \item $\pattern_1 * (\pattern_2 * \pattern_3) = (\pattern_1 * \pattern_2) * \pattern_3$.
    \end{enumerate}
\end{lemma}
The proof of \Cref{lem:associativity} is deferred to \Cref{appendix:associativityproof}.
We can now extend the definition of chainability to a general architecture $\arch$ of arbitrary depth {$L \geq 1$.} %$L \geq 2$.

% The extension of chainability to general architectures of 
% % \mRG{"\ldots sequence \ldots" $\to$ "... factorization architecture ..."} 
% length $L \geq 2$ is as follows.

% \begin{definition}[Chainable butterfly architectures]
\begin{definition}[Chainable architecture]
\label{def:chainableseqDBPs}
    An architecture $\arch := (\pattern_\ell)_{\ell=1}^L$, {$L \geq 2$,}  is \emph{chainable} if $\pattern_\ell$ and $\pattern_{\ell + 1}$ are chainable for each $\ell \in \intset{L-1}$ in the sense of \Cref{def:chainableDBfour}. We then denote $\mathbf{r}(\arch) = (r(\pattern_\ell, \pattern_{\ell+1}))_{\ell = 1}^{L-1} \in \NN^{L-1}$.
    {By convention any architecture of depth $L=1$ is also chainable.}
    \label{def:dbchain}
\end{definition}
\begin{example}
\label{example:ex_chainable}
One can check that the square dyadic butterfly architecture (resp. the Monarch architecture), cf.~\Cref{example:typicalDBfactors}, are chainable, with $\mathbf{r}(\arch) =  (1, \ldots, 1)$ (resp.~$\mathbf{r}(\arch) = (1)$). 
{They are particular cases of the $5$-parameter deformable butterfly architecture of \cite{lin2021deformable}, which is chainable with $\mathbf{r}(\arch) = (1,\ldots,1)$.
{In contrast, the Kaleidoscope architecture {of depth $2L$ with $L \geq 2$} of \Cref{tab:butterfly-parametrization} is {\em not} chainable, because for $\ell = {L+1}$ we have $\pattern_\ell = {(2^{L-1},2,2,1)}$, $\pattern_{\ell+1} = {(2^{L-2},2,2,2)}$, and this pair is not chainable since $2^{L-1}$ does not divide $2^{L-2}$.}
}
\end{example}
We state in the following some useful properties of chainable architectures.
\begin{lemma}
\label{lem:suppdbfactorprod}
    If $\arch = (\pattern_\ell)_{\ell=1}^L$ {with $L \geq 2$} is chainable then $\setButterfly{\arch} \subseteq \setDBfactor{(\pattern_1 * \ldots * \pattern_L)}$, with
    \begin{equation}
        \label{eq:producttheta}
        \pattern_1 * \ldots * \pattern_L = \left( a_1, \frac{b_1 d_1}{d_L}, \frac{a_L c_L}{a_1}, d_L \right).
    \end{equation}
\end{lemma}

\begin{remark}
  \revision{As a consequence, the number of parameters in the pattern \eqref{eq:producttheta} is at most $\|\pattern_1 * \ldots * \pattern_L \|_0 = b_1d_1a_Lc_L$ by \Cref{lemma:numberofparameters}. As a comparison, the number of parameters in the architecture $\arch = (\pattern_\ell)_{\ell=1}^L$ is at most  $\|\arch\|_0:= \sum_{\ell = 1}^L \|\pattern_\ell\|_0 = \sum_{\ell = 1}^L a_\ell b_\ell c_\ell d_\ell$, by \Cref{def:dbarchitecture}. {For example, with the square dyadic butterfly architecture $\|\beta\|_0 = \sum_{\ell=1}^L 2^{L+1} = L 2^{L+1} \ll 2^{2L}= b_1d_1a_Lc_L$ for large $L$}.
}
%   $\arch = (\pattern_\ell)_{\ell=1}^L$, then 
%   $$
%   \|\arch\|_0:= \sum_{\ell = 1}^L \|\pattern_\ell\|_0 = \sum_{\ell = 1}^L a_\ell b_\ell c_\ell d_\ell,
%   $$
%   while a direct consequence of \eqref{eq:producttheta} is that (if $\beta$ is chainable)
%   $$
%   \|\pattern_1 * \ldots * \pattern_L \|_0\leq b_1d_1a_Lc_L.
%   $$}
\end{remark}

\begin{proof}[Partial proof]
    \revision{
    \Cref{prop:stability} yields $\setButterfly{\arch} \subseteq \setDBfactor{(\pattern_1 * \ldots * \pattern_L)}$ when $L=2$. This extends to any $L \geq 2$ by an induction.
    %The result $\setButterfly{\arch} \subseteq \setDBfactor{(\pattern_1 * \ldots * \pattern_L)}$ follows from an induction on $|\arch|$ where the base case ($L = 2$)
    %\TL{can be seen as a corollary of}
    %is a direct consequence of \cref{prop:stability}. 
    We prove \eqref{eq:producttheta} in \Cref{appendix:productmultipleDBparam}.}
\end{proof}
\begin{remark}\label{rem:denseornot}
{As a consequence of this lemma, if the first pattern $\pattern_1$ of a chainable architecture $\arch$ satisfies $a_1 > 1$ then all matrices in $\setButterfly{\arch}$ have a support included in $\mathbf{S}_{\pattern_1 * \ldots * \pattern_L}$, which has zeroes outside its main block diagonal structure (see \Cref{fig:DBfactorillu}). A similar remark holds when $d_L > 1$, and in both cases we conclude that $\setButterfly{\arch}$ does not contain any dense matrix where all entries are nonzero. 
In contrast, when $a_1 = d_L = 1$, it is known for specific architectures that some dense matrices do belong to $\setButterfly{\arch}$. This is notably the case when $\arch$ is the square dyadic butterfly architecture or the Monarch architecture (see \Cref{example:typicalDBfactors}): then we have $\pattern_1 * \ldots * \pattern_L = (a_1, m, n, d_L) = (1, m, n, 1)$ for some integers $m, n$, and indeed the Hadamard (or the DFT matrix, up to bit-reversal permutation of its columns, cf.~\cite{dao2019learning}) is a dense matrix belonging to $\setButterfly{\arch}$.}
    % \RG{DISCUSION A AVOIR: si l'inclusion était stricte il se pourrait qu'on ne puisse pas avoir de matrice dense. Mais dans le cas de square dyadic on a l'exemple de Hadamard / DFT qui sont denses et dans $\setButterfly{\arch}$. Peut-on/veut-on dire qqch de général ? Un équivalent de \Cref{lem:prodDB2} ? Par ailleurs, même si on montre que $\setButterfly{\arch}$ contient en effet au moins une matrice dense, insister sur le fait que cela reste en général un ensemble bien plus restreint que celui de toutes les matrices $m \times n$, comme on le verra sans doute avec la caractérisation complementary low-rank !} \LZc{l'équivalent de \Cref{lem:prodDB2} pour une architecture chainable ce serait \Cref{cor:characterizationofDBmatrix} ou \Cref{cor:characterizationofDBmatrix}. Quand $a_1 = d_L = 1$, il n'y a plus de contrainte de support dans la condition 1 de \Cref{cor:characterizationofDBmatrix} car $\mat{S}_{\pattern_1 * \ldots * \pattern_L} = \mat{1}_{m \times n}$}
\end{remark}
% \begin{remark}
% \RG{As a consequence of this lemma, for any chainable architecture $\arch$ where the first and last patterns $\pattern_1$ and $\pattern_L$ satisfy}
%     %Consequently, if 
%     $a_1 = d_L = 1$ (e.g., the square dyadic butterfly architecture or the Monarch architecture, see \Cref{example:typicalDBfactors}), %then
%     \RG{we have} $\pattern_1 * \ldots * \pattern_L = (a_1, m, n, d_L) = (1, m, n, 1)$ for some integers $m, n$. \RG{Since $(1,m,n,1) = \mbf{1}_{m \times n}$ this} means that the product of $\pattern_\ell$-factors for $\ell \in \intset{L}$ can be a dense matrix. 
%     \RG{DISCUSION A AVOIR: si l'inclusion était stricte il se pourrait qu'on ne puisse pas avoir de matrice dense. Mais dans le cas de square dyadic on a l'exemple de Hadamard / DFT qui sont denses et dans $\setButterfly{\arch}$. Peut-on/veut-on dire qqch de général ? Un équivalent de \Cref{lem:prodDB2} ? Par ailleurs, même si on montre que $\setButterfly{\arch}$ contient en effet au moins une matrice dense, insister sur le fait que cela reste en général un ensemble bien plus restreint que celui de toutes les matrices $m \times n$, comme on le verra sans doute avec la caractérisation complementary low-rank !} \LZc{l'équivalent de \Cref{lem:prodDB2} pour une architecture chainable ce serait \Cref{cor:characterizationofDBmatrix} ou \Cref{cor:characterizationofDBmatrix}. Quand $a_1 = d_L = 1$, il n'y a plus de contrainte de support dans la condition 1 de \Cref{cor:characterizationofDBmatrix} car $\mat{S}_{\pattern_1 * \ldots * \pattern_L} = \mat{1}_{m \times n}$}
% \end{remark}

Next we state an essential property of chainable architectures. It builds on and extends \Cref{lem:associativity}, and corresponds to a form of {\em stability} under pattern multiplication that will serve as a cornerstone to support the introduction of hierarchical algorithms.
\begin{lemma}
	\label{lem:seq-qchainability}
    If $\arch=(\pattern_\ell)_{\ell=1}^L$ is chainable then for each $1 \leq q \leq s < t \leq L$, the patterns $(\pattern_{q} * \ldots * \pattern_s)$ and $(\pattern_{s+1} * \ldots * \pattern_t)$ are well-defined and chainable with $r(\pattern_{q} * \ldots * \pattern_s, \pattern_{s+1} * \ldots * \pattern_t) = r(\pattern_s, \pattern_{s+1})$. 
\end{lemma}

The proof is deferred to \Cref{appendix:seq-qchainabilityproof}.  

% of \Cref{lem:seq-qchainability} is based on \eqref{eq:producttheta} and the definition of chainability (cf.~\Cref{def:chainableDBfour,def:chainableseqDBPs}), and 

\subsection{Non-redundancy}
A first version of our proposed hierarchical factorization algorithm (expressed recursively 
% as 
% \Cref{algo:recursivehierarchicalalgo} 
in \Cref{algo:recursivehierarchicalalgo})
%, or non-recursively as \Cref{algo:hierarchicalalgo} in \Cref{section:normalizedbutterflyfactorization}) 
will be applicable to any chainable architecture $\arch$. However, establishing approximation guarantees as in  \Cref{eq:error-bound-intro} will require a variant of this algorithm (\Cref{algo:modifedbutterflyalgo}) involving certain orthonormalization steps, which are only well-defined if the architecture $\arch$ satisfies an additional {\em non-redundancy} condition. Fortunately, any redundant architecture $\arch$ can be transformed (\Cref{prop:procedure-for-redundant-arch}) into an expressively equivalent architecture $\arch'$ ({\em i.e.}, $\setButterfly{\arch'} = \setButterfly{\arch}$) with reduced number of parameters ($\|\arch'\|_0 \leq \|\arch'\|_0$) thanks to \Cref{algo:redundancyremoval} below. This will be instrumental in introducing the proving approximation guarantees of the final butterfly algorithm \Cref{algo:modifedbutterflyalgo} applicable to {\em any (redundant or not) chainable architecture}.

%Among all chainable patterns $(\pattern_1, \pattern_2)$, some are \emph{redundant} in the following sense.
To define redundancy of an architecture we begin by considering elementary pairs.
\begin{definition}[Redundant architecture]
\label{def:redundant}
	A chainable pair of patterns $\pattern_1 = (a_1, b_1, c_1, d_1)$ and $\pattern_2 = (a_2, b_2, c_2, d_2)$ is \emph{redundant} if $r(\pattern_1,\pattern_2) \geq \min(b_1, c_2)$ (i.e., if $a_1 c_1 \geq a_2 c_2$ or $b_2d_2 \geq b_1d_1$).
    A chainable architecture $\arch = (\pattern_\ell)_{\ell=1}^L$, {$L = |\arch| \geq 1$}, is \emph{redundant} if there exists $\ell \in \intset{L-1}$ such that $(\pattern_{\ell}, \pattern_{\ell+1})$ is redundant. {Observe that by definition, any chainable architecture with $|\arch|=1$ is non-redundant.}
\end{definition}
\begin{remark}
   By \Cref{def:chainableDBfour} we always have $r(\pattern_1,\pattern_2) \leq \min(b_2,c_1)$ for a chainable pair $(\pattern_1, \pattern_2)$, hence a redundant one satisfies $\min(b_1,c_2) \leq r(\pattern_1,\pattern_2) \leq \min(b_2,c_1)$. A \emph{non-redundant} one satisfies $r(\pattern_1, \pattern_2) \leq \min(b_1-1,c_2-1,b_2,c_1)$.
   %   By \Cref{def:chainableDBfour} we always have $r \leq \min(b_2,c_1)$, hence a redundant pair satisfies $\min(b_1,c_2) \leq r \leq \min(b_2,c_1)$. A \emph{non-redundant} one satisfies $r \leq \min(b_1-1,c_2-1,b_2,c_1)$.
\end{remark}
\begin{lemma}
		\label{lem:partialproductnonredundant}
    If
  $\arch = (\pattern_\ell)_{\ell=1}^L$ is chainable and non-redundant then, for any \revision{$1 \leq q \leq s < t \leq L$}, the pair $(\pattern_{q} * \ldots * \pattern_{s}, \pattern_{s+1} * \ldots * \pattern_t)$ is chainable and non-redundant. 
\end{lemma}	
The proof is deferred to \Cref{appendix:partialproductnonredundantproof}.
\begin{example}
    \label{example:redundant}
    The architecture $\arch := (\pattern_1, \pattern_2) := \left( (1, m, r, 1), (1, r, n, 1) \right)$ is chainable, with $r(\pattern_1, \pattern_2) = r$. The set $\setButterfly{\arch}$
    is the set of $m \times n$ matrices of rank at most $r$. $(\pattern_1, \pattern_2)$ is redundant if $r \geq \min(m, n)$. We observe that on this example redundancy corresponds to the case where $\setButterfly{\arch}$ is the set of {\em all} $m \times n$ matrices.
\end{example}
%\Cref{def:redundant} comes with the following lemma.
A (chainable and) redundant architecture is as expressive as a smaller chainable architecture with less parameters. This is first proved for pairs, i.e., $\arch := (\pattern_1, \pattern_2)$.

In order to do this we need the following result, 
characterizing precisely the set of matrices $\setButterfly{\arch}=\{\bX_1\bX_2 \, | \,  \bX_i \in \setDBfactor{\pattern_i}, i \in \intset{2} \}$ as the set of matrices having a support included in $\DBsupport{\pattern_1*\pattern_2}$ and with selected low-rank blocks. It is proved in \cref{app:prodDB2}.
%\ERc{to be moved if we keep this}

% \begin{lemma}
%     \label{lem:prodDB2}
%   \ER{ If $\arch := (\pattern_1, \pattern_2)$ is chainable then (with the notations of \Cref{def:classequivalence}) we have: 
%     \begin{equation}
%     \label{eq:prodDB}
%         \setButterfly{\arch} = \setDBfactor{\pattern_1 * \pattern_2} \cap 
%         \setblocklowrank{\arch}.
%     \end{equation}
%     where we use the following set of matrices of the same dimensions as those in $\setButterfly{\arch}$:
%      \begin{equation}
%         \label{eq:prodDBRank}
%         \setblocklowrank{\arch}
%         := \{\bA: \rank\big(\bA[R_P,C_P]\big) \leq r(\pattern_1,\pattern_2),\ \forall P \in \cP(\bS_{\pattern_1}, \bS_{\pattern_2})
%         \}.
%     \end{equation}}
%     \end{lemma}
    \begin{lemma}
    \label{lem:prodDB2}
   Let $\arch := (\pattern_1, \pattern_2)$ be chainable, and (with the notations of \Cref{def:classequivalence}) 
   consider the following set of matrices of size equal to those in $\setButterfly{\arch}$:
     \begin{equation}
        \label{eq:prodDBRank}
        \setblocklowrank{\arch}
        := \{\bA: \rank\big(\bA[R_P,C_P]\big) \leq r(\pattern_1,\pattern_2),\ \forall P \in \cP(\bS_{\pattern_1}, \bS_{\pattern_2})
        \}.
    \end{equation}
    We have
     \begin{equation}
    \label{eq:prodDB}
        \setButterfly{\arch} = \setDBfactor{\pattern_1 * \pattern_2} \cap 
        \setblocklowrank{\arch}.
    \end{equation}
    \end{lemma}

\begin{lemma}
\label{lemma:consequence-redundant}
    Consider a chainable pair $\arch = (\pattern_1, \pattern_2)$. If $\arch$ is redundant, then the single-factor architecture $\arch' = (\pattern_1 * \pattern_2)$ satisfies : 
    \begin{enumerate}
        \item $\setButterfly{\arch} = \setDBfactor{\pattern_1 * \pattern_2} = \setButterfly{\arch'}$.
        \item $\|\arch'\|_0 = \|\pattern_1 * \pattern_2\|_0 < \|\pattern_1\|_0 + \|\pattern_2\|_0 = \|\arch\|_0$.
    \end{enumerate} 
\end{lemma}
% Let us clarify this notion in the light of the following lemma.
\begin{proof}
By \Cref{lem:prodDB2}   we have $\setButterfly{\arch} = \setDBfactor{\pattern_1 * \pattern_2}  \cap \setblocklowrank{\arch}$ and by \Cref{lem:prodDBBis} we have $|R_P|=b_1, |C_P|=c_2$ for each $P \in \cP(\bS_{\pattern_1}, \bS_{\pattern_2})$.
    The first claim follows from the fact that $\setblocklowrank{\arch}$ is the set of all matrices of appropriate size: indeed for any such matrix $\mat{A}$, the block $\bA[R_P,C_P]$ is of size $b_1 \times c_2$ hence its rank is at most $\min(b_1,c_2)$ which is smaller than or equal to $r(\pattern_1, \pattern_2)$ since $\arch$ is redundant. By definition of $\setblocklowrank{\arch}$ this shows that $\mat{A} \in \setblocklowrank{\arch}$.
    For the second claim, by \Cref{def:dbarchitecture} of $\|\arch\|_0$ and $\|\arch'\|_0$, we only need to prove the strict inequality. Since $(\pattern_1, \pattern_2)$ is (chainable and) redundant, we have 
    either $a_1c_1 \geq a_2c_2$ or $b_2d_2 \geq b_1d_1$, 
    hence by \Cref{lemma:numberofparameters} and \Cref{eq:operatorDBPfour} we obtain
        $\|\pattern_1 * \pattern_2\|_0 = a_2c_2b_1d_1 < a_1c_1b_1d_1 + a_2c_2b_2d_2 = \|\pattern_1\|_0 + \|\pattern_2\|_0.$
\end{proof}

\Cref{lemma:consequence-redundant} serves as a basis to define \Cref{algo:redundancyremoval}, which replaces any chainable (and possibly redundant) architecture by a ``smaller'' non-redundant one.
%{We can always assume that the chainable $\arch$ is not redundant since we can form an ``equivalent'' non-redundant and chainable $\arch'$ using \Cref{algo:redundancyremoval} otherwise. Consider the following proposition.}
    
    \begin{algorithm}[t]
	   \centering
	   \caption{Architecture redundancy removal algorithm}
	   \label{algo:redundancyremoval}
	    \begin{algorithmic}[1]
		    \REQUIRE A chainable $\arch = (\pattern_\ell)_{\ell=1}^L$
      \ENSURE A chainable and non-redundant
      $\arch' = (\pattern_\ell')_{\ell=1}^{L'}$ ($1 \leq L'\leq L$) 
		    \STATE $\arch' \gets \arch$.
		    \WHILE{
      $\arch'$ is redundant (cf.~\Cref{def:redundant})}
                \STATE $(\pattern_{\ell}')_{\ell=1}^{L'} \gets \arch'$
		      \STATE $\ell \gets$ an integer $\ell$ such that  $(\pattern_{\ell}', \pattern_{\ell + 1}')$ is redundant (cf.~\Cref{def:redundant})
		      \STATE $\arch' \gets (\pattern_1', \ldots, \pattern_{\ell-1}', \pattern_\ell' * \pattern_{\ell + 1}', \pattern_{\ell+2}', \ldots, \pattern_{L'}')$ 
        \label{line:update-redundancy}
		    \ENDWHILE
            \RETURN $\arch'$
	   \end{algorithmic}
    \end{algorithm}
  
\begin{proposition}
    \label{prop:procedure-for-redundant-arch}
   For any chainable architecture $\arch = (\pattern_\ell)_{\ell=1}^L$, 
    \Cref{algo:redundancyremoval} stops in finitely many iterations and returns an
    architecture     $\arch'$ such that:
    \begin{enumerate}
        \item $\arch'$ is chainable and non-redundant, and  
        either {a single factor architecture}  $\arch' = (\pattern_1 * \ldots * \pattern_L)$, or {a multi-factor one} $\arch' = (\pattern_1 * \ldots * \pattern_{\ell_1}, \pattern_{\ell_1+1} * \ldots * \pattern_{\ell_2}, \ldots, \pattern_{\ell_p+1} * \ldots * \pattern_L)$ for some indices $1 \leq \ell_1 < \ldots < \ell_p < L$ with $p \in \intset{1, L-1}$;
     \item $\setButterfly{\arch'} = \setButterfly{\arch}$;
        \item $\| \arch' \|_0 \leq \| \arch \|_0$. 
        %\LZc{pourquoi pas $<$?}\RG{Parce que si $\beta$ est déjà non redondant on ne change rien !}
    \end{enumerate}
\end{proposition}

\begin{proof}
    \Cref{algo:redundancyremoval} terminates since $| \arch' |$ decreases at each iteration.
    At each iteration, % $i$, 
    the updated $\arch'$ is obtained by replacing a chainable redundant pair $(\pattern_{\ell}', \pattern_{\ell + 1}')$ by a single pattern $(\pattern_\ell' * \pattern_{\ell + 1}')$. The architecture $\arch'$ remains chainable by \Cref{lem:seq-qchainability} and by chainability of $\arch$, hence the algorithm can continue with no error. 
    Due to the condition of the ``while'' loop, the returned $\arch'$ is either non-redundant with $|\arch'|>1$, or $| \arch' | = 1$ in which case it is in fact also non-redundant by \Cref{def:redundant}. 
        This yields the first condition (a formal proof of the final form of $\arch'$ can be done by an easy but tedious induction left to the reader). Moreover, a straightforward consequence of
        \Cref{lemma:consequence-redundant} is that the update of $\arch'$ in line \ref{line:update-redundancy} does not change $\setButterfly{\arch'}$, and it strictly decreases $\| \arch' \|_0$ if the condition of the ``while" loop is met at least once (otherwise the algorithm outputs $\arch' = \arch$). This yields the two other properties. 
\end{proof}

{
In particular, \Cref{algo:redundancyremoval} applied to a redundant architecture $\arch$ in \Cref{example:redundant} returns $\arch' = \left ( (1, m, r, 1) * (1, r, n , 1) \right) = \left( (1, m, n, 1) \right)$. 
}

\subsection{Constructing a chainable architecture for a target matrix size}\label{sec:archfromsize}

It is natural to wonder what (non-redundant) chainable architectures $\arch = (\pattern_\ell)_{\ell=1}^L$ allow to implement dense matrices of a prescribed size, \revision{in the sense that the class of butterfly matrices $\setButterfly{\arch}$ is sufficiently large to contain at least one dense matrix.} This is the object of the next lemma, which is proved in  \Cref{app:archfromsize}.

\begin{lemma}
\label{lem:archfromsize}
Consider integers $m,n,L \geq 2$. If $\arch = (\pattern_\ell)_{\ell=1}^L$ is a chainable architecture such that $\setButterfly{\arch}$ is made of $m \times n$ matrices and contains at least one dense matrix, then there exists a factorization of $m$ (resp.~of $n$) into $L$ integers $q_\ell$ (resp.~$p_\ell$), $1 \leq \ell \leq L$, and a sequence of $L-1$ integers $r_\ell$, $1 \leq \ell \leq L-1$ such that: with the convention $r_0=r_L=1$, for each $1 \leq \ell \leq L$ we have $\pattern_\ell = (a_\ell,b_\ell,c_\ell,d_\ell)$ with
 \begin{align}
        %n = \prod_{\ell=1}^L p_\ell & \quad \text{and}\ 
        a_\ell &= \prod_{j=1}^{\ell-1} p_j,
        %\quad 1 \leq \ell \leq L
    %    m = \prod_{\ell=1}^L q_\ell &
    \quad \text{and}\  
    d_\ell = \prod_{j=\ell+1}^{L} q_j, \label{eq:ADAsProduct}\\
    %1 \leq \ell \leq L\\
    %b_1 = q_1 & \quad \text{and}\ 
    b_\ell &= q_\ell r_{\ell-1}, 
    \quad \text{and}\ 
   % \\ %0 \leq \ell \leq L\\
    %c_L = p_L & \quad \text{and}\ 
    c_\ell = 
    p_\ell r_\ell. %, 1 \leq \ell \leq L+1.
    \label{eq:BCAsProduct}
\end{align}
Vice-versa, any architecture defined as above with integers such that $n = \prod_{\ell=1}^L p_\ell$ and $m = \prod_{\ell=1}^L q_\ell$ is chainable, and the set $\setButterfly{\arch} \subseteq \RR^{m \times n}$ contains at least one dense matrix. 

The architecture $\arch$ is non-redundant if, and only if, $r_1 < q_1$, $r_{L-1} < p_L$, and
\begin{equation}
    \label{eq:NonRedundancyConditionAnnex}
     \frac{1}{p_\ell} < \frac{r_\ell}{r_{\ell-1}} < q_\ell,\ 2 \leq \ell \leq L-1.
\end{equation}
\end{lemma}
The proof of the following corollary is straightforward and left to the reader.
\begin{corollary}
Consider integers $m,n \geq 2$ and their integer factorizations into $L \geq 2$ integers $p_\ell,q_\ell$ as in \Cref{lem:archfromsize}. \begin{itemize}
    \item If $p_\ell \geq 2$ and $q_\ell \geq 2$ for every $1 \leq \ell \leq L$, then there exists a choice of integers $r_\ell$, $1 \leq \ell \leq L-1$ such that the construction of \Cref{lem:archfromsize} is non-redundant.
\item If either $q_1=1$, $p_L=1$, or $p_\ell q_\ell=1$ for some $2 \leq \ell \leq L-1$, then no choice of $r_\ell$ allows us to obtain a non-redundant architecture.
\end{itemize}
\end{corollary}
% \begin{proof}
%     For the first claim simply $r_\ell:=1$ satisfies all the required constraints.
% \end{proof}

%equivalent to $\min_\ell \max(p_\ell,q_\ell) > 1$ (again possible with special cases for $\ell=1,L$)

As a  consequence, given a matrix size $m \times n$, a chainable architecture compatible with this matrix size  can be built via the following steps:
    \begin{enumerate}
        \item choosing integer sequences $\boldsymbol{p} := (p_\ell)_{\ell=1}^L$, $\boldsymbol{q} := (q_\ell)_{\ell=1}^L$ that factorize $n$ and $m$ (optionally: with the condition that $p_\ell \geq2$ and $q_\ell \geq 2$ for every $\ell$);
        \item choosing an integer sequence $\boldsymbol{r} := (r_\ell)_{\ell=1}^{L-1}$
        (optionally: with the condition that $r_1<q_1$, $r_{L-1} < p_L$, and \eqref{eq:NonRedundancyConditionAnnex} holds with $r_0=r_L:=1$ by convention); 
        \item defining of $\pattern_\ell =  (a_\ell,b_\ell,c_\ell,d_\ell)$ using \eqref{eq:ADAsProduct}-\eqref{eq:BCAsProduct}.
    \end{enumerate}
    Instead of imposing non-redundancy constraints, it is also possible to build a possibly redundant architecture and to exploit the redundancy removal algorithm (\Cref{algo:redundancyremoval}).
    
\begin{remark}
    \revision{Given $\boldsymbol{p}$, $\boldsymbol{q}$ and $\boldsymbol{r}$, the number of parameters in the architecture $\arch := (\pattern_\ell)_{\ell=1}^L$ defined by \eqref{eq:ADAsProduct} and \eqref{eq:BCAsProduct} is at most $\| \arch \|_0 = \sum_{\ell=1}^L p_1 \ldots p_\ell q_{\ell+1} \ldots q_L r_{\ell-1} r_\ell$.}
\end{remark}
    
\begin{remark}
Mathematically oriented readers may notice that for prime $m$ and/or $n$ there are few compatible butterfly architectures. Extending the concepts of this paper to such dimensions would allow to cover known fast transforms for prime dimensions \cite{rader_discrete_1968}. Nevertheless, typical matrix dimensions in practical applications are composite and lead to many more choices. For instance, the dimensions of weight matrices in the vision transformer architecture \cite{dosovitskiy2021an} are commonly 768, 1024, 1280, 3072, 4096, 5120, which enable many possible choices of chainable architectures for implementing dense matrices, beyond the Monarch architecture \cite{dao2022monarch} that was used specifically to accelerate such neural networks. This will be illustrated in \Cref{sec:numrectangle}.
%\todo{Add concrete details e.g. with ViT, attention heads, etc. TODO Léon (idéalement: un/qq exemples déjà traité par Dao, et une archi "nouvelle" ?}
%\todo{Is it worth showing a numerical example in the numerical part where we deal with $m,n$ not powers of $2$, e.g. $m = n = 2 \times 3 \times 5 \times 7 \ldots$ ?}
    %Finally, evenCharacterizing non-redundancy in the boundary case where $q_1 \geq 2$, $p_L \geq 2$, but there exists $2\leq \ell \leq L-1$ such that $p_\ell q_\ell=2$ (i.e. $(p_\ell,q_\ell) \in \{(1,2),(2,1)\}$ is left to future work as it seems of marginal practical interest.
\end{remark}
\section{Hierarchical algorithm under the chainability condition}
\label{sec:hierarchical}
We show how the hierarchical algorithm in \cite{le2022fastlearning,zheng2023efficient}, initially introduced for specific (square dyadic) architectures, can be directly extended to the case where  $\arch$ is \emph{any} chainable architecture. 
% This extension yields \Cref{algo:recursivehierarchicalalgo} presented below. 
{The case where $L = |\arch|=1$ is trivial since Problem~\eqref{eq:butterfly-approximation-pb} is then simply solved by setting $\mat{X}_1$ to be a copy of $\mat{A}$ where all entries outside of the prescribed support are set to zero. We thus focus on $L \geq 2$ and} 
%We 
start with $\arch$ of depth $L = 2$ before considering arbitrary $L \geq 2$.

% Our proposed algorithm is inspired by the hierarchical factorization algorithm in \cite{le2022fastlearning,zheng2023efficient}, which was developed for the square dyadic butterfly factorization (cf.~\Cref{tab:butterfly-parametrization}). 

% In this section, we first discuss how the hierarchical algorithm in \cite{le2022fastlearning,zheng2023efficient}, {which was initially introduced for specific (square dyadic) architectures,} can be directly extended to the case where  $\arch$ is \emph{any} chainable { architecture}, yielding \Cref{algo:hierarchicalalgo} presented below. Then, we argue that \Cref{algo:hierarchicalalgo} without further modification cannot admit an error bound in the form of \eqref{eq:error-bound-intro}. This will motivate us to make a small tweak in \Cref{algo:recursivehierarchicalalgo} to obtain \Cref{algo:modifedbutterflyalgo} presented later in \Cref{section:normalizedbutterflyfactorization}, where we show that this modification allows establishing a bound of the form \eqref{eq:error-bound-intro} in \Cref{sec:error}.

\subsection{Case with $L=2$ factors}
\label{subsec:L=2}
Problem \eqref{eq:butterfly-approximation-pb} with an architecture $\arch = (\pattern_1, \pattern_2)$ is simply an instance of Problem \eqref{eq:FSMF} with $(\bL, \bR) = (\bS_{\pattern_1}, \bS_{\pattern_2})$. 
% Interestingly, 
% {of compatible dimensions, i.e.,} that satisfies \eqref{eq:condition-size-chainable} with $L = 2$.

\begin{lemma}
	\label{lem:qperfectcovering}
Consider the pair of supports $(\mat{L}, \mat{R}) = (\bS_{\pattern_1}, \bS_{\pattern_2})$ associated to any (\emph{not necessarily chainable}) architecture $(\pattern_1, \pattern_2)$.
 The assumptions of \Cref{theorem:tractablefsmf} hold, hence
%    by the pair of supports $(\mat{L}, \mat{R}) = (\bS_{\pattern_1}, \bS_{\pattern_2})$ for any (\emph{not necessarily chainable}) architecture $(\pattern_1, \pattern_2)$. 
for \emph{any} architecture $\arch$ of depth $|\arch|=2$, the two-factor fixed-support matrix factorization algorithm (\Cref{algo:algorithm1}) returns an optimal solution to the corresponding instance of Problem \eqref{eq:butterfly-approximation-pb}.

\end{lemma}

\begin{proof}
%\ERc{SUGGESTION: the proof of the lemma is given by the first argument of lemma 4.19 in section 4}

 %   In general, for any pattern $\pattern$, by 

   % It can be easily verified that the column supports $\{ \matindex{\bS_{\pattern_1}}{:}{j} \}_{j}$ (resp.~the row supports $\{ \matindex{\bS_{\pattern_2}}{i}{:} \}_{i}$) are pairwise disjoint or identical (\TL{due to the support structure of the form $\bI_a \otimes \mbf{1}_{b \times c} \otimes \bI_d$ of a Kronecker-sparse factor\footnote{This is not hard to verity since the columns (resp.~rows) of the Kronecker product $\bA \otimes \bB$ are equal to the Kronecker product of columns and rows of $\bA$ and $\bB$ and the matrices $\bA, \bB$ appearing in our case are either identity matrices or all-one matrices. }}), 
   As in the proof of \Cref{lem:prodDBBis}, the column supports $\{ \matindex{\bS_{\pattern_1}}{:}{j} \}_{j}$ (resp.~the row supports $\{ \matindex{\bS_{\pattern_2}}{i}{:} \}_{i}$) are pairwise disjoint or identical. Hence   the components $\bU_i$ of 
    $\varphi(\bS_{\pattern_1}, \bS_{\pattern_2})$ are pairwise disjoint or identical
     (if their  
    column {\em and} row supports coincide).
    
  \end{proof}

\subsection{Case with $L \geq 2$ factors}
Consider now Problem \eqref{eq:butterfly-approximation-pb} associated with a \emph{chainable} architecture $\arch$ of depth $L \coloneqq |\arch| \geq 2$, and a given target matrix $\mat{A}$. A first proposition of 
hierarchical algorithm, introduced in \Cref{algo:recursivehierarchicalalgo}, is a direct adaptation to our framework of previous algorithms \cite{le2022fastlearning,zheng2023efficient}. It computes an approximate solution by performing successive two-factor matrix factorization  in a certain hierarchical order that is described by a so-called factor-bracketing tree \cite{zheng2023efficient}. Further refinements of the algorithm will later be added to obtain approximation guarantees.

\begin{definition}[Factor-bracketing tree \cite{zheng2023efficient}]
\label{def:factor-bracketing-tree}
    A factor-bracketing tree of $\intset{L}$ for a given integer $L$ is a binary tree such that:
    \begin{itemize}
        \item each node {is} an interval $\intset{q,t} := \{ q, q + 1, \ldots, t \}$ for $1 \leq q \leq t \leq L$;
        %\mRG{To propagate: $r \to q$}
        \item the root is $\intset{L}$;
        \item every non-leaf node $\intset{q,t}$ for $q<t$ has $\intset{q,s}$ as its left child and $\intset{s+1,t}$ as its right child, for a certain $s$ such that $q \leq s < t$;
        \item a leaf {is} a singleton $\intset{q,q}$ for some $q \in \intset{1, L}$.
    \end{itemize}
    Such a tree has exactly $(L-1)$ non-leaf nodes and $L$ leaves.
    % The root is always the node $\intset{1,L}$. Every non-leaf node $\intset{q,t}$ for $r<t$ has exactly two children $\intset{q,s}$ and $\intset{s+1,t}$, for a certain $s$ such that $r \leq s < t$.  
\end{definition}
% \mRG{Make it a definition ?}

% The pseudo-code of this hierarchical algorithm, written recursively, is presented in \Cref{algo:recursivehierarchicalalgo}. 
{Before exposing the limitations of \Cref{algo:recursivehierarchicalalgo} and proposing fixes, let us briefly explain how it works with a focus on its main step in line~\ref{line:recursivemainlinea}.}
Consider any factor-bracketing tree $\mathcal{T}$. \Cref{algo:recursivehierarchicalalgo} computes a matrix $\matseq{\mat{X}}{\intset{q,t}} \in \setDBfactor{\pattern_{q} * \ldots * \pattern_t}$ for each node  $\intset{q,t}$ in a recursive manner. $\pattern_{q} * \ldots * \pattern_t$ is well-defined for any $1 \leq q \leq t \leq L$ because $\arch$ is chainable.
At the root node, we set $\matseq{\mat{X}}{\intset{1,L}} := \mat{A}$. At each non-leaf node $\intset{q,t}$ whose matrix $\bX_{\intset{q,t}}$ is already computed during the hierarchical procedure, and with children $\intset{q,s}$ and $\intset{s+1,t}$, {at line~\ref{line:recursivemainlinea}} we compute $(\bX_{\intset{q,s}}, \bX_{\intset{s+1,t}}) \in \setDBfactor{\pattern_{q} * \ldots * \pattern_s} \times \setDBfactor{\pattern_{s+1} * \ldots * \pattern_t}$ that is solution to the following instance of {the Fixed Support Matrix Factorization Problem}~\eqref{eq:FSMF}:
\begin{equation}
	\label{eq:fsmf-DB}
	\begin{aligned}
		\text{Minimize} \quad  &\| \mathbf{X}_{\intset{q,t}} - \mathbf{X}_{\intset{q,s}} \mathbf{X}_{\intset{s+1,t}} \|_F \\
		\text{Subject to} \quad  & \supp(\mathbf{X}_{\intset{q,s}}) \subseteq \DBsupport{\pattern_{q} * \ldots * \pattern_s},\\
		&\supp(\mathbf{X}_{\intset{s+1,t}}) \subseteq \DBsupport{\pattern_{s+1} * \ldots * \pattern_t}.
	\end{aligned}
\end{equation}
Indeed, by \Cref{lem:qperfectcovering}, Problem \eqref{eq:fsmf-DB} is solved by the two-factor fixed support matrix factorization algorithm (\Cref{algo:algorithm1}), which yields line~\ref{line:fsmf} in \Cref{algo:recursivehierarchicalalgo}. After computing $(\bX_{\intset{q,s}}, \bX_{\intset{s+1,t}})$, we repeat recursively the procedure on these two matrices independently, as per lines \ref{line:left-recursion} and \ref{line:right-recursion}, until we obtain the butterfly factors $(\matseq{\mat{X}}{\intset{\ell,\ell}})_{\ell=1}^L \in \setDBfactor{\arch}$ that yield an approximation $\hat{\mat{A}} := \matseq{\mat{X}}{\intset{1,1}} \ldots \matseq{\mat{X}}{\intset{L,L}} \in \setButterfly{\arch}$ of $\mat{A}$. In conclusion, \Cref{algo:recursivehierarchicalalgo} is a greedy algorithm that seeks the optimal solution at each two-factor matrix factorization problem during the hierarchical procedure.

\begin{algorithm}[t]
	\centering
	\caption{{Hierarchical} factorization algorithm -- recursive version
	% Butterfly factorization algorithm
	} 
	\label{algo:recursivehierarchicalalgo}
	\begin{algorithmic}[1]
		\REQUIRE $\mat{A} \in \CC^{m \times n}$, chainable $\arch = (\pattern_\ell)_{\ell = 1}^L$, factor-bracketing tree
  $\cT$
    \ENSURE $\texttt{factors} \in  \setDBfactor{\arch}$
		\IF{$L = 1$} 
		\RETURN $(\bA \odot \bS_{\pattern_1})$\hfill \(\triangleright\)\COMMENT{$\odot$ is the Hadamard product.}
		\ENDIF
		\STATE $\intset{1, s}, \intset{s+1,L} \gets$ two children of the root $\intset{1,L}$ of $\cT$
		\STATE $(\cT_{\tleft}, \cT_{\tright}) \gets$ the corresponding left and right subtrees of $\cT$
		\STATE $(\pattern_{\tleft}, \pattern_{\tright}) \gets (\pattern_1 * \ldots * \pattern_{s}, \pattern_{s+1} * \ldots * \pattern_L)$
		\STATE \label{line:recursivemainlinea} {($\bX_{\intset{1, s}}, \bX_{\intset{s+1,L}}) \gets$ \Cref{algo:algorithm1}$(\bA, \bS_{\pattern_{\tleft}}, \bS_{\pattern_{\tright}})$}{\label{line:fsmf}}
		\STATE $\texttt{left\_factors} \gets$ \Cref{algo:recursivehierarchicalalgo}($\bX_{\intset{1, s}}, (\pattern_1, \ldots, \pattern_s), \cT_{\tleft}$){\label{line:left-recursion}}
		\STATE $\texttt{right\_factors} \gets$ \Cref{algo:recursivehierarchicalalgo}($\bX_{\intset{s+1,L}}, (\pattern_{s+1}, \ldots, \pattern_L), \cT_{\tright}$){\label{line:right-recursion}}
        \STATE $\texttt{factors} \gets \texttt{left\_factors} \cup \texttt{right\_factors}$
		\RETURN $\texttt{factors}$
	\end{algorithmic}
\end{algorithm}

\subsection{\Cref{algo:recursivehierarchicalalgo} does not satisfy the theoretical guarantee \eqref{eq:error-bound-intro}}
% In this section, we argue that t
However, the control of the approximation error in the form of \eqref{eq:error-bound-intro} for \Cref{algo:recursivehierarchicalalgo} in its current form is {\em impossible}, as illustrated in the following example.

% To illustrate that, let us consider the following example.
\begin{example}
	\label{example:illposedbX1} 
    Consider $\arch = (\pattern_1, \pattern_2, \pattern_3) = \left( (2^{\ell-1}, 2, 2, 2^{3-\ell}) \right)_{\ell=1}^3$, which is the square dyadic architecture of depth $L=3$.
    % Consider the square dyadic architecture of length three, i.e., {$\pattern_1 = (1,2,2,4), \pattern_2 = (2,2,2,2), \pattern_3 = (4,2,2,1)$} and 
    Define $\bA := (\bD\bS_{\pattern_1})\bS_{\pattern_2}\bS_{\pattern_3}$  where $\bD$ is the diagonal matrix with diagonal entries $(0, 1, 1, 1, 0, 1, 1, 1)$.
% Observe that $\bA = (\bD\bS_{\pattern_1})\bS_{\pattern_2}\bS_{\pattern_3}$ where $\bD = \texttt{diag}(0, 1, 1, 1, 0, 1, 1, 1)$, 
Hence, $\bA \in \setButterfly{\arch}$, meaning that any algorithm with a theoretical guarantee \eqref{eq:error-bound-intro} must output
%\mRG{was "recover" mais on ne ``retrouve'' pas les facteurs initiaux} 
butterfly factors whose product is exactly $\bA$.  
However, we claim that this is not the case of \Cref{algo:recursivehierarchicalalgo} with the so-called left-to-right factor-bracketing tree of $\intset{1, 3}$
(defined as the tree where each left child is a singleton).
% , cf.~\Cref{fig:unbalancedtree}). 
To see why, let us apply this algorithm to $\mat{A}$.
\begin{enumerate}
    \item In the first step, the hierarchical algorithm applies {the two-factor fixed support matrix factorization algorithm (\Cref{algo:algorithm1})} with input $(\bA, \bS_{\pattern_1}, \DBsupport{\pattern_2 * \pattern_3})$, and returns  ($\matseq{X}{\intset{1,1}}, \matseq{X}{\intset{2,3}}) \in \setDBfactor{\pattern_1} \times \setDBfactor{\pattern_2 * \pattern_3}$. 
 
    \item At the second step (which is the last one),  
    { \Cref{algo:algorithm1}}
    %\LZ{the two-factor fixed support matrix factorization algorithm}
    is applied to the input $(\bX_{\intset{2,3}}, \DBsupport{\pattern_2}, \DBsupport{\pattern_3})$, and returns  $(\matseq{X}{\intset{2,2}}, \matseq{X}{\intset{3,3}}) \in \setDBfactor{\pattern_2} \times \setDBfactor{\pattern_3}$.
\end{enumerate}

By construction, the first and the fifth row of $\mat{A}$ are null, so the first step can return \emph{many possible optimal solutions} $(\matseq{X}{\intset{1,1}}, \matseq{X}{\intset{2,3}})$ for the considered instance of Problem \eqref{eq:FSMF}, such as $\matseq{X}{\intset{1,1}} := \bD\bS_{\pattern_1}$ %at the end of the first step, with
and %$\matseq{X}{\intset{2,3}}$ of the form:
 \begin{equation*}
     \matseq{X}{\intset{2,3}} = \begin{pmatrix}
         \bB & \mbf{0}_{4 \times 4}\\
         \mbf{0}_{4 \times 4} & \bB
     \end{pmatrix} \quad \text{with} \quad \bB = \begin{pmatrix}
         \alpha & \beta & \gamma & \delta \\
         1 & 1 & 1 & 1 \\
         1 & 1 & 1 & 1 \\
         1 & 1 & 1 & 1 \\
     \end{pmatrix},
 \end{equation*}
where the scalars
 $\alpha, \beta, \gamma, \delta$ can be arbitrary.
 % Indeed, such pairs $(\matseq{X}{\intset{1,1}}, \matseq{X}{\intset{2,3}})$ yield $\bX_{\intset{1,1}} \bX_{\intset{2,3}} = \bA$.
 % , because 
% $\bS_{\pattern_1} = \mbf{1}_{2 \times 2} \otimes \bI_4$ and $\bS_{\pattern_2 * \pattern_3} = \bI_2 \otimes \mbf{1}_{4 \times 4}$.
% 
% as long as they satisfy $\bX_{\intset{1,1}} \bX_{\intset{2,3}} = \bA$.
% In particular, {since $\bS_{\pattern_1} = \mbf{1}_{2 \times 2} \otimes \bI_4$ and $\bS_{\pattern_2 * \pattern_3} = \bI_2 \otimes \mbf{1}_{4 \times 4}$, one can check that a valid pair of this type is} 
% %it can output
% $\matseq{X}{\intset{1,1}} := \bD\bS_{\pattern_1}$ %at the end of the first step, with
% and %$\matseq{X}{\intset{2,3}}$ of the form:
%  \begin{equation*}
%      \matseq{X}{\intset{2,3}} = \begin{pmatrix}
%          \bB & \mbf{0}_{4 \times 4}\\
%          \mbf{0}_{4 \times 4} & \bB
%      \end{pmatrix} \quad \text{where} \quad \bB = \begin{pmatrix}
%          \alpha & \beta & \gamma & \delta \\
%          1 & 1 & 1 & 1 \\
%          1 & 1 & 1 & 1 \\
%          1 & 1 & 1 & 1 \\
%      \end{pmatrix},
%  \end{equation*}
% where the scalars
%  $\alpha, \beta, \gamma, \delta$ can be arbitrary.
%  % \mRG{Put notation $\mbf{0}_{m \times n}$ in notations section} 
 % 
 Then, with the choice $(\alpha, \beta, \gamma, \delta) = (1, -1, 1, -1)$, one can check that 
 % the matrix $\matseq{X}{\intset{2,3}}$ cannot be written as a product of a $\pattern_2$-factor and $\pattern_3$-factor (e.g., using \mRG{partout: "by using" $\to$ "using"} \Cref{theorem:tractablefsmf}).
 % for example). 
 % Hence, the second step of the procedure can only output $\matseq{X}{\intset{2,2}}$ and $\matseq{X}{\intset{3,3}}$ such that $\matseq{X}{\intset{2,2}} \matseq{X}{\intset{3,3}} \neq \matseq{X}{\intset{2,3}}$ and therefore $\matseq{X}{\intset{1,1}} \matseq{X}{\intset{2,2}} \matseq{X}{\intset{3,3}} \neq \mat{A}$. 
 % In particular, 
 the second step of the procedure will always output $\matseq{X}{\intset{2,2}}$ and $\matseq{X}{\intset{3,3}}$ such that $\matseq{X}{\intset{2,2}} \matseq{X}{\intset{3,3}} \neq \matseq{X}{\intset{2,3}}$ and $\matseq{X}{\intset{1,1}} \matseq{X}{\intset{2,2}} \matseq{X}{\intset{3,3}} \neq \mat{A}$. In conclusion, this example\footnote{At first sight, this seems to contradict the so-called exact recovery property \cite{le2022fastlearning,zheng2023efficient} of \Cref{algo:recursivehierarchicalalgo} in the case of square dyadic butterfly factorization. This is not the case, since the statement of these exact recovery results 
 % \RG{TODO: Pointer vers un lemme/thm précis ?}
 includes a technical assumption excluding matrices with zero columns/rows \cite[Theorem 3.10]{zheng2023efficient}, which is not satisfied by $\mat{A}$ here.} shows that the output of \Cref{algo:recursivehierarchicalalgo} cannot satisfy the theoretical guarantee \eqref{eq:error-bound-intro}.
\end{example}

% \begin{remark}

% \end{remark}

% This impossibility of establishing an error bound as in \eqref{eq:error-bound-intro} for \Cref{algo:recursivehierarchicalalgo}
% results from the ambiguity of the choice of optimal factors $(\matseq{X}{\intset{1,s}}, \matseq{X}{\intset{s+1,L}})$ given by \Cref{algo:algorithm1} in line~\ref{line:fsmf} of \Cref{algo:recursivehierarchicalalgo}. There are plenty of optimal pairs of factors at each iteration, and the choice at line~\ref{line:fsmf} will affect the following factorizations in the recursion because it changes the input matrices of \Cref{algo:algorithm1} at lines~\ref{line:left-recursion} and \ref{line:right-recursion}.
% To avoid the issue illustrated in \Cref{example:illposedbX1}, 
% \Cref{section:normalizedbutterflyfactorization} proposes to modify \Cref{algo:recursivehierarchicalalgo} so that it produces ``good choices'' of input matrices at lines~\ref{line:left-recursion} and \ref{line:right-recursion} for \Cref{algo:algorithm1}, in the sense that it allows for a provable guarantee of the modified hierarchical algorithm with an error bound of the form \eqref{eq:error-bound-intro}.
% \mRG{grammaire anglaise "to allow for doing" ou "to allow something" (sans "for")} 

The inability to establish an error bound as in \eqref{eq:error-bound-intro} for \Cref{algo:recursivehierarchicalalgo}
is due to the ambiguity for the choice of optimal factors $(\matseq{X}{\intset{1,s}}, \matseq{X}{\intset{s+1,L}})$ returned by the two-factor fixed support matrix factorization algorithm called at line~\ref{line:fsmf} in \Cref{algo:recursivehierarchicalalgo}, cf.~\Cref{rem:nonuniquesolution}.
At each iteration, there are \emph{multiple} optimal pairs of factors, and the choice at line~\ref{line:fsmf} 
impacts subsequent factorizations in the recursive procedure.
% , since it changes the input matrices for \Cref{algo:algorithm1} at lines~\ref{line:left-recursion} and \ref{line:right-recursion}.
% will affect the following factorizations in the recursion because it changes the input matrices of 
% To avoid the issue illustrated in \Cref{example:illposedbX1},
To guarantee an error bound of the type \eqref{eq:error-bound-intro},
\Cref{section:normalizedbutterflyfactorization} proposes a revision of \Cref{algo:recursivehierarchicalalgo}, where, among all the possible choices, the modified algorithm selects \emph{specific} input matrices at lines \ref{line:left-recursion} and \ref{line:right-recursion}.

\section{Butterfly algorithm with error guarantees}
% \section{\LZ{Polynomial hierarchical algorithm with \RG{error} guarantees% on the error bound
% }}
\label{section:normalizedbutterflyfactorization}
% In order to introduce 
{We now propose a}
%The proposed 
modification of the hierarchical algorithm (\Cref{algo:recursivehierarchicalalgo}) {using} %includes some 
\emph{orthonormalization operations} that are novel in the context of butterfly factorization. It is based on an unrolled version of \Cref{algo:recursivehierarchicalalgo} {and will be endowed with error guarantees stated and proved in the next section.}
% order to obtain error guarantees of the form \eqref{eq:error-bound-intro}, let us

% is based on two ingredients: an unrolled version of \Cref{algo:recursivehierarchicalalgo}, and some novel \emph{orthonormalization operations} in the context of butterfly factorization.

% which allows us to explore the set of possible input matrices for each call of \Cref{algo:algorithm1} (cf.~Line~\ref{line:fsmf}, \Cref{algo:recursivehierarchicalalgo}) and our proposed orthonormalization operator, which shows a specific choice that makes the control of the approximation error possible. 

% To introduce the modified version of \Cref{algo:recursivehierarchicalalgo} with guarantees on the error bound, we need two elements: an unrolled version of \Cref{algo:recursivehierarchicalalgo}, which allows us to explore the set of possible input matrices for each call of \Cref{algo:algorithm1} (cf.~Line~\ref{line:fsmf}, \Cref{algo:recursivehierarchicalalgo}) and our proposed orthonormalization operator, which shows a specific choice that makes the control of the approximation error possible. 

%\subsection{A non-recursive version for \Cref{algo:recursivehierarchicalalgo}}
\subsection{{Butterfly algorithm with orthonormalization}}
%\subsection{Orthonormalization operations in the hierarchical algorithm with error guarantees for non-redundant chainable architecture}
\label{sec:modifiedalgo}
% A factor bracketing tree $\mathcal{T}$ of $\intset{L}$ can be described equivalently using a permutation $\sigma$ of $\intset{L-1}$. Indeed, 

% Before describing the proposed modification, we first need to introduce an unrolled version of \Cref{algo:recursivehierarchicalalgo}.
% To unroll the 
% algorithm described in 
% \LZ{recursive}
% \Cref{algo:recursivehierarchicalalgo}, we observe that 
The %role of the 
factor-bracketing tree $\cT$ in \Cref{algo:recursivehierarchicalalgo} 
%is to 
% encode  
% how to ``split'' the factorization. 
{describes in which order the successive $L - 1$ two-factor matrix factorization steps are performed, where $L := | \arch |$.} 
An equivalent way to describe this hierarchical order is to store a permutation $\sigma := (\sigma_\ell)_{\ell = 1}^{L - 1}$ of $\intset{L-1}$, by saving each splitting index $s \in \intset{L-1}$ that corresponds to the maximum integer in the left child $\intset{q,s}$ of each non-leaf node $\intset{q,t}$ of $\mathcal{T}$ (cf.~\Cref{def:factor-bracketing-tree}). 
We can then propose a non-recursive version of \Cref{algo:recursivehierarchicalalgo}, described in \Cref{algo:modifedbutterflyalgo},
where for any non-empty integer interval we use the shorthand
\begin{equation}\label{def:patterninterval}
    \pattern_{\intset{p,q}} := \pattern_{p} * \ldots * \pattern_{q}.
\end{equation} 
Skipping (for now) lines~\ref{line:beginfor1}-\ref{line:endfor2}, these two algorithms are equivalent when $\mathcal{T}$ and $\sigma$ {match}, 
and thus the new version still suffers from the pitfall highlighted in \Cref{example:illposedbX1} regarding error guarantees. This can however be overcome by introducing additional \emph{pseudo-orthonormalization operations} (lines~\ref{line:beginfor1}-\ref{line:endfor2}),
% -- they are 
involving orthogonalization of certain blocks of the matrix, as explicitly described in   \Cref{algo:exchange} in \Cref{app:orthonormalDB}. 
%\todo{renommer en "partial" -orthonormalization" partout ?}
% \mLZ{j'ai enlevé ici que c'est défini que si c'est non redondant car ça répète avec le paragraphe suivant}

%  Combined
%{Combining this modification} with redundancy removal () we obtain an algorithm that is applicable to any chainable architecture and will be shown in \Cref{sec:error} to  satisfy the desired approximation guarantees. }

%\todo{TODO: séparer titre de commentaires caption de l'algo \TL{OK};}
\begin{algorithm}[ht]
	\centering
	\caption{Butterfly factorization algorithm -- unrolled version \HL{(with pseudo-orthonormalization)}. \\
 NB: removing \HL{blue code} yields a non-recursive equivalent to \Cref{algo:recursivehierarchicalalgo}, applicable even to {\em redundant} $\arch$, but not endowed with the guarantees of \Cref{theorem:mainresults,theorem:bettermainresults}.) }
	\label{algo:modifedbutterflyalgo}
	\begin{algorithmic}[1]
    \REQUIRE $\mat{A} \in \CC^{m \times n}$, \HL{ non-redundant}, chainable $\arch = (\pattern_\ell)_{\ell=1}^L$, permutation $(\sigma_\ell)_{\ell = 1}^{L - 1}$ of $\intset{L-1}$
    \ENSURE $\texttt{factor} \in \setDBfactor{\arch}$
        {
        \IF {L = 1}
        \RETURN $(\mat{A} \odot 
        %\mat{S}_{\pattern_1 * \ldots * \pattern_L}
        {\mat{S}_{\pattern_1}}
        )
        $ \hfill \(\triangleright\)\COMMENT{$\odot$ is the Hadamard product.}
        \ENDIF
        }
		\STATE $\texttt{partition} \gets ( {I}_1)$ where we denote ${I}_1 = 
  %\intset{L - 1}
  \intset{1,L}
  $ 
  
  {\label{line:partition_init}}
		% \STATE $\texttt{factors}  \gets ( \matseq{X}{{I}_1} )$ where we denote $\matseq{X}{{I}_1} = \mat{A}$
        \STATE $\texttt{factors} \gets (\mat{A})$ {\label{line:factorinit}}
		\FOR{$J = 1, \ldots, L - 1$}
		\STATE $({I}_j)_{j=1}^J \gets \texttt{partition}$ {\label{line:partition_update}}
		\STATE $(\matseq{X}{{I}_j})_{j=1}^J \gets \texttt{factors}${\label{line:factors_before_ortho}}
        \STATE {$s:=\sigma_J$}{\label{line:s}}
  \STATE {$j \gets $ the unique $j \in \intset{J}$ such that ${I}_j := \intset{q,t} \ni s$ % := \sigma_J$
  }{\label{line:defjmodifiedalgo}}\\
        % \(\triangleright\) {\color{blue} \emph{Begin of orthonormalization operations}} \(\triangleleft\)\\
        % 
		\color{blue}
        \FOR{$k = 1, \ldots, j-1$}{\label{line:beginfor1}}
		\STATE {$(\matseq{X}{{I}_k}, \matseq{X}{{I}_{k+1}}) \gets $ \Cref{algo:exchange}$(\pattern_{{I}_k}, \pattern_{{I}_{k+1}}, \matseq{X}{{I}_k}, \matseq{X}{{I}_{k+1}}, \texttt{column})$}{\label{line:colorthonormal}}
		% $r_t$-column-orthonormalization of $\matseq{X}{{I}_k}$ with $t = \max {I}_k$
		\ENDFOR {\label{line:endfor1}} 
		\FOR{$k = J, \ldots, j+1$}{\label{line:beginfor2}}
		\STATE {$(\matseq{X}{{I}_{k-1}}, \matseq{X}{{I}_{k}}) \gets $ \Cref{algo:exchange}$(\pattern_{{I}_{k-1}}, \pattern_{{I}_k}, \matseq{X}{{I}_{k-1}}, \matseq{X}{{I}_{k}}, \texttt{row})$}{\label{line:roworthonormal}} 
		\ENDFOR
  {\label{line:endfor2}}
		\color{black}\\
        % \(\triangleright\) {\color{blue} \emph{End of orthonormalization operations} } \(\triangleleft\)\\
		\STATE $(\pattern_{\tleft}, \pattern_{\tright}) \gets (\pattern_{q} * \ldots * \pattern_{s}, \pattern_{s + 1} * \ldots * \pattern_t)$
		\STATE {$(\matseq{X}{\intset{q,s}}, \matseq{X}{\intset{s+1,t}}) \gets$ \Cref{algo:algorithm1}$(
    {\matseq{X}{{I}_j}}
%\matseq{X}{\intset{q,t}}
  , \DBsupport{\pattern_{\tleft}}, \DBsupport{\pattern_{\tright}})$}{\label{line:fsmfmodifiedalgo}}
		\STATE $\texttt{partition} \gets ({I}_1, \ldots, {I}_{j-1}, \intset{q,s}, \intset{s+1,t}, {I}_{j+1}, \ldots, {I}_J)$ {\label{line:setpartition}}
		\STATE $\texttt{factors} \gets (\matseq{X}{{I}_1}, \ldots, \matseq{X}{{I}_{j-1}}, \matseq{X}{\intset{q,s}}, \matseq{X}{\intset{s+1,t}}, \matseq{X}{{I}_{j+1}}, \ldots, \matseq{X}{{I}_J})${\label{line:list_factors_at_end}}
		\ENDFOR
		\RETURN $\texttt{factors}$
	\end{algorithmic}
\end{algorithm}
%\LZc{TODO: changer preuve de \Cref{theorem:mainresults}}

%\mRG{Ne devrait-on pas donner des noms plus explicites aux algorithmes ? \LZ{oui, à faire à la fin quand on aura décidé des noms} Done: TODO: mettre nom des algos dans le texte la ou ca parait pertinent}

% In this section, we first \LZ{introduce} 
% \Cref{algo:modifedbutterflyalgo}, \LZ{which is a variation} of \Cref{algo:recursivehierarchicalalgo,algo:hierarchicalalgo} that will be proved to admit an error bound of the form \eqref{eq:error-bound-intro}.
% \mRG{Je vais reprendre ma lecture detaillee ici}

%{\bf \em  {Orthonormalization operations under the non-redundancy assumption.}}
{
These pseudo-orthonormalization operations in this new butterfly algorithm (\Cref{algo:modifedbutterflyalgo})
% The goal of the pseudo-orthonormalization operations in lines~\ref{line:beginfor1}-\ref{line:endfor2} of \Cref{algo:modifedbutterflyalgo} is to 
rescale the butterfly factors $(\matseq{X}{{I}_k})_{k=1}^J$ without changing their product, in order to make a specific choice of $\matseq{X}{\intset{q,t}}$ given as input to \Cref{algo:algorithm1} at line~\ref{line:fsmfmodifiedalgo} during 
subsequent steps of the algorithm, 
while constructing factors $\matseq{X}{{I}_1}, \ldots, \matseq{X}{{I}_{j-1}}, %\matseq{X}{j+1},
{\matseq{X}{{I}_{j+1}}} \ldots, \matseq{X}{{I}_J}$ that are pseudo-orthonormal in the following sense.
% As it will be shown in \Cref{sec:error}, this will lead to provable error guarantees of the algorithm.
}

{
\begin{definition}[Left and right $r$-unitary factors]
    \label{def:q-unitary}
    Consider a pattern $\pattern = (a,b,c,d)$ and  $r \in \NN$. A $\pattern$-factor $\bX$ {(cf~\Cref{def:dbfactorfour})} is left-{$r$}-unitary (resp.~right-{$r$}-unitary) if $r \mid c$ (resp.~$r \mid b$) and for any {$\pattern'$}-factor $\bY$ satisfying $r(\pattern, \pattern') = r$ (resp.~$r(\pattern', \pattern) = r$), we have: $\|\bX\bY\|_F = \|\bY\|_F$ (resp.~$\|\bY\bX\|_F = \|\bY\|_F$). 
\end{definition}
}

\begin{remark}\label{rem:leftrightnotplainunitary}
    {The {notions of left/right-$r$-unitary factor introduced  in \Cref{def:q-unitary} are relaxed versions of the usual column/row orthonormality}. In particular, a left/right-$r$-unitary factor is only required to preserve the Frobenius norm of a set of chainable factors upon left/right matrix multiplication. Therefore, a $\pattern=(a,b,c,d)$-factor (cf.~\Cref{def:dbfactorfour}) with orthonormal columns (resp.~rows) is left-$r$-unitary (resp.~right-$r$-unitary) for any $r \mid c$ (resp.~$r \mid b$). The other implication is not true since $\frac{1}{\sqrt{2}} \bI_2 \otimes \mbf{1}_{2 \times 2}$ is a left-$1$-unitary $(2,2,2,1)$-factor but it is not a column orthonormal matrix. We name our operation \emph{pseudo-orthonormalization} to avoid confusion with the usual orthonormalization operation.}   

    {More importantly, left/right-$r$-unitary notions also share certain properties with column/rows orthonormality such as the stability under matrix multiplication and norm preserving upon both left and right multiplication. These properties will be detailed in \Cref{app:orthonormalDB}.} 
\end{remark}

Using \Cref{def:q-unitary}, we can describe the purpose of the pseudo-orthonormalization operation used in the new butterfly algorithm (\Cref{algo:modifedbutterflyalgo}) as follows:
\begin{lemma}
    \label{lemma:role-pseudo-orthogonality}
   At the $J$-th iteration of \Cref{algo:modifedbutterflyalgo}, denote\footnote{Observe that by construction $t_i = q_{i+1}-1$ whenever $i+1 \in \intset{J}$.} $I_i = \intset{q_i,t_i}$ for any $i \in \intset{J}$. 
    After pseudo-orthonormalization operations (cf.~line~\ref{line:beginfor1}-\ref{line:endfor2} - \Cref{algo:modifedbutterflyalgo}), the $\pattern_{I_i}$-factor $\bX_{{I}_i}$ for $i = 1, \ldots, j - 1$ is left-$r(\pattern_{t_i},\pattern_{t_i+1})$-unitary,
    and the $\pattern_{I_i}$-factor $\bX_{{I}_i}$ for $i = j + 1, \ldots, J$ is right-$r(\pattern_{q_{i}-1}, \pattern_{q_{i}})$-unitary, where $j$ is the integer defined in line \ref{line:defjmodifiedalgo}.
    
    %$r_i,t_i$ are the first and the last elements of ${I}_i, i \in \intset{J}$ respectively.
\end{lemma}
This result plays a key role in deriving a guarantee for \Cref{algo:modifedbutterflyalgo} in \Cref{sec:error}. It is proved in \Cref{app:role-pseudo-orthogonality}.

% for two-layer, fixed-support matrix factorization.

\begin{remark}\label{rem:redundantcase}
As detailed in  \Cref{app:orthonormalDB}  the orthonormalization operations are well-defined  only under the non-redundancy assumption. When the architecture $\arch$ is redundant, by the redundancy removal algorithm (\Cref{algo:redundancyremoval}) we can reduce it to a non-redundant architecture $\arch'$  {that is} expressively equivalent to $\arch$ (i.e, $\setButterfly{\arch'} = \setButterfly{\arch}$) and apply the algorithm to it. This yields an approximation $\mat{A} \approx \prod_{\ell = 1}^{L'} \bX_\ell'$ with $L' :=  |\arch'|$ and $(\matseq{X}{\ell'})_{\ell'=1}^{L'} \in \setDBfactor{\arch'}$. By \Cref{lemma:consequence-redundant} we can then construct $(\matseq{\mat{X}}{\ell})_{\ell=1}^L \in \setDBfactor{\arch}$ 
%such that $\prod_{\ell = 1}^{L'} \bX_\ell' = \prod_{\ell = 1}^L \bX_\ell$. this 
that yields an approximation $\mat{A} \approx \prod_{\ell = 1}^{L} \bX_\ell$ with the same approximation error as $\prod_{\ell = 1}^{L'} \bX'_\ell$.
Therefore, in the following we only consider non-redundant chainable architectures.
\end{remark}
%, since they will automatically provide guarantees to \Cref{algo:algoforanychainable}.

% in three steps: remove the redundancy of $\arch$ to obtain a non-redundant architecture $\arch'$ that is expressively equivalent to $\arch$ ($\setButterfly{\arch'} = \setButterfly{\arch}$); construct an approximate solution $(\matseq{\mat{X}'}{\ell'})_{\ell'=1}^{L'}$ via \Cref{algo:modifedbutterflyalgo} to problem \eqref{eq:butterfly-approximation-pb}

% {\bf \em Complexity analysis.}
\subsection{Complexity analysis}
\label{subsec:complexity}
It is not hard to see that the proposed butterfly algorithm (\Cref{algo:modifedbutterflyalgo}) %,algo:algoforanychainable}) 
has polynomial complexity with respect to the sizes of the butterfly factors and the target matrix, since they only perform {a polynomial number of} standard matrix operations such as matrix multiplication, QR and SVD decompositions. 
% The complexity of these versions are provided more explicitly as follows:
\begin{theorem}[Complexity analysis]
    \label{theorem:complexity}
    Consider a chainable architecture $\arch = (\pattern_\ell)_{\ell = 1}^L$ where $\pattern_\ell = (a_\ell,b_\ell,c_\ell,d_\ell)$, and a target matrix $\bA$ of size ${m \times n}$. Define 
    \[
    M_\arch := \max_{\ell \in \intset{L}} a_\ell c_\ell,
    \quad 
    N_\arch = \max_{\ell \in \intset{L}}b_\ell d_\ell.
    \]
    When $\arch$ is non-redundant we have $M_\arch \leq m$ and $N_\arch \leq n$.
     With 
    %Denote $\|\mathbf{r}(\arch)\|_1, \|\mathbf{r}(\arch)\|_\infty$ the $\ell_1$ and $\ell_\infty$ norm of 
    the vector $\mathbf{r}(\arch)$ of~\Cref{def:chainableseqDBPs}, the complexity is bounded by:
    \begin{itemize}
    % [leftmargin=*]
        \item $\mathcal{O}(\|\mathbf{r}(\arch)\|_1M_\arch N_\arch)$ for  \Cref{algo:recursivehierarchicalalgo} %,algo:hierarchicalalgo}, 
%        {in general, and $\mathcal{O}(\|\mathbf{r}(\arch)\|_1 mn)$ when $\arch$ is not redundant};
        \item $\mathcal{O}\left((\|\mathbf{r}(\arch)\|_1 + |\arch|^2\|\mathbf{r}(\arch)\|_\infty)mn\right)$ for \Cref{algo:modifedbutterflyalgo}
        (with a non-redundant $\arch$).
        %when $\arch$ is not redundant.\todo{removed ref to algorithm "for any chainable" here: can remove the corresponding part of the proof ?}
        %\item $\mathcal{O}\left((\|\mathbf{r}(\arch)\|_1 + |\arch|^2\|\mathbf{r}(\arch)\|_\infty)mn + \|\mathbf{r}(\arch)\|_1M_\arch N_\arch\right)$ for \Cref{algo:algoforanychainable}.
    \end{itemize}
\end{theorem}

The proof of \Cref{theorem:complexity} is in \Cref{appendix:complexity}. {The complexity bounds in \Cref{theorem:complexity} are generic for any matrix size $m,n$, chainable $\arch$ and factor-bracketing tree $\mathcal{T}$ (or equivalent permutation $\sigma$). {They can be} %It can be significantly
improved for specific $\arch$. For example, in the case of the square dyadic butterfly, \cite{le2022fastlearning,zheng2023efficient} showed that the complexity of the hierarchical algorithm (\Cref{algo:recursivehierarchicalalgo}) is $\mathcal{O}(n^2)$ where $n = 2^L$ instead of $\mathcal{O}(\|\mathbf{r}(\arch)\|_1n^2) = \mathcal{O}(n^2\log n)$. This is optimal in the sense that it already matches the space complexity of the target matrix.}

\section{Guarantees on the approximation error}
\label{sec:error}

One of the main contributions of this paper is to show that the new butterfly algorithm (\Cref{algo:modifedbutterflyalgo}) outputs an approximate solution to Problem \eqref{eq:butterfly-approximation-pb} that satisfies an error bound of the type \eqref{eq:error-bound-intro}. 
\subsection{Main results}
Our error bounds are based on the following relaxed problem.
\begin{definition}[{First-level factorization}]
 \label{def:firstlevelfactorization-new}
Given a chainable $\arch:= (\pattern_\ell)_{\ell=1}^L$ {with $L \geq 2$}, we define for {each splitting index} $s \in \intset{L-1}$ {the two-factor ``split'' architecture}:
 \begin{equation*}
     \splitarch{\arch}{s} := (\pattern_1 * \ldots * \pattern_s, \pattern_{s+1} * \ldots * \pattern_L).
 \end{equation*}

 {When $L=2$ we have $\splitarch{\arch}{1}=\arch$. For any target matrix $\mat{A}$ we consider the problem}
	\begin{equation}
		\label{eq:firstlevelfact-new}
        E^{\splitarch{\arch}{s}}(\mat{A}) := \min_{(\mbf{X},\mbf{Y}) \in \setDBfactor{\splitarch{\arch}{s}}} \|\mbf{A} - \mbf{X}\mbf{Y}\|_F = \min_{\bB \in \setButterfly{\splitarch{\arch}{s}}} \|\bA - \bB\|_F.
	\end{equation}
\end{definition}

The following two theorems are the central {theoretical} results of this paper. The first one bounds the approximation error of \Cref{algo:modifedbutterflyalgo} in the general case where $\sigma$ can be any permutation. The second one is a tighter bound specific to the case where $\sigma$ is the identity permutation,  corresponding to the so-called {\em unbalanced tree} of \cite{zheng2023efficient}. %\ER{does the correspoding tree have a name? \TL{It is called unbalanced tree in the paper of Leon}}

\begin{theorem}[{Approximation error, arbitrary permutation $\sigma$, \Cref{algo:modifedbutterflyalgo}}]
	\label{theorem:mainresults}
    Let $\arch$ be a \emph{non-redundant chainable} architecture {of depth $L \geq 2$}.
    For any target matrix $\mat{A}$ and permutation $\sigma$ of $\intset{L-1}$ with $L = | \arch |$, 
    \Cref{algo:modifedbutterflyalgo} with inputs $(\bA, \arch, \sigma)$ returns butterfly factors $(\bX_\ell)_{\ell = 1}^L \in \setDBfactor\arch$ such that
	\begin{equation}
	\label{eq:MainBoundTriangle}
        \|\mbf{A} - \matseq{X}{1} \ldots \matseq{X}{L} \|_F \leq \sum_{k = 1}^{L - 1} E^{\splitarch{\arch}
        {\sigma_k}
        }(\mbf{A}).
	\end{equation} 
\end{theorem}
\begin{theorem}[{Approximation error, identity permutation $\sigma$, \Cref{algo:modifedbutterflyalgo}}]
	\label{theorem:bettermainresults}
	Assume that $\sigma$ is either the identity permutation, 
	$\sigma = (1, \ldots, L - 1)$ or its ``converse", $\sigma = (L-1,\ldots,1)$.
	%is the identity permutation, \mRG{would it also work with the flipped identity permutation? \TL{You mean reversed identity permutation? Yes, it works but the proof needs some modifications, e.g., rows to columns, etc.}} 
	Under the assumptions and with the notations of \Cref{theorem:mainresults}, \Cref{algo:modifedbutterflyalgo} with inputs $(\mat{A}, \arch, \sigma)$ 
	returns butterfly factors $(\matseq{X}{\ell})_{\ell=1}^L \in \setDBfactor{\arch}$ such that:
	\begin{equation}
		\label{eq:MainBoundPythagore}
        \|\mbf{A} - \matseq{X}{1} \ldots \matseq{X}{L} \|_F^2 \leq \sum_{k = 1}^{L - 1}[E^{\splitarch{\arch}{k}}(\mbf{A})]^2.
	\end{equation} 
\end{theorem}
{For $L=2$ both results yield $\|\mbf{A} - \matseq{X}{1}\matseq{X}{2} \|_F \leq E^{\arch}(\mat{A})$, i.e. the algorithm is optimal.}

%\RG{Bounding the right-hand side in both theorems in terms of $\max_s E^{\arch_s}(\bA)$ recovers a bound directly comparable to the known bound \eqref{eq:bound-liu}. This is however a pessimistic worst-case estimate, and before} \mTL{I do not think these bounds are comparable. I said it to you but now I realized that I was wrong.}
Before 
proving these theorems in \Cref{sec:proof-mainresults-new}, we state and prove their main consequences: the quasi-optimality of  \Cref{algo:modifedbutterflyalgo}, a ``complementary low-rank'' characterization of butterfly matrices, {and the existence of an optimum for Problem \eqref{eq:butterfly-approximation-pb} when $\arch$ is chainable.}

\subsection{Quasi-optimality of \Cref{algo:modifedbutterflyalgo}}
\label{sec:quasi-optimality}
The theorems imply that butterfly factors obtained via \Cref{algo:modifedbutterflyalgo} satisfy an error bound of the form \eqref{eq:error-bound-intro}.

%\ER{make this a thm?}
\begin{theorem}[{Quasi-optimality of \Cref{algo:modifedbutterflyalgo}}]
	\label{cor:mainresults}
    Let $\arch$ be \emph{any chainable} architecture of arbitrary depth $L \coloneqq |\arch| \geq 1$. For any target matrix $\mat{A}$, the outputs $(\matseq{X}{\ell})_{\ell = 1}^L$ of \Cref{algo:modifedbutterflyalgo} with inputs $(\mat{A}, \arch, \sigma)$ for arbitrary permutation $\sigma$ satisfy:
	\begin{equation}
    \label{eq:mainresults}
		\|\mbf{A} - \matseq{X}{1} \ldots \matseq{X}{L} \|_F \leq ({\max(L,\,2)} - 1) E^\arch(\mbf{A}).
	\end{equation} 
    When $\sigma$ is the identity permutation, the outputs also satisfy
    {the finer bound}:
    	\begin{equation}
		\label{eq:betterbound}
		\|\mbf{A} -  \matseq{X}{1} \ldots \matseq{X}{L} \|_F \leq \sqrt{{\max(L,\,2)} - 1} \, E^\arch(\mbf{A}).
	\end{equation} 
\end{theorem}
{For $L \in \{1,2\}$ the output of \Cref{algo:modifedbutterflyalgo}  is thus indeed optimal.}
%\todo{A pretty similar result, but proved for the case of tensor train decomposition is given in \cite[Corollary 2.4]{oseledets2011tensortrain}. To be discussed. RG: Seems important indeed.}
%{As opposed to \eqref{eq:bound-liu}, the error bounds in this corollary compare the approximation error to the {\em minimal possible} error} {$ E^\arch(\mbf{A})$.}\todo{Léon: do we keep this sentence here? Note that \eqref{eq:bound-liu} has not been introduced yet in the text. \TL{Good point}}
\Cref{tab:constant} summarizes the consequences of \Cref{cor:mainresults} for some standard examples of chainable $\arch$. {The constant $C_\arch$ scales linearly or sub-linearly with respect to ${L} = |\arch|$, the number of factors. Since most part of the existing architectures have length $\mathcal{O}(\log n)$ with $n$ the size of the matrix, the growth of $C_\arch$ is very slow in many practical cases.}

A result reminiscent of \Cref{cor:mainresults} appears in the quite different context of tensor train decomposition \cite[Corollary 2.4]{oseledets2011tensortrain}.
The proof of \Cref{cor:mainresults} has a similar structure and is based on the following lemma. First, we use the fact that the errors in \eqref{eq:firstlevelfact-new} lower bound the error in \eqref{eq:error-bound-intro}, by \Cref{def:firstlevelfactorization-new}.

\begin{lemma}
	\label{lem:lowerboundE-new}
    If the architecture $\arch$ of depth $L := | \arch |$ is chainable then 
    \begin{equation}
	   \label{eq:inclusionBtheta-new} 
    \forall s \in \intset{L-1}, \quad \setButterfly{\arch} \subseteq \setButterfly{\splitarch{\arch}{s}}.
\end{equation}
    Consequently, for any matrix $\mat{A}$ {the quantity $E^\arch(\mbf{A})$ defined in \eqref{eq:butterfly-approximation-pb} satisfies}: 
    \begin{equation}
        E^\arch(\mbf{A}) \geq \max_{1 \leq s \leq L - 1} E^{\splitarch{\arch}{s}}(\mat{A}).
    \end{equation}
\end{lemma}

\begin{proof}
If $\mbf{B} \in \setButterfly{\arch}$, there exist $(\matseq{X}{\ell})_{\ell=1}^L \in \setDBfactor{\arch}$
 such that $\mbf{B} = \matseq{X}{1} \ldots \matseq{X}{L}$.
By \Cref{lem:suppdbfactorprod}, $\bX_1\ldots\bX_s \in \setDBfactor{(\pattern_1 * \ldots *\pattern_s)}$ and $\bX_{s+1}\ldots\bX_L \in \setDBfactor{(\pattern_{s+1} * \ldots * \pattern_L)}$. % for any $s \in \intset{L-1}$. 
\end{proof}

\begin{proof}[Proof of \Cref{cor:mainresults}]
    We start by proving \eqref{eq:mainresults}. We consider two possibilities for the depth $L := |\arch|$ of the non-redundant, chainable architecture $\arch$:
    \begin{itemize}
        \item If $L = 1$: we have $\arch =\{\pattern\}$ for some pattern $\pattern$. The projection of $\bA$ onto $\setDBfactor{\pattern}$ is simply $\bA \odot \bS_{\pattern} \in \setDBfactor{\pattern}$, which is exactly the output computed by the algorithm. Hence the obtained factor $\bX_1$ satisfies $\|\mbf{A} - \matseq{X}{1}\|_F = E^{\arch}(\mbf{A})$.
        \item If $L \geq 2$: by \Cref{lem:lowerboundE-new} and \Cref{theorem:mainresults} $\|\mbf{A} - \matseq{X}{1} \ldots \matseq{X}{L} \|_F \leq (L-1) E^{\arch}(\mbf{A})$.
    \end{itemize}
    In both cases, we have $\|\mbf{A} - \matseq{X}{1} \ldots \matseq{X}{L} \|_F \leq ({\max(L,2)} - 1) E^{\arch}(\mbf{A})$.
    The proof for \eqref{eq:betterbound} is similar, the only difference being that: $\|\mbf{A} - \matseq{X}{1} \ldots \matseq{X}{L} \|_F^2 \leq ({\max(L,2)} - 1) [E^{\arch}(\mbf{A})]^2$. The result is proved by taking the square root on both sides. 
\end{proof}

%\todo{Commented out remark on algorithm "for any chainable"}
% \begin{remark}
%     Because the generic factorization method - \Cref{algo:algoforanychainable} - uses \Cref{algo:modifedbutterflyalgo}, it inherits naturally the bounds \eqref{eq:mainresults} and \eqref{eq:betterbound}, depending on the choice of $\sigma'$  in Line~\ref{line;sigmaprime} of \Cref{algo:algoforanychainable}. In fact, the bounds in \Cref{cor:mainresults} for a redundant architecture can be even tighter since the depth $L'$ of its non-redundant counterpart ($\arch'$ - output of \Cref{algo:redundancyremoval}) is strictly smaller than $L$.
% \end{remark} 

\subsection{{Complementary} low-rank characterization of butterfly matrices}
\label{sec:low-rank-characterization-new}
{Another important consequence of}
\Cref{theorem:mainresults} is a characterization of matrices admitting an exact butterfly factorization associated with a chainable $\arch$.
{This allows (when $\arch$ is chainable) to {\em verify} whether or not a given matrix $\mat{A}$ admits exactly a butterfly factorization associated with $\arch$, by checking the rank of a polynomial number of specific submatrices of $\mat{A}$. This is feasible using SVDs, and contrasts with the \emph{synthesis} definition of $\setButterfly{\arch}$ given by \eqref{eq:butterfly-factorization-intro}, which is a priori harder to verify since it requires checking the {\em existence} of an exact factorization of $\mat{A}$.} 

\begin{definition}[Generalized complementary low-rank property]
\label{def:general-clr-new}
    Consider a chainable architecture $\arch := (\pattern_\ell)_{\ell=1}^L$.
    % and denote
    % \begin{equation}
    % \label{eq:DefSupportAggregated}
    % \mat{S}_{q,t} := \matseq{S}{\pattern_{q} * \ldots * \pattern_t},\quad 1 \leq q \leq t \leq L,    \end{equation}
    % \todo{Puisque cette notation est introduite ici, autant la ré-utiliser dans toutes les preuves où ces matrices apparaissent} 
    A matrix $\mat{A}$ satisfies the \emph{generalized complementary low-rank property} associated with $\arch$ if it satisfies:
         \begin{enumerate}
         \item $\supp(\mat{A}) \subseteq %\mat{S}_{1, L} 
         \bS_{\pattern_1 * \ldots * \pattern_L}
         $;
         \item $\rank(\bA[R_P,C_P]) \leq r(\pattern_\ell, \pattern_{\ell + 1})$ for each $ P \in 
         \cP(
         %\bS_{1, \ell}, \bS_{\ell+1,L}
         \bS_{\pattern_1 * \ldots * \pattern_\ell},\bS_{\pattern_{\ell+1} * \ldots * \pattern_L}
         )$ and $\ell \in \intset{L-1}$ (with the notations of \Cref{def:classequivalence}, \Cref{def:chainableDBfour}).
     \end{enumerate}   
\end{definition}
\revision{An illustration of the complementary low-rank property is given in \Cref{fig:complementary-low-rank} of \Cref{app:low-rank-characterization}.}
We show in \Cref{cor:equivalence-clr} of \Cref{app:low-rank-characterization} that this generalized definition indeed coincides with the notion of a complementary low-rank property (\Cref{def:complementary-low-rank}) from the literature \cite{li2015butterfly}, for every architecture $\arch$ with patterns such that $a_1 = d_L=1$, i.e., architectures such that $\setButterfly{\arch}$ contains some dense matrices, see \Cref{rem:denseornot}.

The following results show that a matrix admits an exact butterfly factorization associated with $\arch$ if, and only if, it satisfies the associated generalized complementary low-rank property. Note that the complementary low-rank property induced by a chainable butterfly architecture requires the same low-rank constraint for all submatrices at the same level $\ell \in \intset{L}$, as opposed to the classical complementary low-rank property (\Cref{def:complementary-low-rank}) where these constraints can be different for the submatrices at a same level $\ell \in \intset{L}$.

\begin{corollary}[Characterization of $\setButterfly{\arch}$ for chainable $\arch$]
	\label{cor:characterizationofDBmatrix}
	If $\arch := (\pattern_\ell)_{\ell=1}^L$ is chainable with $L \geq 2$ then, with the notations of \Cref{def:firstlevelfactorization-new} and \Cref{lem:prodDB2}:
	\begin{equation}
    \label{eq:low-rank-characterization-new}
	   {\setButterfly{\arch} = \bigcap_{\ell = 1}^{L - 1} \setButterfly{\splitarch{\arch}{\ell}} = \setDBfactor{(\pattern_1 * \ldots * \pattern_L)} \cap \bigcap_{\ell=1}^{L-1} \setblocklowrank{\splitarch{\arch}{\ell}}.}
	\end{equation}
\end{corollary}

\begin{proof}
    The second equality in \eqref{eq:low-rank-characterization-new} is a reformulation based on \Cref{lem:prodDB2}, so it only remains to prove the first equality.
    The inclusion $\setButterfly{\arch} \subseteq \bigcap_{\ell = 1}^{L - 1} \setButterfly{\splitarch{\arch}{\ell}}$ is a consequence of \Cref{lem:lowerboundE-new}. We now prove the other inclusion.
    
    First, consider the case of a non-redundant $\arch$.
    For $\bA \in \bigcap_{\ell = 1}^{L - 1} \setButterfly{\splitarch{\arch}{\ell}}$, we have $E^{\splitarch{\arch}{\ell}}(\mat{A}) = 0$ for each $\ell \in \intset{L-1}$. By \Cref{theorem:mainresults}, \Cref{algo:modifedbutterflyalgo} with inputs $(\mat{A}, \arch, \sigma)$ and arbitrary permutation $\sigma$ returns $(\matseq{X}{\ell})_{\ell=1}^L \in \setDBfactor{\arch}$ such that  $\|\bA - \matseq{X}{1} \ldots \matseq{X}{L}\|_F \leq \sum_{\ell=1}^{L-1} E^{\splitarch{\arch}{\ell}}(\mat{A}) = 0$. Thus, $\|\bA - \matseq{X}{1} \ldots \matseq{X}{L}\|_F = 0$ and $\bA = \matseq{X}{1} \ldots \matseq{X}{L} \in \setButterfly{\arch}$. This proves $\bigcap_{\ell = 1}^{L - 1} \setButterfly{\splitarch{\arch}{\ell}} \subseteq \setButterfly{\arch}$.

    For redundant $\arch$, consider $\arch' =(\pattern'_\ell)_{\ell=1}^{L'}$ returned by the redundancy removal algorithm (\Cref{algo:redundancyremoval}) with input $\arch$. By \Cref{prop:procedure-for-redundant-arch}: $\setButterfly{\arch'}=\setButterfly{\arch}$. Moreover, by the same proposition, $\arch'$ is of the form $(\pattern_1 * \ldots * \pattern_{\ell_1}, \pattern_{\ell_1+1} * \ldots * \pattern_{\ell_2}, \ldots, \pattern_{\ell_p+1} * \ldots * \pattern_L)$ for some indices $1 \leq \ell_1 < \ldots < \ell_p < L$ with $p \in \intset{L-1}$. Therefore, for any $s \in {\intset{L'-1}}$, there exists  $\ell(s) \in {\intset{L - 1}}$
    %\ER{replace $s$ with $\ell$?} 
    such that 
    $\splitarch{\arch'}{{s}} = \splitarch{\arch}{\ell(s)}$,
%    $\splitarch{\arch}{{s}} = \splitarch{\arch'}{s'}$,
{by associativity of the operator $*$ (\Cref{lem:associativity}).} Thus,
    \begin{equation*}
        \bigcap_{\ell = 1}^{L-1} \setButterfly{\splitarch{\arch}{\ell}} \subseteq %\bigcap_{\ell' = 1}^{L'-1}
        \bigcap_{s=1}^{L'-1}
        \setButterfly{\splitarch{\arch}{\ell(s)}} 
        =
        \bigcap_{s=1}^{L'-1}
        \setButterfly{\splitarch{\arch'}{s}} %\ell'}} 
        = \setButterfly{\arch'} = \setButterfly{\arch},
    \end{equation*}
    where in the first equality we used the result proved above for non-redundant $\arch'$.
\end{proof}

\subsection{Existence of an optimum}
\Cref{cor:characterizationofDBmatrix} 
also allows us to prove the existence of optimal solutions for Problem \eqref{eq:butterfly-approximation-pb} when $\arch$ is chainable.

\begin{theorem} [Existence of optimum in butterfly approximation]
    \label{cor:optimum-exists-in-bf}
    If $\arch$ is chainable, then for any target matrix $\bA$, Problem \eqref{eq:butterfly-approximation-pb} admits a minimizer.
\end{theorem}
\begin{proof}
    The set of matrices of rank smaller than a fixed constant is closed, and closed sets are stable under finite intersection, so by the characterization of $\setButterfly{\arch}$ from \Cref{cor:characterizationofDBmatrix}, the set $\setButterfly{\arch}$ is closed. Therefore, Problem \eqref{eq:butterfly-approximation-pb} is equivalent to a projection problem on the non-empty (it contains the zero matrix) closed set $\setButterfly{\arch}$, hence it always admits a minimizer.
\end{proof}

The rest of the section is dedicated to the proofs of \Cref{theorem:mainresults,theorem:bettermainresults}. Readers more interested in numerical aspects of the proposed butterfly algorithms can directly jump to \Cref{sec:experimentbutterfly}.

\subsection{Proof of \Cref{theorem:mainresults,theorem:bettermainresults}}
\label{sec:proof-mainresults-new}
{Consider an iteration number $J \in \intset{L-1}$, and denote  $(I_k)_{k=1}^J$ the partition obtained after line \ref{line:partition_update} and}
$({\bX}_{I_k})_{k = 1}^J$ the list $\texttt{factors}$ obtained \emph{after} the pseudo-orthonormalization operations in lines~\ref{line:beginfor1}-\ref{line:endfor2}, 
at the $J$-th iteration of \Cref{algo:modifedbutterflyalgo}. 
With $s := \sigma_J$ and $j$ defined in line~\ref{line:defjmodifiedalgo} {of \Cref{algo:modifedbutterflyalgo}} and
 $\intset{q,t} := I_j$, %$s := \sigma_J$,
 denote
\begin{align}
    \bX_{\tleft}^{(J)} := {\bX}_{I_1} \ldots {\bX}_{I_{j - 1}}, \label{eq:productleftright-new}
    \qquad
    \bX_{\tright}^{(J)} := {\bX}_{I_{j+1}} \ldots {\bX}_{I_J}, %\label{eq:productright-new}
\end{align}
with {the} convention that $\bX_{\tleft}^{(J)}$ (resp.~$\bX_{\tright}^{(J)}$) is the identity matrix of size $a_1b_1d_1$ if $j = 1$ (resp.~$a_L c_L d_L$ if $j = J$).
We also denote $(\matseq{X}{\intset{q,s}}, \matseq{X}{\intset{s+1, t}})$ the matrices computed in line~\ref{line:fsmfmodifiedalgo} {of \Cref{algo:modifedbutterflyalgo}}, as well as
\begin{equation}
\label{eq:DefBJ}
{\mbf{B}}_J := \mat{X}_{\tleft}^{(J)} \matseq{X}{\intset{q,s}} \matseq{X}{\intset{s+1, t}} \mat{X}_{\tright}^{(J)}
\end{equation}
and $R_J := \|\mbf{A} - {\mbf{B}}_J \|_F$. 
{Note that ${\mbf{B}}_J$ is the product of butterfly factors in the list \texttt{factors} at the end of the iteration $J$ (line~\ref{line:list_factors_at_end}).}
In particular, ${\mat{B}}_{L-1} \in \setButterfly{\arch}$ is the product of the butterfly factors returned by the algorithm after $L-1$ iterations. By convention we also define ${\mbf{B}}_0 := \mat{A}$ and $R_0 := 0$. 

Our goal is to control $R_{L-1} = \|\mbf{A} - {\mbf{B}}_{L-1} \|_F$. To this end, it is sufficient to track the evolution of the sequence $(\bB_0, \ldots, \bB_{L-1})$. The following lemma enables a description for the relation between two consecutive matrices $\bB_{J-1}$ and $\bB_J$, $1 \leq J \leq L - 1$.

\begin{lemma}
    \label{lemma:orthogonal-left-n-right}
    Consider a chainable architecture $\arch:= (\pattern_\ell)_{\ell = 1}^L$ and a partition of $\intset{L}$ into consecutive intervals $\{I_1, \ldots, I_J\}$. For each $i \in \intset{J}$, let $I_i = \intset{q_i,t_i}$, $\bX_{I_i}$ be a $(\pattern_{q_i} * %\pattern_{r_i + 1} * 
    \ldots * \pattern_{t_i})$-factor  {and $\bB := \bX_{I_1}\ldots\bX_{I_{J}}$}. Given $j \in \intset{J-1}$, if each $\bX_{I_i}$ for $i = 1, \ldots, j - 1$ is left-$r(\pattern_{t_i},\pattern_{t_{i} + 1})$-unitary, and if each $\bX_{I_i}$ for $i = j + 1, \ldots, J$ is right-$r(\pattern_{q_{i} - 1}, \pattern_{q_{i}})$-unitary, then for any optimal factorization ({solution} of \eqref{eq:fsmf-DB}) $(\matseq{X}{\intset{q,s}}, \matseq{X}{\intset{s+1, t}})$ of $\bX_{I_j}$ with $q = q_j, t = t_j$, 
    {the matrix $$\bB' := (\bX_{I_{1}}\ldots\bX_{I_{j-1}})\matseq{X}{\intset{q,s}} \matseq{X}{\intset{s+1, t}}(\bX_{I_{j+1}}\ldots\bX_{I_J})$$ satisfies}
    \begin{equation}\label{eq:BJprojection}
    \bB' \in \cB^{\arch_{s}}\quad \text{and}\quad 
        \|\bB' - \bB\|_F = E^{{\arch_s}}(\bB),
    \end{equation}
    with $\arch_s$ as in \Cref{def:firstlevelfactorization-new}.
\end{lemma}
This lemma {(proved in \Cref{appendix:keylemmaproof-new})} has a direct corollary ({obtained} by combining it with \Cref{lemma:role-pseudo-orthogonality}, {which ensures that each factor $\mathbf{X}_{I_i}$ is indeed $r$-unitary as needed}).
\begin{lemma}
    \label{lemma:relation-bl-bl+1}
    With the setting {of} {\Cref{theorem:mainresults}}, %\mRG{Expliciter les hypotheses pour que ce lemme soit autosuffisant? A discuter} 
    for $J \in \intset{L-1}$, {the matrix} $\bB_J$ {defined in \eqref{eq:DefBJ}} is {a projection} of $\bB_{J - 1}$ onto $\cB^{\arch_{\sigma_J}}$ (cf.~\Cref{def:firstlevelfactorization-new}), i.e., 
    \begin{equation*}
    \bB_J \in \cB^{\arch_{\sigma_J}}\quad \text{and}\quad 
        \|\bB_{J-1} - \bB_J\|_F = E^{
        %\cB
        \arch
        _{\sigma_J}}(\bB_{J-1}).
    \end{equation*}
\end{lemma}

{These two lemmas show the role of the pseudo-orthonormalization, since without this operation and its consequence, i.e., \Cref{lemma:role-pseudo-orthogonality}, {the conclusions of} \Cref{lemma:relation-bl-bl+1} would not hold. Thus, we can describe the sequence $\bB_j$ for $j = 1, \ldots, L - 1$ as {a sequence of subsequent projections}:
\begin{equation*}
    \bB_0 = \bA \overset{\text{Proj}_{\cB^{\arch_{\sigma_1}}}}{\longrightarrow} \bB_1 \overset{\text{Proj}_{\cB^{\arch_{\sigma_2}}}}{\longrightarrow} \ldots \overset{\text{Proj}_{\cB^{\arch_{\sigma_{L-1}}}}}{\longrightarrow} \bB_{L-1} \in \cB^\arch.
\end{equation*}
where $\text{Proj}_S$ is the projection operator onto the set $S$. We note that since there might be more than one projector, $\text{Proj}_S$ is a set-valued mapping.}

Moreover, the sequence of architectures $(\arch_s)_{s=1}^{L-1}$ defined in \Cref{def:firstlevelfactorization-new} possesses another nice property, described in the following lemma:

\begin{lemma}
    \label{lemma:relation-projectors}
    Consider a chainable architecture $\arch = (\pattern_\ell)_{\ell = 1}^L$, a matrix $\bM$ of appropriate size, and integers $s,q \in \intset{L - 1}$. If $\bN$ is a projection of $\bM$ onto $\cB^{\arch_{q}}$ then
    \begin{equation}
        E^{\arch_s}(\bM) \geq E^{\arch_s}(\bN)
    \end{equation}
    %where $\bM_r$ is a projection of $\bM$ onto $\cB^{\arch_r}$. % and $\arch_s$ is defined as in \Cref{def:firstlevelfactorization-new}. 
\end{lemma}
\Cref{lemma:orthogonal-left-n-right} and \Cref{lemma:relation-projectors} are proved in \Cref{appendix:keylemmaproof-new} and \Cref{appendix:relation-projectors} respectively. In the following, we admit these results and prove \Cref{theorem:mainresults,theorem:bettermainresults}.

\begin{proof}[Proof of \Cref{theorem:mainresults,theorem:bettermainresults}]
    First, we prove that: 
    \begin{equation}
        \label{eq:comparison-to-A}
        E^{\arch_s}(\bA) \geq E^{\arch_s}(\bB_J), \quad \forall s,J \in \intset{L-1}. 
    \end{equation}
    Indeed, for any $J \in \intset{L-1}$, by \Cref{lemma:relation-bl-bl+1}, $\bB_{J}$ is a projection of $\bB_{J-1}$ onto 
    %$\cB^{\arch_{\sigma_J}}$. 
    $\cB^{\splitarch{\arch}{\sigma_J}}$.
    %with $q:= \sigma_J$. 
    %\ER{propagate $r$ below?(we still use $\sigma_J$ below) \RG{Done + $r \to q$} and do we want to change letter now that $r$ is the rank?}
   % \ER{I think we don't need to introduce $q:= \sigma_J$.}
    Hence, by~\Cref{lemma:relation-projectors}:
    \begin{equation*}
        E^{\arch_s}(\bB_{J-1}) \geq E^{\arch_s}(\bB_J), \quad \forall s \in \intset{L-1}.
    \end{equation*}
    {Since $\bA=\bB_0$}, applying the above inequality recursively yields \eqref{eq:comparison-to-A}.

    % We now derive the error bound in \Cref{theorem:mainresults}, which is true for an arbitrary permutation $\sigma$, as follows:
    % \begin{equation*}
    %     \begin{aligned}
    %         \|\bA - \bB_{L-1}\| &= \|(\bB_0 - \bB_1) + (\bB_1 - \bB_2) + \ldots + (\bB_{L - 2} - \bB_{L - 1})\|_F\\
    %         &\leq \sum_{J = 1}^{L-1} \|\bB_{J - 1} - \bB_J\|_F \qquad (\text{triangle inequality})\\
    %         &= \sum_{\ell = 1}^{L-1} E^{\arch_{\sigma_J}}(\bB_{J-1}) \qquad (\text{by~\Cref{lemma:relation-bl-bl+1}})\\
    %         &\leq \sum_{J = 1}^{L-1} E^{\arch_{\sigma_J}}(\bA) \qquad (\text{by~\Cref{eq:comparison-to-A}})\\
    %         &= \sum_{s = 1}^{L-1} E^{\arch_{s}}(\bA) \qquad (\text{since $\sigma$ is a permutation of $\intset{L-1}$}).
    %     \end{aligned}
    % \end{equation*}
    { %Replace?: 
    We now derive the error bound in \Cref{theorem:mainresults}, which is true for an arbitrary permutation $\sigma$, as follows. 
    \begin{equation*}
        \begin{aligned}
            \|\bA - \bB_{L-1}\| &= \|(\bB_0 - \bB_1) 
            %+ (\bB_1 - \bB_2) 
            + \ldots + (\bB_{L - 2} - \bB_{L - 1})\|_F 
            \leq \sum_{J = 1}^{L-1} \|\bB_{J - 1} - \bB_J\|_F 
             \end{aligned}
    \end{equation*}
    and for each $J \in \intset{L-1}$ it holds:
    %\mRG{Ici il faut garder la dépendance en $J$ donc la notation $\sigma_J$}
     \begin{equation*}
           \|\bB_{J - 1} - \bB_J\|_F\overbrace{=}^{\textrm{\Cref{lemma:relation-bl-bl+1}}}  E^{\arch_{\sigma_J}}(\bB_{J-1}) 
            \overbrace{\leq}^{\textrm{\Cref{eq:comparison-to-A}}} E^{\arch_{\sigma_J}}(\bA)  
    \end{equation*}
    and since $\sigma$ is a permutation of $\intset{L-1}$ it holds $\sum_{J = 1}^{L-1} E^{\arch_{\sigma_J}}(\bA) 
            = \sum_{s = 1}^{L-1} E^{\arch_{s}}(\bA) $. }
    The proof when $\sigma$ is the identity or its ``converse" is %nearly 
    similar: we only replace the first 
    %two lines 
    {equation} by:
    \begin{equation*}
        \|\bA - \bB_{L-1}\|^2 = \sum_{J = 1}^{L-1} \|{\mbf{B}}_{J - 1} - {\mbf{B}}_J\|_F^2
    \end{equation*}
    Indeed, as proved in \Cref{appendix:bettermainresultsproof-new} using an orthogonality argument\footnote{This argument has been used in \cite{liu2021butterfly} to prove \eqref{eq:bound-liu}. We adapt their argument to our context for the self-containedness of our paper.}, we have
    %\mRG{Does the orthogonality condition also hold for other $\sigma$, e.g. the flipped identity?\TL{No. I tested it numerically.}\RG{what about if we replace $p\geq J$ by $p \leq J$?}\TL{Ah ok, ONLY for flipped identity, we have a similar orthogonality argument. For general permutation, the property fails. I thought you ask for general permutation.}}
	\begin{equation}
		\label{eq:orthogonalpair-new}
		\forall J {\in \intset{L-1}},\quad \forall p \in \intset{J,L-1}, \quad \ip{{\mbf{B}}_{J - 1} - {\mbf{B}}_{J}}{{{\mbf{B}}_{p}}} = 0.
    \end{equation}
    Hence:
	\begin{equation*}
		\begin{aligned}
			\|\mbf{A} - \mbf{B}_{L-1}\|_F^2 &= \|({\mbf{B}}_0 - {\mbf{B}}_1) + \ldots + ({\mbf{B}}_{L - 2} - {\mbf{B}}_{L-1})\|_F^2\\ 
			&= \sum_{J = 1}^{L-1} \|{\mbf{B}}_{J - 1} - {\mbf{B}}_J\|_F^2 + 2\sum_{J = 1}^{L-1} \sum_{p > J}\ip{{\mbf{B}}_{J - 1} - {\mbf{B}}_{J}}{{\mbf{B}}_{p-1} - {\mbf{B}}_{p}}\\
			&\overset{\eqref{eq:orthogonalpair-new}}{=} \|{\mbf{B}}_0 - {\mbf{B}}_1\|_F^2 + \ldots + \|{\mbf{B}}_{L - 2} - {\mbf{B}}_{L-1} \|_F^2.
		\end{aligned}
	\end{equation*}
\end{proof}

\subsection{Comparison with existing error bounds}
\label{sec:comparison}
% \todo{NEW}
% \todo{Léon: continue here, il faut reformuler certaines phrases (closest, four parameters / six parameters}
The result in the literature that 
% is closest
is most related
to our proposed error bounds (\cref{theorem:mainresults} and \cref{theorem:bettermainresults}) is the bound in  \cite{liu2021butterfly}, which  takes the form 
\begin{equation}
\label{eq:bound-liu}
    \| \mathbf{A} - \hat{\mathbf{A}} \|_F \leq C_n \epsilon_0 \| \mathbf{A} \|_F, \quad \text{with } C_n = \mathcal{O}({\sqrt{\log n}}),
\end{equation}
where $\mathbf{A}$ is an $n \times n$ matrix and $\epsilon_0$ is the maximum relative error $\| \mathbf{M} - \hat{\mathbf{M}} \|_F / \| \mathbf{M} \|_F$ across all blocks $\mathbf{M}$ on which the algorithm performs low-rank approximation, with $\hat{\mathbf{M}}$ a best low-rank approximation of $\mathbf{M}$. 

It is noteworthy that these bounds are not immediately comparable because of the difference in problem formulations. First of all, the sparsity patterns of the factors in \cite{liu2021butterfly} cannot be expressed by four parameters of a pattern $\pattern$. Instead, using the Kronecker supports introduced in this work, the support constraints in \cite{liu2021butterfly} can be described using six parameters $(a,b,c,d,e,f)$ as $\bI_a \otimes \mbf{1}_{b \times c} \otimes \bI_d \otimes \mbf{1}_{e \times f}$, 
% one can use six parameters $(a,b,c,d,e,f)$ and describe the support constraint by $\bI_a \otimes \mbf{1}_{b \times c} \otimes \bI_d \otimes \mbf{1}_{e \times f}$, 
a structure also referred to as the ``block butterfly'' \cite{chen2022pixelated}. 
% This is also known as `block butterfly''
However, these supports can be transformed into supports of the form considered in \cref{def:dbfactorfour}  by multiplying by appropriate permutation matrices $\bP, \bQ$ to the left and right, respectively, such that $\bP(\bI_a \otimes \mbf{1}_{b \times c} \otimes \bI_d \otimes \mbf{1}_{e \times f})\bQ = \bI_{a} \otimes \mbf{1}_{be \times cf} \otimes \bI_{d}$.
% \cite[Appendix C.2]{gonon2024inference}. 
% Thus, the sparsity structures can be considered the same.
Thus, the sparsity structures can be regarded as equivalent.
    % However and more importantly
    But more importantly, in the formulation of the factorization problem, in \cite{liu2021butterfly} the architecture is not fixed {\em a priori} as 
    % ours. 
    it is in our case.
  %  Instead, given a \emph{fixed} level of relative error  $\epsilon$ and \emph{fixed} number of low-rank constraints, \cite{liu2021butterfly} will find {\em a posteriori} a sparse architecture that allows for an $\epsilon$-approximation that satisfies the considered low-rank constraints.
   Instead, given a \emph{fixed} level of relative {error} $\epsilon$, \cite{liu2021butterfly} will find {\em a posteriori} a sparse architecture associated to some low-rank constraints that are adapted to the considered input matrix, so that the low-rank approximations in the butterfly algorithm satisfy the target error $\epsilon$ for the considered input matrix.
    % a sparse but sufficiently \todo{what do we mean by "sufficiently dense" and "viable" here?} dense architecture so that $\epsilon$-approximation is viable. 
    
    %Even though the constant $C_n$ in the error bound grows slowly with respect to the matrix size $n$, 
    Even if the bounds are not directly comparable, we can say that our bound improves with respect to  \eqref{eq:bound-liu} because (i) it allows us to compare the approximation error $\| \mathbf{A} - \hat{\mathbf{A}} \|_F$ to the \emph{best} approximation error, that is, the minimal error $\| \mathbf{A} - \mathbf{A}^* \|_F$ with $\mathbf{A}^*$ satisfying \emph{exactly} the complementary low-rank property associated to the prescribed butterfly architecture; (ii) it holds for any factor-bracketing tree (\Cref{def:factor-bracketing-tree}), while \eqref{eq:bound-liu} is tight to a specific tree; (iii) %it
    the constant $C_\arch$ is computable \emph{a priori}, while  the quantity $\epsilon_0$ in \eqref{eq:bound-liu} can only be determined algorithmically {\em after} applying the butterfly algorithm on the target matrix $\mathbf{A}$, i.e., when $\hat{\bA}$ is available and the error $\|\bA-\hat{\bA}\|_F$ can be directly computed. 
Moreover, the algorithm in \cite{liu2021butterfly} suffers from the same limitations  of \Cref{algo:recursivehierarchicalalgo} (cf.  \Cref{example:illposedbX1}): $\epsilon_0$ can be {\em strictly positive} (and arbitrarily close to one) for some target matrices $\bA$ {\em even if} they satisfy exactly the complementary low-rank property.
    
    %Thus, if one forces a comparison between our results and \cite{liu2021butterfly}, then that of \cite{liu2021butterfly} can be given by:
  %  \begin{equation}
   %     \|\mbf{A} - \matseq{X}{1} \ldots \matseq{X}{L} \|_F \leq \sqrt{L-1}\epsilon\|\bA\|,
    %\end{equation}
    %where $\epsilon$ is the maximum relative error $\|\bX_{\intset{q,t}} - \bX_{\intset{q,s}}\bX_{\intset{s+1,t}}\|_F/\|\bX_{\intset{q,t}}\|_F$ when solving \eqref{eq:fsmf-DB}.

\section{{Numerical experiments}}
\label{sec:experimentbutterfly}
% In this section, w
% {We} 
% % conduct several experiments to 
% compare \Cref{algo:hierarchicalalgo}, \Cref{algo:modifedbutterflyalgo}, and other 
% % available
% \LZ{existing}
% algorithms in the literature
% \LZ{for solving numerically different instances of \eqref{eq:butterfly-approximation-pb}.}
We now illustrate the empirical behaviour of the proposed hierarchical algorithm for Problem \eqref{eq:butterfly-approximation-pb}.
All methods are implemented in Python 3.9.7 using the \texttt{PyTorch 2.2.1} package. Our implementation codes are given in \cite{le2024codebutterfly}.
%\todo{Donner un lien vers du code \TL{OK}}
%\todo{Cela a-t-il un sens de faire une implem pyfaust ? Le code sera-t-il diffusé ?} \mLZ{oui ce serait top}
Experiments are conducted on an Apple M3, 2.8 GHz, in float-precision. The following experiments consider the decomposition of real-valued matrices, so we 
% }Since all the target matrices in the experiment are of real values, we
implement \Cref{algo:algorithm1,algo:recursivehierarchicalalgo,algo:modifedbutterflyalgo} with real {numbers} instead of complex ones.
% \LZc{le code actuel n'est que implémenté pour les matrices réels. Sera disponible à un endroit (dépot hal par exemple, code pytorch).} 
% \todo{à changer sur anchiale1 machine du cbp. Pour l'instant: \Cref{subsec:exp-hierarchical-vs-gradient} fait sur l'ordinateur de Tung, \Cref{subsec:exp-ortho} fait sur CBP anchiale1.}
%\mRG{Serait-ce un bon endroit pour dire que concretement quand on a une matrice $m \times n$ il faut choisir $\arch$, et qu'on peut s'appuyer sur \Cref{lem:archfromsize}? \TL{Peut-etre plus tard.}} 
\subsection{Hierarchical algorithm \emph{vs.}~existing methods}
\label{subsec:exp-hierarchical-vs-gradient}
We consider Problem \eqref{eq:butterfly-approximation-pb} associated with the square dyadic butterfly architecture with $L=10$ factors. The target matrix $\mat{A}$ is the $1024 \times 1024$ Hadamard matrix. % of size $1024 \times 1024$. 
The compared methods are:

% \LZc{TODO}

% We compare the proposed hierarchical algorithm with and without orthonormalization operations (\Cref{algo:hierarchicalalgo,algo:modifedbutterflyalgo}) to a gradient-based method \cite{dao2019learning} and an alternating least square algorithm (ALS) \cite{lin2021deformable}. Below is a description of these methods.
\begin{itemize}
	\item {\bf \em Our butterfly algorithm
% \mRG{ici on sait que $\arch$ est non-redondante. Si on ne le savait pas on ferait tourner \Cref{algo:algoforanychainable}. Quel serait le surcout ? 
 % \LZ{on ferait en effet tourner \Cref{algo:redundancyremoval} et si $\arch' = \arch$ alors il n'y a pas beaucoup de surcoût. Sinon, il y aurait peut être un surcoût mais faible également car il y a moyen de faire la ligne 4 de \Cref{algo:algoforanychainable} rapidement}
 %} 
 (%\Cref{algo:hierarchicalalgo} and 
 \Cref{algo:modifedbutterflyalgo}, with or without pseudo-orthonormalization operations)}: we use the permutation $\sigma$ of $\intset{10}$ corresponding to a balanced factor-bracketing tree \cite{zheng2023efficient,le2022fastlearning}
 (see  \Cref{fig:factor-bracketing-tree-balanced} in \Cref{app:MoreNumerical}), i.e., $\sigma = (5, 2, 1, 3, 4, 7, 6, 8, 9)$. Since $\bA$ admits an exact factorization, the results (except computation times) are not expected to depend on $\sigma$ as already documented  \cite{zheng2023efficient,le2022fastlearning}. %The role of $\sigma$ will soon be investigated in the noisy case.\todo{Enelver si on ne le fait pas}

 % \LZc{Faire plutôt une figure (ici ou annexe) qui illustre l'arbre avec sur les noeuds les valeurs de sigma}
% \mRG{Pourquoi cette permutation?}
 % \LZ{ça correspond en fait à au balanced factor-bracketing tree}
 %corresponding to the \emph{balanced} factor-bracketing tree, defined as the binary tree where the nodes of a same level have the same cardinal.
 % $\sigma=(1, \ldots, L-1)$. 
 % $P = \{1, \ldots, N - 1\}$ (left-to-right factorization). 
    In line~\ref{line:svd} of the two-factor fixed support matrix factorization algorithm (\Cref{algo:algorithm1}), 
    the best low-rank approximation is computed via truncated SVD: following \cite{le2022fastlearning}, we compute the full\footnote{For the considered matrix dimension this is faster than partial SVD. In higher dimensions, partial SVD can further optimize the algorithm, see the detailed discussion in \cite{zheng2023efficient}.} SVD  $\bU\bD\bV^\top$ {of} the submatrix $\bA[R_P, C_P]$ where the diagonal entries of $\mat{D}$ are the singular values in decreasing order, and we set the factors $\bH = \bU[\col{1}] \mat{D}^{1/2}$  and $\bK = {\bD}^{1/2}\bV[\col{1}]^\top$ {since $r=1$}. %\todo{est-ce correct?}
    %\todo{pourquoi noter $\bU[\col{\intset{r}}]$ ce qui est juste $\bU[\col{1}]$ ?}
    %\LZc{ajouter footnote: for further optimization, voir papier simods pour des solvers SVDs différents.} 
    % {Note that we used \Cref{algo:modifedbutterflyalgo} instead of \Cref{algo:algoforanychainable} because we know that the considered architecture $\arch$ is not redundant}\footnote{{With a redundant architecture $\arch$, we would have to use \Cref{algo:algoforanychainable}, incurring an overhead compared to \Cref{algo:modifedbutterflyalgo} due to line~\ref{line:reconstruct-factors-in-beta}, but it should be negligeable according to \Cref{theorem:complexity}.}}.
     % and } $\bU[\col{\intset{r}}]$ (resp.~$\bV[\col{\intset{r}}]$) is the submatrix corresponding to the first $q$ columns of $\bU$ (resp.~$\bV$)
    % \mLZ{attentio le texte dit que H est orthonormalisé mais dans le code ce n'est pas le cas? refaire les expériences par sécurité.}
    %\mLZ{Pour l'instant on fait l'implémentation naive de l'approximation de rang faible, mais il y a sans doute mieux à faire, cf.~le papier SIMODS où l'on teste pleins de solver.}
    
	\item {\bf \em Gradient-based method} \cite{dao2019learning}: 
    % Thanks to the parameterization $\mat{A} = \matseq{X}{1} \ldots \matseq{X}{L}$ of a butterfly matrix $\mat{A} \in \setButterfly{\arch}$ given by  \LZ{\eqref{eq:butterfly-param} for a $\arch = (\pattern_\ell)_{\ell=1}^L$}, 
    Using the parameterization $ \matseq{X}{1} \ldots \matseq{X}{L}$ for a butterfly matrix in  $\setButterfly{\arch}$, this method uses (variants of) gradient descent to optimize all nonzero entries $\matseq{X}{\ell}[i, j]$ for $(i, j) \in \supp(\matseq{S}{\pattern_\ell})$ and $\ell \in \intset{L}$ to minimize \eqref{eq:butterfly-approximation-pb}. We use the protocol of \cite{dao2019learning}: we perform $100$ iterations of ADAM\footnote{The learning rate is set as $0.1$, and we choose $(\beta_1, \beta_2) = (0.9, 0.999)$.} \cite{adam}, followed by $20$ iterations of L-BFGS \cite{liu1989limited}\footnote{L-BFGS terminates when the norm of the gradient is smaller than $10^{-7}$.}. 
    % \LZ{Our implementation for this method in our experiments is faster than implementations from \cite{dao2019learning,lin2021deformable}.}  
    Besides directly benchmarking the implementation from \cite{dao2019learning}, we propose a new implementation of this gradient-based method which is faster than the one of \cite{dao2019learning}. Please consult our codes \cite{le2024codebutterfly} for more details.
    %\todo{TODO: add "see the reproducible code \href{?} for details \TL{OK}}.
    %(see  \Cref{app:detail-new-implementation-gradient-based} for more details about this new implementation).} \ER{keep this? details are missing in the appendix?}
   % \mLZ{Besides directly benchmarking the implementation from [22], we also benchmark a custom faster implementation (donner détail de quoi il s'agit). Quelques lignes en annexe. Ajouter sur le graph une implémentation de base et la nouvelle?} 
    % \LZc{je m'en occupe}
    % \mLZ{nouveau: on propose un algo rapide pour multiplier les twiddles}\mRG{Cad? veut-on expliquer/documenter qq part (en annexe ?)} 
    % \LZ{en gros on a amélioré la méthode gradient based par rapport à l'implémentation de Dao.}
    % In this experiment, we choose $a = 100, b = 20$ to ensure the convergence. 
	\item {\bf \em Alternating least square{s} (ALS)} \cite{lin2021deformable}: At each iteration of this iterative algorithm, we optimize the nonzero entries of a given factor $\matseq{X}{\ell}$ for some $\ell \in \intset{L}$ while fixing the others, by solving a linear regression problem. 
 % We reuse the ALS implementation of \cite{lin2021deformable}.  
\end{itemize}

\begin{figure}[t]
	\centering
	\includegraphics[width=\textwidth]{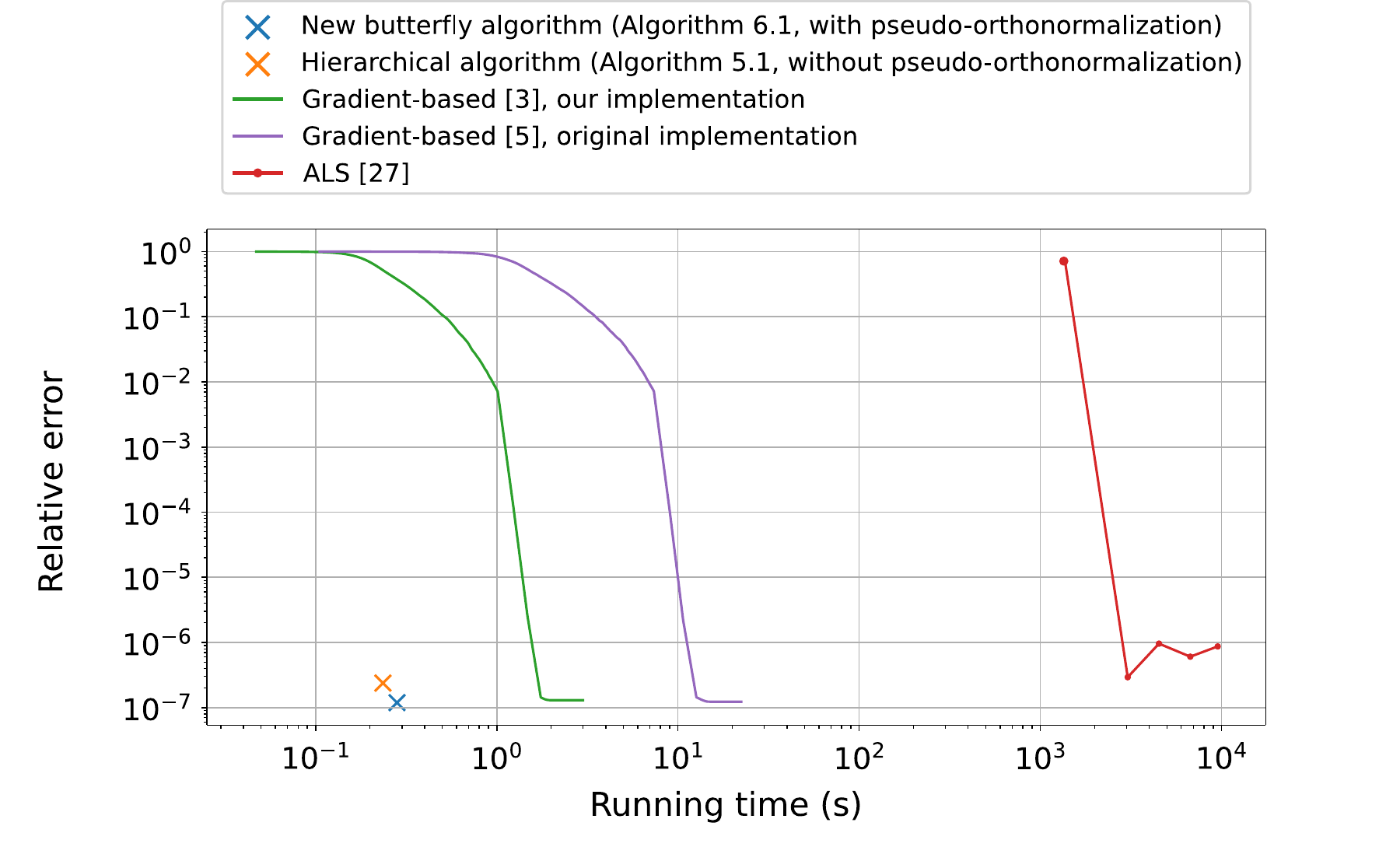}
	\caption{Relative approximation errors {defined} as $\|\bA - \hat{\bA}\|_F / \|\bA\|_F$ \emph{vs.} running time of the different algorithms. The target matrix $\mat{A}$ is the Hadamard matrix of size $1024 \times 1024$, and $\hat{\mat{A}}$ is the computed approximation for Problem \eqref{eq:butterfly-approximation-pb} associated with the square dyadic butterfly architecture.
% 	\RG{Oui, idem sur la Fig. }\ER{la legende me semble toujours petite}
	}
	\label{fig:comparisonbetweenmethods}
\end{figure}
%\todo{Change orthonormalization to pseudo-orthonormalization in legends of \Cref{fig:comparisonbetweenmethods} (and other figures) ?}

\Cref{fig:comparisonbetweenmethods} shows that the {different methods find an approximate solution nearly up to machine precision}\footnote{{%More precisely, 
{Yet,} hierarchical algorithms and gradient-based methods are more accurate than ALS {by more than one} order of magnitude.}}, 
% \mRG{Faut-il aussi commenter les différences de précision ? 
% } 
but hierarchical algorithms are several orders of magnitude faster than the gradient-based method \cite{dao2019learning} and ALS \cite{lin2021deformable}. 
{Using \Cref{algo:modifedbutterflyalgo} {\em without} pseudo-orthonormalization operations}
%\Cref{algo:hierarchicalalgo} 
is also faster than {with these operations.} We however show the positive practical impact of pseudo-orthonormalization on noisy problems in the following section.

%\Cref{algo:modifedbutterflyalgo} since it does not perform the additional orthonormalization operations.
% It can be seen that  methods can recover the exact factorization of the Hadamard 
% \LZ{matrix}, up to machine precision, 
% \LZ{since} our implementation in Pytorch uses $32$-bit float number format. In terms of running time, \Cref{algo:hierarchicalalgo} and \Cref{algo:modifedbutterflyalgo} are faster than the other algorithms, with gains of several orders of magnitude. 
% \LZ{In the same setting,}
% \Cref{algo:hierarchicalalgo} is also faster than \Cref{algo:modifedbutterflyalgo} since it does not perform the \LZ{additional} orthonormalization \LZ{operations}.
% \mLZ{\Cref{fig:comparisonbetweenmethods} is nice but we need better formatting (labels, legend, PDF format, fontsize...)}

\subsection{
To orthonormalize or not to orthonormalize?}\nopunct
\label{subsec:exp-ortho}
We now study in practice the impact of the pseudo-orthonormalization operations in the new butterfly algorithm, in terms of running time and approximation error, at different scales of the matrix size $n$ with 
$n \in \{2^i \mid i = 7, \ldots, 13 \}$.
%\ER{replace? : $n=2^i$, for $i=7,\dots,13$}.
We consider Problem \eqref{eq:butterfly-approximation-pb} associated with a chainable architecture  $\arch = (\pattern_1, \ldots, \pattern_4)$ such that:
%\mRG{Ne serait-ce pas opportun de faire lien avec \Cref{lem:archfromsize}?\TL{Ajoute}\RG{Ajouter une caractérisation de ce $\beta$ en annexe \Cref{app:archfromsize}?}}  :
\begin{itemize}
    \item Each factor is of size $(a_\ell b_\ell d_\ell, a_\ell c_\ell d_\ell) = (n, n)$ 
    %, where we denote $\pattern_\ell = (a_\ell, b_\ell, c_\ell, d_\ell)$;
    % \item $b_\ell = c_\ell$ for all $\ell \in \intset{4}$;
    \item $\setButterfly{\arch}$ contains some dense matrices: $\pattern_1 * \pattern_2 * \pattern_3 * \pattern_4 = (1, n, n, 1)$;
    \item In the complementary low-rank characterization of  $\setButterfly{\arch}$, the rank constraint on the submatrices is $r \geq 2$, i.e., $\mathbf{r}(\arch) = (r, r, r)$. We do not choose $r=1$ because %otherwise, 
    the pseudo-orthonormalization operations would then be 
    equivalent to some simple rescaling.
\end{itemize}

Among all of such architectures  (that can be characterized and found using \cref{lem:archfromsize}), we choose the one with the smallest number of parameters, i.e., yielding the smallest $\|\arch\|_0$.
The considered target matrix is $\bA = \tilde{\bA} + \epsilon \frac{\| \tilde{\bA} \|_F}{\| \bE \|_F} \bE$, 
% \RG{le bruit n'est-il pas plutot réglé comme $\frac{\|\tilde{\bA}\|}{\|\bE\|}$?} 
where $\tilde{\bA} := \bX_1 \bX_2 \bX_3 \bX_4$, the entries of $\bX_\ell \in \setDBfactor{\pattern_\ell}$ for $\ell \in \intset{4}$ are i.i.d.~sampled from the uniform distribution in the interval $[0,1]$, $\bE$ is an i.i.d. centered Gaussian matrix with the standard deviation $1$, and $\epsilon =0.1$ is the noise level.
The permutation $\sigma$ for the butterfly algorithm is $\sigma = (2, 1, 3)$,  which corresponds to the \emph{balanced} factor-bracketing tree of $\intset{4}$.

\begin{figure}[t]
    \begin{subfigure}[b]{0.48\linewidth}
        \centering
        \includegraphics[width=\linewidth]{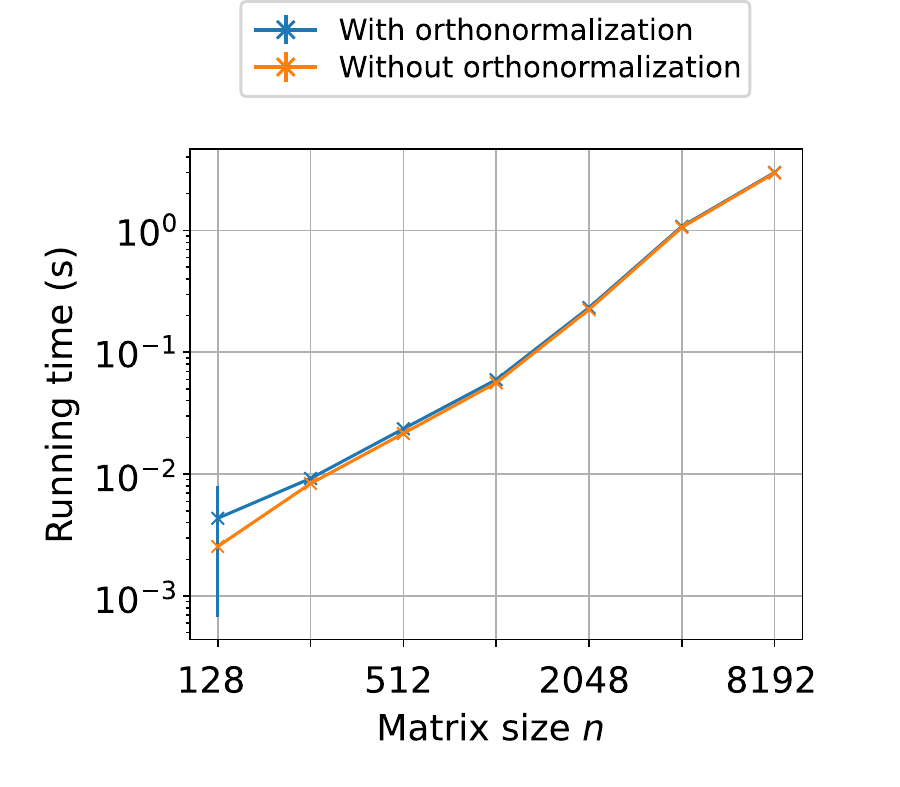}
        \caption{Running time \emph{vs.}~size $n$}
        \label{fig:time-hierarchical}
    \end{subfigure}
    \hfill
    \begin{subfigure}[b]{0.48\linewidth}
        \centering
        \includegraphics[width=\linewidth]{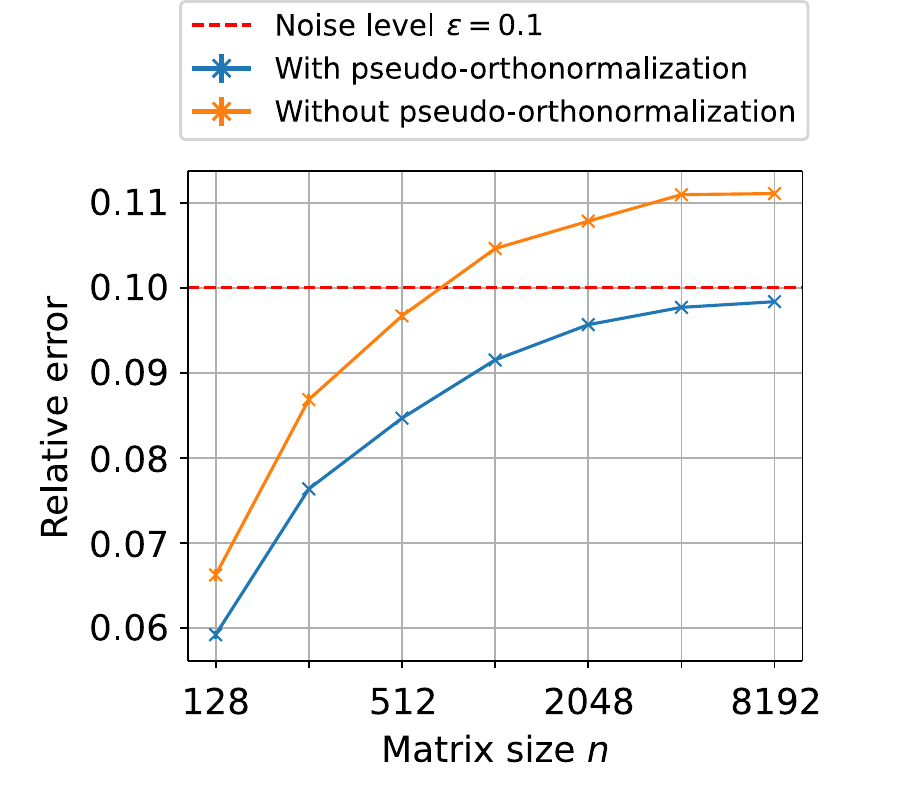}
        \caption{Relative error $\frac{\| \mat{A} - \hat{\mat{A}}\|_F}{\| \mat{A} \|_F}$ \emph{vs.}~size $n$}
        \label{fig:error-hierarchical}
    \end{subfigure}
    \caption{Running time and 
% relative error approximation
{the relative approximation errors} \emph{vs.} the matrix size $n$, for %\Cref{algo:hierarchicalalgo} (without orthonormalization) and 
\Cref{algo:modifedbutterflyalgo} (with {or without} pseudo-orthonormalization), for the instance of Problem \eqref{eq:butterfly-approximation-pb} described in \Cref{subsec:exp-ortho} {with $L = 4$, $r=4$}. We show mean and standard deviation on the error bars over $10$ repetitions of the experiment.
% \LZ{aligner}
}
\end{figure}

%{\bf \em Results.}
\Cref{fig:time-hierarchical} shows that the difference in running time between the new butterfly algorithm (\Cref{algo:modifedbutterflyalgo}) \emph{with} %{(\Cref{algo:hierarchicalalgo})} 
and \emph{without} %{(\Cref{algo:recursivehierarchicalalgo})} 
pseudo-orthonormalization is negligible, in the regime of large matrix size $n \geq 512$. This means that, asymptotically, the time of orthonormalization operations is not the bottleneck, which is coherent with our complexity analysis given in \Cref{theorem:complexity}. 

In terms of the approximation error, \Cref{fig:error-hierarchical} shows that the hierarchical algorithm {(\Cref{algo:modifedbutterflyalgo})} {\em with} pseudo-orthonormalization %(\Cref{algo:hierarchicalalgo})} 
returns a smaller {(i.e., better)} approximation error. % when we perform orthornormalization operations. 
Moreover, the {relative} error with pseudo-orthonormalization error is always smaller than the {relative} noise level {$\epsilon = 0.1$ (cf.~\Cref{fig:butterfly-approximation-other-values-epsilon} for other values of $\epsilon$)}, 
% \mRG{On a l'impression qu'en augmentant la taille ce ne serait plus le cas ? Ne faudrait-il pas essayer aussi une expérience où l'on fait varier $\epsilon$? Donne dans le texte une idée de l'ordre de grandeur du temps de calcul ? \LZc{ajouter les graphes en annexe.}}
%\mRG{Voir s'il est opportun d'utiliser certaines des figures du mail de Léon du 16/03 \LZ{oui c'est fait}}
which is not the case for the hierarchical algorithm without pseudo-orthonormalization {(\Cref{algo:recursivehierarchicalalgo})}. In conclusion, besides yielding error guarantees of the form \eqref{eq:error-bound-intro}, the pseudo-orthonormalization operations in our experiments also lead to better approximation in practice.

\subsection{Numerical assessment of the bounds }\label{sec:numrectangle}
In this section, we numerically evaluate the bounds given by \cref{theorem:mainresults} and \cref{theorem:bettermainresults}. The setting is identical to that of the previous section, except that:
\begin{enumerate}
    \item We run the new butterfly algorithm (\cref{algo:modifedbutterflyalgo}) with multiple permutations associated to various factor-bracketing trees: identity permutation (left-to-right tree, cf.~\Cref{example:illposedbX1} or \cite{zheng2023efficient}), balance permutation (corresponding to the so-called ``balanced tree'' \cite{zheng2023efficient}) %\todo{add crossref \TL{OK}}
    and a random one. 
    \item The number of factors is $L = 5$. The size of the matrix is $4608 \time 4608$. The number $4608 = 3^2 \times 2^9$ is a common matrix size appearing when one replaces convolutional layers by a product of butterfly factors \cite{lin2021deformable} (the factor $3^2$ corresponds to a convolutional kernel of size $3 \times 3$). This matrix size is non-dyadic, which partly illustrates the versatility of butterfly factors in replacing linear operators.
    \item The experiments are repeated $10$ times. {The plots in \cref{fig:comparisonbetweenpermutations} show the average and standard deviation values of the relative error and of the bounds. {The standard deviations are not visible due to their small size and the log scale of the plot.}
    % the  is very small and thus it is not visible. 
    % This is also due to the log-scale of the plot. 
    }
\end{enumerate}
\begin{figure}[t]
	\centering
	\includegraphics[width=\textwidth]{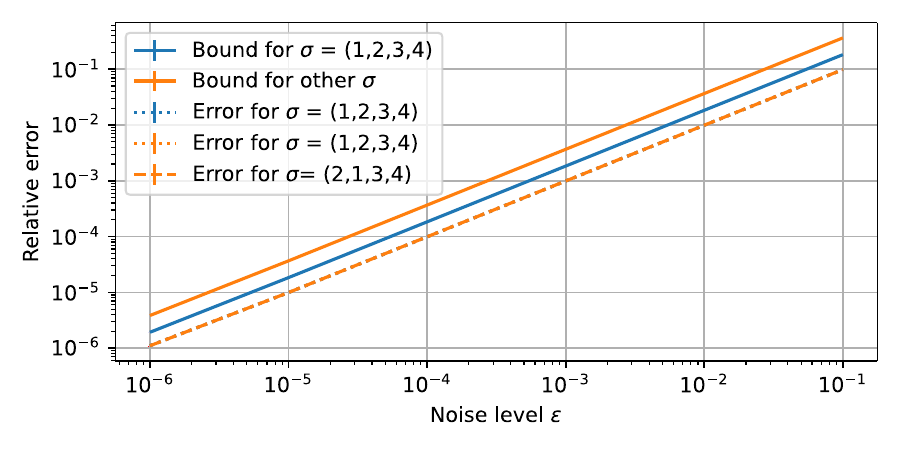}
	\caption{Bounds and numerical errors of \Cref{algo:modifedbutterflyalgo} with different permutations (associated to different factor bracketing trees): identity permutation $(1,2,3,4)$, ``balanced'' permutation $(2, 1, 3, 4)$,  randomized permutation ($(4, 1, 2, 3)$ in this experiment). The bound for the identity permutation is given by \Cref{theorem:bettermainresults} while the others are given by \cref{theorem:mainresults}. {Vertical bars show standard deviations over 10 experiments, but they are not visible on the graph because they are small. Note that all error curves (discontinuous lines) are overlapping, so we are seeing only one of them.}}
	\label{fig:comparisonbetweenpermutations}
\end{figure}
We report the relative errors and the bounds corresponding to each permutation in \cref{fig:comparisonbetweenpermutations}. It is well illustrated by \cref{fig:comparisonbetweenpermutations} that our theoretical results in \cref{theorem:mainresults} and \cref{theorem:bettermainresults} are correct. The bound in  \cref{theorem:bettermainresults} (\cref{fig:comparisonbetweenpermutations}-a) is visibly tighter than the one in \cref{theorem:mainresults} (\cref{fig:comparisonbetweenpermutations}-b,c) although there is no significant difference among the errors obtained by the three permutations corresponding to different factor bracketing trees {(the errors for the three different permutations are almost identical, which makes only one error plot visible). }

\section{Conclusion}
% In this chapter, we
We proposed a {general definition of (deformable) butterfly architectures, together with a} new butterfly algorithm for the problem of (deformable) butterfly factorization \eqref{eq:butterfly-approximation-pb}, endowed with new guarantees on the approximation error of the type \eqref{eq:error-bound-intro}, under the condition that the associated architecture $\arch$ {satisfies a so-called {\em chainability} condition.} %is chainable. 
The proposed algorithm involves some novel orthonormalization operations in the context of butterfly factorization. We discuss some perspectives of this work.

{\bf \em Tightness of the error bound.}
%\todo{peut-on enlever les couleurs dans la conclusion? Oui}
The constants $C_\arch$ in  \Cref{cor:mainresults} scale linearly or even %and
sub-linearly with respect to
the depth $L = |\arch|$ of the architecture. 
%\LZ{but linearly with the matrix size $n$ as shown in \Cref{tab:constant}}. 
Note that the quasi-linear constant $C_n$ in the existing bound \eqref{eq:bound-liu} is not comparable with the constants $C_\arch$ in  \Cref{cor:mainresults}, due to the presence of $\epsilon_0$ in the bound \eqref{eq:bound-liu}, whereas the constants $C_\arch$ in  \Cref{cor:mainresults} controls the ratio between the approximation error and the minimal error.
% \LZc{rappeler que dans la table c'est lineaire ou sous linéaire par rapport à la taille de matrices.} 
A natural question is whether the constants $C_\arch$ in  \Cref{cor:mainresults} are tight for an error bound of the type \eqref{eq:error-bound-intro}. If not, can the bounds for the proposed be algorithm be sharpened by a refined theoretical analysis, or is there another algorithm that yields a smaller constant $C_\arch$ in the error bound?
 % $C_\arch$, especially {the one in} \Cref{theorem:bettermainresults}, {is} tight for \Cref{algo:modifedbutterflyalgo}

{\bf \em Randomized algorithms for low-rank approximation.}
\Cref{algo:recursivehierarchicalalgo,algo:modifedbutterflyalgo} need to access all the elements of the target matrix $\bA \in \CC^{m \times n}$. Thus, the complexity of all algorithms is {at least} $\mathcal{O}(mn)$. This complexity, however, fails to scale for large $m, n$ (e.g., up to $10^6$).
    % \LZ{like for instance, $m,n = \Omega(10^6)$}. 
    {Assuming that the target matrix admits a butterfly factorization associated with $\arch$, i.e., $\bA \in \setButterfly{\arch}$}, is it possible to recover the butterfly factors of $\bA$, with a faster algorithm, ideally of complexity $\mathcal{O}(\|\arch\|_0)$? Note that this question was already considered in \cite{li2015butterfly,li2017interpolative} where randomized algorithms for low-rank approximation \cite{halko2011svd,liberty2007randomized} {are leveraged} in the context of butterfly factorization.
    % and certain algorithms can achieve the lower-bound $\|\arch\|_0$. 
    The question is therefore whether we can still prove some  theoretical guarantees of the form \eqref{eq:error-bound-intro} for butterfly algorithms with {such} %randomized 
    algorithms. % for low-rank approximation.
    % in such a context. As {illustrated} in \Cref{sec:experimentbutterfly}, {the proposed orthonormallization operations in \Cref{algo:modifedbutterflyalgo}} can significantly improve  the approximation quality.

{\bf \em {Algorithms beyond the chainability assumption}.} %\LZc{Certe, la chainabilité est suffisante pour avoir des garanties, mais on se demande si chanianbilité est nécessaire. En fait il existe des architectures non convertes par la chainabilité, mais qui ont quand même des garanties, par exemple la notion de transpose chainable (expliquer à haut niveau pourquoi, mais les détails en annexe.). Et on discute en annexe à quel point on peut relaxer chainable et transpose chainable (à nouveau rester à haut niveau). Parler aussi Kaleidoscope: ni chainable, ni transpose chainable. donc pas du tout couvert pour notre framework, donc notre algo ne s'applique pas. Question: peut-on avoir un algorithme avec garanties, qui fait mieux que descente de gradient?}
{Although chainability is a sufficient condition for which we can design an algorithm with guarantees on the approximation error, it is natural to ask whether it is also a necessary condition. There exist, in fact, non-chainable architectures for which we can still build an algorithm yielding an error bound \eqref{eq:error-bound-intro}. For instance, this is the case of \emph{arbitrary} architectures of depth $L=2$ (see \Cref{lem:qperfectcovering}); or of architectures that satisfy a \emph{transposed} version of \Cref{def:chainableseqDBPs}: such architectures can easily be checked to cover some {\em non-chainable} architectures, and \Cref{algo:modifedbutterflyalgo} can be easily adapted to have guarantees similar to \Cref{theorem:mainresults,theorem:bettermainresults}. Therefore, chainability in the sense of \Cref{def:chainableseqDBPs} is not necessary for theoretical guarantees of the form \eqref{eq:error-bound-intro}.} 
%{However, {as detailed in \Cref{appendix:chainability}}, we can show that it is \emph{necessary} to either consider the definition of chainability in the sense of \Cref{def:chainableseqDBPs}, or the transpose version of it, in order to preserve the axioms of stability by matrix multiplication (\Cref{prop:stability}) and associativity (\Cref{lem:associativity}),} {which are at the core for our proofs in \Cref{sec:error} (in particular they are important to ensure that the hierarchical algorithm works with guarantees for \emph{any} factor-bracketing tree).}
% These axioms will be at the core of our analysis in the following sections (in particular they are important to ensure that the hierarchical algorithm works with guarantees for \emph{any} factor-bracketing tree, cf.~\Cref{sec:hierarchical} below). 
% Our argument is provided in \Cref{appendix:chainability}. 
%{In other words, to find a more general condition on $\arch$, we need to develop a new proof technique that does not rely on either property. These new techniques, if possible, might allow us to deal with architectures that are neither chainable nor transposed-chainable, such as the Kaleidoscope architecture (cf.~\Cref{tab:existingparameterization}).}
Whether algorithms with performance guarantees can be derived beyond the three above-mentioned cases remains open.

{\bf \em Efficient implementation of butterfly matrix multiplication.}
    {While the complexity of matrix-vector multiplication by a $n \times m$ butterfly matrix associated with $\arch$ is theoretically subquadratic if $\| \arch \|_0 = o(mn)$,} a {practical} fast implementation {of such a matrix multiplication}
    is not straightforward \cite{gonon2024inference}, since dense matrix multiplication algorithms are competitive, e.g., on GPUs. {Future work can focus on efficient implementations of butterfly matrix multiplication on different kinds of hardware, in order to harness all of the benefits of the butterfly structure for large-scale applications, like in machine learning for instance.}

{\bf \em Butterfly factorization with unknown permutations.}
    In general, it is necessary to take into account row and column permutations in the butterfly factorization problem for more flexibility of the {butterfly} model $\setDBfactor{\arch}$ for an architecture $\arch$. For instance, as mentioned above, the DFT matrix admits a square dyadic butterfly factorization \emph{up to the bit-reversal permutation of column indices}. Therefore, as in \cite{zheng2023butterfly}, the general approximation problem that takes into account row and column permutations is: 
    \begin{equation}
    \label{eq:unknown-perm}
        \inf_{(\matseq{X}{\ell})_{\ell=1}^L \in \setDBfactor{\arch}, \mat{P}, \mat{Q}} \| \mat{A} - \mat{Q}^\top \matseq{X}{1} \ldots \matseq{X}{L} \mat{P} \|_F,
    \end{equation} 
    where $\mat{P}$, $\mat{Q}$ are unknown permutations part of the optimization problem. Without any further assumption on the target matrix $\mat{A}$, solving this approximation problem is conjectured to be difficult. Future work can further study this more general problem, based on the existing heuristic proposed in \cite{zheng2023butterfly}.
    % Thus, for the rest of the paper, we will assume that the permutation matrices $\mat{P}$ and $\mat{Q}$ are known and fixed, which leads to problem \eqref{eq:butterfly-approximation-pb}.
    % As discussed in \Cref{rem:permutation}, it is also interesting to see whether minimizing the    \LZ{entries} of butterfly factors altogether with the permutation matrices $\bP,\bQ$ is tractable or not. 
    % \LZc{I removed the last sentence because it repeats with the remark?}

% \LZc{idée principale: guarantees.
% Perspectives: réduire la constante, algorithmes randomiser pour réduire la complexité, implémentation GPU, permutations.}

\section*{Acknowledgement} The authors gratefully acknowledge the support of the Centre Blaise Pascal’s IT test platform
at ENS de Lyon for Machine Learning facilities. The platform operates the SIDUS solution \cite{quemener2013sidus} developed by Emmanuel Quemener. This work was partly supported by the ANR, project AllegroAssai ANR-19-CHIA-0009, and by the SHARP ANR project ANR-23-PEIA-0008 in the context of the France 2030 program. Tung Le was supported by AI Interdisciplinary Institute ANITI funding, through the French "Investments for the Future – PIA3" program under the grant agreement ANR-19-PI3A0004, and Air Force Office of Scientific Research, Air Force Material Command, USAF, under grant numbers FA8655-22-1-7012.
%\todo{Acknowledge financial support: ANR, ERC, bla bla}
% \LZc{TODO} Thanks to . (did we use CBP?)

\bibliography{ref}
\bibliographystyle{plain}
\appendix

% \LZc{TODO: uniformize notation with the main text}
% \mLZ{j'ai enlevé les supplementary notations car c'est mieux de les introduire localement}
% \section{Supplementary notations}
% To ease the reading when it comes to chainable $\arch$, we will use the following shorthands: \todo{do we really need this? can we use this locally?}
% \begin{definition}
% 	\label{def:supportfactordef}
% 	Given a chainable $\arch:= \seqparamfour$, define: 
% 	\begin{enumerate}[leftmargin=*, wide, labelwidth=!, labelindent=0pt]
% 		\item $\mbf{S}_\ell^\arch$ (or simply $\mbf{S}_\ell$) $:= \bS_{\pattern_\ell} = \bS_\paramfouri{\ell}$.
% 		\item $\mbf{S}_{k,\ell}^\arch$ (or $\mbf{S}_{k,\ell}$) $: = \mbf{S}_{(\pattern_k * \ldots * \pattern_\ell)}, \forall k \leq \ell$.
% 		\item $\pattern_{k,\ell} = \pattern_k * \ldots * \pattern_\ell$.
% 		\item $\setDBfactor{i} := \setDBfactor{\pattern_i}$.
% 		\item $\setDBfactor{k, \ell}:= \setDBfactor{\pattern_{k,\ell}}$. 
% 	\end{enumerate}
% 	We make {the} convention that $\mbf{S}_{1,0} = \mbf{I}_{a_1b_1d_1}$ and $\mbf{S}_{L+1, L} = \mbf{I}_{a_Lc_Ld_L}$.
% \end{definition}

\section{Proof for results in \Cref{section:preliminaries}}
\label{appendix:proofpreliminaries}
% \LZc{prêt à la relecture}
This section is devoted to the  proof of \Cref{theorem:tractablefsmf}. To that end we first introduce the following technical lemma.
%\todo{peut-on enlever les couleurs sur cette page? Oui}
\begin{lemma}
    \label{lemma:submatrixequality}
    Consider support constraints $\bL \in \{0,1\}^{m \times s}$, $\bR\in \{0,1\}^{r \times n}$ satisfying the conditions of \Cref{theorem:tractablefsmf}. With the notations $R_P \subseteq \intset{m}$, $C_P \subseteq \intset{n}$, $\cP[\bL,\bR]$ from \Cref{def:classequivalence}, for each $(\bX, \bY)$ such that $\supp(\bX)\subseteq\bL, \supp(\bY)\subseteq\bR$, we have:
    \begin{equation}
        \label{eq:submatrixequality}
        \forall P \in \cP(\bL,\bR), \quad (\bX\bY)[R_P, C_P] = \bX[R_P, P]\bY[P,C_P].
    \end{equation}
    % where $R_P,C_P$ are defined as in \Cref{def:classequivalence}.
\end{lemma}
\begin{proof}
    For any pair of matrices $(\bX,\bY) \in \CC^{m \times r} \times \CC^{r \times n}$:
    \begin{equation}
    \label{eq:sum-rankone-contributions}
        \bX\bY = \sum_{i = 1}^r \bX[\col{i}]\bY[\row{i}] = \sum_{P \in \cP(\bL,\bR)} \bX[\col{P}]\bY[\row{P}].
    \end{equation}
    For $P, \tilde{P} \in \cP(\bL,\bR)$ with $P \neq \tilde{P}$, the components $\bU_P$ and $\bU_{\tilde{P}}$ of \Cref{def:classequivalence} are disjoint by assumption on $\mat{L}, \mat{R}$. Since $\supp(\bX[\col{\tilde{P}}]\bY[\row{\tilde{P}}]) \subseteq \bU_{\tilde{P}}$, we have $(\bX[\col{\tilde{P}}]\bY[\row{\tilde{P}}])[R_P,C_P] = \mbf{0}$,
    % If $(\bL,\bR)$ satisfies the conditions of \Cref{theorem:tractablefsmf}, then $(\bX[\col{\tilde{P}}]\bY[\row{\tilde{P}}])[R_P,C_P] = \mbf{0}$ for any $P, \tilde{P} \in \cP(\bL,\bR)$ such that $P \neq \tilde{P}$, since $\supp(\bX[\col{\tilde{P}}]\bY[\row{\tilde{P}}]) \subseteq \bU_{\tilde{P}}$ and $\bU_P$ and $\bU_{\tilde{P}}$ are disjoint by assumption on $\mat{L}, \mat{R}$. 
    and {by \eqref{eq:sum-rankone-contributions}} it follows that we have
%$ %\begin{equation*}
\(
        (\bX\bY)[R_P, C_P] = (\bX[\col{P}]\bY[\row{P}])[R_P,C_P] = \bX[R_P,P]\bY[P,C_P].
        \)
        %$
        %\end{equation*}
\end{proof}

The following proof of \Cref{theorem:tractablefsmf} is mainly taken from \cite{le2022spurious}, but we additionally compute the infimum value of Problem \eqref{eq:FSMF}.
\begin{proof}[Proof for \Cref{theorem:tractablefsmf}]
    %In this proof, w
    We use the shorthand $\cP$ for $\cP(\bL,\bR)$, and denote $\Sigma := \{ (\mat{X}, \mat{Y}) \, | \, \supp(\mat{X}) \subseteq \mat{L}, \, \supp(\mat{Y}) \subseteq \mat{R} \}$. Recall that $(\matseq{U}{i})_{i=1}^r := \varphi(\mat{X}, \mat{Y})$ with $\varphi$ from \Cref{def:rankonesupportcontribution}.
    %Due to the hypothesis of this theorem, for any \emph{distinct} members $P, P' \in \cP(\bL,\bR)$, we have: 
	%\begin{equation}
	%	\label{eq:Hadamardproduct}
	%	\bX[\col{P'}]\bY[\row{P'}] \odot \mS_{P} = \mathbf{0}
	%\end{equation}
    Let $(\mat{X}, \mat{Y}) \in \Sigma$. Then, $\supp(\bX[\col{P}]\bY[\row{P}]) \subseteq \bU_P$ for any $P \in \cP$, hence $\supp(\bX\bY) \subseteq \bigcup_{P \in \cP} \bU_P$, and $(\bX\bY) \odot \overline{\bU_\cP} = \mbf{0}$ where we denote $\overline{\bU_\cP} := (\intset{m} \times \intset{n}) \setminus \left(\bigcup_{P \in \cP} \bU_P\right)$.
 % 
	% Define $\overline{\bU_\cP} := (\intset{m} \times \intset{n}) \setminus \left(\bigcup_{P \in \cP} \bU_P\right)$. 
 % Because $\supp(\bX[\col{P}]\bY[\row{P}]) \subseteq \bU_P, \forall P \in \cP$, we have: $\supp(\bX\bY) \subseteq \bigcup_{P \in \cP} \bU_P$. Thus, $(\bX\bY) \odot \overline{\bU_\cP} = \mbf{0}$.
 % 
    Moreover, by assumption, $\bU_{P}$ and $\bU_{\tilde{P}}$ are disjoint for any $P, \tilde{P} \in \mathcal{P}$ such that $P \neq \tilde{P}$, so:
	\begin{equation}
		\label{eq:decomposedisjointoverlapping}
		\begin{split}
			\|\bA - \bX\bY\|_F^2 &= \left(\sum_{P \in \mc{P}}\|(\bA -\bX\bY) \odot \bU_{P}\|_F^2\right) + \|(\bA -\bX\bY) \odot \overline{\bU_\mc{P}}\|_F^2\\
			&= \left(\sum_{P \in \mc{P}}\|(\bA - \bX\bY)[R_P,C_P]\|_F^2\right) + \|\bA \odot \overline{\bU_\mc{P}}\|_F^2\\
			&= \left(\sum_{P \in \mc{P}} \| \matindex{\bA}{R_P}{C_P} - \matindex{(\bX\bY)}{R_P}{C_P} \|_F^2 \right) + \|\bA \odot \overline{\bU_\mc{P}}\|_F^2\\
	       &\overset{\eqref{eq:submatrixequality}}{=} \left(\sum_{P \in \mc{P}}\|\matindex{\bA}{R_P}{C_P} - \matindex{\mat{X}}{R_P}{P} \matindex{\bY}{P}{C_P} \|_F^2\right) + \|\bA \odot \overline{\bU_\mc{P}}\|_F^2.
		\end{split}
	\end{equation}
 
    Since $\mathcal{P}$ is a partition, this implies that each term in the sum $\sum_{P \in \mc{P}}\|\matindex{\bA}{R_P}{C_P} - \matindex{\mat{X}}{R_P}{P} \matindex{\bY}{P}{C_P} \|^2$ involves columns of $\mat{X}$ and rows of $\mat{Y}$ that are not involved in other terms of the sum. 
    Moreover, we remark that for any $P \in \mathcal{P}$, the matrix $\matindex{\mat{X}}{R_P}{P} \matindex{\bY}{P}{C_P}$ is of rank at most $|P|$. In other words, minimizing the right-hand-side of \eqref{eq:decomposedisjointoverlapping} with respect to $(\mat{X}, \mat{Y}) \in \Sigma$ is equivalent to minimize each term of the sum for $P \in \mathcal{P}$, which is the problem of finding the best rank-$|P|$ approximation of $\matindex{\bA}{R_P}{C_P}$. This yields the claimed equation \eqref{eq:explicit-formula-fsmf}.
% 
% 
    % \begin{equation}
    %     \min_{(\mat{X}, \mat{Y}) \in \Sigma} \|\bA - \bX\bY\|_F^2 = \left(\sum_{P \in \mc{P}}\|\matindex{\bA}{R_P}{C_P} - \matindex{\mat{X}}{R_P}{P}\matindex{\bY}{P}{C_P}\|^2\right) + \|\bA \odot \overline{\bU_\mc{P}}\|^2.
    % \end{equation}
% 
    % \mLZ{ancienne version en commentaire latex}
	% Therefore, if we ignore the constant term $\|\bA \odot \bar{\bU}_\mc{P}\|^2$,  the function $L(\bX,\bY)$ is decomposed into a sum of functions $\|\matindex{\bA}{R_P}{C_P} - \matindex{\mat{X}}{R_P}{P}\matindex{\bY}{P}{C_P}\|^2$, which are instances of low-rank matrix approximation problem. Since all the optimized parameters are $\{(\matindex{\mat{X}}{R_P}{P}, \bY[{P,C_P}])\}_{P \in \mc{P}}$, an optimal solution of $L$ is $\{(\bX^\star[R_P,P], \bY^\star[P,C_P])\}_{P \in \mc{P}}$, where $(\bX^\star[R_P,P], \bY^\star[P,C_P])$ is a  minimizer of $\|\bA_{R_P,C_P} - \bX[R_P,P]\bY[P,C_P]\|^2$ {which is computed efficiently using a truncated SVD. Since the blocks associated with distinct $P$ are disjoint, these SVDs can be performed blockwise, in any order, and even in parallel.}
    \end{proof}

\section{Proof for results in \Cref{sec:DB-factorization} and \cref{subsec:L=2}}
% \LZc{prêt à la relecture}

\subsection{Proof of \Cref{prop:stability}}
\label{app:stability-chainable-pair}

\begin{proof}
    Let $r = r(\pattern_1, \pattern_2)$. By \Cref{def:chainableDBfour}, $a_1 \mid a_2$ and $d_2 \mid d_1$, so there are two integers $q, s$ such  that $a_2 = qa_1$ and $d_1 = sd_2$.  Since $a_1c_1/a_2 = b_2d_2/d_1 = r$ by \Cref{def:chainableDBfour}, this yields $c_1 = a_2r/a_1 = rq$, $b_2 = d_1r / d_2 = rs$. Thus,
	\begin{equation*}
		% \label{eq:productDB}
		\begin{aligned}
			\mbf{S}_{\pattern_1} \mbf{S}_{\pattern_2} &= \left(\mbf{I}_{a_1} \otimes \mbf{1}_{b_1 \times rq} \otimes \mbf{I}_{d_1}\right)\left(\mbf{I}_{a_2} \otimes \mbf{1}_{rs \times  c_2} \otimes \mbf{I}_{d_2}\right)\\
            &= \left(\mbf{I}_{a_1} \otimes \mbf{1}_{b_1 \times q} \otimes \mbf{1}_{1 \times r} \otimes \mbf{I}_{d_1}\right)\left(\mbf{I}_{a_2} \otimes \mbf{1}_{r \times 1} \otimes \mbf{1}_{s \times c_2} \otimes \mbf{I}_{d_2}\right)\\
            &= \left[(\mbf{I}_{a_1} \otimes \mbf{1}_{b_1 \times q}) \otimes \mbf{1}_{1 \times r} \otimes \mbf{I}_{d_1}\right]\left[\mbf{I}_{a_2} \otimes \mbf{1}_{r \times 1} \otimes (\mbf{1}_{s \times c_2} \otimes \mbf{I}_{d_2})\right]\\
			&\overset{(\star)}{=} \left(\mbf{I}_{a_1} \otimes \mbf{1}_{b_1 \times q}\right) \otimes \left(\mbf{1}_{1 \times r} \one{r \times 1}\right) \otimes \left(\mbf{1}_{s \times c_2} \otimes \mbf{I}_{d_2}\right)\\
			&= \left(\mbf{I}_{a_1} \otimes \mbf{1}_{b_1 \times q}\right) \otimes \left(r\mbf{1}_{1 \times 1} \right) \otimes \left(\mbf{1}_{s \times c_2} \otimes \mbf{I}_{d_2}\right)\\
            &= r\left(\bI_{a_1} \otimes \one{b_1s \times qc_2} \otimes \bI_{d_2}\right)\\
			&= r\mbf{S}_{\pattern_1 * \pattern_2} \quad \left(\text{because } \frac{b_1d_1}{d_2} = b_1s \text{ and } \frac{a_2c_2}{c_1} = qc_2\right).
		\end{aligned}
	\end{equation*}
	%\LZc{indicate in \eqref{eq:productDB} where we use the following property of Kronecker product} 
    We can use the equality $(\mbf{A} \otimes \mbf{C} \otimes \bE)(\mbf{B} \otimes \mbf{D} \otimes \bF) = (\mbf{AB}) \otimes (\mbf{CD}) \otimes (\bE\bF)$ in $(\star)$ because, according to our conditions for chainability, the sizes of $\mbf{A}, \mbf{B}, \mbf{C}, \mbf{D}, \bE, \bF$ in $(\star)$ make the matrix products $\mbf{AB}, \bC\bD$ and $\mbf{EF}$ well-defined. 
    %\LZc{It would be more intuitive for the reader to show \eqref{eq:productDB} before the conditions of \cref{def:chainableDB}.}
\end{proof}

\subsection{Proof of \Cref{lem:prodDBBis}}
\label{app:prodDBBis}

We prove the claims in \cref{lem:prodDBBis} one by one.
%as follows:

\begin{enumerate}
\item Follows by the definition of $R_P$ and $C_P$.

\item 
%The third claim is proved as in the proof of \cref{lem:qperfectcovering} (see \cref{subsec:L=2}).

 It can be easily verified that 
 %By \Cref{def:dbfactorfour}, 
    the column supports $\{ \matindex{\bS_{\pattern_1}}{:}{j} \}_{j}$ (resp.~the row supports $\{ \matindex{\bS_{\pattern_2}}{i}{:} \}_{i}$) are pairwise disjoint or identical, due to the support structure of the form $\bI_a \otimes \mbf{1}_{b \times c} \otimes \bI_d$ of a Kronecker-sparse factor. Indeed, the columns (resp.~rows) of the Kronecker product $\bA \otimes \bB$ are equal to the Kronecker product of columns and rows of $\bA$ and $\bB$ and the matrices $\bA, \bB$ appearing in our case are either identity matrices or all-one matrices. Thus, for each $P \in \cP(\bS_{\pattern_1}, \bS_{\pattern_2})$ the sets   $ R_P\times C_P=\matindex{\bS_{\pattern_1}}{:}{i} \times  \matindex{\bS_{\pattern_2}}{i}{:}$ are pairwise disjoint.
    
%     \item The fact that $|P| = r(\pattern_1,\pattern_2)$
% %The first equality of the second claim
% follows from \eqref{eq:prodDB}, \LZ{whose proof is deferred to \Cref{app:prodDB2}}. 
% The second equality \ERc{which one??}is a consequence of the rank-one decomposition of matrix multiplication, i.e., $\bX\bY = \sum_{i} \bX[\col{i}]\bY[\row{i}]$.

\item For any $P \in \cP(\DBsupport{\pattern_{1}}, \DBsupport{\pattern_{2}})$, by \Cref{def:dbfactorfour} we have $|R_P| = b_1, |C_P| = c_2$.  Setting $r := r(\pattern_1,\pattern_2)$, we have
%, because by \Cref{def:dbfactorfour}, the support of any column in $\mat{S}_{\pattern_1}$ and any row in $\mat{S}_{\pattern_2}$ is of cardinal $b_1$ and $c_2$, respectively.
%Moreover, we have: 
$$\sum_{P \in \mathcal{P}(\bS_{\pattern_1}, \bS_{\pattern_2})} |P| \, \rankonesuppseq{P} 
\stackrel{\textrm{Def. \ref{def:classequivalence}}}{=} \sum_{\rankonesuppseq{i} \in \varphi(\bS_{\pattern_1}, \bS_{\pattern_2})} \rankonesuppseq{i} = \bS_{\pattern_1} \bS_{\pattern_2}  \stackrel{\textrm{Prop. \ref{prop:stability}}}{=} r \bS_{(\pattern_1 * \pattern_2)},$$ where %the first equality comes from \Cref{def:classequivalence}, 
the second equality simply comes from the rank-one decomposition of matrix multiplication.
%, and the third equality comes from \Cref{prop:stability}. Because 
Since all $\bU_P$'s have pairwise disjoint supports (from point 2), we get $|P| = r$ for each $P \in \cP(\DBsupport{\pattern_{1}}, \DBsupport{\pattern_{2}})$. 
%\end{proof}

\item $\supp(\bS_{\pattern_1 * \pattern_2}) = \supp(\bS_{\pattern_1}\bS_{\pattern_2})$ is a consequence of \Cref{prop:stability}. The second equality, $\supp(\bS_{\pattern_1}\bS_{\pattern_2}) = \cup_{P \in \cP} R_P \times C_P$, is a consequence of the rank-one decomposition of the matrix multiplication, i.e., $\bX\bY = \sum_{i} \bX[\col{i}]\bY[\row{i}]$, and of the fact that $\bS_{\pattern_1}, \bS_{\pattern_2}$ have non-negative coefficients (thus, the support is equal to the union, and it is not just a subset).
\end{enumerate}

\subsection{Proof for \Cref{lem:associativity}}
\label{appendix:associativityproof}
%\begin{proof}
	Denote $\pattern_\ell = (a_\ell, b_\ell, c_\ell, d_\ell)$ for $\ell \in \intset{3}$. Let us show that $\pattern_1$ and $\pattern_2 * \pattern_3$ are chainable and $r(\pattern_1, \pattern_2*\pattern_3) = r(\pattern_1, \pattern_2)$. Since $(\pattern_2, \pattern_3)$ is chainable, by definition of $*$ (\Cref{def:chainableDBfour}), we have:
	\begin{equation*}
		(\tilde{a}, \tilde{b}, \tilde{c}, \tilde{d}) = \ttheta := \pattern_2 * \pattern_3 = \left(a_2, \frac{b_2d_2}{d_3}, \frac{a_3c_3}{a_2}, d_3\right).
	\end{equation*}
	We then verify that $(\pattern_1, \ttheta)$ satisfies {the} conditions of \Cref{def:chainableDBfour}:
	\begin{enumerate}
        \item By chainability of $(\pattern_1, \pattern_2)$, we have $r(\pattern_1,\pattern_2) := a_1c_1/a_2 = b_2d_2/d_1 \in \NN$. Therefore: $a_1c_1/\tilde{a} = a_1c_1/a_2 = r(\pattern_1,\pattern_2) =  b_2d_2 / d_1 = \tilde{b}\tilde{d} / d_1 \in \mathbb{N}$. This means that $r(\pattern_1, \pattern_2*\pattern_3) = r(\pattern_1, \pattern_2)$. 
        % Since $(\pattern_2, \pattern_3)$.
		\item By chainability of $(\pattern_1, \pattern_2)$, we have $a_1 \mid a_2 = \tilde{a}$. 
		%Since $\tilde{a} = a_2$ by definition, we have $a_1 \mid \tilde{a}$.
		\item By chainability of $(\pattern_2, \pattern_3)$ (resp.~of $(\pattern_1, \pattern_2)$) we have $\tilde{d} = d_3 \mid d_2 \mid d_1$. % (resp.~) hence $\tilde{d} = d_3 \mid d_1$. %hence $\tilde{d} \mid d_1$ because $\tilde{d} = d_3$ by definition.
	\end{enumerate}
	In conclusion, $(\pattern_1, \pattern_2 * \pattern_3)$ is chainable with $r(\pattern_1, \pattern_2 * \pattern_3) = r(\pattern_1, \pattern_2)$. Computing $\pattern_1 * (\pattern_2 * \pattern_3)$ explicitly by \eqref{eq:operatorDBPfour} gives $\pattern_1 * (\pattern_2 * \pattern_3) = \left(a_1, \frac{b_1d_1}{d_3}, \frac{a_3c_3}{a_1}, d_3\right)$.
    Similarly, we can show that $(\pattern_1 * \pattern_2, \pattern_3)$ is also chainable with $r(\pattern_1 * \pattern_2, \pattern_3) = r(\pattern_2, \pattern_3)$, and we can indeed verify that $(\pattern_1 * \pattern_2) * \pattern_3 = \pattern_1 * (\pattern_2 * \pattern_3)$ using \eqref{eq:operatorDBPfour}.
	% The computing of $(\pattern_1 * \pattern_2) * \pattern_3$ yields the same pattern, which ends the proof.
%\end{proof}

\subsection{Proof of \Cref{lem:suppdbfactorprod} - explicit formula for $(\pattern_1 * \ldots * \pattern_L)$ in \eqref{eq:producttheta}}
\label{appendix:productmultipleDBparam}
% \begin{corollary}
%     \label{cor:productmultipleDBparam}
    
% \end{corollary}

%\begin{proof}
    We show that $\pattern_1 * \ldots * \pattern_L = \left( a_1, \frac{b_1 d_1}{d_L}, \frac{a_L c_L}{a_1}, d_L \right)$, for each chainable architecture $\arch = (\pattern_\ell)_{\ell=1}^L = (a_\ell, b_\ell, c_\ell, d_\ell)_{\ell=1}^L$ of depth $L \geq 2$.
    The proof is an induction on $L \geq 2$. 
  %  \begin{itemize}
   %     \item 
   If $L = 2$, the result comes from \eqref{eq:operatorDBPfour}.
   %     \item
   Let $L \geq 2$, and assume that the statement holds for any chainable architecture of depth $L$. 
        % We prove that it is still valid for $\arch = (\pattern_\ell)_{\ell = 1}^{L+1}$ of length $L + 1$. Indeed, by induction hypothesis, we have:
        Consider a chainable architecture $\arch := (\pattern_\ell)_{\ell=1}^{L+1} = (a_\ell,b_\ell,c_\ell,d_\ell)_{\ell=1}^{L+1}$ of depth $L+1$. By the induction hypothesis, we have $\pattern_1 * \ldots * \pattern_L = \left( a_1, \frac{b_1 d_1}{d_L}, \frac{a_L c_L}{a_1}, d_L \right)$.
        Therefore,
            \begin{align*}
                \pattern_1 * \ldots * \pattern_L * \pattern_{L + 1} &= \left( a_1, \frac{b_1 d_1}{d_L}, \frac{a_L c_L}{a_1}, d_L \right) * (a_{L+1},b_{L+1},c_{L+1}, d_{L+1})\\
                &= \left(a_1, \frac{b_1d_1d_L}{d_Ld_{L+1}}, \frac{a_{L+1}c_{L+1}}{a_1}, d_{L+1}\right)\\
                &= \left(a_1, \frac{b_1d_1}{d_{L+1}}, \frac{a_{L+1}c_{L+1}}{a_1}, d_{L+1}\right).
            \end{align*}
%        which ends the proof.
   % \end{itemize}
%\end{proof}

\subsection{Proof for \Cref{lem:seq-qchainability}}
\label{appendix:seq-qchainabilityproof}
%\begin{proof}
    By \eqref{eq:producttheta}, we have:
%    \todo{replace $r$ by another index ($q$?)}
    \begin{equation}
    \label{eq:explicit-pattern-r-s-t}
        \begin{aligned}
            (a, b, c, d) &= \pattern := \pattern_{q} * \ldots * \pattern_s = \left(a_{q}, \frac{b_{q}d_{q}}{d_s}, \frac{a_sc_s}{a_{q}}, d_s\right),\\
            (a', b', c', d') &= \pattern' := \pattern_{s+1} * \ldots * \pattern_t = \left(a_{s+1}, \frac{b_{s+1}d_{s+1}}{d_t}, \frac{a_t c_t}{a_{s+1}}, d_t\right).\\
        \end{aligned}
    \end{equation}
    We show the chainability of $(\pattern, \pattern')$ by verifying the conditions of \Cref{def:chainableDBfour}:
    % Verifying three conditions of chainability gives us:
    \begin{enumerate}
        \item By chainability of $(\pattern_s, \pattern_{s+1})$, $r(\pattern_s,\pattern_{s+1}) = a_sc_s / a_{s + 1} = b_{s+1}d_{s+1} / d_s \in \mathbb{N}$. This means that $ac / a' = a_sc_s / a_{s + 1}   = r(\pattern_s,\pattern_{s+1}) = b_{s+1}d_{s+1} / d_s = b'd'/d$ and $r(\pattern, \pattern') = r(\pattern_s,\pattern_{s+1}) \in \mathbb{N}$.
        % \item First condition: $ac / a' = a_sc_s / a_{s + 1} = r(\pattern_s,\pattern_{s+1})$, $b'd'/d = b_{s+1}d_{s+1} / d_s = r(\pattern_s,\pattern_{s+1})$ (because $(\pattern_s, \pattern_{s+1})$ is chainable).
        \item By chainability of $\arch$, $a_\ell \mid a_{\ell+1}$ for $\ell \in \intset{L-1}$, so $a= a_{q} \mid a_{s+1} =a'$  since $q \leq s$.
        %, hence $a \mid a'$ because $a = a_{q}$ and $a' = a_{s+1}$.
        \item Similarly
        %By chainability of $\arch$, 
        $d_{\ell + 1} \mid d_\ell$ for all $\ell \in \intset{L-1}$, so
        $d'= d_t \mid d_s =d$ 
        %since $s \leq t$.
        %, hence $d' \mid d$ because $d' = d_t$ and $d = d_s$.
        % and $d_{\ell + 1} \mid d_\ell, \forall \ell = 1, \ldots, L-1$, we have: $a_l \mid a_{s+1}$ and $d_r \mid d_s$.
    \end{enumerate}
%\end{proof}

\subsection{Proof for \Cref{lem:partialproductnonredundant}}
\label{appendix:partialproductnonredundantproof}
%    \todo{replace $r$ by another index ($q$?)}

    The explicit formulas for $(a,b,c,d) = \pattern := \pattern_{q} * \ldots * \pattern_{s}$ and $(a',b',c',d') = \pattern' := \pattern_{s+1} * \ldots * \pattern_t$ are given in \eqref{eq:explicit-pattern-r-s-t}. 
    % The proof is based on the explicit formulas of $\pattern_{q} * \ldots * \pattern_{s}$ and $\pattern_{s+1} * \ldots * \pattern_t$, which are given by:
    % \begin{equation}
    %     \begin{aligned}
    %         (a, b, c, d): = \pattern_{q} * \ldots * \pattern_s &= \left(a_{q}, \frac{b_{q}d_{q}}{d_s}, \frac{a_sc_s}{a_{q}}, d_s\right)\\
    %         (a', b', c', d'): = \pattern_{s+1} * \ldots * \pattern_t &= \left(a_{s+1}, \frac{b_{s+1}d_{s+1}}{d_t}, \frac{a_tc_t}{a_{s+1}}, d_t\right)\\
    %     \end{aligned}
    % \end{equation}
    By \Cref{lem:seq-qchainability}, $(\pattern, \pattern')$ is chainable, with $r(\pattern, \pattern') = r(\pattern_s, \pattern_{s+1})$. To show the non-redundancy of $(\pattern, \pattern')$, it remains to show $r(\pattern, \pattern') < \min(b,c')$.
    Let us show $r(\pattern, \pattern') < b$.
    % 
    % We just need to verify that $q = b'd'/d = ac/a' < \min(b,c')$ because they are already chainable, as proved in \Cref{lem:seq-qchainability}. 
% 
    % We will prove that $q = b'd'/d < b$. A similar argument will yield $q = ac/a' < c'$. As a consequence, $q < \min(b,c)$ and the proof is concluded.
% 
    %In fact, since $b'd'/d = b_{s+1}d_{s+1}/d_s, b = b_1d_1/d_s$, proving $b'd'/d \leq b$ is equivalent to proving $b_{s+1}d_{s+1} \leq b_1d_1$. This is, however, correct thanks to \eqref{eq:boundonGBP}.
    % 
    % 
    Because $\arch$ is not redundant, by \Cref{def:redundant}, we have
    $r(\pattern_\ell, \pattern_{\ell+1}) < b_{\ell}$ for any $\ell \in \intset{L-1}$. But $r(\pattern_\ell, \pattern_{\ell+1})= b_{\ell + 1}d_{\ell + 1}/d_{\ell}$ by \Cref{def:chainableDBfour}. Therefore, $b_{\ell + 1}d_{\ell + 1} < b_\ell d_\ell$ for any $\ell \in \intset{L-1}$. Thus, since $q \leq s$, we have $b_{s+1}d_{s+1} < b_{q}d_{q}$. A fortiori, $\frac{b_{s+1}d_{s+1}}{d_s} < \frac{b_{q}d_{q}}{d_s}$. But by \eqref{eq:explicit-pattern-r-s-t}, $\frac{b_{q}d_{q}}{d_s} = b$ and $\frac{b_{s+1}d_{s+1}}{d_s} = \frac{b'd'}{d} = r(\pattern, \pattern')$. In conclusion, $r(\pattern, \pattern') = \frac{b_{s+1}d_{s+1}}{d_s} < \frac{b_{q}d_{q}}{d_s} = b$.
    % \begin{equation*}
    %     b_{r}d_{r} > \ldots > b_{s+1}d_{s+1} \implies  b = \frac{b_{q}d_{q}}{d_s} > \frac{b_{s+1}d_{s+1}}{d_s} = \frac{b'd'}{d} = q.
    % \end{equation*}
    A similar argument yields $r(\pattern, \pattern') < c'$. This ends the proof.

\subsection{Proof of \Cref{lem:prodDB2}}
\label{app:prodDB2}
{We first show that
 $(\bS_{\pattern_1},\bS_{\pattern_2})$ satisfies the condition of \Cref{theorem:tractablefsmf}. 
As in the proof of point 2 of \Cref{lem:prodDBBis}
 %By \Cref{def:dbfactorfour}, 
    the column supports $\{ \matindex{\bS_{\pattern_1}}{:}{j} \}_{j}$ (resp.~the row supports $\{ \matindex{\bS_{\pattern_2}}{i}{:} \}_{i}$) are pairwise disjoint or identical, hence %respectively.
    %: this is not hard to verity since the columns (resp.~rows) of the Kronecker product $\bA \otimes \bB$ are equal to the Kronecker product of columns and rows of $\bA$ and $\bB$ and the matrices $\bA, \bB$ appearing in our case are either identity matrices or all-one matrices. 
    %As a consequence, 
    the components $\bU_i$ of %any pair of elements     in the tuple 
    $\varphi(\bS_{\pattern_1}, \bS_{\pattern_2})$ are pairwise disjoint or identical
    %is either identical 
    (if their %corresponding 
    column {\em and} row supports coincide).}

Therefore, 
for any matrix $\mat{A}$, we have $\mat{A} \in \setButterfly{\arch}$ if, and only if: 
\begin{equation*}
    \begin{split}
        &\min_{(\mat{X}, \mat{Y}) \in \setDBfactor{\arch}} \| \mat{A} - \mat{X} \mat{Y} \|_F^2 = 0 \\
        \overset{\eqref{eq:explicit-formula-fsmf}}{\iff}& \sum_{P \in \cP(\bS_{\pattern_1},\bS_{\pattern_2})} \min_{\bB, \rank(\bB) \leq |P|}\|\bA[R_P,C_P] - \bB\|_F^2 + \sum_{(i,j) \notin \supp(\bS_{\pattern_1} \bS_{\pattern_2})} \bA[i,j]^2 = 0 \\
        \overset{\eqref{eq:DB-support}}{\iff}& \sum_{P \in \cP(\bS_{\pattern_1},\bS_{\pattern_2})} \min_{\bB, \rank(\bB) \leq |P|}\|\bA[R_P,C_P] - \bB\|_F^2 + \sum_{(i,j) \notin \supp(\bS_{\pattern_1 * \pattern_2})} \bA[i,j]^2 = 0\\
        \iff& \begin{cases}
            \rank(\bA[R_P,C_P]) \leq |P|, \quad \forall P \in \cP(\bS_{\pattern_1},\bS_{\pattern_2}) \\
            \bA \in \setDBfactor{\pattern_1 * \pattern_2}
        \end{cases}.
    \end{split}
\end{equation*}
This proves \eqref{eq:prodDB}.

\subsection{Proof of \Cref{lem:archfromsize}}
\label{app:archfromsize}

    %First we consider necessary conditions. Consider a chainable $\arch = (\pattern_\ell)_{\ell=1}^L$ with $\pattern_\ell = (a_\ell,b_\ell,c_\ell,d_\ell)$, and 
   % \todo{Léon: on devrait mettre ce lemme plus haut, dans l'appendix \Cref{app:DB-factorization}}
   By \Cref{lem:suppdbfactorprod} we have $\pattern 
    := \pattern_1 * \ldots * \pattern_L = (a_1, b_1d_1/d_L,a_Lc_L/a_1,d_L)$ and $\setButterfly{\arch} \subseteq \setDBfactor{\pattern}$. Since  $\setButterfly{\arch}$ contains some dense matrices, by \Cref{def:dbfactorfour} we must have $a_1=d_L=1$ hence $\pattern = (1,b_1d_1,a_Lc_L,1)$, and the matrices in $\setDBfactor{\pattern}$ are of size $b_1d_1 \times a_Lc_L$. As a result $b_1d_1=m$ and $a_Lc_L=n$, so that $a_L \mid n$, $d_1 \mid m$, and $b_1 = m/d_1$ and $c_L = n/a_L$. Now, by chainability (see \Cref{def:chainableDBfour}) we also have $a_1 \mid \ldots \mid a_L$ and $d_L \mid \ldots \mid d_1$ and, $a_\ell c_\ell/a_{\ell+1} = b_{\ell+1}d_{\ell+1}/d_\ell = r(\pattern_\ell,\pattern_{\ell+1}) := r_\ell$ for each $\ell \in \intset{L-1}$.
  
    With the convention $a_{L+1}:=n$ and $d_0 := m$, the quantities $p_\ell := a_{\ell+1}/a_\ell$ and $q_\ell := d_{\ell-1}/d_\ell$, $1 \leq \ell \leq L$ are thus integers, and we have $b_1 = m/d_1 = d_0/d_1 = q_{1}$, $c_L = n/a_L = a_{L+1}/a_L = p_{L}$. We obtain $m = \prod_1^L p_\ell$, $n = \prod_1^L q_\ell$, and \eqref{eq:ADAsProduct}.
    For $1 \leq \ell \leq L-1$ we have
    \begin{align}
    %    1 & = a_1 \mid \ldots \mid a_L \mid a_{L+1} := n\\
     %   1& = d_L \mid \ldots \mid d_1 \mid d_0 := m\\
        %b_1 & = q_1 %\frac{m}{d_1}=\frac{d_0}{d_1}
        %\quad \text{and}\ 
        b_{\ell+1} = \frac{d_\ell}{d_{\ell+1}} r_\ell = q_{\ell+1}r_\ell,\quad \text{and}\ 
        %\ \forall \ell \in \intset{L-1}
        &
        %\\
                %c_L = \frac{n}{a_L}=\frac{a_{L+1}}{a_L} \quad \text{and}\ 
c_\ell = \frac{a_{\ell+1}}{a_\ell} r_\ell = p_\ell r_\ell. %[],\ \forall \ell \in \intset{L-1}
    \end{align}
    By convention $r_0=r_L=1$, hence we also have $b_1 = q_1 r_1$ and $c_L = p_L r_L$, so that \eqref{eq:BCAsProduct} indeed holds.
%   \begin{align}
% %        n = \prod_{\ell=1}^L p_\ell & \quad \text{and}\ a_\ell = \prod_{j=1}^{\ell-1} p_\ell,\quad 1 \leq \ell \leq L\\
%  %       m = \prod_{\ell=1}^L q_\ell &
%   %  \quad \text{and}\  d_\ell = \prod_{j=\ell+1}^{L} q_\ell, 1 \leq \ell \leq L\\
%     %b_1 = q_1 & \quad \text{and}\ 
%     b_\ell &= q_\ell r_{\ell-1}, 0 \leq \ell \leq L\\
%     %c_L = p_L & \quad \text{and}\ 
%     c_\ell &= 
%     p_\ell r_\ell, 1 \leq \ell \leq L+1.
% \end{align}
Vice versa it is not difficult to check that given any such integers $p_\ell$, $q_\ell$, $r_\ell$ 
the expressions~\eqref{eq:ADAsProduct}-\eqref{eq:BCAsProduct} yield a chainable architecture enabling the construction of $\one{m \times n}$, which is a dense matrix.
   
We now deal with non-redundancy. By \Cref{def:redundant}, each pair $\pattern_\ell,\pattern_{\ell+1}$, $1 \leq \ell \leq L-1$ is non-redundant if, and only if, $r_\ell := r(\pattern_\ell,\pattern_{\ell+1}) < \min(b_\ell,c_{\ell+1}) = \min(q_\ell r_{\ell-1},p_{\ell+1}r_{\ell+1})$. This reads 
 \begin{align*}
     \frac{r_\ell}{r_{\ell-1}} < q_\ell
     \quad \text{and}\ 
     \frac{r_{\ell+1}}{r_\ell} > \frac{1}{p_{\ell+1}},
     \ 
     1 \leq \ell \leq L-1
     \end{align*}
     or equivalently (recall that $r_0=r_L:=1$ by definition): $r_1<q_1$, $r_{L-1} < p_L$, and \eqref{eq:NonRedundancyConditionAnnex}.

\section{%Details on the p
Pseudo-orthonormalization operations of %in 
\Cref{section:normalizedbutterflyfactorization}}
\label{app:orthonormalDB}
%The goal of this section is to clarify the nature of the orthonormalization operations introduced in \Cref{algo:modifedbutterflyalgo}, which involve the procedure described in \Cref{algo:exchange}. We start by giving a formal definition of orthonormality for a $\pattern$-\ksf{}(\Cref{app:def-orthonormality-butterfly-factor}). Then we describe important properties of such orthonormal butterfly factors (\Cref{app:prop-ortho-butterfly-factor}), in order to explain the nature of \Cref{algo:exchange} (\Cref{app:lemma-core-algo-exchange}). 
%\todo{Enlever couleurs jusqu'à C.1 inclus}
The goal of this section is to describe
the pseudo-orthonormalization operations mentioned
in {the new butterfly algorithm (\Cref{algo:modifedbutterflyalgo})}, which involve the procedure described in \Cref{algo:exchange}.

\begin{algorithm}[t]
	\centering
	\caption{Column/row-pseudo-orthonormalization} 
	\label{algo:exchange}
	\begin{algorithmic}[1]
		\REQUIRE Non-redundant $(\pattern_1, \pattern_2)$, $\bX \in  \setDBfactor{\pattern_1}$, $\bY \in \setDBfactor{\pattern_2}$, $u \in \{ \texttt{column}, \texttt{row}\}$
        \ENSURE $(\tilde{\bX},\tilde{\bY}) \in \setDBfactor{\pattern_1} \times \setDBfactor{\pattern_2}$ such that $\tilde{\bX} \tilde{\bY} = \mat{X} \mat{Y}$
		\STATE $(\tilde{\bX}, \tilde{\bY}) \gets (\mbf{0}, \mbf{0})$
		\FOR{$P \in \mathcal{P}(\bS_{\pattern_1}, \bS_{\pattern_2})$ (cf.~\Cref{def:classequivalence})} %\COMMENT{TEST}
  %\RG{See \Cref{def:classequivalence} for the notation \mathcal{P}(\cdot,\cdot)$}  }
		\IF{$u$ is \texttt{column}}
		\STATE {$(\bQ, \bR) \gets$ QR-decomposition of $\bX[R_P, P]$}{\label{line:qrdecomposition}}
		\STATE {$\tilde{\bX}[R_P, P] \gets \bQ$}{\label{line:columnorthonormalization}}
		\STATE {$\tilde{\bY}[P,C_P] \gets \bR\bY[P, C_P]$}{\label{line:matrix-multiplication}}
		\ELSIF{$u$ is \texttt{row}}
		\STATE $(\bQ, \bR) \gets$ QR-decomposition of $\bY[P, C_P]^\top$
		\STATE {$\tilde{\bX}[R_P, P] \gets \bX[R_P, P] \bR^\top$}
		\STATE $\tilde{\bY}[P,C_P] \gets \bQ^\top${\label{line:roworthonormalization}}
		% \STATE $\tilde{\bX}^1[R,Z] = \bX^1[R,Z]\bU\bD, \tilde{\bX}^2[Z,C] = \bV^\top$.
		\ENDIF
		\ENDFOR
		\RETURN $(\tilde{\bX},\tilde{\bY})$
	\end{algorithmic}
\end{algorithm}

{First of all, let us highlight that the orthonormalization operations are well-defined  only under the non-redundancy assumption. 
 In \Cref{algo:exchange}, the input pair of patterns $(\pattern_1, \pattern_2)$ is assumed to be chainable and non-redundant, thus by \Cref{lem:prodDBBis} and \Cref{def:redundant}, we have $| R_P | \leq | P |$ and $| C_P | \leq | P |$, which makes the operations at lines  \ref{line:columnorthonormalization} and \ref{line:roworthonormalization} in \Cref{algo:exchange} well-defined.
In \Cref{algo:modifedbutterflyalgo}, the architecture $\arch$ is assumed to be non-redundant. By \Cref{lem:partialproductnonredundant}, this means that the pair $(\pattern_{{I}_k}, \pattern_{{I}_{k+1}})$ at line \ref{line:colorthonormal} or the pair $(\pattern_{{I}_{k-1}}, \pattern_{{I}_k})$ at line \ref{line:roworthonormal} are chainable and non-redundant. This makes the call to \Cref{algo:exchange} at these lines well-defined. }

In the following, we start by providing properties of left/right-$r$-unitary matrices (cf.~\Cref{def:q-unitary}),
%definition. 
while the second part of this section is devoted to the proof of \Cref{lemma:role-pseudo-orthogonality}.

\subsection{Properties of left/right-$r$-unitary factors}
\label{app:prop-ortho-butterfly-factor}
We introduce properties related to left/right-$r$-unitary factor from \Cref{def:q-unitary}.
%\ER{pseudo}-orthonormal butterfly factors.

\subsubsection{Norm preservation under left and right matrix multiplication}
%\ER{I replaced $X\leftarrow\mat{A}_1$ etc}
\begin{lemma}
	\label{lem:qunitarypreservesfrob}
	Consider a chainable $\arch := (\pattern_1, \pattern_2, \pattern_3)$ and $\mat{A}_i \in \setDBfactor{\pattern_i}$ {for $i \in \intset{3}$}. If $\mat{A}_1$
    %Then, for any $\mat{A}_1 \in \setDBfactor{\pattern_1}$ that 
    is left-$r(\pattern_1, \pattern_2)$-unitary
    and %, any $\mat{A}_2 \in \setDBfactor{\pattern_2}$, and any $\mat{A}_3 \in \setDBfactor{\pattern_3}$ that 
    $\mat{A}_3$ is right-$r(\pattern_2, \pattern_3)$-unitary then: 
	\begin{equation*}
		\|\mat{A}_1\mat{A}_2\mat{A}_3\|_F = \|\mat{A}_2\|_F.
	\end{equation*}  
\end{lemma}

\begin{proof}
	%The proof combines the associativity of $*$ and the definition of \RG{(left or right)-$q$-}unitary factors. Indeed, b
 By \Cref{lem:associativity} we have $r(\pattern_1, \pattern_2 * \pattern_3) = r(\pattern_1, \pattern_2)$, hence
        \begin{align*}
            \|\mat{A}_1\mat{A}_2\mat{A}_3\|_F &= \|\mat{A}_2\mat{A}_3\|_F \quad (\text{since } \mat{A}_2\mat{A}_3 \in \setDBfactor{\pattern_2 * \pattern_3}, r(\pattern_1, \pattern_2 * \pattern_3) = r(\pattern_1, \pattern_2))\\
            &= \|\mat{A}_2\|_F \quad (\text{by~\Cref{def:q-unitary}}).
        \end{align*}
\end{proof}
% \begin{corollary}
% \label{lem:orthonormalfactorisorthogonal}
% 	Consider a left-$r$-unitary $\pattern$-factor $\bX$ and a pairs of $\pattern'$-factors $(\bY_1, \bY_2)$. If $r(\pattern,\pattern') = r$, then:
% 	\begin{equation*}
% 		\langle \mbf{X}\mbf{Y}_1, \mbf{X}\mbf{Y}_2 \rangle = \langle \mbf{Y}_1, \mbf{Y}_2 \rangle 
% 	\end{equation*}  
% 	where $\langle\cdot,\cdot\rangle$ denotes the usual Frobenius inner product of two matrices. 
% \end{corollary}
% \begin{proof}
%     The proof uses the following equality:
%     \begin{equation*}
%         \langle \mbf{X}\mbf{Y}_1, \mbf{X}\mbf{Y}_2 \rangle = \frac{1}{2}(\|\bX\bY_1\|_F^2 + \|\bX\bY_2\|_F^2 - \|\bX(\bY_1 + \bY_2)\|_F^2)
%     \end{equation*}
%     and the norm preservation property of left-$r$-unitary definition.
% \end{proof}
\subsubsection{Stability under matrix multiplication} 
Similarly to classical orthonormal matrices, left/right-$r$-unitary factors also enjoy a form of stability under matrix multiplication, in the following sense.

\begin{lemma}
	\label{lemma:stableundermatrixmultiplication}
    Consider a chainable pair $(\pattern_1, \pattern_2)$
    and $\mat{A}_i \in \setDBfactor{\pattern_i}$ {for $i \in \intset{2}$}.
    \begin{enumerate}
        \item If $\mat{A}_1$ is  left-$r(\pattern_1, \pattern_2)$-unitary and $\mat{A}_2$ is left-$r$-unitary for some integer $r$, then the product $\bA_1\bA_2 \in \setDBfactor{\pattern_1 * \pattern_2}$ is left-$r$-unitary.
        \item If $\mat{A}_1$  is  right-$r$-unitary for some integer $r$ and $\mat{A}_2$ is right-$r(\pattern_1, \pattern_2)$-unitary, then the product $\bA_1\bA_2 \in \setDBfactor{\pattern_1 * \pattern_2}$ is right-$r$-unitary.
    \end{enumerate}
\end{lemma}

\begin{proof}
%[Proof for \Cref{lemma:stableundermatrixmultiplication}]
	We only prove the first claim. The second one can be dealt with similarly.
Denote $\pattern_i = (a_i, b_i, c_i, d_i)$ {for $i \in \intset{2}$}. By \eqref{eq:operatorDBPfour}, we have:
    \begin{equation*}
        \pattern_1 * \pattern_2 = \left(a_1, \frac{b_1d_1}{d_2}, \frac{a_2c_2}{a_1}, d_2\right).
    \end{equation*}
    One can verify that $r \mid c_2 \mid a_2c_2/a_1$ (since $a_1 \mid a_2$).
    
 %   Consider $\bX, \bY$ that are left-$r(\pattern_1, \pattern_2)$-unitary and left-$r$-unitary factors respectively, it is sufficient to verify that for all $\pattern_3$ satisfying $r(\pattern_2, \pattern_3) = q$ and $\pattern_3$-factor $\bZ$, we have: $\|\bX\bY\bZ\|_F = \|\bZ\|_F$.

  %  Note that $r(\pattern_1, \pattern_2) = r(\pattern_1, \pattern_2 * \pattern_3)$ (cf.~\Cref{lem:associativity}), we have:
   % \begin{equation*}
    %    \begin{aligned}
     %       \|\bX\bY\bZ\|_F &= \|\bY\bZ\|_F \quad (\text{since } \bY\bZ \in \setDBfactor{\pattern_2 * \pattern_3}, r(\pattern_1, \pattern_2 * \pattern_3) = r(\pattern_1, \pattern_2))\\
      %      &= \|\bZ\|_F \quad (\text{by definition}).
      %  \end{aligned}
    %\end{equation*}
 %   This concludes the proof.

    Given any pattern $\pattern_3$ satisfying $r(\pattern_1*\pattern_2, \pattern_3) = r$ and any $\pattern_3$-factor $\mat{A}_3$, we have 
    %we have: $\|\mat{A}_1\mat{A}_2\mat{A}_3\|_F = \|\mat{A}_3\|_F$.}
 %\ER{ We have:}
    \begin{equation*}
            \|\mat{A}_1\mat{A}_2\mat{A}_3\|_F = \|\mat{A}_2\mat{A}_3\|_F   = \|\mat{A}_3\|_F,
    \end{equation*}
   where the first equality comes from the fact that $ \mat{A}_2\mat{A}_3 \in \setDBfactor{\pattern_2 * \pattern_3}$,  $\mat{A}_1$ is $r(\pattern_1,\pattern_2)$-left-unitary  and $ r(\pattern_1, \pattern_2 * \pattern_3)=r(\pattern_1, \pattern_2) $ from \Cref{lem:associativity}; while the second equality holds because   $\mat{A}_2$ is left-$r$-unitary and $ r(\pattern_1* \pattern_2, \pattern_3)=r(\pattern_2, \pattern_3)=r$ from \Cref{lem:associativity}.
    This concludes the proof.
\end{proof}

%\subsubsection{Preservation of the inner product under matrix multiplication}
\subsection{Characterization of $r$-unitary factors and proof of \Cref{lemma:role-pseudo-orthogonality}}
We explicitly characterize left/right-$r$-unitary factors.
This characterization %new definition 
reveals why pseudo-orthonormalization generates left/right-$r$-unitary factors (i.e., why \Cref{lemma:role-pseudo-orthogonality} holds). 
%\subsubsection{An alternative \ER{version} of \Cref{def:q-unitary}}
We first explain why $r \mid c$ is required in \Cref{def:q-unitary}.

\begin{lemma}
    Consider $\pattern:= (a,b,c,d)$ a factor pattern and $q\in\mathbb{N}$. 
    \begin{itemize}
        \item 
There exists $\pattern': = (a',b',c',d')$ such that $r(\pattern, \pattern') = r$ if, and only if, $r \mid c$. 
     \item 
There exists $\pattern': = (a',b',c',d')$ such that $r(\pattern', \pattern) = r$ if, and only if, $r \mid b$. 
\end{itemize}
\end{lemma}
\begin{proof}
We prove the first claim, the second one is proved similarly.
    If $(\pattern, \pattern')$ is chainable, then (by \Cref{def:chainableDBfour}) $r(\pattern, \pattern') = ac/a'$ and $a \mid a'$, thus $c = r(a'/a)$ and $r \mid c$. Conversely if $r \mid c$ then $\pattern' = (ac/r, r, c, d)$ satisfies the requirements.
\end{proof}
% \RG{Similarly we have}
% \begin{lemma}
%     Consider $\pattern:= (a,b,c,d)$ a factor pattern and $q\in\mathbb{N}$. There exists a pattern $\pattern': = (a',b',c',d')$ satisfying $r(\pattern', \pattern) = r$ if and only if $r \mid b$. 
% \end{lemma}
%A similar result explaining the requirement $r \mid b$ can be announced and proved similarly.
Next, we consider the partition $\cP = \cP(\bS_{\pattern}, \bS_{\pattern'})$ of $\intset{q}$ (with $q = acd$) and the sets $R_P, C_P$ (for $P \in \cP$) from \Cref{def:classequivalence}, with chainable $(\pattern,\pattern')$. By~\Cref{lem:prodDBBis}, the sets $R_P \times C_P$ are pairwise disjoint, and for any $\pattern$-pattern $\bX$ and $\pattern'$-pattern $\bY$ we have
\(
\supp(\bX\bY) \subseteq 
\supp(\bS_{\pattern}\bS_{\pattern'})
= \supp(\bS_{\pattern * \pattern'}) = 
\cup_{P \in \cP} R_P \times C_P,
\)
so that
\begin{equation}
		\label{eq:decomposeorthonormal}
		\|\bX\bY\|_F^2 =  \sum_{P \in \cP}  \|(\bX\bY)[R_P,C_P]\|_F^2
  \overset{\eqref{eq:submatrixequality}}{=} 
   \sum_{P \in \cP}
  \|\bX[R_P,P]\bY[P,C_P]\|_F^2.
\end{equation}
This hints that the left-$r$-unitary of $\pattern$-factor $\bX$  can be chararacterized via the blocks $\bX[R_P,P]$, $P \in \cP$. To explicit this characterization  we prove
that if $\pattern'$ satisfies $r(\pattern,\pattern') = r$ then 
$\cP = \cP(\bS_{\pattern},\bS_{\pattern'})$ does not depend on $\pattern'$. Since it partitions the set $\intset{q}$, which indexes the columns of $\bX$, we denote it $\cP_{\mathtt{col}}(\pattern,r)$.
\begin{lemma}
	\label{lem:invariance-partition-col}
  %  Consider a pattern $\pattern := (a,b,c,d)$ and a natural number $r \mid c$. The partition $\mathcal{P}(\mat{S}_{\pattern}, \mat{S}_{\pattern'})$ does not depend on $\pattern'$ for any $\pattern'$ satisfying $(\pattern, \pattern')$ is chainable with $r(\pattern, \pattern') = r$. In particular, $\mathcal{P}(\mat{S}_{\pattern}, \mat{S}_{\pattern'})$ is given by:
   Consider a pattern $\pattern := (a,b,c,d)$, an integer $r \mid c$, and $\pattern'$ such that $(\pattern, \pattern')$ is chainable with $r(\pattern, \pattern') = r$.
   The partition $\cP := \cP(\mat{S}_{\pattern}, \mat{S}_{\pattern'})$ of \Cref{def:classequivalence} does not depend on $\pattern'$. We denote it $\mathcal{P}_{\mathtt{col}}(\pattern,r)$.    
   We have $|\cP|=acd/r=a'd$.
 %  \ER{ Consider a pattern $\pattern := (a,b,c,d)$ and a natural number $r \mid c$. For any $\pattern'$ such that $(\pattern, \pattern')$ is chainable with $r(\pattern, \pattern') = r$, the partition $\mathcal{P}(\mat{S}_{\pattern}, \mat{S}_{\pattern'})$ (cf.~\Cref{def:classequivalence}) does not depend on $\pattern'$. In particular, $\mathcal{P}(\mat{S}_{\pattern}, \mat{S}_{\pattern'})$ is 
   %given by
 %   \RG{the collection of all the index sets}
 %   :}
	% \begin{equation}
	% 	\label{eq:partitionformule}
	% 	%\{ 
 %  P_{t,k} %\}_{t, k} 
 %  := 
 %  %\left \{ 
 %  \{k + (t-1) dr + (j-1) d 
 %  \}_{j \in \intset{r}} 
 %  \,\  \RG{\text{where}}\ 
 %  %| \, 
 %  (t,k) \in \intset{ac/r} \times \intset{d}. %\right \}.
	% \end{equation} 
\end{lemma}
We can similarly define $\mathcal{P}_{\mathtt{row}}(\pattern,r)$ when $r \mid b$, with an analog property $\mathcal{P}(\mat{S}_{\pattern'}, \mat{S}_{\pattern}) = \mathcal{P}_{\mathtt{row}}(\pattern,r)$ whenever $(\pattern',\pattern)$ is chainable with $r(\pattern',\pattern)=r$.
The proof of 	\Cref{lem:invariance-partition-col} is slightly postponed to immediately state and prove our main characterization.
%To go further we observe that $\cP(\bS_{\pattern}\bS_{\pattern'})}$ only is independent of 
% \begin{lemma}
%     \label{lem:invariance-partition-row}
%    % Consider a pattern $\pattern := (a,b,c,d)$ and a natural number $r \mid b$. The partition $\mathcal{P}(\mat{S}_{\pattern'}, \mat{S}_{\pattern})$ does not depend on $\pattern'$ for any $\pattern'$ satisfying $(\pattern', \pattern)$ is chainable with $r(\pattern', \pattern) = r$.
%   \ER{  Consider a pattern $\pattern := (a,b,c,d)$ and a natural number $r \mid b$. For any $\pattern'$ such that $(\pattern', \pattern)$ is chainable with $r(\pattern', \pattern) = r$, the partition $\mathcal{P}(\mat{S}_{\pattern'}, \mat{S}_{\pattern})$ does not depend on $\pattern'$ .}
% \end{lemma}
\begin{theorem}
	\label{theorem:orthonormalDBfactor}
	Consider a pattern $\pattern:=(a,b,c,d)$ and a natural number $r$ such that $r \mid c$ (resp.~$r \mid b$). Let $\bX$ be a $\pattern$-factor. The following claims are equivalent:
    \begin{enumerate}
        \item $\bX$ is left-$r$-unitary (resp.~right-$r$-unitary).
        \item For each $P \in \cP_{\mathtt{col}}(\pattern, r)$ (resp.~$P \in \cP_{\mathtt{row}}(\pattern, r)$), 
        %$\bX[C_P,P]$
        $\bX[R_P,P]$
        (resp.~%$\bX[P,R_P]$
        $\bX[P,C_P]$) has orthogonal columns (resp.~rows) ($R_P,C_P,P$ are defined as in \Cref{def:classequivalence}).
    \end{enumerate}
\end{theorem}

\begin{proof}[Proof of~\Cref{theorem:orthonormalDBfactor}]
    We prove the claim for left-$r$-unitarity (right-$r$-unitarity is proved similarly). Let $\pattern'$ be a pattern satisfying $r(\pattern,\pattern') = r$. 
    Denoting $\cP := \cP(\bS_{\pattern},\bS_{\pattern'})$, by \Cref{lem:invariance-partition-col} we have $\cP=\cP_{\mathtt{col}}(\pattern,r)$. We exploit~\eqref{eq:decomposeorthonormal}.
    %By \Cref{lem:qperfectcovering}, we have $\supp(\bX\bY) \subseteq \cup_{P \in \cP(\bS_{\pattern}, \bS_{\pattern'})} R_P \times C_P = \cup_{P \in \cP_{\mathtt{col}}(\bS_{\pattern}, q)} R_P \times C_P$ (\TL{by \Cref{def:qpartition}}) where $\{ R_P \times C_P \}_{P}$ are pairwise disjoint, we have  Hence:
	% \begin{equation}
	% 	%\label{eq:decomposeorthonormal}
	% 	\|\bX\bY\|_F^2 =  \sum_{P \in \cP_{\mathtt{col}}(\bS_{\pattern}, q)} \|(\bX\bY)[R_P,C_P]\|_F^2.
	% \end{equation}
    %Each summand in 
    %By\eqref{eq:decomposeorthonormal} 
    %can be explicitly calculated as follows:
    % \begin{equation*}
    %     \|(\bX\bY)[R_P,C_P]\|_F^2 \overset{\eqref{eq:submatrixequality}}{=} \|\bX[R_P,P]\bY[P,C_P]\|_F^2.
    % \end{equation*}
	\begin{enumerate}[leftmargin=*]
	    \item Assume that $\bX$ is left-$r$-unitary, and fix an arbitrary $P \in \cP$. For any $\pattern'$-factor $\bY$ such that $\bY[R_{P'},P'] = \mbf{0}$ for all $P' \in \cP$, $P' \neq P$, by left-$r$-unitarity of $\bX$ we have:
        \begin{equation*}
            \begin{aligned}
                \|\bX[R_P,P]\bY[P,C_P]\|_F^2
                \stackrel{\eqref{eq:decomposeorthonormal}}{=}
                \|\bX\bY\|_F^2 = 
                %\|(\bX\bY)[R_P,C_P]\|_F^2 = 
                \|\bY\|_F^2 =  \|\bY[P,C_P]\|_F^2.
            \end{aligned}
        \end{equation*}
        Thus $\bX[R_P,P]$ preserves the Frobenius norm of $\bY[P,C_P]$ upon left multiplication. Since this holds for any choice of $\bY[P,C_P]$, the matrix $\bX[R_P,P]$ has orthogonal columns. %The same argument is applied for any $P \in \cP_{\mathtt{col}}(\bS_{\pattern}, q)$, which 
        This shows the first implication.
        \item Assume that $\bX[R_P,P]$ has orthogonal columns for each $P \in \cP$. For each $\pattern'$-factor $\bY$ and $P \in \cP$ we have 
        $\|\bX[R_P,P]\bY[P,C_P]\|_F^2 = \|\bY[P,C_P]\|_F^2$, thus,
        \begin{align*}
            \|\bX\bY\|_F^2                  \stackrel{\eqref{eq:decomposeorthonormal}}{=}
 \sum_{P \in \cP} \|\bX[R_P,P]\bY[P,C_P]\|_F^2 
 & = \sum_{P \in \cP} \|\bY[P,C_P]\|_F^2 \\
&  = \sum_{P \in \cP} \|\bY[P,:]\|_F^2
 = \|\bY\|_F^2.
        \end{align*}
        where the first equality of the second row results from the fact that the sub-matrix $\bY[
        \row{P}]$ has the form  $[\begin{smallmatrix} \bY[P, C_P] & \mbf{0}\end{smallmatrix}]$ up to some column permutations. To see this fact, we remind readers that by definition of $\cP = \cP(\bS_\pattern,\bS_{\pattern'})$ (see \Cref{def:classequivalence}) all rows of 
        $\bS_{\pattern'}[\row{P}]$ are equal and their support is $C_P$, and since $\bY \in \setDBfactor{\pattern'}$ the corresponding rows have support included in
        %$\bY$ has the same support  constraints and all supports are equal to
        $C_P$ . 
%\todo{Ici il faut expliquer la première égalité de la seconde ligne \TL{Done}}
  %      since $\cP$ is a partition of $\intset{q}$, which indexes the set of rows of $\bY$.
	\end{enumerate}
 %   The proof is concluded.
\end{proof}

To prove \Cref{lem:invariance-partition-col} we will actually show that under its assumptions we have $\cP(\bS_{\pattern}, \bS_{\pattern'}) = \cP_{\mathtt{col}}(\pattern,r)$ where $\cP_{\mathtt{col}}(\pattern,r)$ is specified as follows:
\begin{definition}
    \label{def:formula-partition}
    Consider a pattern $\pattern := (a,b,c,d)$ and a natural number $r \mid c$. For each pair of integers $(t,k) \in \intset{ac/r} \times \intset{d}$ denote
\begin{equation}
		\label{eq:partitionformuledef}
      P_{t,k}  :=   \{k + (t-1) dr + (j-1) d \}_{j \in \intset{r}} \subseteq \intset{q}\ \text{with}\ q:= acd,
	\end{equation} 
and define $\mathcal{P}_{\mathtt{col}}(\pattern,r) := \{P_{t,k}\}_{t,k}$.
An illustration for these sets is given on \cref{fig:classequivalence}.
\begin{figure}[H]
	\centering
	\includegraphics[width=1.0\textwidth]{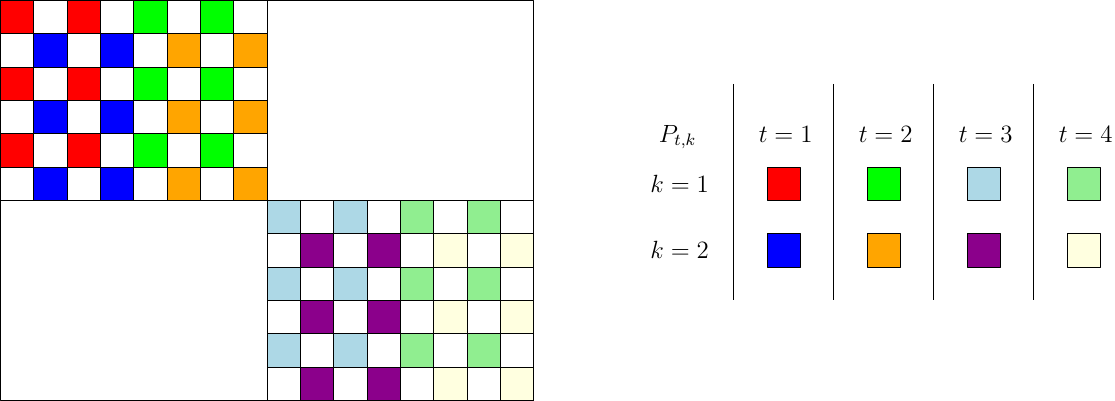}
	\caption{The partition $\cP_{\mathtt{col}}(\pattern, r)$ with $\pattern = (2, 3, 4, 2)$ and $r = 2$. Each set $P_{t,k}, (t,k) \in \intset{4} \times \intset{2}$ gathers the indices of the columns of $\bS_\pattern$ of a given color. See also \Cref{fig:DBfactorillu}. Same color indicates columns of the same sets $P_{t,k}$ (cf.~\cref{def:formula-partition}).}
	\label{fig:classequivalence}
\end{figure}
\end{definition}

\begin{proof}[Proof of \Cref{lem:invariance-partition-col}]
It is easy to check that $\cP' := \cP_{\mathtt{col}}(\pattern,r)$ partitions the index set $\intset{q}$, with $q = acd$, into $|\cP'|=acd/r$ components $P = P_{t,k} \in \cP'$ of equal cardinality $|P|=r$. Since, by \Cref{lem:prodDBBis}, the partition $\cP := \cP(\mat{S}_{\pattern}, \mat{S}_{\pattern'})$ of $\intset{q}$ has the same property, it is sufficient to prove that, for any $P \in \cP$, there exists $(t, k) \in \intset{ac/r} \times \intset{d}$ such that $P = P_{t,k}$ (such a pair $(t,k)$ is necessarily unique since we consider partitions).

    Given an arbitrary $P \in \mathcal{P}(\mat{S}_{\pattern}, \mat{S}_{{\pattern}'})$, we prove first that there exists a unique $t \in \intset{ac/r}$ such that $P \subseteq P_t := \bigcup_{k \in \intset{d}} P_{t,k}$.  
    Notice (see also \Cref{fig:classequivalence}) that each $P_t $ is an interval of length $dr$ and since (by chainability of $(\pattern, {\pattern}')$) we have $acd = {a}' {b}' {d}'$, $ac/r = {a}'$, and $dr = {b}' {d}'$, 
 \begin{equation}
    \label{eq:partition-It}
        P_t
        = \intset{(t-1)dr + 1, tdr} = \intset{(t-1)b'd' + 1, tb'd'}.
        % = \bigcup_{k \in \intset{d}} \{k + (t-1) dr + (j-1) d \}_{j \in \intset{r}}
    \end{equation}
    Further observe that $\{P_t\}_{t \in \intset{ac/r}}$ partitions
    %In other words: $\{ I_{t} \}_{t \in \intset{ac/r}} = \{ (t-1) {b}' {d}' + 1, t {b}' {d}' \}_{t \in \intset{{a}'}}$ is 
    $\intset{q} = \intset{{a}' {b}' {d}'}$ into ${a}'$ consecutive intervals of length ${b}' {d}'$.     In other words, $P_t$ indexes the rows of the $t$-th block of $\bS_{\pattern'}$, cf.~\Cref{fig:DBfactorillu}.
    %Since $\{P_t\}_{t \in \intset{ac/r}}$ partitions $\intset{q} = \intset{acd}$,
    As a result, there exists at least one index $t$ such that $P_t \cap P \neq \emptyset$. We next prove the uniqueness of such a $t$, which implies $P \subseteq P_t$ as we consider a partition of $\intset{q}$.
    %into $ac/r$ consecutive intervals of length $dr$. 
    %, because $\{ P_{t,k} \}_{t, k}$ and $\mathcal{P}(\mat{S}_{\pattern}, \mat{S}_{{\pattern}'})$ have the same cardinal: indeed, by \Cref{lem:prodDB2}, $| P | = q$ for any $P \in \mathcal{P}(\mat{S}_{\pattern}, \mat{S}_{{\pattern}'})$, so the cardinal of $\mathcal{P}(\mat{S}_{\pattern}, \mat{S}_{{\pattern}'})$ is $acd / q$ (since there are $acd$ columns in $\mat{S}_{\pattern}$), which is precisely the cardinal of the set \eqref{eq:partitionformule}.
    % as:
    % \begin{align*}
    %     \forall t \in \intset{ac/r}, \quad I_{t} := \intset{(t-1)dr + 1, tdr}. 
    % \end{align*}
%For each $t \in \intset{ac/r}$, the interval $I_{t}$ of can be partitioned into as:
           %Now consider $P \in \cP(\bS_\pattern,\bS_{\pattern'})$. 
    Consider two indices $t,t'$ such that $P_t \cap P \neq \emptyset$ and $P_{t'} \cap P \neq \emptyset$, and $i \in P_t \cap P$, $i' \in P_{t'} \cap P$. By \Cref{def:classequivalence}, since $i,i' \in P$ and $P \in \mathcal{P}(\mat{S}_{\pattern}, \mat{S}_{{\pattern}'})$, 
    we have $\mat{S}_{{\pattern}'}[i, :] = \mat{S}_{{\pattern}'}[i',:]$. %Let us show that this implies $t=t'$.
    Since $\bS_{\pattern'} = \bI_{a'} \otimes \mbf{1}_{b' \times c'} \otimes \bI_{d'}$ is block-diagonal with $a'$ blocks of size $b'd' \times c'd'$, the row $\mat{S}_{{\pattern}'}[i, :]$ is supported in a subset of the interval $\intset{(\ell-1)c'd' + 1, \ell c'd'}$ with $\ell \in \intset{a'}$ the unique integer such that $i \in \intset{(\ell-1)b'd'+1,\ell b'd'} = P_\ell$. Since $i \in P_t$ we have $\ell=t$. The same holds for $\mat{S}_{{\pattern}'}[i', :]$ with $\ell'=t'$. Since both rows are identical, we deduce that $\ell=\ell'$, hence $t=t'$.
    % \TL{Due to the first claim of the technical \cref{lemma:row-col-inclusion} (stated and proved right after this proof), the support of each row 
    % %\todo{column? toujours pas compris}
    % of $\bS_{\pattern'}$ lies entirely in intervals of the form $\intset{(t-1)c'd' + 1, tc'd'}$ for some $t \in \intset{a'}$. Thus,\todo{Désolé il va falloir m'expliquer ce "thus"} $t = t'$.} %this is only possible if $t=t'$.
    % \todo{\textcolor{red}{TODO Tung: preuve à corriger}; faire un lemme séparé si le résultat peut être ré-utilisé ailleurs.\TL{OK}}
    % %Ajouter des détails ici?\TL{J'ai fait une tentative}}\RG{\bf Dsl toujours pas compris: chaque ligne a un support dans un bloc $\intset{tc'd'+1,(t+1)c'd'}$, pour $t \in \int{a'}$, mais ce n'est pas $P_t$?}
    
     Considering now the unique $t \in \intset{ac/r}$ such that $P \subseteq{P_t}$, since $\{P_{t,k}\}_{k \in \intset{d}}$ partitions $P_t$, there must exist $k \in \intset{d}$ such that $P_{t,k} \cap P \neq \emptyset$. Again, we now prove the uniqueness of such a $k$, which implies that $P \subseteq P_{t,k}$, and since $| P | = q = |P_{t,k}|$, we will obtain  $P = P_{t,k}$ as claimed. 
     Consider two indices $k,k' \in \intset{d}$ such that $P_{t,k} \cap P \neq \emptyset$ and $P_{t,k'} \cap P \neq \emptyset$, and $i \in P \cap P_{t,k}$ and $i' \in P \cap P_{t,k'}$.
     By construction (see~\cref{def:formula-partition}), we have $
     %(i \mod d) = k
     {i \equiv k \mod d}
     $ and $
     {i \equiv k' \mod d}
     %(i' \mod d) = k'
     $.
      To continue we observe that $\{P_{t,k}\}_{k \in \intset{d}}$ partitions the interval $P_t$ into $d$ (disjoint) subsets of integers of cardinality $r$ and the elements in such subsets are equally spaced by a distance $d$.
      %\todo{VERIFIER \TL{deja verifie}}
    Moreover, as above, since $i,i' \in P$ and $P \in \cP(\bS_\pattern,\bS_{\pattern'})$, we have
     $\mat{S}_{\pattern}[:,i] = \mat{S}_{\pattern}[:,i']$.
     By the structure 
     $\mat{S}_{\pattern} = \mat{I}_{a} \otimes \mat{1}_{b \times c} \otimes \mat{I}_d$, 
     for each $k$ the columns of $\bS_\pattern$ such that $\supp(\bS[\col{i}]) \subseteq P_{t,k}$
     %in the same group \todo{Que veux-tu dire par "group" ici?} 
     share in fact the same support, which is disjoint from the other column supports 
     %and those from different groups have disjoint supports}
     (see illustration of \cref{fig:classequivalence}). 
     %\ER{can we distinguish them by the color? \TL{Added}}
     %This yields
     %\todo{Tung: peux-tu détailler?\TL{Fait}} $k=k'$. 
     %A fortiori, indices belonging to different $P_{t,k}, P_{t,k'}$ for $k \neq k'$ cannot be in the same equivalence class $P \in \mathcal{P}(\mat{S}_{\pattern}, \mat{S}_{{\pattern}'})$, by \Cref{def:classequivalence}.
 
%    \todo{RG: je reprends ici}
%      \todo{RG: Tung, peux-tu clarifier la fin de preuve stp ? }
      Finally, by \Cref{def:chainableDBfour} and chainability of $(\pattern, \pattern')$, we have $r = r(\pattern,\pattern') = ac / a'$, so that $|\cP|=|\cP'| = acd/r = a'd$ as claimed.
\end{proof}

We conclude the section with a lemma that relates $\cP_{\mathtt{col}}(\cdot,r)$ for various patterns. 
%to $\cP_{\mathtt{col}}(\pattern_1 * \pattern_2, q)$.
\begin{lemma}
    \label{cor:invariant_pc}
    Consider a chainable architecture $\arch = (\pattern_\ell)_{\ell = 1}^L$ and a natural number $r \mid c_L$ where $\pattern_L = (a_L,b_L,c_L,d_L)$.
    %such that $\cP_{\mathtt{col}}(\pattern_L, q)$ is well-defined (in the sense of \Cref{lem:invariance-partition-col}). 
    Then $r \mid c$ where $\pattern_1 * \ldots * \pattern_L = (a,b,c,d)$, and $\cP_{\mathtt{col}}(\pattern_1 * \ldots * \pattern_L, r) = %$ is also well-defined and equals  
    \cP_{\mathtt{col}}(\pattern_L, r)$.
\end{lemma}

% \begin{lemma}
% 	\label{lem:invariant_pc}
% 	Consider a pattern $\pattern_2:=(a_2,b_2,c_2,d_2)$ and a natural number $r \mid c_2$. For every $\pattern_1$ such that $(\pattern_1,\pattern_2)$ is chainable, $\cP_{\mathtt{col}}(\pattern_1 * \pattern_2, q)$ is also well-defined and $\cP_{\mathtt{col}}(\pattern_2, q) = \cP_{\mathtt{col}}(\pattern_1 * \pattern_2, q)$. 
% \end{lemma}
    %   \todo{TODO : replace $P_{t,k}$, $P_t$ by $P_*$; in algorithm (and proofs), replace $P_k$ by $I_k$}
       %we may later want to rather use $I$ everywhere instead of $P$, including in $\Cref{def:classequivalence}$). This would avoid the notation conflict with $P \subset \intset{L-1}$ in the algorithm. \TL{You mean $\cI_{\mathtt{col}}$ instead of $\cP_{\mathtt{col}}$?}Non, mais on en discutera de vive voix.}

\begin{proof}
	We prove the result for $L=2$. The general case follows by induction. Consider $\pattern_i = (a_i,b_i,c_i,d_i)$, with $i \in \intset{2}$, and an integer $r \mid c_2$.
	%Assume that $\pattern_1 = (a_1, b_1, c_1, d_1)$. 
%	First, it is easy to verify that $\cP_{\mathtt{col}}(\pattern_1 * \pattern_2, q)$ is well-defined (cf.~\Cref{lemma:stableundermatrixmultiplication}).\todo{Detail? Pas clair pour moi}
	%	To prove that $\cP_{\mathtt{col}}(\pattern_2, q) = \cP_{\mathtt{col}}(\pattern_1 * \pattern_2, q)$, we use the
	By definition of the $*$ operator in \Cref{eq:operatorDBPfour} we have
	\begin{equation*}
		(a,b,c,d):=\pattern_1 * \pattern_2 = \left(a_1, \frac{b_1d_1}{d_2}, \frac{a_2c_2}{a_1}, d_2\right) 
	\end{equation*}
	Since $(\pattern_1,\pattern_2)$ are chainable, we have $a_1 \mid a_2$. Since $r \mid c_2$, we have $r \mid (a_2/a_1)c_2$, i.e., $r \mid c$ as claimed.
	A direct calculation then shows that $\cP_{\mathtt{col}}(\pattern_2, r) = \cP_{\mathtt{col}}(\pattern_1 * \pattern_2, r)$, 
	since for each $t \in \intset{ac/r} = \intset{a_2c_2/r}$, $k \in \intset{d}=\intset{d_2}$, \eqref{eq:partitionformuledef}  yields the very same set $P_{t,k}$ (as $dr = d_2r$ and $d=d_2$).
	%which share all elements of the form:
% 	\begin{equation*}
% 		\{k + tq_2d_2, k + d_2 + tq_2d_2, \ldots, k + d_2(q_2 - 1) + tq_2d_2\}, \forall k \in \intset{d_2}, \forall t = 0, \ldots, a_2c_2/r_2 - 1.
% 	\end{equation*}
%  \textcolor{white}{a}
\end{proof}

\subsubsection{Proof of \Cref{lemma:role-pseudo-orthogonality}}
\label{app:role-pseudo-orthogonality}
This %\Cref{lemma:role-pseudo-orthogonality} 
is a direct %an immediate 
corollary of \Cref{lemma:exchange} below.

%\TL{On the final note on the proof of \Cref{lemma:role-pseudo-orthogonality},
With the notations of \Cref{lemma:role-pseudo-orthogonality}, the constant $r$ involved in \Cref{lemma:exchange} applied to $\pattern = \pattern_{I_i}$ and $\pattern' = \pattern_{I_{i+1}}$ for $i = 1, \ldots, j-1$ is indeed equal to $r(\pattern_{t_i}, \pattern_{t_{i} + 1})$:
% is due to the following reasoning:
\begin{enumerate}
    \item %It is not hard to see that 
    The intervals $I_i:=\intset{q_i,t_i}$  in $\texttt{partition}$ are a sorted partition of $\intset{1,L}$. Therefore, $t_i + 1 = q_{i+1}$ for each $i$.
    \item By~ \Cref{lem:seq-qchainability}, $r(\pattern_{I_i},\pattern_{I_{i+1}}) = r(\pattern_{t_i},\pattern_{q_{i+1}}) = r(\pattern_{t_i},\pattern_{t_i + 1})$.
\end{enumerate}

%\todo{Homogénéiser: l'algorithme utilise $\bX,\bY$, le lemme $\bX_1,\bX_2$ \TL{Done}}
\begin{lemma}
	\label{lemma:exchange}
    Consider a non-redundant chainable pair $(\pattern, \pattern')$, and $(\bX, \bY) \in \setDBfactor{\pattern} \times \setDBfactor{\pattern'}$.
    Denote $r := r(\pattern,\pattern')$.
    \Cref{algo:exchange} with input $u \in \{ \texttt{column},\texttt{row}\}$ returns $(\tilde{\bX}, \tilde{\bY}) \in \setDBfactor{\pattern} \times \setDBfactor{\pattern'}$ such that $\tilde{\bX} \tilde{\bY} = \bX\bY$ and:
	\begin{itemize}
		\item if  $u = \texttt{column}$ then $\tilde{\bX}$ is left-$r$-unitary;
		\item otherwise if $u = \texttt{row}$ then $\tilde{\bY}$ is right-$r$-unitary.
  %and $\tilde{\bX}_1 \tilde{\bX}_2 = \bX_1\bX_2$.
	\end{itemize}
\end{lemma}
% The proof of \Cref{lemma:exchange} is demonstrated in \Cref{appendix:proof3lemmas} altogether with its corollary: \Cref{lem:productleftandrightorthonormal}. Below, we list several properties of orthonormal butterfly factors, which will be used to prove for other remaining results.

\begin{proof}
% [Proof of \Cref{lemma:exchange}]%[Sketch of the proof]
	We only prove the first point, as the second point can be addressed similarly. Denoting $\cP = \cP(\bS_{\pattern}, \bS_{\pattern'})$, by \Cref{lem:prodDBBis}, for any $(\tilde{\mat{X}}, \tilde{\mat{Y}}) \in \setDBfactor{\pattern} \times \setDBfactor{\pattern'}$, $\supp(\bX\bY)$ and $\supp(\tilde{\bX} \tilde{\bY})$ are both included in $\bigcup_{P \in \cP} R_P \times C_P$ where $\{ R_P \times C_P \}_{P \in \cP}$ are pairwise disjoint. As a result, we have:
     \begin{equation*}
     \begin{split}
          \bX\bY = \tilde{\bX}\tilde{\bY} &\iff \forall P \in \cP, \quad (\bX\bY)[R_P,C_P] = (\tilde{\bX}\tilde{\bY})[R_P,C_P] \\
         &\iff \forall P \in \cP, \quad (\bX\bY)[R_P,C_P] = \tilde{\bX}[R_P,P]\tilde{\bY}[P,C_P],
     \end{split}
     \end{equation*}
    where the second equivalence comes by \Cref{lemma:submatrixequality} and by chainability of $(\pattern_1, \pattern_2)$.
    Moreover, by \Cref{theorem:orthonormalDBfactor}, $\tilde{\bX} \in \setDBfactor{\pattern}$ is left-$r$-unitary if, and only if, $\tilde{\bX}[R_P,P]$ has orthonormal columns for each $P \in \cP_{\mathtt{col}}(\pattern, r) = \cP(\bS_{\pattern},\bS_{\pattern'}) = \cP$. 

    The output $(\tilde{\mat{X}}, \tilde{\mat{Y}})$ of  \Cref{algo:exchange} 
%    \todo{Ligne 2 de \Cref{algo:exchange}, pointer vers  \Cref{def:classequivalence} pour la définition de la notation $\mathcal{P}(\cdot,\cdot)$?}
is built via the QR-decomposition of ${\bX}[R_P,P] = \bQ\bR$, and setting 
    %\todo{RG: il faut faire bien ressortir ici pourquoi on a besoin de l'hypothese de non-redondance \TL{Done}}
    $\tilde{\bX}[R_P,P] = \bQ$, $\tilde{\bY}[P,C_P] = \bR\bY[P,C_P]$, which is possible because $(\pattern, \pattern')$ is assumed to be non-redundant. 
    Indeed, to enable %since 
    $\bQ \in \RR^{|R_P| \times |P|}$ to have %has 
    orthonormal columns, we must have $|R_P| \geq |P|$, or equivalently, $b \geq q$ (assuming that $\pattern = (a,b,c,d)$, see \Cref{lem:prodDBBis}). This is the non-redundancy criterion (cf.~\Cref{def:redundant}).
    Thus, $\tilde{\bX}[R_P,P]$ has orthonormal columns and $\tilde{\bX}[R_P,P]\tilde{\bY}[P,C_P] = \bQ\bR\bY[P,C_P] = \bX[R_P,P]\bY[P, C_P] = (\bX\bY)[R_P,C_P]$.
%This yields the claim of the lemma.
\end{proof}

% \LZ{A more formal presentation of \Cref{algo:exchange} will be provided in \Cref{app:orthonormalDB}.}

\subsubsection{Further useful technical properties of $r$-unitary factors}

\begin{lemma}
	\label{lemma:orthonormasubmatrix}
	Let $\bX$ be a left-$r$-unitary $\pattern$-factor with $\pattern = \paramfour$. 
	For any $t \in \intset{ac/r}$,
	%For 
% 	any subset of column indices of the form \RG{$P_t$ given in ~\eqref{eq:partition-It}, with $t \in \intset{ac/\RG{?}}$} 
% 	\todo{$P_\ell \to P_\ell$} $I_\ell = \intset{1 + \ell rd, (\ell+1)rd}, \forall \ell =  0, \ldots, ac/r - 1$, 
	the submatrix $\bX[:,P_t]$ %I_\ell, \ell =t-1
	with $P_t$ as in~\eqref{eq:partition-It} has orthonormal columns.
\end{lemma}

\begin{proof}
	%Since $\bA$ is left-$r$-unitary, 
	From the proof of \Cref{lem:invariance-partition-col} the sets $(P_{t,k})_{k \in \intset{d}}$ from \Cref{def:formula-partition} partition $P_{t}$ into $d$ subsets of integers of cardinality $r$ and the elements of each subsets are equally spaced by a distance $d$. We first show that the columns of $\bX[\col{P_t}]$ coming from distinct blocks $P:=P_{t,k}$, $P':=P_{t,k'}$ (where $1 \leq k \neq k' \leq d$) are mutually orthogonal, as they have disjoint supports. 
	Due to the structure of $\bS_\pattern = \bI_a \otimes \mbf{1}_{b \times c} \otimes \bI_d$, if two columns share the same support, the difference of their indices is divisible by $d$. Otherwise, their supports are disjoint. Thus, $P$ and $P'$ have disjoint column supports. 
	
		%\todo{todo: details \TL{OK}}
		To conclude we show that each submatrix $\bX[\col{P_{t,k}}]$ has orthonormal columns.
	Indeed, by \Cref{def:formula-partition}, we have $P := P_{t,k} \in \cP_{\mathtt{col}}(\pattern,r)$, hence by \Cref{theorem:orthonormalDBfactor} the block $\bX[R_P,P]$ has orthonormal columns. Since $R_P$ is the support of the columns of $\bS_{\pattern}[\col{P}]$, and as $\supp(\bX) \subseteq \supp(\bS_\pattern)$, the submatrix $\bX[\col{P}]$ is zero outside of the block $\bX[R_P,P]$ hence it also has orthonormal columns.

%\sout{Indeed, by \Cref{def:classequivalence}, the set $R_P$ (resp.~$R_{P'}$) is the support of all the columns of $\bS_{\pattern}[\col{P}]$ (resp.~of $\bS_{\pattern}[\col{P'}]$), and $R_P \cap R_{P'} = \emptyset$ because \ldots}

%	following form:\todo{Est-ce que l'expriomer comme intersection entre un intervalle et l'ensemble des entiers avec une certaine congruence pourrait simplifier l'analyse ?}
%	\begin{equation*}
% 		Z_{k,t}:= \{k + trd, k + d + trd, \ldots, k + d(r - 1) + trd\}, \forall k = 1, \ldots, d, \forall t = 0, \ldots, ac/r - 1.
% 	\end{equation*}
% 	It is noteworthy that $%I_\ell
% 	\RG{P_t}
% 	= \bigcup_{k = 1}^d Z_{k,\ell}$ \RG{with $\ell = t-1$}.

%	from the \emph{same} equivalence classes are
% 	\todo{Preuve a finir de simplifier; déplacer résultat comme corollaire plus tot ?}
% 	The supports of columns of different equivalence classes $Z_{k,t}, k = 1, \ldots, d$ are pairwise disjoint (since sub-blocks are diagonal matrices of size $d \times d$). Thus, columns of $\bA$ from the different equivalence classes $Z_{k,t}, k = 1, \ldots, d$ are orthogonal. 

%	also pairwise orthogonal and they are also of unit norm (cf.~\Cref{remark:equivalentdeforthonormal}). Thus, $\bA[:,I_\ell]$ has orthonormal columns.
\end{proof}
\begin{remark}
%\RG{REMI: Otez-moi d'un gros doute: les ensembles $P_t$ sont deux à deux disjoints, ceci implique-t-il que {\em toutes} les colonnes de $\bX$ sont orthonormales ? Si oui il y a une grosse simplification du papier à faire ! Sinon, rappeler / détailler un peu ici le contre-exemple de la \Cref{rem:leftrightnotplainunitary} de matrice left-$r$-unitary qui n'a pas des colonnes orthonormales. Il s'agit alors de dire que les supports des colonnes de $\bX[\col{P_t}]$ et $\bX[\col{P_{t'}}]$, $t \neq t'$, ne sont pas forcément disjoints. \TL{Added remarks}}
Although the columns of $\bX[\col{P_t}]$ are orthonormal and the sets $P_t$ are pairwise disjoint, it does not imply that $\bX$ has orthonormal columns. For example, consider the sets $P_1$ (indices of columns in red and blue) and $P_2$ (columns in green and orange)
%Indeed, as 
illustrated in \cref{fig:classequivalence}.
%, $P_1$ and $P_2$ (resp.~$P_3$ and $P_4$) 
Even though these two sets
are disjoint %but 
the supports of their columns are not. Thus, the matrix can be left-$2$-unitary without having orthonormal columns. This example also explains our earlier %\cref{rem:redundantcase}
\Cref{rem:leftrightnotplainunitary}.
\end{remark}

% \begin{remark}\todo{Fusionner avec preuve précédente}
% 	\label{remark:equivalentdeforthonormal}
% 	Yet, (another) equivalent definition of \Cref{def:q-unitary} is to require $\bA[:,P]$ (resp.~$\bA[P,:]$) to have orthogonal columns (resp.~rows). Indeed, in the case of column-orthonormality, the submatrix $\bA[\col{P}]$ has the form: $\bA[\col{P}] = \left(\begin{smallmatrix} \bA[R_P,P]\\ \mbf{0} \end{smallmatrix}\right)$ up to some row permutation. Thus, $\bA[\col{P}]$ has orthonormal columns if, and only if, $\bA[R_P,P]$ has orthonormal columns. The same reasoning holds in the case of row-orthonormality. This leads to \Cref{algo:exchange} %the following algorithm 
%  to implement the pseudo-orthogonalisation operations.
% \end{remark}
\begin{corollary}
	\label{cor:orthonormalsubmatrix-new}
	Consider a chainable pair of patterns $(\pattern, \pattern')$. If $\bX \in \setDBfactor{\pattern}$ is left-$r(\pattern, \pattern')$-unitary, then for any column index $i$ of $\bS_{\pattern'}$ the submatrix $\bX[\col{T_i}]$, where $T_i = \supp(\bS_{\pattern'}[\col{i}])$, has orthonormal columns.
\end{corollary}
\begin{proof}
    Denote $\pattern = (a,b,c,d)$, $\pattern' = (a', b', c', d')$, $r = r(\pattern,\pattern')$. 
    Since $\bS_{\pattern'} = \bI_{a'} \otimes \mbf{1}_{b' \times c'} \otimes \bI_{d'}$ is block-diagonal with $a'$ blocks of size $b'd' \times c'd'$, the set $T_i$ is a subset of the interval
     $P_t = \intset{(t-1)b'd'+1, t b'd'}$ with $t \in \intset{a'}$ the unique integer such that $i \in \intset{(t-1)c'd',t c'd'}$. By chainability of $(\pattern,\pattern')$ we have $b'd'=dr$.
    %
    % By the second claim of \cref{lemma:row-col-inclusion},\todo{mettre directemnt l'argument des blocs ici} one can verify
    % %\todo{Faire un lemme qui regroupe toutes ces propriétés utiles qui semblent dispersées et réu-tilisées dans diverse preuves? Ou insérer dans le \Cref{lem:prodDBBis}?\TL{OK, see \cref{lemma:row-col-inclusion}}} 
    % that $T_i$ is a subset of an interval $P_t := \intset{1 + (t-1)\tilde{b}\tilde{d}, t \tilde{b}\tilde{d}} = \intset{1 + (t-1)dr, t dr}$ for a certain $t \in \intset{\tilde{a}}$, where the equality $\tilde{b}\tilde{d} = dr$ comes from the chainability of $(\pattern, \tilde{\pattern})$.
 %
    By \Cref{lemma:orthonormasubmatrix}, $\bX[:,P_t]$ has orthonormal columns, and therefore so does $\bX[:,T_i]$.
\end{proof}

The following corollary will be handy to prove \eqref{eq:AnnexProofLocalBlock}, in the proof of \Cref{lem:originalkeylemma-new}.

\begin{corollary}
\label{cor:orthonormalsubmatrix-readytouse}
    Consider a chainable architecture of $\arch = (\pattern_1, \ldots, \pattern_L)$. Denote 
%\todo{Do we really need this lemma?}
    $$\pattern_1' := \pattern_1 * \ldots * \pattern_{q-1}, \quad \pattern_2' := \pattern_{q} * \ldots * \pattern_t, \quad \pattern_3' := \pattern_{t + 1} * \ldots * \pattern_L.$$ 
    If $\bX \in \setDBfactor{\pattern'_1}$ and $\bZ \in \setDBfactor{\pattern'_3}$ are respectively left-$r(\pattern_1', \pattern_2')$-unitary and right-$r(\pattern_2', \pattern_3')$-unitary, then $\bX[\col{R}]$ (resp.~$\bZ[\row{C}]$) has orthonormal columns (resp.~rows) for any column (resp.~row) support $R$ (resp.~$C$) of $\bS_{\pattern_2'}$, i.e., whenever $R = \supp(\bS_{\pattern_2'}[\col{i}])$ (resp.~$C = \supp(\bS_{\pattern_2'}[\row{i}])$) for some row (resp.~column) index $i$.
\end{corollary}
\begin{proof}
    We consider the case of columns of $\bX[\col{R}]$. The other case can be dealt with similarly. The proof follows from~\Cref{cor:orthonormalsubmatrix-new} because:
    \begin{enumerate}
        \item By~\Cref{lem:seq-qchainability}, $(\pattern_{1}', \pattern_{2}')$ is chainable with $r(\pattern_{1}', \pattern_{2}') = r(\pattern_{q-1}, \pattern_q)$;
        \item $\mat{X}$ is a left-$r(\pattern_1', \pattern_2')$-unitary $\pattern'_1$-factor;
        \item $R = \supp(\mat{S}_{\pattern_2'}[:,i])$ for some column index $i$.
        %\todo{Léon: $R = \supp(\mat{S}_{\pattern_2'}[:, i])$ instead of $R = \supp(\mat{S}_{\pattern_2'}[i, :])$?}
    \end{enumerate}
\end{proof}

\section{Complexity of hierarchical algorithms - proof of \Cref{theorem:complexity}}
\label{appendix:complexity}
% \LZc{prêt à la relecture}
%This section is devoted to prove \Cref{theorem:complexity}. 
We analyze the complexity of each of the components involved in the proposed hierarchical algorithms.

\subsection{Complexity of \Cref{algo:algorithm1}}
This algorithm essentially performs several low-rank approximations that are typically computed using truncated SVD.

\begin{lemma}
\label{lem:complexity-fsmf}
    Consider $(\mat{L}, \mat{R})$ satisfying the condition of \Cref{theorem:tractablefsmf}.
    For any matrix $\mat{A}$, the complexity of the two-factor fixed support matrix factorization algorithm (\Cref{algo:algorithm1}) with input $\mat{A}, \mat{L}, \mat{R}$ is $$\mathcal{O} \left( \sum_{P \in \cP(\bL,\bR)} |P||R_P||C_P| \right).$$
\end{lemma}

\begin{proof}
    The algorithm performs the best rank-$|P|$ approximation of a submatrix of size $|R_P| \times |C_P|$ for each $P \in \cP(\bL,\bR)$ and the complexity of the truncated SVD at order $k$ for an $m \times n$ matrix is $\mathcal{O}(kmn)$ \cite{halko2011svd}.
\end{proof}

We apply this complexity analysis to the case where $(\mat{L}, \mat{R})$ are butterfly supports corresponding to a chainable pair of patterns.

\begin{lemma}
\label{lem:complexity-fsmf-butterfly}
    Consider chainable patterns %$(\pattern, \pattern')$, with 
    $\pattern = (a,b,c,d)$, $\pattern'=(a',b',c',d')$. Denoting $r := r(\pattern, \pattern')$, the complexity $C(\pattern,\pattern')$ of {the two-factor fixed support matrix factorization algorithm (\Cref{algo:algorithm1})} with input $\mat{A}, \mat{S}_{\pattern}, \mat{S}_{\pattern'}$ is
    \begin{equation}
        % \label{eq:complexity2factors}
        C(\pattern, \pattern') = \mathcal{O}(r
        a'bc'd
        ).
    \end{equation}
\end{lemma}

\begin{proof}
      By \Cref{lem:prodDBBis}, for each $P \in \cP := \cP(\bS_\pattern,\bS_{\pattern'})$ we have $|P|=r$, $|R_P|=b$ and $|C_P|=c'$. 
      By \Cref{lem:invariance-partition-col} we have $|\cP|=a'd$.
      By
      \Cref{lem:complexity-fsmf} it follows that
      $C(\pattern, \pattern') = \mathcal{O}(a'drbc')$.
      %\todo{We need to prove $|\cP|=a'd$ to conclude. Do we/can we state this in a lemma ?}
 %     By \Cref{def:chainableDBfour} and chainability of $(\pattern, \pattern')$, we have $r = ac / a'$.
          %   %      $C(\pattern, \pattern') = \mathcal{O}(acdbc')$.     By chainability of $(\pattern, \pattern')$, we have $r = ac / a'$. 
  
    % By \Cref{lem:qperfectcovering,lem:complexity-fsmf}, \todo{why ? Do we have $|\cP| = acd/r$? Quote a lemma that gives $|\cP|$} $C(\pattern, \pattern') =\mathcal{O}\left(\frac{acd}{r}|P||R_P||C_P|\right)$. 
    
    % Since $r = ac / a'$ by \Cref{def:chainableDBfour},  $C(\pattern, \pattern') = \mathcal{O}(ra'c'bd)$.
% $\mathcal{O}\left(\frac{acd}{r}|P||R_P||C_P|\right) = \mathcal{O}(abcdc') = $
\end{proof}

\begin{remark}
    In practice, best low-rank approximations in \Cref{algo:algorithm1} via truncated SVDs can be computed in parallel. For {arbitrary} $\bL,\bR$ as in \Cref{lem:complexity-fsmf} this can decrease the complexity down %up
    to $\mathcal{O}(\max_{P \in \cP(\bL,\bR)} |P||R_P||C_P|)$ when parallelizing accross $| \cP(\bL,\bR) |$ processes. For butterfly factors as in \Cref{lem:complexity-fsmf-butterfly}, this decreases the complexity down to $\mathcal{O}(rbc')$.
\end{remark}

\subsection{Complexity of \Cref{algo:recursivehierarchicalalgo}} %,algo:hierarchicalalgo}}

%These two algorithms are
This algorithm is based on \Cref{algo:algorithm1}.
% , which performs low-rank approximation that is typically computing using truncated SVD. We start by the following remark.
% 
% \begin{remark}
% \label{rem:complexityalgo1}
%     The complexity of the $k$-truncated SVD for an $m \times n$ matrix is $\mathcal{O}(kmn)$ \cite{halko2011svd}, so the complexity of \Cref{algo:algorithm1} is $\mathcal{O}(\HL{\sum}_{P \in \cP(\bL,\bR)} |P||R_P||C_P|)$. This complexity does \emph{not} take into account the computation of {the partition} $\cP(\bL, \bR)$, {which is often negligible compared to the complexity of the truncated SVDs.}
%     In practice, best low-rank approximations in \Cref{algo:algorithm1} via truncated SVDs can be computed in parallel. This can decrease the complexity up to $\mathcal{O}(\HL{\max}_{P \in \cP(\bL,\bR)} |P||R_P||C_P|)$ when parallelizing accross $| \cP(\bL,\bR) |$ processes. \todo{put this remark at the right place. Do we need this?}
% \end{remark}
We prove the first point of \Cref{theorem:complexity} in the following lemma.
\begin{lemma}
    \label{lem:complexityalgorecursive}
    Consider a non-redundant chainable architecture $\arch$ and a matrix $\bA$ of size $m \times n$. With notations of \Cref{theorem:complexity}, the complexity of {the hierarchical algorithm (\Cref{algo:recursivehierarchicalalgo})} %,algo:hierarchicalalgo} 
    with inputs $\arch$, $\bA$ and any factor-bracketing tree $\mathcal{T}$ is at most: 
    \begin{itemize}
        \item $\mathcal{O}(\|\mathbf{r}(\arch)\|_1M_\arch N_\arch)$ in the general case; %\ER{remove? Fix proof?}
        \item $\mathcal{O}(\|\mathbf{r}(\arch)\|_1mn)$ if $\arch$ is non-redundant.
    \end{itemize}
\end{lemma}

\begin{proof}
    Since \Cref{algo:recursivehierarchicalalgo} performs $(L-1)$ factorizations of the form of Problem \eqref{eq:fsmf-DB} {using 
    the two-factor fixed-support matrix factorization algorithm (\Cref{algo:algorithm1})}, its complexity is equal to the sum of the complexity of each of these $(L-1)$ factorizations.

    % Consider a factorization where we recall the notation $C(\cdot, \cdot)$ from \eqref{eq:complexity2factors}. 
    {Fix $1 \leq q \leq s < t \leq L$. By \Cref{lem:seq-qchainability}, $r_s := r(\pattern_q * \ldots * \pattern_s, \pattern_{s + 1} * \ldots * \pattern_t) = r(\pattern_s, \pattern_{s+1})$.}
    {By \eqref{eq:producttheta} we have}
    \begin{equation*}
        \pattern_q * \ldots * \pattern_s = \left( a_q, \frac{b_q d_q}{d_s}, \frac{a_s c_s}{a_q}, d_s \right),\ 
        \pattern_{s+1} * \ldots * \pattern_t = \left( a_{s+1}, \frac{b_{s+1} d_{s+1}}{d_t}, \frac{a_t c_t}{a_{s+1}}, d_t \right),
    \end{equation*}
    % By \Cref{rem:complexityalgo1}, the complexity of \Cref{algo:algorithm1} for problem \eqref{eq:butterfly-approximation-pb} associated with a chainable $\arch := (\pattern, \pattern')$ is $C(\pattern, \pattern') = \mathcal{O}((acd/r)|P||R_P||C_P|)$ where $q = r(\pattern,\pattern')$, because there are exactly $acd/r$ equivalence classes (\Cref{def:classequivalence}) and $|P| = r(\pattern,\pattern'), |R_P| = b, |C_P| = c'$ by \Cref{lem:prodDB2}. Hence,
    % \begin{equation}
    %     \label{oldeq:complexity2factors}
    %     C(\pattern, \pattern') = \mathcal{O}\left(\frac{acd}{q}|P||R_P||C_P|\right) = \mathcal{O}(abcdc') = \mathcal{O}(qa'c'bd),
    % \end{equation}
    % since $q = ac / a'$ by \Cref{def:chainableDBfour}.
    hence,  \Cref{lem:complexity-fsmf-butterfly} yields  $C(\pattern_q * \ldots * \pattern_s, \pattern_{s + 1} * \ldots * \pattern_t) = \mathcal{O}(r_s a_t c_t b_q d_q)$, which is upper bounded by $\mathcal{O}(r_s M_\arch N_\arch)$, by definition of $M_\arch := \max_\ell a_\ell c_\ell, N_\arch := \max_\ell b_\ell d_\ell$.
    % In addition, the values $r_s$ form a permutation of the vector $\mathbf{r}(\arch)$ (cf.~\Cref{def:chainableseqDBPs}). 
    Therefore, the overall complexity of \Cref{algo:recursivehierarchicalalgo}
    %/\Cref{algo:hierarchicalalgo} 
    is upper bounded by: 
    \begin{equation}
        \label{eq:boundonq}
        \sum_{s = 1}^{L-1} \mathcal{O}(r_s M_\arch N_\arch) = \mathcal{O}(\|\mathbf{r}(\arch)\|_1M_\arch N_\arch).
    \end{equation}
    When $\arch$ is non-redundant, by \Cref{def:redundant} we have $a_\ell c_\ell < a_{\ell + 1}c_{\ell + 1}$ and $b_\ell d_\ell > b_{\ell + 1}d_{\ell + 1}$ for all $\ell \in \intset{L-1}$. Since $m=a_1b_1d_1$ and $n = a_Lc_Ld_L$ (cf.~\Cref{lem:suppdbfactorprod} and
    \Cref{def:dbfactorfour}) we obtain
%      \begin{equation}
%        \label{eq:boundonGBP}
%         \begin{split}
            $M_\arch = \max_{\ell} a_\ell c_\ell < a_Lc_L = n/d_L \leq n$ and
            $N_\arch = \max_{\ell} b_\ell d_\ell < b_1d_1 = m/a_1 \leq m$.
 %       \end{split}.
 %   \end{equation}
    % $a_1c_1 < \ldots < a_Lc_L$ and $b_1d_1 > \ldots > b_Ld_L$.
    % Consequently:
    % \begin{equation}
    %     \label{eq:boundonGBP}
    %     \forall \ell \in \intset{L-1}, \quad \begin{cases}
    %         a_\ell c_\ell < a_Lc_L = n/d_L \leq n \\
    %         b_\ell d_\ell < b_1d_1 = m/a_1 \leq m
    %     \end{cases}.
    % \end{equation}
    % Thus, we have: $M_\arch = \max_{\ell \in \intset{L}} a_\ell c_\ell \leq m$. Similarly, $N_\arch \leq n$, which ends the proof.
    \end{proof}

    \subsection{Complexity of \Cref{algo:exchange}}
    We now analyze the complexity of the construction of orthonormal butterfly factors in \Cref{algo:exchange}.
    % Our next step is to analyze the complexity of each call of \Cref{algo:exchange}.
    \begin{lemma}
        \label{lem:complexityexchange}
        {The complexity of {column/row-pseudo-orthonormalization (\Cref{algo:exchange})} with input $(\pattern_1, \pattern_2, \bX, \bY, u)$ for any $u \in \{\texttt{column}, \texttt{row} \}$ and any non-redundant chainable pair $(\pattern_1, \pattern_2)$ is $$\mathcal{O}(r(\pattern_1, \pattern_2)(\| \pattern_1 \|_0 + \| \pattern_2 \|_0)),$$
        with the notation $\|\pattern\|_0$ from \Cref{lemma:numberofparameters}.}
    \end{lemma}
    \begin{proof}
    We only consider the case $u = \texttt{column}$, since the other case can be dealt with similarly. Denote $r = r(\pattern_1, \pattern_2)$ and $\pattern_\ell = (a_\ell, b_\ell, c_\ell, d_\ell)$ for $\ell \in \intset{2}$. By \Cref{lem:prodDBBis} for each $P \in \cP(\bS_\pattern,\bS_{\pattern'})$ we have $|P|=r$, $|R_P|=b_1$, $|C_P|=c_2$, and we have $r \leq \min(b_1,c_2)$ (cf.~\Cref{def:chainableDBfour}). Thus, 
    at each iteration of \Cref{algo:exchange}: 
    \begin{itemize}
        \item the complexity of line~\ref{line:qrdecomposition} is $\mathcal{O}(|R_P||P|^2) = \mathcal{O}(b_1r^2)$, since the complexity of QR-decomposition of an $m \times n$ matrix is $\mathcal{O}( \min(m,n)mn)$ \cite[Section 5.2]{matrixcomputation};
        \item the complexity of line~\ref{line:columnorthonormalization} is $\mathcal{O}(|R_P||P|) = \mathcal{O}(b_1r)$;
        \item the complexity of line~\ref{line:matrix-multiplication} is $\mathcal{O}(|P|^2|C_P|) = \mathcal{O}(c_2r^2)$.
    \end{itemize}
    Overall, the complexity of an iteration is $\mathcal{O}(r^2(b_1 + c_2))$. There are $a_1c_1d_1/r = a_2b_2d_2/r$ equivalence classes (by \Cref{lem:invariance-partition-col} and chainability, that implies $a_1c_1d_1=a_2b_2d_2$, see \Cref{def:chainableDBfour}), and by \Cref{lemma:numberofparameters} $a_\ell b_\ell c_\ell d_\ell = \|\pattern_\ell\|_0$, hence the overall complexity is
    \begin{equation*}
    \begin{split}
        \mathcal{O}\left(\frac{a_1c_1d_1}{r}r^2b_1 + \frac{a_2b_2d_2}{r}r^2c_2\right) &= \mathcal{O}(r(a_1b_1c_1d_1 + a_2b_2c_2d_2)) \\
        &= \mathcal{O}(r(\pattern_1, \pattern_2)(\| \pattern_1 \|_0 + \| \pattern_2 \|_0)).
    \end{split}
    \end{equation*}
    % This concludes the proof.
\end{proof}

\subsection{Complexity of \Cref{algo:modifedbutterflyalgo}}
% We are now able to prove  \Cref{theorem:complexity}.
We can now prove the second point of \Cref{theorem:complexity}, where we assume that $\arch$ is not redundant so that $M_\arch \leq m$ and $N_\arch \leq n$.
%\todo{Missing details on $\|\pattern_I\|_0$ for the intervals appearing in the call to \Cref{algo:exchange} in lines~\ref{line:colorthonormal} and \ref{line:roworthonormal} of \Cref{algo:modifedbutterflyalgo}\TL{Added}}
    % We prove the complexity of the algorithms one by one:
    % \begin{enumerate}[leftmargin=*]
    %     \item \Cref{algo:recursivehierarchicalalgo,algo:hierarchicalalgo}: This is given by \Cref{lem:complexityalgorecursive}.
    %     \item \Cref{algo:modifedbutterflyalgo} with the assumption on the non-redundancy of $\arch$: 
    % \end{enumerate}
%\mRG{Je dois reprendre ici}
    By \Cref{lem:complexityexchange} 
    %and the inequality in \eqref{eq:boundonGBP},
    the complexity for each call to \Cref{algo:exchange} in lines~\ref{line:colorthonormal} and \ref{line:roworthonormal} of the new butterfly algorithm (\Cref{algo:modifedbutterflyalgo}) is given by $\mathcal{O}(r(\pattern_{I_{k-1}}, \pattern_{I_k})(\|\pattern_{I_{k-1}}\|_0 + \|\pattern_{I_k}\|_0))$ where $\pattern_{I} = \pattern_q * \ldots \pattern_s$ if $I := \intset{q,s}$. 
    By \cref{lem:seq-qchainability}, $r(\pattern_{I_{k-1}}, \pattern_{I_k}) = r(\pattern_s, \pattern_{s+1})$ for some $1 \leq s \leq L - 1$, thus is bounded by $\|r(\arch)\|_\infty$ (cf.~\cref{def:chainableseqDBPs}). 
    Moreover, $\pattern_{I_{k}}$ (resp.~$\pattern_{I_{k-1}}$) has the form $(a_q, \frac{b_qd_q}{d_s}, \frac{a_sc_s}{a_q}, d_s)$ for some $1 \leq q \leq s \leq L$ (cf.~\cref{lem:suppdbfactorprod}). 
    Thus, $\|\pattern_{I_{k}}\|_0$ (resp.~$\|\pattern_{I_{k-1}}\|_0$) is bounded (recall the notation $\|\pattern\|_0$ from \Cref{lemma:numberofparameters}) by $\max_{q,s} b_qd_qa_sc_s \leq M_\arch N_\arch \leq mn$. In conclusion, the complexity of each call to \cref{algo:exchange} is at most 
    $\mathcal{O}(mn\|\mathbf{r}(\arch)\|_{\infty})$.
    % where $\|\mathbf{r}(\arch)\|_{\infty}:= \max_{\ell = 1, \ldots, |\arch|-1} r(\pattern_\ell, \pattern_{\ell + 1})$.
    {At the $J$-th iteration for some $J \in \intset{L-1}$, }
    there are at most $(J - 1)$ calls to \Cref{algo:exchange}, so the {total} complexity for the orthonormalization operations accross all the $|\arch|-1$ iterations is at most $\mathcal{O}(|\arch|^2mn\|\mathbf{r}(\arch)\|_{\infty})$. Therefore, the complexity of \Cref{algo:modifedbutterflyalgo} is given by $$\cO\left((|\arch|^2\|\mathbf{r}(\arch)\|_{\infty} + \|\mathbf{r}(\arch)\|_{1}) mn\right).$$

\section{Proof of the results of \Cref{sec:error}}
\subsection{Proof of \Cref{lemma:orthogonal-left-n-right}}
\label{appendix:keylemmaproof-new}
% The proof is based on the following two results.
% \begin{lemma}
% 	\label{lem:originalkeylemma-new}
% 	Consider a non-redundant chainable architecture $\arch:= (\pattern_\ell)_{\ell = 1}^L$, and integers $(r,s,t)$ such that $1 \leq r \leq s < t \leq L$.
%     Denote $\bX$ any \ksf{}in $\setDBfactor{\pattern_1 * \ldots * \pattern_{r-1}}$ that is left-$r(\pattern_{r-1}, \pattern_{r})$-unitary \RG{(cf.~\Cref{def:q-unitary})} if $r > 1$, otherwise $\bX$ is the identity matrix of size $a_1 b_1 d_1$. 
%     Denote $\bZ$ any \ksf{}in $\setDBfactor{\pattern_{t+1} * \ldots * \pattern_{L}}$ that is right-$r(\pattern_{t}, \pattern_{t+1})$-unitary if $t < L$, otherwise $\bZ$ is the identity matrix of size $a_L c_L d_L$. 
%     Then, for any $\mat{Y} \in \setDBfactor{\pattern_{q} * \ldots * \pattern_{t}}$:
% 	\begin{equation}
%     \label{eq:orthonormality-on-left-right-equivalent-problem-new}
% 		E^{\splitarch{\arch}{s}}(\mbf{XYZ}) = \inf_{\mat{Y}_1, \mat{Y}_2} \{\|\mbf{Y} - \mbf{Y}_1\mbf{Y}_2\|_F \mid (\mbf{Y}_1, \mbf{Y}_2) \in \setDBfactor{\pattern_{q} * \ldots \pattern_s} \times \setDBfactor{\pattern_{s+1} * \ldots * \pattern_{t}} \}. 
% 	\end{equation}
% \end{lemma}

% \TL{New proof here}
% Admitting \Cref{lem:originalkeylemma-new}, we can prove \Cref{lemma:orthogonal-left-n-right} as follows:
%\begin{proof}[Proof of \Cref{lemma:orthogonal-left-n-right}]
  %\todo{RG: I rewrote the whole proof}
   First, we prove that $\bB' \in \cB^{\arch_s}$. By definition 
    \begin{equation*}
        \bB' = \underbrace{\left(\bX_{P_1}\ldots\bX_{P_{j-1}}\bX_{\intset{q,s}}\right)}_{\bB_{(\text{left})}'}\underbrace{\left(\bX_{\intset{s+1,t}}\bX_{P_{j+1}}\ldots\bX_{P_J}\right)}_{\bB_{(\text{right})}'}
    \end{equation*}
and since by assumption $\mathbf{X}_{P_i} \in \setDBfactor{\pattern_{q_i} * \ldots * \pattern_{t_i}}$ for every $i$, and
    %  By construction (cf.~line~\ref{line:fsmfmodifiedalgo} of \Cref{algo:modifedbutterflyalgo}), each factor appearing above can be written (for appropriate integers $k,\ell$) as
    % \begin{equation}\label{eq:Xpartbelongs}
    %      \bX_{\intset{k,\ell}} 
    %      \in \setDBfactor{(\pattern_k * \ldots * \pattern_\ell)},
    %  \end{equation}
     $\arch$ is chainable, a recursive application of \Cref{prop:stability} yields that $\bB'_{(\text{left})} \in \setDBfactor{\pattern_1 * \ldots * \pattern_{s}}$, $\bB'_{(\text{right})} \in \setDBfactor{\pattern_{s+ 1} * \ldots * \pattern_{L}}$. Hence, $\bB' \in \cB^{\arch_s}$.
     
%    Since $\bX_{\intset{s,f}} \in \cB^{(\pattern_s * \ldots * \pattern_f)}$, by \Cref{prop:stability}, we have: $\bB_J^{(\text{left})} \in \setButterfly{\pattern_1 * \ldots * \pattern_{\sigma_J}}$, $\bB_J^{(\text{right})} \in \setButterfly{(\pattern_J + 1) * \ldots * \pattern_{L}}$. Hence, $\bB_J \in \cB^{\arch_{\sigma_J}}$.

We now prove~\Cref{eq:BJprojection}. 
%Denote $q_i := r(\pattern_{i},\pattern_{i+1})$ for $i \in \intset{L-1}$. 
By \Cref{lem:seq-qchainability} and chainability of $\arch = (\pattern_\ell)_{\ell=1}^L$, the patterns $\pattern'_1 := \pattern_1 * \ldots * \pattern_{q-1}$, $\pattern'_2:=\pattern_{q} * \ldots * \pattern_t$ and $\pattern'_3:=\pattern_{t+1} * \ldots * \pattern_{L}$ are well-defined, and the two-factor architectures $\arch_\mathtt{left}:=(\pattern'_1,\pattern'_2)$ and  $\arch_\mathtt{right}:=(\pattern'_2,\pattern'_3)$ are both chainable with \[
r(\pattern'_1,\pattern'_2) = r(\pattern_{q-1},\pattern_{q})
\ \text{and}\ 
r(\pattern'_2,\pattern'_3) = r(\pattern_{t},\pattern_{t+1}).
\]
% \RG{the two-factor architecture} $(\pattern_1 * \ldots * \pattern_{r-1}, \pattern_{q} * \ldots * \pattern_t)$ is chainable with $r(\pattern_1 * \ldots * \pattern_{r-1}, \pattern_{q} * \ldots * \pattern_t) = r(\pattern_{r-1}, \pattern_r)$, and $(\pattern_{q} * \ldots * \pattern_t, \pattern_{t+1} * \ldots * \pattern_{L})$ is chainable with $r(\pattern_{q} * \ldots * \pattern_t, \pattern_{t+1} * \ldots * \pattern_{L}) = r(\pattern_t, \pattern_{t+1})$. 
Moreover, by assumption, for each $i \in \intset{j-1}$ the factor $\bX_{P_i}$  is left-$r(\pattern_{t_i},\pattern_{t_i+1})$-unitary hence, 
using the notations $\bX_{\tleft}^{(J)}$, $\bX_{\tright}^{(J)}$ of \eqref{eq:productleftright-new}, multiple applications of \Cref{lemma:stableundermatrixmultiplication} yield that $\bX_{\tleft}^{(J)} \in \setDBfactor{\pattern'_1}$ is left-$r(\pattern_{r-1},\pattern_r)$-unitary if $r \geq 2$ (if $r = 1$ then by convention $\bX_{\tleft}^{(J)}$ is the identity). Similarly, $\bX_{\tright}^{(J)} \in \setDBfactor{\pattern'_3}$ is right-$r(\pattern_{t},\pattern_{t+1})$-unitary when $t \leq L-1$ (and the identity when $t = L$).
Finally, by \Cref{prop:stability}, both ${\bX}_{\intset{q,t}}$ and  $\matseq{X}{\intset{q,s}} \matseq{X}{\intset{s+1, t}}$ (and therefore their difference) belong to $\setDBfactor{(\pattern_{q} * \ldots * \pattern_t)}=\setDBfactor{\pattern'_2}$, and denoting $\tilde{\pattern}_1 := \pattern_{q} * \ldots * \pattern_{s-1}$ and $\tilde{\pattern}_2 = \pattern_{s+1} * \ldots * \pattern_t$, since $\matseq{X}{\intset{q,s}}$ and $\matseq{X}{\intset{s+1, t}}$ solve~\eqref{eq:fsmf-DB} we obtain
  %Therefore, combining \eqref{eq:normpreserve1-new} and \eqref{eq:keylemma-new} with the definition of ${\mbf{B}}_J$, we obtain
    \begin{equation*}
        \begin{aligned}
        \|{\mbf{B}}' - {\mbf{B}}\|_F &
        = \| \bX_{\tleft}^{(J)} ({{\bX}_{\intset{q,t}}} - \matseq{X}{\intset{q,s}} \matseq{X}{\intset{s+1, t}}) \bX_{\tright}^{(J)}\|_F
        \hfill \quad \text{by definitions of } \mat{B},\mat{B}'
       \\
       &
       =
        \| {{\bX}_{\intset{q,t}}} - \matseq{X}{\intset{q,s}} \matseq{X}{\intset{s+1, t}} \|_F
        \hfill \quad  \text{by \Cref{lem:qunitarypreservesfrob}}
        \\
        &
        =
        \inf_{\mat{Y}_i \in \setDBfactor{\tilde{\pattern}_i},i=1,2}
        \| {{\bX}_{\intset{q,t}}} - \mat{Y}_1 \mat{Y}_2 \|_F
        \hfill \quad \text{by }  \eqref{eq:fsmf-DB}\\
        &
        =
        E^{\splitarch{\arch}{s}}(\bX_{\tleft}^{(J)}\bX_{\intset{q,t}}\bX_{\tright}^{(J)})
        \quad \text{by \Cref{lem:originalkeylemma-new} below} 
        \\
        & = E^{\splitarch{\arch}{s}}({\mbf{B}}) 
        \hfill \quad \text{by definition of } \mat{B}.
    \end{aligned}
  \end{equation*}
  %\RG{using the two following equalities that we now prove:}
    %To finish the proof, we \RG{establish below} %need 
    %the following properties: {for any} $J \in \intset{L-1}$, we have:
    % \begin{align}
    %     \label{eq:normpreserve1-new}
    %     \| \bX_{\tleft}^{(J)} ({{\bX}_{\intset{q,t}}} - \matseq{X}{\intset{q,s}} \matseq{X}{\intset{s+1, t}}) \bX_{\tright}^{(J)}\|_F &= \| {{\bX}_{\intset{q,t}}} - \matseq{X}{\intset{q,s}} \matseq{X}{\intset{s+1, t}} \|_F,\\
    %     \label{eq:keylemma-new}
    %     E^{\splitarch{\arch}{s}}(\bX_{\tleft}^{(J)}\bX_{\intset{q,t}}\bX_{\tright}^{(J)}) &= \|\bX_{\intset{q,t}} - \bX_{\intset{q,s}}\bX_{\intset{s+1,t}}\|_F.
    % \end{align}
 %   where $\bX_{\tleft}^{(J)}$ and $\bX_{\tright}^{(J)}$ are given at \eqref{eq:productleftright-new}. % and \eqref{eq:productright-new}.
%  We u
%  Therefore, \RG{(*) and (**)}
%  %\eqref{eq:normpreserve1-new} and \eqref{eq:keylemma-new} 
%  follow from \Cref{lem:qunitarypreservesfrob} and \Cref{lem:originalkeylemma-new} respectively.

The penultimate line used the following result, which we prove next.
%proof is based on the following two results.
\begin{lemma}
	\label{lem:originalkeylemma-new}
	Consider a non-redundant chainable architecture $\arch:= (\pattern_\ell)_{\ell = 1}^L$, integers $(q,s,t)$ such that $1 \leq q \leq s < t \leq L$, and denote 
 \begin{align*}
\pattern'_1 := \pattern_1 * \ldots * \pattern_{q-1}, \quad \pattern'_2  := \pattern_{q} * \ldots * \pattern_{t}, \quad
\pattern'_3 := \pattern_{t+1} * \ldots * \pattern_{L},\\ 
\tilde{\pattern}_1 := \pattern_{q} * \ldots \pattern_{s}, \quad
\tilde{\pattern}_2 := \pattern_{s+1} * \ldots * \pattern_t.
 \end{align*}
    Assume that $\bX \in \setDBfactor{\pattern'_1}$ is left-$r(\pattern_{{q}%r
    -1}, \pattern_{{q}%r
    })$-unitary (if ${q}%r
    > 1$) or  the identity matrix of size $a_1 b_1 d_1$ (if ${q}%r
    =1$), and that $\bZ \in \setDBfactor{\pattern'_3}$ is right-$r(\pattern_{t}, \pattern_{t+1})$-unitary (if $t < L$) or  the identity matrix of size $a_L c_L d_L$ (if $t=L$). 
Then for any $\mat{Y} \in \setDBfactor{\pattern'_2}$ we have :
	\begin{equation}
    \label{eq:orthonormality-on-left-right-equivalent-problem-new}
		E^{\splitarch{\arch}{s}}(\mbf{XYZ}) = \inf_{\mat{Y}_i \in \setDBfactor{\tilde{\pattern}_i}, i=1,2}
  \|\mbf{Y} - \mbf{Y}_1\mbf{Y}_2\|_F. 
	\end{equation}
\end{lemma}
Before proving the lemma observe that for each pair of factors $\bY_i \in \setDBfactor{\tilde{\pattern}_i}$, $i = 1,2$, by \Cref{lem:prodDB2} we have $(\bX\bY_1, \bY_2\bZ) \in \setDBfactor{\arch_s}$ (cf.~\Cref{def:firstlevelfactorization-new}), hence the inequality
     \begin{equation*}
%         \label{eq:LHSsmallerRHS}
         \begin{aligned}
            E^{\arch_s}(\bX\bY\bZ) 
            \leq \|\bX\bY\bZ - \bX\bY_1\bY_2\bZ\|_F^2
            &= \|\bX\underbrace{(\bY - \bY_1\bY_2)}_{\in \setDBfactor{\pattern_2'}}\bZ\|_F^2\\
            &= \|\bY - \bY_1\bY_2\|_F^2 \quad (\text{by \Cref{lem:qunitarypreservesfrob}}).
        \end{aligned}
     \end{equation*}
     Rather than proving the converse inequality, we proceed by characterizing both hand sides of~\eqref{eq:orthonormality-on-left-right-equivalent-problem-new} via spectral properties of appropriate blocks of $\bX\bY\bZ$ (resp.~of $\bY$).
     
 \begin{proof}[Proof of \Cref{lem:originalkeylemma-new}]
We denote $\cP := \cP(\bS_{\tilde{\pattern}_1},\bS_{\tilde{\pattern}_2})$ and $\cP' = \cP(\bS_{\pattern'_1 * \tilde{\pattern}_1},\bS_{\tilde{\pattern}_2*\pattern'_3})$
the partitions of $\intset{{p}%q
}$ (where ${p}%q
$ is the number of columns of matrices in $\setDBfactor{\pattern'_1 *\tilde{\pattern}_1}$ and in $\setDBfactor{\tilde{\pattern}_1}$, as well as the number of rows of matrices in $\setDBfactor{\tilde{\pattern}_2}$ and in $\setDBfactor{\tilde{\pattern}_2 * \pattern'_3}$).
%\todo{Léon: je crois qu'il y a un conflit de notations pour le $q$ qui a été introduit ici? ça ne peut pas être le même $q$ que celui de l'énoncé du lemme?}
By \Cref{lem:prodDBBis}, the sets $R_P,C_P$ for $P \in \cP$ satisfy  
\begin{align}
    R_P &= \supp(\bS_{\tilde{\pattern}_1}[\col{i}]),\quad C_P =\supp(\bS_{\tilde{\pattern}_2}[\row{i}]),\quad \forall i \in P.\label{eq:Partition1Supports}
    \end{align}
To avoid confusions, for $P \in \cP'$ we use the distinct notation/property
\begin{align}
    R'_P &= \supp(\bS_{\pattern'_1 * \tilde{\pattern}_1}[\col{i}]),\quad 
C'_P =\supp(\bS_{\tilde{\pattern}_2 * \pattern'_3}[\row{i}]),\quad \forall i \in P.\label{eq:Partition2Supports}
    \end{align}
    With $r := r(\pattern_s,\pattern_{s+1})$, 
    by \Cref{lem:seq-qchainability} we have 
    $r(\tilde{\pattern}_1,\tilde{\pattern}_2) = r(\pattern'_1*\tilde{\pattern}_1,\tilde{\pattern}_2*\pattern'_3) = r$. 
    
    With these notations, and denoting $\prank_{r}(\cdot)$
    %\todo{Léon: c'est $\mathtt{rank}_{r}(\cdot)$ ou $\prank_{r}(\cdot)$?} 
    a best rank-$r$ approximation (in the Frobenius norm) to a matrix, we first characterize the left-hand side (LHS) and the right-hand side (RHS) of \eqref{eq:orthonormality-on-left-right-equivalent-problem-new} separately. Exploiting the (left/right)-*-unitarity of $\bX$ and $\bZ$, by \Cref{lem:qunitarypreservesfrob} we have $\|\bX\bY\bZ\|_F^2=\|\bY\|_F^2$. Moreover, by~\Cref{lem:qperfectcovering}, all constraint support patterns satisfy the assumptions of \Cref{theorem:tractablefsmf}. Decomposing $\mbf{XYZ}$ (resp.~$\bY$) into the corresponding blocks $R_P \times C_P$, using \Cref{theorem:tractablefsmf}, and summing the resulting equalities, we obtain 
    %\todo{\TL{Ce resultat n'est pas encore mentionne? Ou il est trivial?}}\mRG{est-ce plus clair?}
    \begin{align}
        \mathtt{LHS} 
        &= E^{\splitarch{\arch}{s}}(\mbf{XYZ})^2 
        = \underbrace{\|\mat{X} \mat{Y} \mat{Z} \|_F^2}_{= \| \mat{Y} \|_F^2} -  \sum_{P \in \cP'} \|\mathtt{rank}_{r}\big((\mat{X} \mat{Y} \mat{Z})[R'_P,C'_P]\big)\|_F^2\\
        \mathtt{RHS}
        & = \|\bY\|_F^2 -  \sum_{P \in \cP}\|\mathtt{rank}_{r}\big(\bY[R_{P},C_{P}]\big)\|_F^2.
    \end{align}
    To conclude, we prove that $\cP = \cP' = \cP_{\mathtt{col}}(\pattern_s,{r})$, and that 
    \begin{equation}
    \label{eq:AnnexBlockLowRankApproxEquality}    
    \|\mathtt{rank}_{r}\big((\mat{X} \mat{Y} \mat{Z})[R'_P,C'_P]\big)\|_F^2 =\|\mathtt{rank}_{r}\big(\bY[R_{P},C_{P}]\big)\|_F^2,\quad \forall P \in \cP.
    \end{equation}

{\bf Proof that $\cP=\cP'$.} Since $r(\tilde{\pattern}_1,\tilde{\pattern}_2) = r(\pattern'_1*\tilde{\pattern}_1,\tilde{\pattern}_2*\pattern'_3) = r$ (this implies that the assumption $r \mid c$ holds where $c$ is such that $\pattern'_1*\tilde{\pattern}_1 = (a,b,c,d)$ -- respectively such that $\tilde{\pattern}_1 = (a,b,c,d)$ -- for some $a,b,d$), by~\Cref{lem:invariance-partition-col} and \Cref{cor:invariant_pc} we have
\begin{align*}
        \cP'
        & = \cP(\bS_{\pattern'_1*\tilde{\pattern}_1}, \bS_{\tilde{\pattern}_2 * \pattern'_3}) 
        = \cP_{\mathtt{col}}(\pattern'_1*\tilde{\pattern}_1, r)
        = \cP_{\mathtt{col}}(\pattern_1 * \ldots * \pattern_s, r) 
        = \cP_{\mathtt{col}}(\pattern_s, r)\\
         \cP
        & = \cP(\bS_{\tilde{\pattern}_1}, \bS_{\tilde{\pattern}_2}) 
        = \cP_{\mathtt{col}}(\tilde{\pattern}_1, r)
        = \cP_{\mathtt{col}}(\pattern_{{q}%r
        } * \ldots * \pattern_s, r) 
        = \cP_{\mathtt{col}}(\pattern_s, r).
\end{align*}

{\bf Proof of~\eqref{eq:AnnexBlockLowRankApproxEquality}.}
An important step is to show that for each $P \in \cP$, the matrix $(\mbf{XYZ})[R'_P,C'_P]$ is, up to some permutation of rows and columns and addition of zero rows and columns, equal to $\mbf{X}[\col{R_P}]\mbf{Y}[R_P,C_P]\mbf{Z}[\row{C_P}]$. 
For this, we prove that
    \begin{align}
        \label{eq:decomposition-xyz-new2}
        &\bX\bY\bZ = \sum_{P \in \cP} \bX[\col{R_P}]\bY[R_P, C_P]\bZ[\row{C_P}]\\
    \label{eq:support-xyz}        
&\supp(\bX[\col{R_P}]\bY[R_P, C_P]\bZ[\row{C_P}])  \subseteq R'_P \times C'_P.
    \end{align}
Indeed, since $\pattern'_2 = \tilde{\pattern}_1 * \tilde{\pattern}_2$, by \Cref{lem:prodDBBis}, we have 
      \(
      \supp(\bS_{\pattern'_2}) = \bigcup_{P \in \cP} R_P \times C_P
      \)
      where we recall that $\cP := \cP(\bS_{\tilde{\pattern}_1},\bS_{\tilde{\pattern}_2})$, and the sets $R_P \times C_P$ are pairwise disjoint.       
      Since $\bY \in \setDBfactor{\pattern'_2}$ it follows that $\supp(\bY) \subset \bigcup_{P \in \cP} R_P \times C_P$, hence the decomposition~\eqref{eq:decomposition-xyz-new2}. Moreover, by \Cref{lem:prodDBBis},  the integers $k:=|R_P|$ and $\ell := |C_P|$ are independent of $P \in \cP$, and since $\bX \in \setDBfactor{\pattern'_1}$ and $\bZ \in \setDBfactor{\pattern'_3}$, for each $P \in \mathcal{P}$ we have
	\begin{align*}
			\supp(\bX[\col{R_{P}}] \, \bY[R_{P},C_{P}]  \, \bZ[\row{C_{P}}]) 
			& \subseteq \supp(\bS_{\pattern'_1}[\col{R_P}] \, \mbf{1}_{k  \times \ell} \, \bS_{\pattern'_3}[\row{C_P}]) \\
            &=\supp(\bS_{\pattern'_1}[\col{R_P}] \, \mbf{1}_{k \times 1} \, \mbf{1}_{1 \times \ell} \, \bS_{\pattern'_3}[\row{C_P}]).
		\end{align*}
By~\eqref{eq:Partition1Supports}-\eqref{eq:Partition2Supports}, for any $i \in P$, $\bS_{\tilde{\pattern}_1}[\col{i}] = \mbf{1}_{R_P}$ (the indicator vector of the set $R_P$), and $\supp(\bS_{\pattern'_1 * \tilde{\pattern}_1}) =  \mbf{1}_{R'_P}$, hence by elementary linear algebra and \Cref{prop:stability}
\begin{align*}
\bS_{\pattern'_1}[\col{R_P}] \mbf{1}_{k \times 1}
=
\bS_{\pattern'_1} \mbf{1}_{R_P}
=
\bS_{\pattern'_1} \bS_{\tilde{\pattern}_1}[\col{i}]
\propto
\bS_{\pattern'_1 *\tilde{\pattern}_1}[\col{i}] = \mbf{1}_{R'_P}.
\end{align*}
Similarly $ \mbf{1}_{1 \times \ell} \, \bS_{\pattern'_3}[\row{C_P}] \propto \mbf{1}_{C'_P}^\top$. This establishes~\eqref{eq:support-xyz}: as a consequence the supports of the summands in the right-hand side of \eqref{eq:decomposition-xyz-new2} are pairwise disjoint. 

Thus, for any $P \in \cP$, the product $\mbf{X}[\col{R_P}]\mbf{Y}[R_P,C_P]\mbf{Z}[\row{C_P}]$ is, up to some permutation of rows and columns and deletion of zero rows and columns, equal to $(\mbf{XYZ})[R'_P,C'_P]$, as claimed. This implies that their best approximation of rank $r$ has the same Frobenius norm. 
 Therefore, to establish~\eqref{eq:AnnexBlockLowRankApproxEquality}, we  only need to prove 
\begin{equation}
\label{eq:AnnexProofLocalBlock}    
\|\prank_{r}\big(\mbf{X}[\col{R_P}]\mbf{Y}[R_P,C_P]\mbf{Z}[\row{C_P}]\big)\|_F^2 = 
\|\prank_{r}\big(\bY[R_P,C_P]\big)\|_F^2, 
\ \forall P \in \cP.
\end{equation}
This is again a consequence of the left/right-$r$-unitarity of $\bX$/$\bZ$: by~\Cref{cor:orthonormalsubmatrix-readytouse}, the columns (resp.~rows) of $\bX[\col{R_P}]$ (resp.~$\bZ[\row{C_P}]$) are orthonormal. %Hence \eqref{eq:AnnexProofLocalBlock}.
    % By using \Cref{cor:orthonormalsubmatrix-new}, the columns of $\bX[\col{R_P}]$ are orthonormal. Indeed, \Cref{cor:orthonormalsubmatrix-new} can be used because:
    % \begin{enumerate}
    %     \item $(\pattern_{1,r-1}, \pattern_{r,s})$ is chainable with $r(\pattern_{1,r-1}, \pattern_{r,s}) = r(\pattern_{r-1}, \pattern_r)$ (cf.~\Cref{lem:seq-qchainability})
    %     \item The factor $\mat{X}$ is assumed to be left-$r(\pattern_{r-1}, \pattern_r)$-unitary.
    %     \item $R_P = \supp(\mat{S}_{r,s}[i, :])$ for some $i \in P$  by~\eqref{eq:Partition1Supports}-\eqref{eq:Partition2Supports}.
    % \end{enumerate}
    % Similarly, the rows of $\bZ[\row{C_P}]$ are orthonormal. In conclusion, $\mbf{X}[\col{R_P}]\mbf{Y}[R_P,C_P]\mbf{Z}[\row{C_P}]$ shares the same set of singular values (including multiplicities) with $\mbf{Y}[R_P,C_P]$. 
    % %So does $(\mbf{XYZ})[R^1_P,C^1_P]$. 
    % This ends the proof.
\end{proof}

\subsection{Proof of \Cref{lemma:relation-projectors}}
\label{appendix:relation-projectors}
%\todo{RG: I reordered the section. To be re-read.}
%\todo{culeur a enlever}
%\subsubsection{Main proof of \Cref{lemma:relation-projectors}}
%\begin{proof}[Proof of \Cref{lemma:relation-projectors}]
    We consider a chainable $\arch = (\pattern_\ell)_{\ell=1}^L$ and 
    $1 \leq s,{q} \leq L-1$. Our goal is to show that $E^{\arch_s}(\bM) \geq E^{\arch_s}(\bN)$ where $\bN$ is a projection of $\bM$ onto $\cB^{\splitarch{\arch}{{q}}}$.
    When $s = {q}$ the result is trivial as $E^{\splitarch{\arch}{{s}}}(\bM) \geq 0 = E^{\splitarch{\arch}{q}}(\bN) = E^{\splitarch{\arch}{s}}(\bN)$, so we focus on the case $s \neq q$.  We give the detailed proof when $s<q$ and outline its adaptation when $s>q$.
    
    Assume that $s<q$ and  denote $(\pattern,\pattern')$ the chainable patterns such that $\splitarch{\arch}{s} = (\pattern,\pattern')$ and $\cP := \cP(\bS_\pattern,\bS_{\pattern'})$ (cf.~\Cref{def:classequivalence}). By \Cref{lem:prodDBBis}, all classes $P \in \cP$ have the same cardinality, denoted $r$. By \Cref{theorem:tractablefsmf}, recalling that $\rankprojerrsq{r}{\cdot}$ denotes the squared error of best rank-$r$ approximation in the Frobenius norm,
    we have
    \(
      E^{\arch_s}(\bM) 
      = c_\bM + \sum_{P \in \cP}
            \rankprojerrsq{r}{\bM[R_P,C_P]}
            %\\
            %\left(\sum_{(i,j) \notin \supp(\bS_{\pattern_1 * \ldots * \pattern_L})} \bM[i,j]^2\right)\\
            %E^{\arch_s}(\bN) &= \sum_{P \in \cP}
%            E^{\cM_{|Q|}}(\bN[R_Q,C_Q]).
            %\RG{\rankprojerrsq{q}{\bN[R_P,C_P]}}+c_\bN
    \)
 where $c_\bM \geq 0$, and similarly for $E^{\arch_s}(\bN)$ with $c_\bN=0$ since $\supp(\bN) \subseteq \bS_{\pattern_1 * \ldots * \pattern_L}$ (by \Cref{prop:stability},
 %\todo{add correct link \TL{Done}} 
 since $\bN \in \cB^{\arch_{q}}$).
 %we have $\supp(\bN) \subseteq \bS_{\pattern_1 * \ldots * \pattern_L}$.
  %  Note that
 %We need to prove that: $E^{\arch_s}(\bN) \leq E^{\arch_s}(\bM)$. 
 Considering an arbitrary
   $P \in \cP$, we will prove that
 \begin{equation}
\label{eq:AnnexLowRankErrorIneq}
            \rankprojerrsq{r}{\bM[R_P,C_P]}
        \geq 
        \rankprojerrsq{r}{\bN[R_P,C_P]}.
 \end{equation} 
  This will yield the conclusion.
 To establish~\eqref{eq:AnnexLowRankErrorIneq} we first observe that, as a simple consequence of Eckart–Young–Mirsky theorem on low rank matrix approximation \cite{eckart1936approximation}, for  $\bU \in \RR^{m \times n}$ and $r \leq \min(m,n)$ we have
    \begin{equation}
        \label{eq:best-rank-r-distance}
        \rankprojerrsq{r}{\bU} 
        = 
        \Tr(\bU\bU^\top) - \sum_{i = 1}^{r} \lambda_i(\bU\bU^\top) =
        \Tr(\bU^\top\bU) - \sum_{i = 1}^{r} \lambda_i(\bU^\top\bU),
    \end{equation}
    with $\lambda_i(\cdot)$ the $i$-th largest eigenvalue of a symmetric matrix. We also will use the following lemma (the proof of all intermediate lemmas is slightly postponed).
\begin{lemma}
    \label{lemma:sub-additive}
    If $\bA, \bB \in \RR^{n \times n}$ are symmetric positive semi-definite (PSD) then
    \begin{equation}
        \label{eq:monotone}
        \Tr(\bA) - \sum_{i = 1}^{q} \lambda_i(\bA) \leq \Tr(\bA + \bB) - \sum_{i = 1}^{q} \lambda_i(\bA + \bB),
        \quad \forall 1 \leq q \leq n.
    \end{equation}
\end{lemma}
Denoting $\bK := \bM-\bN$ and
\[
\bU := \bM[R_P,C_P],\quad \bV := \bN[R_P,C_P],\quad  \bW := \bU-\bV = \bK[R_P,C_P],
\]
%we distinguish whether $s<r$ or $s>r$. When $s<r$, 
we will soon show that 
\begin{equation}\label{eq:OrthoConditionLowRankAnnex}
    \bV\bW^\top = %\bW\bV^\top =
    \mat{0}_{|R_P|\times |R_P|}.
\end{equation} 
%\todo{RG: no need to study $\bW \bV^\top$, it is the transpose of $\bV\bW^\top$\TL{OK}}
Since $\bU = \bV+\bW$, this implies
%\todo{RG: I think we then need to consider $\bU\bU^\top$ here and to adapt these lines \ldots \TL{Done}} 
$\bU\bU^\top = \bV\bV^\top + \bV\bW^\top + \bW\bV^\top + \bW\bW^\top = \bV\bV^\top+\bW\bW^\top$, hence $\bB := \bU\bU^\top-\bV\bV^\top = \bW\bW^\top \succeq 0$. 
With $\bA := \bV\bV^\top$ we have $\bA+\bB = \bU\bU^\top$ and the following derivation then yields~\eqref{eq:AnnexLowRankErrorIneq} as claimed:
     \begin{align*}
            \rankprojerrsq{r}{\bM[R_P,C_P]}
            %)
            &\overset{\eqref{eq:best-rank-r-distance}}{=} \Tr(\bU\bU^\top) - \sum_{i = 1}^{r} \lambda_i(\bU\bU^\top)
            =  \Tr(\bA+\bB) - \sum_{i = 1}^{r} \lambda_i(\bA + \bB)\\
            &\overset{\eqref{eq:monotone}}{\geq} \Tr(\bB) - \sum_{i = 1}^{r} \lambda_i(\bB)
            \overset{\eqref{eq:best-rank-r-distance}}{=} 
            \rankprojerrsq{r}{\bN[R_P,C_P]}.
    \end{align*}
   
To prove~\eqref{eq:OrthoConditionLowRankAnnex} we use the following lemma.
  \begin{lemma}
    \label{lemma:inclusion}
    %\todo{lemme unifié, à revérifier une dernière fois}
    Let $\arch=(\pattern_\ell)_{\ell=1}^L$ be a chainable architecture, and $1 \leq i < j \leq L-1$. For $\ell \in \{i,j\}$ denote $\cP_\ell := \cP(\bS_{\pattern_1 * \ldots * \pattern_\ell}, \bS_{\pattern_{\ell+1} * \ldots * \pattern_L})$.
%
%    For any $P \in \cP_i:= \cP(\bS_{\pattern_1 * \ldots * \pattern_i}, \bS_{\pattern_{i+1} * \ldots * \pattern_L})$ and $Q \in \cP_{j}:= \cP(\bS_{\pattern_1 * \ldots * \pattern_r}, \bS_{\pattern_{r+1} * \ldots * \pattern_L})$, 
    %\RG{TODO: define $\cP_s$, $\cP_{row}$??? \TL{Done}}
    %\todo{Et pour $r<s$? \TL{Tu changes le role des $r$ et $s$.}}
        For each $P \in \cP_i$, $Q \in \cP_j$
        \begin{enumerate}
        \item 
        %$R_P \subseteq R_Q$ or $R_P \cap R_Q = \emptyset$. 
        If $R_P \cap R_Q \neq \emptyset$ then $R_P \subseteq R_Q$
        \item 
        %$C_Q \subseteq C_P$ or $C_Q \cap C_P = \emptyset$.
                If $C_P \cap C_Q \neq \emptyset$ then $C_Q \subseteq C_P$ (reverse inclusion compared to $R_P$ and $R_Q$).
        \end{enumerate}
        For each $P \in \cP_i$ denote \(
    \cP_j(P):= \{Q \in \cP_j: R_P \cap R_Q \neq \emptyset\ \text{and}\ C_P \cap C_Q \neq \emptyset\},
    \). Similarly, for $Q \in \cP_j$, \(
    \cP_i(Q):= \{P \in \cP_i: R_P \cap R_Q \neq \emptyset\ \text{and}\ C_P \cap C_Q \neq \emptyset\}
    \). We have
        \begin{enumerate}[resume]
                %\todo{Je crois que c'est bien le sens de l'inclusion, mais à vérifier}
           \item $C_P$ is the disjoint union of $C_Q$, $Q \in \cP_j(P)$, and $R_P \subseteq R_Q$ for each $Q \in \cP_j(P)$.
           \item $R_Q$ is the disjoint union of $R_P$, $P \in \cP_i(Q)$, and $C_Q \subseteq C_P$ for each $P \in \cP_i(Q)$.
    \end{enumerate}
    % the set 
    % %$R_P$ (resp.~$C_Q$) 
    % \RG{$R_Q$ (resp.~$C_P$)}
    % \RG{is the union of all}
    % %can be partitioned into several 
    % %$R_Q$ (resp.~$C_P$).
    % \RG{$R_P$ (resp.~$C_Q$)} \RG{that intersect it.}
    %    \end{enumerate}
    %\todo{Dernier claim pas formulé clairement, et pas encore vérifié \TL{Il me semble OK.}}
\end{lemma}

% \begin{lemma}
%     \label{lemma:partition}
%     %\todo{Nouveau lemme qui rassemble ce dont on a besoin sauf erreur de ma part \TL{Il est bon}}
%     Let $\arch=(\pattern_\ell)_{\ell=1}^L$ be a chainable architecture and $1 \leq i < j \leq L-1$. For $\ell \in \{i,j\}$ denote $\cP_\ell := \cP(\bS_{\pattern_1 * \ldots * \pattern_\ell}, \bS_{\pattern_{\ell+1} * \ldots * \pattern_L})$.
%     Given any $P \in \cP_i$ denote 
%     \[
%     \cP_j(P):= \{Q \in \cP_j: R_P \cap R_Q \neq \emptyset\ \text{and}\ C_P \cap C_Q \neq \emptyset\}.
%     \]
%     The set $R_Q$ is the disjoint union of $R_P$, $Q \in \cP_j(P)$, and $C_Q \subseteq C_P$ for each $Q \in \cP'(P)$.
% \end{lemma}

To prove \eqref{eq:OrthoConditionLowRankAnnex} for an arbitrary $P \in \cP$, first recall that $\splitarch{\arch}{s} = (\pattern,\pattern')$ with $\pattern = \pattern_1 * \ldots * \pattern_s$, $\pattern' = \pattern_{s+1} * \ldots * \pattern_{L}$ and that $\cP = \cP(\bS_{\pattern},\bS_{\pattern'})$. Similarly $\splitarch{\arch}{q} = (\tilde{\pattern}, \tilde{\pattern}')$ with
$\tilde{\pattern} = \pattern_1 * \ldots * \pattern_{q}$, $\tilde{\pattern}' = \pattern_{q+1} * \ldots * \pattern_{L}$, and we introduce $\cP' := \cP(\bS_{\tilde{\pattern}}, \bS_{\tilde{\pattern}'})$.
With the notation of \Cref{lemma:inclusion}, we have $\cP = \cP_i$ and $\cP' = \cP_j$ with $i=s<q=j$, and we denote $\cP'(P) := \cP_j(P)$.
% By \Cref{lemma:inclusion} \RG{with $i:=s < r =: j$}, for each $Q \in \cP'(P)$  we have 
%     %$C_P \subseteq C_Q$ 
%     \RG{$C_Q \subseteq C_P$}\todo{Je pense que c'est l'inclusion que l'on a. Je ne vois pas comment on poursuit la preuve.}
%     hence, 
%     %denoting $\cP'(P) = \{Q_1,\ldots,Q_k\}$, 
%     the sets $R_Q$, $Q \in \cP'(P)$ are pairwise disjoint: otherwise, there would exist 
%     $Q, Q' \in \cP'(P), Q \neq Q'$
%     %two indices $1 \leq p < q \leq k$ 
%     such that $(R_{Q} \times C_{Q}) \cap (R_{Q'} \times C_{Q'})  \supseteq (R_{Q} \cap R_{Q}) \times C_P \neq \emptyset$. This is impossible due to the butterfly structure (rank-one supports are pairwise disjoint).
%     In addition, $R_P =  %\cup_{l = 1}^k R_{Q_l}
%     \RG{\cup_{Q \in \cP'(P)} Q}
%     $ because $\cup_{Q \in \cP'} R_Q \times C_Q = \cup_{P \in \cP} R_P \times C_P = \supp(\bS_{\pattern_1 * \ldots * \pattern_L})$.\todo{Peut-on inclure tous les résultats du paragraphe dans le lemme? \TL{Ils sont deja dans le lemme, voir la derniere phase de Lemma E.3.}}
  %  
%   % Therefore, $R_P =  \amalg_{Q \in \cP'(P)} R_{Q}$ is the disjoint union of $R_Q$, $Q \in \cP'(P)$.
     By \Cref{lemma:inclusion}:
     \begin{itemize}
         \item $C_P$ is the disjoint union of $C_Q$, $Q \in \cP'(P)$;
         \item $R_P \subseteq R_Q$ for each $Q \in \cP'(P)$.
         \end{itemize}
     As a  result of the first fact, $\bV := \bN[R_P,C_P]$ is (up to column permutation) the horizontal concatenation of blocks $\bN[R_P,C_Q]$, $Q \in \cP'(P)$, and similarly for $\bW := \bK[R_P,C_P]$. Establishing \eqref{eq:OrthoConditionLowRankAnnex} is thus equivalent to proving that 
     %for every $Q \in \cP'(P)$
     \begin{equation}\label{eq:BlockOrthoConditionAppendix}
     \bN[R_P, C_Q]\bK[R_{P}, C_{Q}]^\top = \bK[R_{P}, C_Q]\bN[R_{P}, C_{Q}]^\top = \mbf{0},
     %_{|R_P|\times |R_{P}|},
     \quad \forall Q \in \cP'(P)
     \end{equation}
    since
        \begin{equation*}
            \bV\bW^\top = \sum_{Q \in \cP'(P)} \bN[R_P, C_Q]\bK[R_{P}, C_{Q}]^\top.
            %,\quad \bW\bV^\top = \sum_{Q \in \cP'(P)} \bK[R_P, C_Q]\bN[R_{P}, C_{Q}]^\top\\
        \end{equation*}
    To prove \eqref{eq:BlockOrthoConditionAppendix}, consider % Indeed, for
    $Q \in \cP'(P)$: since $R_P \subseteq R_Q$, we have
    \begin{equation*}
        \begin{aligned}
            \rowrange(\bK[R_{P}, C_{Q}]) &\subseteq \rowrange(\bK[R_{Q}, C_{Q}]),\\
            \rowrange(\bN[R_{P}, C_{Q}]) &\subseteq \rowrange(\bN[R_{Q}, C_{Q}])\\
        \end{aligned}
    \end{equation*}
    with $\rowrange(\cdot)$ the span of the rows of a matrix. 
    We use the following classical lemma.
    \begin{lemma}
    %\todo{Move the statement within the proof of Lemma 7.11; keep the proof here \TL{I do not understand.}}
    \label{lemma:orthogonal-space}
    Consider a non-redundant and chainable architecture $\arch=(\pattern_\ell)_{\ell=1}^L$.
    Denote $\cP_j =\cP(\bS_{\pattern_1 * \ldots * \pattern_j}, \bS_{\pattern_{j+1} * \ldots * \pattern_L})$, where $j \in \intset{L-1}$. If $\bB$ is a projection of a matrix $\bA$ onto $\arch_j$, then we have:
        \begin{align*}
            \colrange(\bA[R_Q, C_Q] - \bB[R_Q,C_Q]) &\perp \colrange(\bB[R_Q, C_Q]), & \forall Q \in \cP_j,\\
            \rowrange(\bA[R_Q, C_Q] - \bB[R_Q,C_Q]) &\perp \rowrange(\bB[R_Q, C_Q]), & \forall Q \in \cP_j.
        \end{align*}
    %where $\colrange$ and $\rowrange$ \todo{introduce these notations no later than the first time they appear \TL{Done}} are the linear subspaces spanned by the columns and rows of a matrix respectively. 
    %The notation $\cS_1 \perp \cS_2$ indicates that for any pair of vectors $(x_1, x_2) \in \cS_1 \times \cS_2$ ($\cS_1, \cS_2$ are some linear subspaces), we have: $\pr{x_1}{x_2}{} := x_1^\top x_2 = 0$.  
\end{lemma}
    By~\Cref{lemma:orthogonal-space} with $j=q$ we further have:
    \begin{equation*}
    \rowrange(\bK[R_{Q}, C_{Q}]) = 
    \rowrange(\bM[R_{Q}, C_{Q}] - \bN[R_{Q}, C_{Q}]) \perp \rowrange(\bN[R_{Q}, C_{Q}]) %, \forall Q \in \cP'(P).
    \end{equation*}
    hence
    \(
        \rowrange(\bK[R_{P}, C_{Q}]) \perp 
        \rowrange(\bN[R_{P}, C_{Q}]) %, \forall Q \in \cP'(P).
    \)
    % Therefore, the submatrix $\bM[R_P, C_P]$ can be written as 
    % %(assuming that $\cP'(P):=\{Q_1, \ldots, Q_k\}$):
    % \begin{equation*}
    %     \bM[R_P, C_P] = \begin{pmatrix}
    %         \bM[R_{Q_1}, C_P] \\
    %         \vdots\\
    %         \bM[R_{Q_k}, C_P]
    %     \end{pmatrix} = \begin{pmatrix}
    %         \bN[R_{Q_1}, C_P] + \bK[R_{Q_1}, C_P]\\
    %         \vdots\\
    %         \bN[R_{Q_k}, C_P] + \bK[R_{Q_k}, C_P]
    %     \end{pmatrix} = \bN[R_P, C_P] + \bK[R_P,C_P]
    % \end{equation*}
    and~\eqref{eq:BlockOrthoConditionAppendix} holds %for $Q \in \cP'(P)$ 
    as claimed.
    %$\bN[R_P, C_Q]^\top\bK[R_{P}, C_Q] = \bK[R_{P}, C_Q]^\top\bN[R_{P}, C_Q] = \mbf{0}_{|C_Q|\times |C_Q|}, \forall Q \in \cP'(P)$.
  %  \textcolor{red}{TODO: Need also to prove \eqref{eq:BlockOrthoConditionAppendix} for $Q \neq Q' \in \cP'(P)$ \HL{There is no need for it}.}
    %$\bN[R_P, C_Q]^\top\bK[R_{P}, C_{Q'}] = \bK[R_{P}, C_Q]^\top\bN[R_{P}, C_{Q'}] = \mbf{0}_{|C_Q|\times |C_Q|}$ for $Q \neq Q'$}
    % Therefore,
    % \begin{equation}
    %     \label{eq:relationL123}
    %     \begin{aligned}
    %         \bM[R_Q, C_Q]^\top\bM[R_Q, C_Q] &= \left(\sum_{l = 1}^k \bN[R_{P_l}, C_Q]^\top\bN[R_{P_l}, C_Q]\right) + \left(\sum_{l = 1}^k \bK[R_{P_l}, C_Q]^\top\bK[R_{P_l}, C_Q]\right)\\
    %         &= \bN[R_Q, C_Q]^\top\bN[R_Q, C_Q] + \bK^\top\bK.
    %     \end{aligned}
    % \end{equation}
 
%\end{proof}
%The proof of \eqref{eq:OrthoConditionLowRankAnnexBis}  is analogous.

 We proceed similarly when $s>q$: we prove an analog of \eqref{eq:AnnexLowRankErrorIneq} for each $Q \in \cP'$ instead of each $P \in \cP$. For this we establish a variant of \eqref{eq:OrthoConditionLowRankAnnex}, where $\bU,\bV,\bW$ are the $R_Q \times C_Q$ blocks instead of $R_P \times C_P$:
    \begin{equation}
    \label{eq:OrthoConditionLowRankAnnexBis}
    \bV^\top\bW = %\bW^\top\bV =
    \mat{0}_{|C_{q}|\times|C_{q}|},
\end{equation} 
so that $\bU^\top\bU = \bA + \bB$ where $\bA := \bV^\top\bV$ and $\bB = \bW^\top\bW\succeq 0$. %\todo{\ldots and $\bU^\top\bU$ here \TL{Done}}
Since $s > q$, by \Cref{lemma:inclusion} again, $\bV=\bN[R_P, C_P]$ and $\bW=\bK[R_P,C_P]$ are \emph{vertical} concatenations of $\bN[R_Q,C_P]$ and $\bK[R_Q, C_P]$, $P \in \cP'(Q)$, respectively,
\eqref{eq:OrthoConditionLowRankAnnexBis} can be deduced from an analog to   
\eqref{eq:BlockOrthoConditionAppendix}:
\begin{equation}
    \label{eq:BlockOrthoConditionAppendixBis}
    \bN[R_Q, C_P]^\top\bK[R_{Q}, C_{P}] = \bK[R_{Q}, C_P]^\top\bN[R_{Q}, C_{P}] = \mbf{0},
    \forall P \in \cP(Q) := \cP_i(Q),
    %_{|C_P|\times |C_{P}|}
\end{equation}
(with the notations of \Cref{lemma:inclusion}), which is a direct consequence of
\begin{equation*}
    \begin{aligned}
        \colrange(\bK[R_{Q}, C_{P}]) &\subseteq \colrange(\bK[R_{Q}, C_{Q}]),\\
        \colrange(\bN[R_{Q}, C_{P}]) &\subseteq \colrange(\bN[R_{Q}, C_{Q}])\\
        \colrange(\bK[R_{Q}, C_{P}]) &\perp 
        \colrange(\bN[R_{Q}, C_{P}]).
    \end{aligned}
\end{equation*}
where $\colrange(\cdot)$ is the linear span of the columns of a matrix. %\eqref{eq:BlockOrthoConditionAppendixBis} follows correspondingly.

To conclude, we now prove \Cref{lemma:sub-additive}, \Cref{lemma:inclusion}, and \Cref{lemma:orthogonal-space}.

\begin{proof}[Proof of \Cref{lemma:sub-additive}]
    The set $\mathbb{S}^n$ of symmetric {$n \times n$} matrices is convex, and the function $f: \mathbb{S}^n \mapsto \RR: \bA \mapsto \sum_{i = 1}^q \lambda_i(\bA)$ is convex \cite[Problem 3.26]{ConvexOptimization}. In addition, $f$ is positively homogeneous, i.e., $f(t\bA) = tf(\bA), \forall t \geq 0$. Therefore, for any $\bA, \bB \in \mathbb{S}^n$, we have:
    \begin{align*}
        f(\bA) + f(\bB) 
        &= 
        2\cdot\frac{1}{2}
        \left(f(\bA) + f(\bB)\right) 
        \geq 
        2f\left(\frac{\bA + \bB}{2}\right) 
        = f(\bA + \bB)\\
        \intertext{so that if in addition $\bB \in \mathbb{S}^n$ is positive semi-definite (PSD) we get}
    \Big[\Tr(\bA + \bB) - \sum_{i = 1}^{q} \lambda_i(\bA + \bB)\Big] &- \Big[\Tr(\bA) - \sum_{i = 1}^{q} \lambda_i(\bA)\Big]\\
            &= [\Tr(\bA + \bB) - f(\bA + \bB)] - [\Tr(\bA) - f(\bA)]\\
            &= \Tr(\bB) - [f(\bA + \bB) - f(\bA)]\\
            &\geq \Tr(\bB) - f(\bB) \geq 0.  
    \end{align*}
    The fact that $\bB$ is PSD was only used in the last inequality.
    %\todo{Léon: the notation $\mathbb{S}^n_+$ is not introduced before? \ER{I think we can just remove the +}} 
\end{proof}

%To establish the relation between the projection onto different $\arch_s, s \in \intset{L - 1}$, we need to study the relation between $(R_P, C_P)$ and $(R_Q, C_Q)$ for $P \in \cP_s, Q \in \cP_{\mathtt{row}}, s \neq r$. Their relations are shown in the following lemma:

% \begin{proof}[Proof of \Cref{lemma:partition}]
% %\todo{To be checked then moved}
% We use the following lemma
  
%     In addition
%     \todo{Relies on part of \Cref{lemma:inclusion} that I did not check yet. Toujours pas clair pour moi, on en parlera demain.}, $R_P =  %\cup_{l = 1}^k R_{Q_l}
%     \RG{\cup_{Q \in \cP'(P)} \TL{R}_Q}
%     $ because $\cup_{Q \in \cP'} R_Q \times C_Q = \cup_{P \in \cP} R_P \times C_P = \supp(\bS_{\pattern_1 * \ldots * \pattern_L})$.
% \qed
% \end{proof}
\begin{proof}[Proof of~\Cref{lemma:inclusion}]
    {\bf Preliminaries.} Denote $\pattern_\ell = (a_\ell, b_\ell, c_\ell, d_\ell)$ for $\ell \in \intset{L}$. By \Cref{def:firstlevelfactorization-new} and the definition \eqref{eq:operatorDBPfour} of the operator $*$, we have
    \begin{align*}
        \splitarch{\arch}{i} = \Big(\underbrace{\big(a_1, \tfrac{b_1d_1}{d_i}, \tfrac{a_i c_i}{a_1}, d_i\big)}_{=:\pattern = (a,b,c,d)}, \underbrace{\big(a_{i+1}, \tfrac{b_{i+1}d_{i+1}}{d_L}, \tfrac{a_Lc_L}{{a_{i+1}}}, d_L\big)}_{=:\pattern'=(a',b',c',d')}\Big).
    \end{align*}

{Consider an arbitrary column $\mat{s}$ of $\bS_{\pattern}$. By the structure of $\bS_{\pattern}$ (cf.~\Cref{fig:DBfactorillu}) there exists a block index $1 \leq k \leq a=a_1$ and an index $1 \leq \ell \leq d=d_i$ such that $\mat{s}$ is equal to the $\ell$-th column of the $k$-th ``group'' of $cd$ columns of $\bS_{\pattern}$, so that
%\todo{Léon: il me semble que si on part du principe que les indices commencent à 1 (et non pas à 0 comme en informatique) il faudrait écrire $(t \mod d) =  \ell + 1$ à la place de $(t \mod d) =  \ell$? Idem en bas pour $(t \mod d') =  \ell$ qu'il faudrait remplacer en $(t \mod d') =  \ell+1$?}
    \begin{align*}
    \supp(\mat{s}) 
    &= \intset{(k-1)bd+1,kbd} \cap \{t \in \mathbb{Z}: 
    {t \equiv \ell\mod d} %(t \mod d) =  \ell
    \}\\
    &= \underbrace{\intset{(k-1)b_1d_1 + 1,kb_1d_1}}_{=:T_k} \cap \{t \in \mathbb{Z}: 
    {t\equiv \ell\mod d_i} %(t \mod d_i) =  \ell
    \}
    %= \{kbd+\ell+nd: 0 \leq n < b\}
    %= \{kb_1d_1+\ell+nd_i: 0 \leq n < b_1d_1/d_i\} 
    =: R_{k,\ell}^i.
\end{align*}
    Similarly, for each row $\mat{s}'$ of $\bS_{\pattern'}$, we have
       \begin{align*}
    \supp(\mat{s}') 
    &= 
    \intset{(k-1)c'd'+1,kc'd'} \cap \{t \in \mathbb{Z}: 
    {t \equiv \ell\mod d'}
    %(t \mod d') =  \ell
    \}\\
    &= \intset{(k-1)\tfrac{a_Lc_Ld_L}{a_{i+1}} + 1,k\tfrac{a_Lc_Ld_L}{a_{i+1}}} \cap \underbrace{\{t \in \mathbb{Z}: 
    {t \equiv \ell\mod d_L}
    %(t \mod d_L) =  \ell
    \}}_{=:T'_\ell} =:C^i_{k,\ell}
%=    \{kc'd'+\ell+nd': 0 \leq n < c'\}
 %   = \{k\tfrac{a_Lc_Ld_L}{a_{i+1}}+\ell+nd_L: 0 \leq n < \tfrac{a_Lc_L}{a_{i+1}}\} =: C_{k,\ell}^i
    \end{align*}
    for some $1 \leq k \leq a'= a_{i+1}$, $1 \leq \ell \leq d'= d_L$.
    The same holds with $j$ instead of $i$.
}

    {\bf Claims 1 and 2.}    
    Consider now $P \in \cP_i$ and $Q \in \cP_j$. 
    Assume that $R_P \cap R_Q \neq \emptyset$. By the definition of $\cP_i$ and $\cP_j$ (\Cref{def:classequivalence}) and the above considerations, there exists indices $k,k',\ell,\ell'$ such that $R_P=R^i_{k,\ell}$ and $R_Q = R^j_{k',\ell'}$, hence $R_P \cap R_Q \subset T_k \cap T_{k'}$, and therefore $k = k'$. As a result, $R_P$ (resp.~$R_Q$) is exactly the subset of all integers in $T_k$ satisfying a certain congruence modulo $d_i$ (resp.~modulo $d_j$). Since $i < j$, by chainability, we have $d_j \mid d_i$, hence 
    %\todo{Tung or Elisa: can you complete the missing details \TL{Done}} 
    there are only two possibilities for $\{t \in \Z, 
    {t \equiv \ell\mod d_i}
    %(t \mod d_i) = \ell
    \}$ and $\{t \in \Z, 
    {t \equiv \ell'\mod d_j}
    %(t \mod d_j) = \ell'
    \}$:
    %\todo{Léon: changer ici si le todo précédent est correct}:
    they are either disjoint or the former is a subset of the latter. Since $R_P \cap R_Q \neq \emptyset$ the case of an empty intersection is excluded, and %As a consequence, $\ell = \ell'$ and
    we obtain $R_P \subseteq R_Q$ as claimed.
    
    A similar reasoning shows that if $C_P \cap C_Q \neq \emptyset$ then  $C_P = C^i_{k,\ell}$ and $C_Q = C^j_{k',\ell'}$ with $\ell = \ell'$. Denoting $p = a_Lc_Ld_L/a_{i+1}$ and $q = a_Lc_L d_L/a_{j+1}$, we have $p/q = a_{j+1}/a_{i+1}$. By chainability, since $i<j$, we have $a_{i+1} \mid a_{j+1}$ hence $q \mid p$. It follows that $C_P$ (resp.~$C_Q$) is exactly the subset of all integers in $T'_\ell$ belonging to $\intset{{(k-1)}%k
    p+1,{k}%(k+1)
    p}$ (resp.~$\intset{{(k'-1)}%k'
    q+1,{k'}%(k'+1)
    q}$). Since $C_P$ intersects $C_Q$, these two intervals must intersect, and it is easy to check that since $q \mid p$ the second must then be a subset of the first. We obtain $C_Q \subseteq C_P$ as claimed.
    
    %  Therefore,\todo{Peut-on avoir un argument court et simple ?} the support of columns (resp.~rows) of $\bS_{\pattern_k^1}$ (resp.~$\bS_{\pattern_k^2}$), which is also $R_P$ (resp.~$C_P$) for some $P \in \cP_i$ is given by:
    % \begin{equation*}
    %     \begin{aligned}
    %         (kb_1d_1 + \ell,  kb_1d_1 + \ell + d_i, \ldots, kb_1d_1 + \ell + b_1d_1 - d_i), &\; 0 \leq k \leq a_1, 1 \leq \ell \leq d_i\\
    %         \left(h\frac{a_Lc_Ld_L}{a_{s+1}} + p, h\frac{a_Lc_Ld_L}{a_{s+1}} + p + d_L, \ldots, h\frac{a_Lc_Ld_L}{a_{s+1}} + p + \frac{a_Lc_Ld_L}{a_{s+1}} - d_L\right), &\; 0 \leq h \leq a_{s+1}, 1 \leq p \leq d_L   
    %     \end{aligned}
    % \end{equation*}

{\bf Claims 3 and 4.} We prove Claim 3, the proof of Claim 4 is analogous. Consider $P \in \cP_i$. The fact that $R_P \subseteq R_Q$ for each $Q \in \cP_j(P)$ is a direct consequence of Claim 1 and the definition of $\cP_j(P)$. Let us now show that the sets $C_Q$, $Q \in \cP_j(P)$ are pairwise disjoint. For this consider $Q, Q' \in \cP_j(P)$
    such that $C_Q \cap C_{Q'} \neq \emptyset$. By the first claim of \Cref{lemma:inclusion}
    % \todo{Léon: why are we using \Cref{lemma:inclusion} in the proof of \Cref{lemma:inclusion}?}
    and the definition of $\cP_j(P)$, we have $R_P \subseteq R_Q$ and $R_P \subseteq R_{Q'}$, hence 
    \[
    (R_{Q} \times C_{Q}) \cap (R_{Q'} \times C_{Q'})  \supseteq R_{P} \times (C_Q \cap C_{Q'}) \neq \emptyset. 
    \]
     By \Cref{lem:seq-qchainability} the pair $(\pattern_1 * \ldots * \pattern_\ell, \pattern_{\ell + 1} * \ldots, \pattern_L)$ is chainable, hence by \Cref{lem:qperfectcovering} it satisfies the assumptions of     \Cref{theorem:tractablefsmf}, i.e., the sets $R_Q \times C_Q$, $Q \in \cP_j$ are pairwise disjoint.
    This shows that $Q=Q'$.

   % By chainability, we have: $a_1 \mid a_2 \ldots \mid a_L$ and $d_L \mid d_{L-1} \mid \ldots \mid d_1$. Therefore, $R_P \subseteq R_Q$ or $R_P \cap R_Q = \emptyset$. Similarly, $C_Q \subseteq C_P$ or $C_Q \cap C_P = \emptyset$. 
   %\todo{The bit below I did not check \TL{Checked}}
   Finally, since $\cup_{P \in \cP_i} R_P \times C_P = \cup_{Q \in \cP_j} R_Q \times C_Q = \supp(\bS_{\pattern_1 * \ldots * \pattern_L})$ (see, e.g., \Cref{lem:prodDBBis}),
   we have $\cup_{P \in \cP_i} R_P = \cup_{Q \in \cP_{j}} R_Q = \intset{a_1b_1d_1}$ and $\cup_{P \in \cP_i} C_P = \cup_{Q \in \cP_{j}} C_Q = \intset{a_Lb_Ld_L}$. Combined with the proved inclusions $R_P \subseteq R_Q$ and $C_Q \subseteq C_P$ (under non-empty intersection conditions), this implies that $R_{Q}$ (resp.~$C_{P}$) 
   is the (disjoint) union of all
   %can be partitionned into several 
   $R_{P}$ (resp.~$C_{Q}$) that intersect it.
\end{proof}

%\subsubsection{Technical results related to the projection onto $\cB^{\arch_j}$}
\begin{proof}[Proof of~\Cref{lemma:orthogonal-space}]
    By~\eqref{eq:decomposedisjointoverlapping}, $\bB$ is a projection of $\bA$ onto $\cB^{\arch_j}$ if and only if $\supp(\bB) \subseteq \bS_{\pattern_1 * \ldots * \pattern_L}$ and its subblocks satisfy
\[
        \bB[R_P, C_P] \in 
        \prank_{|P|}(\bA[R_P,C_P])
%        \argmin_{\bM \in \RR^{|R_P| \times |C_P|}, \rank(\bM) \leq |P|} \|\bA[R_P,C_P] - \bM\|_F, &\forall P \in \cP_s.
\]
where we recall that $\prank_{r}(\mat{X}) := \argmin_{\bM: \rank(\bM) \leq r} \|\mat{X}-\mat{M}\|_F$. It is thus sufficient to prove the following claim for the low-rank matrix approximation problem: Given a matrix $\bD \in \RR^{m \times n}$, if a matrix $\bC$ is a projection of $\bD$ onto the set $\cM_r$ of matrices of rank at most $r$, then:
    \begin{align*}
            \colrange(\bD - \bC) \perp \colrange(\bC)\quad \text{and}\quad
            \rowrange(\bD - \bC) \perp \rowrange(\bC).
    \end{align*}
    This result is classical. We re-prove it here for self-completeness. %\RG{Write $\bD = \sum_i \sigma_i \mat{u}_i \mat{v}_i^\top$ the SVD of $\bD$. We prove the result when $\sigma_r > \sigma_{r+1}$. In that case, by the Eckart-Young-Mirsky theorem on low-rank matrix approximation \cite{eckart1936approximation} the projection $\bC$ is unique and given by $\bC = \sum_{i=1}^r \sigma_i \mat{u}_i\mat{v}_i^\top$, so that $\colrange(\bC) = \$
        One can re-write the projection onto $\cM_r$ as the following optimization problem:
    \[
            \underset{\bX \in \RR^{m \times r}, \bY \in \RR^{r \times n}}{\text{Minimize}} \; f(\bX, \bY)\quad \text{where}\quad f(\bX,\bY) := \tfrac{1}{2}\|\bX\bY-\bD\|_F^2.
    \]
    At an optimum $\bC = \bX\bY$ (which always exists, take the truncated SVD for example): %we have %, the gradient is zero, i.e.,
     %Calculating the gradients of $f(\bX, \bY)$ w.r.t. $\bX$ and $\bY$ gives us:
    \[
        \frac{\partial f}{\partial \bX} = (\bX\bY-\bD)\bY^\top  = \mbf{0}, \qquad \frac{\partial f}{\partial \bY} = \bX^\top(\bX\bY-\bD)  = \mbf{0}. 
    \]
        %$\frac{\partial f}{\partial \bX} = \mbf{0}, \frac{\partial f}{\partial \bY} = \mbf{0}$. 
        or equivalently: 
    \(
            \rowrange(\bC - \bD) \perp \rowrange(\bY)
            \quad \text{and}\quad
            \colrange(\bC - \bD) \perp \colrange(\bX).
    \)
    %where $\bC = \bX^\star\bY^\star$. 
    Since $\colrange(\bC) \subseteq \colrange(\bX), \rowrange(\bC) \subseteq \rowrange(\bY)$, this yields the claim. 
\end{proof}

% \begin{proof}
%     Assume that $(\bU, \bD, \bV) \in \RR^{m \times m} \times \RR^{m \times n} \times \RR^{n \times n}$ is the Singular Value Decomposition of $\bA$. That means $\bV, \bU$ are orthogonal matrices, $\bD$ is a (generalized) diagonal matrix whose $\bD[i,i], i \leq \min(m,n)$ is equal to the $i$th largest singular value $\lambda_i$ of $\bA$ and $\bA = \bU\bD\bV^\top$. By Eckart–Young–Mirsky theorem, 
%     \begin{equation*}
%         E^{\cM_r}(\bA) = \sum_{i = r + 1}^{\min(m,n)} \lambda_i(\bA)^2.
%     \end{equation*}
%     Moreover, $\bA^\top\bA = \bV\bK\bV^\top$ where:
%     \begin{equation*}
%         \bK = \texttt{diag}(\lambda_1(\bA)^2, \lambda_2(\bA)^2, \ldots, \lambda_{\min(m,n)}(\bA)^2, \underbrace{0, 0, \ldots, 0, 0}_{(n - \min(m,n))-\text{times}})
%     \end{equation*}
%     Therefore,
%     \begin{equation*}
%         \begin{aligned}
%             \Tr(\bA^\top\bA) - \sum_{i = 1}^{\min(n,r)} \lambda_i(\bA^\top\bA) &= \sum_{i = 1}^{\min(m,n)} \lambda_i(\bA)^2 - \sum_{i = 1}^{\min(n,r)} \lambda_i(\bA)^2 \quad (\text{since $\Tr(\bV\bK\bV^\top) = \Tr(\bK)$})\\
%             &= \sum_{i = \min(n,r) + 1}^{\min(m,n)} \lambda_i(\bA)^2 = \sum_{i = r + 1}^{\min(m,n)} \lambda_i(\bA)^2.
%         \end{aligned}
%     \end{equation*}
%     The last equality holds since: if $r \leq n$, it holds trivially. Otherwise, both quantities are equal to zero. 
% \end{proof}
% \ER{point to a textbook?}

\subsection{Proof for \eqref{eq:orthogonalpair-new}}
\label{appendix:bettermainresultsproof-new}
% \todo{Léon: continue here}
\begin{proof}
The proof is given when $\sigma$ is the identity permutation $\sigma = (1,\ldots,L-1)$. The proof when $\sigma$ is the ``converse" permutation $\sigma = (L-1,\ldots,1)$ is similar, replacing rows by columns, left-*-unitarity by right-*-unitarity, etc.
We begin by preliminaries on key matrices involved in the expression of the matrices $\bB_J$, $\bB_{J-1}$ and $\bB_p$ appearing in \eqref{eq:orthogonalpair-new}.
%specify some key matrices \ER{properties?} to be analyzed. 
We then highlight simple orthogonality conditions which imply~\eqref{eq:orthogonalpair-new}. Finally we prove these orthogonality conditions.
%\RG{Il faudrait commencer par expliquer qu'avec le choix de $\sigma$ à la $J$-ème itération on a $P_j = \ldots$ en ligne 4 et/ou $= ???$ en ligne 19, et donc $s=\ldots$, $j,r,t = \ldots$}

{\bf Preliminaries.} Since $\sigma_\ell=\ell$, $\ell \in \intset{L-1}$, it is not difficult to check by induction on $J \in \intset{L-1}$
%(involving lines~\ref{line:s},\ref{line:defjmodifiedalgo},\ref{line:setpartition}) 
that at the $J$th iteration of \cref{algo:modifedbutterflyalgo}, at line~\ref{line:fsmfmodifiedalgo}, we have 
%where $I_\ell$ = \{\ell\}$, $\ell < J$ and $I_J = \intset{J,L}$.
    \begin{enumerate}
    %\item \RG{$\mathtt{partition} = (I_1,\ldots, I_J)$ where $I_\ell$ = \{\ell\}$, $\ell < J$ and $I_J = \intset{J,L}$}
    \item $\mathtt{partition} = (I_1,\ldots,I_J)$ where $I_\ell=\{\ell\}$, $\ell \in \intset{J-1}$ and $I_J = \intset{J,L}$;
        \item $s = J$ (from line~\ref{line:s}).
        \item $j = J$ and $I_j = I_J = \intset{J, L}$, i.e., $q = J$, $t = L$ (line~\ref{line:defjmodifiedalgo}).
    \end{enumerate}
    
    Denoting $(\matseq{\hat{\mat{X}}}{\ell})_{\ell=1}^L$ the final output of \Cref{algo:modifedbutterflyalgo},
 one can also check by induction %   we remark 
 that the list \texttt{factors} obtained at the end of the $J$-th iteration (cf.~lines~\ref{line:fsmfmodifiedalgo} and \ref{line:list_factors_at_end} of \Cref{algo:modifedbutterflyalgo}) is a tuple of the form:
    \begin{equation*}
        (\hat{\bX}_{1}, \ldots, \hat{\bX}_{J - 1}, \bX_{\intset{J,J}},\bX_{\intset{J+1,L}}).
    \end{equation*}
    Indeed, the value of the first $J-1$ factors in the list \texttt{factors} are left-$r_\ell$-unitary factors for appropriate $r_\ell$, $\ell \in \intset{J - 1}$ (cf.~\Cref{lemma:role-pseudo-orthogonality}). Therefore, their values during pseudo-orthonormalization operations in the next iterations $J+1, J+2, \ldots, L-1$ do not change anymore, which means that they are equal to $\hat{\bX}_{1}, \ldots, \hat{\bX}_{J - 1}$. The factor $\bX_{\intset{J,J}}=\bX_{\intset{q,s}}$ will be pseudo-orthonormalized at the next iteration if $J<L-1$.

    {\bf Expression of $\bB_J$.} By the convention of \eqref{eq:productleftright-new} and the above observation, we have
    $\bX_{\tleft}^{(J)} = {\bX}_{\intset{1,1}} \ldots {\bX}_{\intset{J-1,J-1}}
    = \prod_{\ell = 1}^{J - 1} \hat{\mbf{X}}_\ell
    $ (by convention this is the identity if $J=1$) and $\bX_{\tright}^{(J)} = \bI_{a_Lb_Ld_L}$.  
    The matrices $\bX_{\intset{q,s}}$ $\bX_{\intset{s+1,t}}$, $\bX_{I_j}$ of line~\ref{line:fsmfmodifiedalgo} of \Cref{algo:modifedbutterflyalgo}
    correspond to $\matseq{\mat{X}}{\intset{J, J}}$, $\matseq{\mat{X}}{\intset{J+1, L}}$ and $\matseq{\mat{X}}{\intset{J, L}}$, hence the matrix $\bB_J$ from~\eqref{eq:DefBJ} is
        % \todo{à vérifier car je ne suis pas sûr d'avoir compris de quelles matrices il s'agit ?} at line~\ref{line:fsmfmodifiedalgo} of \Cref{algo:modifedbutterflyalgo}, and reuse  
    %     \RG{Moreover $\bX_{\tleft}^{(J)} = {\bX}_{I_1} \ldots {\bX}_{I_{J - 1}}, \bX_{\tright}^{(J)} := \bI_{a_Lb_Ld_L}$ (cf.~\Cref{eq:productleftright-new}).}
    \begin{equation}
    \label{def:matrice-P_J-new}
        \mat{B}_J 
         := \mat{X}_\tleft^{(J)} \matseq{\mat{X}}{\intset{q, s}} \matseq{\mat{X}}{\intset{s+1, t}}\mat{X}_\tright^{(J)}
            = 
            \mat{X}_\tleft^{(J)} \matseq{\mat{X}}{\intset{J, J}} \matseq{\mat{X}}{\intset{J+1, L}},\quad \forall J \in \intset{L-1}.
    \end{equation}
    Given the nature of \eqref{eq:orthogonalpair-new} we also express $\bB_{J-1}$ and $\bB_p$ for $p \in \intset{J,L-1}$.
    
  Equation~\eqref{def:matrice-P_J-new} also reads $\bB_J = \bX_{\intset{1,1}} \ldots \bX_{\intset{J,J}}\bX_{\intset{J+1,L}}$, hence for $J \in \intset{2,L-1}$
    \begin{align}
        \mat{B}_{J-1} &= \bX_{\intset{1,1}} \ldots \bX_{\intset{J-1,J-1}}\bX_{\intset{J,L}}\notag\\
        & = \mat{X}_\tleft^{(J)} \matseq{\mat{X}}{\intset{J, L}} \label{def:matrice_BJ-1-new}.
    \end{align}
    This indeed holds for all $J \in \intset{L-1}$: for $J=1$, the convention below~\eqref{eq:DefBJ} yields $\bB_{J-1} = \bB_0 := \bA$, and $\bX_{\intset{1,L}} = \bA$ (line~\ref{line:factorinit} of \Cref{algo:modifedbutterflyalgo}). As  $\mat{X}_\tleft^{(J)}$ is the identity, this shows \eqref{def:matrice_BJ-1-new} for $J=1$.
    
    For $p>J$ we have $\bX_{\tleft}^{(p)} = \bX_{\tleft}^{(J)} \matseq{\hat{\mat{X}}}{J} \ldots \matseq{\hat{\mat{X}}}{p-1}$, and one easily deduces from \eqref{def:matrice-P_J-new} that
    \begin{equation}
        \mat{B}_{p} = \bX_{\tleft}^{(J)} \left( \prod_{\ell = J}^{p - 1} \hat{\mbf{X}}_\ell \right) \matseq{\mat{X}}{\intset{p,p}} \matseq{\mat{X}}{\intset{p+1,L}},
        \quad \forall J,p \in \intset{L-1}\ \text{such that}\  p \geq J,
    \end{equation}
    where again by convention an empty matrix product is the identity.

    {\bf Expression of $\langle \bB_{J-1}-\bB_J,\bB_p\rangle$.}
    From now on we fix $J \in \intset{L}$ and $p \in \intset{J,L-1}$, and rewrite $\langle \bB_{J-1}-\bB_J,\bB_p\rangle$ where we recall that the inner-product is associated to the Frobenius norm on matrices. Denote
%    Fix $J \in \intset{L-1}$. Denote
    \begin{equation}
    \label{eq:xleftJ-appearing-new}
    \begin{split}
         \widetilde{\bB}_{J-1} &:= \matseq{\mat{X}}{\intset{J,L}}, \\
         \widetilde{\bB}_J &:= \matseq{\mat{X}}{\intset{J,J}} \matseq{\mat{X}}{\intset{J+1,L}}, \\
         \widetilde{\bB}_{p} &:= \left( \prod_{\ell = J}^{p - 1} \hat{\bX}_\ell \right) \matseq{\mat{X}}{\intset{p,p}}
         \matseq{\mat{X}}{\intset{p+1,L}}.
%         \matseq{\hat{\mat{X}}}{\intset{p+1,L}}
         %\quad 
         %p > J
         %\ \text{(NB: if $p=J$ then $\bP_p=\bP_J$ is already defined)}.
    \end{split}
    \end{equation}
    Note that the expression of $\widetilde{\bB}_p$ falls back on $\bB_J$ when $p=J$.
 %   \todo{Est-ce vraiment $\hat{\bX}_{\intset{p+1,L}}$ ou plutôt $\bX_{\intset{p+1,L}}$? Traiter mieux cas $p=J$}
    We observe that for $k \in \{J-1,J,p\}$, $\widetilde{\bB}_k$ is a $\pattern'$-factor with $\pattern'=\pattern_J * \ldots * \pattern_L$ by \cref{prop:stability} since $\bX_{\intset{p,p}} \in \setDBfactor{\pattern_p}$, $\bX_{\intset{p+1,L}} \in \setDBfactor{\pattern_{p+1} * \ldots * \pattern_L}$ and (when $p>J$)
    $\hat{\bX}_\ell \in \setDBfactor{\pattern_\ell}$ for $\ell \in \intset{J, p - 1}$. %    we then use .  % \todo{Expliquer pourquoi \TL{Fait}}.
    
    Since the factors $\hat{\bX}_\ell$ for $\ell \in \intset{J-1}$ are $\pattern_\ell$-factors, with chainable patterns $\pattern_\ell$, and left-$r_\ell$-unitary for appropriate $r_\ell$'s, by \Cref{lem:suppdbfactorprod} and a recursive use of \Cref{lemma:stableundermatrixmultiplication} 
    %\todo{Vérifier que l'on invoque les lemmes appropriés dans tout ce paragraphe et compléter si nécessaire \TL{Fait}} 
    their product $\bX:= \bX_{\tleft}^{(J)}$ is a left-$r$-unitary $\pattern$-factor with $\pattern := (\pattern_1 * \ldots * \pattern_{J - 1})$ and
    $r=r(\pattern_{J-1},\pattern_J)$. As $\bY_1 := \widetilde{\bB}_{J-1}-\widetilde{\bB}_{J}$ and $\bY_2 := \widetilde{\bB}_{p}$ are $\pattern'$-factors, and since by \cref{lem:seq-qchainability}, $r(\pattern,\pattern') =r(\pattern_{J-1},\pattern_J)=r$,  left-$r$-unitarity of $\bX$ and the parallelogram law yield
 
      \begin{equation*}
        \langle \mbf{X}\mbf{Y}_1, \mbf{X}\mbf{Y}_2 \rangle = \frac{1}{2}(\|\bX\bY_1\|_F^2 + \|\bX\bY_2\|_F^2 - \|\bX(\bY_1 + \bY_2)\|_F^2)
        = \langle \bY_1,\bY_2\rangle.
    \end{equation*}
    % and the norm preservation property of left-$r$-unitary definition.
    %
    %   
    % The second inequality is a consequence of the left-$r$-unitarity of $\bX_{\tleft}^{(J)}$ \todo{explain why}, which we briefly detail:\todo{detail}
    %
    Combining this with the expressions of $\bB_{J-1},\bB_J,\bB_p$, we obtain:
    \[
    \ip{{\mbf{B}}_{J - 1} - {\mbf{B}}_{J}}{\mbf{B}_{p}} = 
    \ip{\bX_{\tleft}^{(J)} ( \widetilde{\bB}_{J-1} - \widetilde{\bB}_J)}{\bX_{\tleft}^{(J)} \widetilde{\bB}_{p}}
    =
    \ip{\widetilde{\bB}_{J-1} - \widetilde{\bB}_J}{ \widetilde{\bB}_{p}}.
    \]
    
    {\bf Orthogonality conditions and their proof.}
    Consider the partition $\cP := \cP(\bS_{\pattern_{\tleft}}, \bS_{\pattern_{\tright}}
    %{J+1} * \ldots * \pattern_{L}}
    )$ with $\pattern_\tleft := \pattern_J$ and $\pattern_\tright = \pattern_{J+1} * \ldots * \pattern_L$.
       For each $k \in \{J-1,J,p\}$, since $\widetilde{\bB}_k$ is a $\pattern'$-factor and $\pattern' = \pattern_\tleft * \pattern_\tright$, %(\pattern_{J+1} * \ldots * \pattern_L$), 
       by \Cref{lem:prodDBBis} we have $\supp(\widetilde{\bB}_k) \subseteq \bigcup_{P \in \cP} R_P \times C_P$ 
       %\todo{quote appropriate lemma \TL{Fait}} 
       hence
    \[
    \ip{\widetilde{\bB}_{J-1} - \widetilde{\bB}_J}{ \widetilde{\bB}_{p}}
    = 
    \sum_{P \in \cP} \ip{(\widetilde{\bB}_{J-1} - \widetilde{\bB}_J)[R_P,C_P]}{ \widetilde{\bB}_{p}[R_P,C_P]}.
    \]
    It follows that $\ip{{\mbf{B}}_{J - 1} - {\mbf{B}}_{J}}{\mbf{B}_{p}} =0$ 
    %      Hence, for any  $p \geq J$:
    % \begin{equation*}
    % \begin{split}
    %     &\ip{{\mbf{B}}_{J - 1} - {\mbf{B}}_{J}}{\TL{\RG{\mbf{B}}_{p}}} = 0 \\
    %     {\iff} &\ip{\bX_{\tleft}^{(J)} ( \mat{P}_{J-1} - \mat{P}_J)}{\bX_{\tleft}^{(J)} \TL{\mbf{P}_{p}}} = 0 \quad (\text{due to } {\eqref{eq:xleftJ-appearing-new} + \eqref{def:matrice-P_J-new}})\\
    %     {\iff} &\ip{  \mat{P}_{J-1} - \mat{P}_J }{ \TL{\mbf{P}_{p}} } = 0 \quad \text{(due to \Cref{lem:orthonormalfactorisorthogonal} below)}\\
    %     {\iff} &\sum_{P \in \cP} \left \langle  \left(\mat{P}_{J-1} - \mat{P}_J \right)[R_P, C_P], \, \TL{\mbf{P}_{p}}[R_P, C_P] \right \rangle = 0 \quad (\text{due to } {\eqref{eq:partition-proof-bettermainresults-new}}).
    % \end{split}
    %   \end{equation*}
    %  and show that $\langle \bB_{J-1}-\bB_J,\bB_p\rangle = 0$ 
     is implied by the orthogonality conditions
     \begin{equation}
		\label{eq:orthogonalcolumnspanblock-new}
		\ip{(\widetilde{\bB}_{J - 1} - \widetilde{\bB}_J)[R_P,C_P]}{\widetilde{\bB}_{p}[R_P, C_P]}  = 0, \quad \forall P \in \cP. %, p \geq J,
	\end{equation}

    To conclude it remains to prove that \eqref{eq:orthogonalcolumnspanblock-new} indeed holds for a fixed $P \in \cP$. By the definition of $\matseq{\mat{X}}{\intset{J, J}}=\bX_{\intset{q,s}}$, $\matseq{\mat{X}}{\intset{J+1, L}}=\bX_{\intset{s+1,t}}$ and $\matseq{\mat{X}}{\intset{J, L}} = \bX_{\intset{q,t}}$ at line~\ref{line:fsmfmodifiedalgo} of \Cref{algo:modifedbutterflyalgo}
    %where \Cref{algo:algorithm1} is called with patterns $\pattern_\tleft,\pattern_\tright$. By
    and line~\ref{line:svd} of \Cref{algo:algorithm1}, we obtain that 
    \begin{align*}
    \widetilde{\bB}_J[R_P, C_P] 
    & \stackrel{\eqref{eq:xleftJ-appearing-new}}{=}
    (\matseq{\mat{X}}{\intset{J, J}} \matseq{\mat{X}}{\intset{J+1, L}})[R_P,C_P]\\
    & \stackrel{    \textrm{\Cref{lemma:submatrixequality}}
    }{=}
    \bX_{\intset{J,J}}[R_P,P]\bX_{\intset{J+1,L}}[P,C_P]
\intertext{is the best rank-$|P|$ approximation of}
    \widetilde{\bB}_{J-1}[R_P, C_P]
        & \stackrel{\eqref{eq:xleftJ-appearing-new}}{=}
\matseq{\mat{X}}{\intset{J, L}}[R_P, C_P].
    \end{align*}
    
    If %the latter is of rank at most $|P|$ then 
    the two matrices are equal then \eqref{eq:orthogonalcolumnspanblock-new} trivially holds, otherwise we have 
    %we necessarily have 
    \begin{align*}
         \rank(\widetilde{\bB}_J[R_P,C_P]) 
        & = \rank(\mat{X}_{\intset{J, J}}[R_P, P]) = |P|\\        \colrange(\widetilde{\bB}_J[R_P,C_P]) 
        & = \colrange(\mat{X}_{\intset{J, J}}[R_P, P]).
        \end{align*}
    Since $\mat{X}_{\intset{J, J}}$ is pseudo-orthonormalized into $\hat{\bX}_J$ at the next iteration $J+1$ (by design of \Cref{algo:exchange}), we also have
    \[
    \colrange(\mat{X}_{\intset{J, J}}[R_P, P]) = \colrange(\hat{\bX}_J[R_P, P]).
    \]
    Combining these equalities with \Cref{lemma:orthogonal-space} %\todo{check that the lemma is indeed easily applicable here \TL{Checked}} 
    we obtain
     \begin{align*}
     \colrange(\hat{\bX}_J[R_P, P])
     =\colrange(\widetilde{\bB}_J[R_P,C_P]) \perp
      \colrange((\widetilde{\bB}_J-\widetilde{\bB}_{J-1})[R_P, C_P]).
        \end{align*}
    The orthogonality condition     \eqref{eq:orthogonalcolumnspanblock-new} is then implied by the fact that
    \begin{align*}
       \colrange(\widetilde{\bB}_{p}[R_P, C_P]) & \subseteq
       \colrange(\hat{\bX}_J[R_P, P])
        \end{align*}
    which is trivial if $p=J$, and if $p>J$
    follows from the fact by \eqref{eq:xleftJ-appearing-new} we have $\widetilde{\bB}_p = \hat{\bX}_J \bZ$
    for some $\bZ := \in \setDBfactor{\pattern_\tright}$.
\end{proof}

\section{On the generalization of the complementary low-rank property}
\label{app:low-rank-characterization}
% \LZc{prêt à la relecture}
We show that the generalized complementary low-rank property associated with a chainable $\arch$ given in \Cref{def:general-clr-new} coincides, under some assumption on $\arch$, with the classical definition of the complementary low-rank property given in \cite{li2015butterfly}. 

\subsection{Definition of the classical complementary low-rank property \cite{li2015butterfly}}
We start by reformulating the definition of \cite{li2015butterfly}, based on the following terminology.
    A cluster tree $T$ of a set of indices $\intset{n}$ with depth $L$ is a tree where:
    \begin{itemize}
        \item the nodes %of the tree 
        are subsets of $\intset{n}$;
        \item the root is $\intset{n}$;
        \item each non-leaf node has non-empty children that %form a 
        partition %of 
        their parent;
        \item the only leaves are at level $L$. 
    \end{itemize}
    %We fix the
    By convention %that
    the root nodes are at level 0. The set of all nodes at level $\ell \in \intset{L}$ is denoted $\treelevel{T}{\ell}$.
Notice that, %An important remark is that, 
by definition of a cluster tree $T$, the set of nodes $\treelevel{T}{\ell}$ form a partition of the root node for each level $\ell$, and $\treelevel{T}{\ell+1}$ is finer than $\treelevel{T}{\ell}$, in the following sense.

\begin{definition}[{Finer partitions \cite[Definition 1.11]{hackbusch2015hierarchical}}]
    Given two partitions $P$ and $\tilde{P}$ of $\intset{n}$, $P$ is \emph{finer} than $\tilde{P}$ if for all $I \in P$ there is a $\tilde{I} \in \tilde{P}$ with $I \subseteq \tilde{I}$. 
    % This transitive relationship is written as $\partitioncol_1 \preceq \partitioncol_2$ or $\partitioncol_2 \succeq \partitioncol_1$.
\end{definition}

We can now give a definition that covers the classical notion of complementary low-rank property \cite{li2015butterfly}.
% We can now %are now able to 
% give the classical definition of the complementary low-rank property.\todo{reformuler offline positionnement}
%\mRG{Since it is classical, can we also cite a paper/book where it appears ? \LZ{Yes, it first appeared in \cite{li2015butterfly}}}
\begin{definition}[``Classical'' complementary low-rank property]
\label{def:complementary-low-rank}
    Consider two cluster trees $\treerow$ and $\treecol$ of $\intset{m}$ and $\intset{n}$ with the same depth $L$, and a set of integer {\em rank constraints} $\mathcal{R} := \{ r_{R, C} \, | \, (R, C) \in  \treelevel{\treerow}{\ell} \times \treelevel{\treecol}{L - \ell + 1}, \, \ell \in \intset{L} \} \subseteq \mathbb{N}$.
    % , and with dyadic subtrees. 
    % Denote $\partitionrow_\ell$ (respectively $\partitioncol_\ell$) the set of nodes of $\treerow$ (respectively $\treecol$) at level $\ell \in \intset{L}$. 
    A matrix $\mat{A}$ of size $m \times n$ satisfies the \emph{complementary low-rank property for $\treerow$ and $\treecol$ with rank constraints $\mathcal{R}$
    %, up to precision $\epsilon > 0$}, 
    if
    the submatrix $\matindex{\mat{A}}{R}{C}$ 
    has rank at most $r_{R,C}$
    %can be well-approximated by a rank-$r_{R,C}$ matrix up to precision $\epsilon$ 
    for each $(R, C) \in \treelevel{\treerow}{\ell} \times \treelevel{\treecol}{L - \ell + 1}$ and $\ell \in \intset{L}$.}
    %\todo{Flou / redondant sur $\epsilon $ et "well-approximated". Quelle métrique?}
    % $\mat{A}$ admits low-rank submatrices at blocks $\bigcup_{\ell=1}^L P_\ell$ where $P_\ell$ is the product block partition constructed from $\treelevel{\treerow}{\ell}$ and $\treelevel{\treecol}{L - \ell + 1}$.
    % \begin{enumerate}
    %     \item $\treerow$ and $\treecol$ have dyadic subtrees, in the sense that each subtree of the root node only possesses non-leaf nodes that have exactly two children;
    %     \item 
    % \end{enumerate}
\end{definition}

%\begin{remark}
With this definition, a matrix $\bA$ satisfies the complementary low-rank property of \cite{li2015butterfly,li2018multidimensional} if it is $\epsilon$-close in the Frobenius norm to a matrix $\tilde{\bA}$ satisfying \Cref{def:complementary-low-rank} in the particular case where $\treerow$ and $\treecol$ are {\em dyadic trees} or {\em quadtrees}, where it is assumed that $\max_{R,C} r_{R,C}$ is bounded poly-logarithmically in $1/\epsilon$ and in  matrix size.

\subsection{Relation with the generalized complementary low-rank property (\Cref{def:general-clr-new})}
A cluster tree yields a hierarchical partitioning of a given set of indices. Similarly, %the following proposition shows that, 
under some appropriate conditions, a chainable architecture $\arch$ also yields a hierarchical partitioning of the row and column indices, which leads to two cluster trees.

% \todo{en fait on va plutôt rédiger de la manière suivante: on traduit chaque }

\begin{proposition}
\label{prop:partitions-for-clr}
    Consider a chainable architecture $\arch = (\pattern_\ell)_{\ell=1}^L$ where $\pattern_\ell := (a_\ell, b_\ell, c_\ell, d_\ell)$ for $\ell \in \intset{L}$,  and denote $m \times n$ the size of the matrices in $\setButterfly{\arch}$, as well as
    \begin{equation}
    \label{eq:DefSupportAggregated}
     \mat{S}_{q,t} := \matseq{S}{\pattern_{q} * \ldots * \pattern_t},\quad 1 \leq q \leq t \leq L.    \end{equation}
%    \todo{Cette notation est-elle utilisable ailleurs?} 
%     Denoting $\mat{S}_{q,t} := \matseq{S}{\pattern_{q} * \ldots * \pattern_t}$ for any $1 \leq q \leq t \leq L$ 
     Recalling the notation from \Cref{def:classequivalence}, define for all $\ell \in \intset{L-1}$: 
    % \LZc{pourquoi on a appelé ça $P$ et pas $I$?}
    % \todo{ce n'est pas tout à fait correct c'est plutôt $L - \ell$ pour l'arbre des lignes}
    \begin{equation}\label{eq:DefBlockPartitionTrees}
        \begin{split}
            % \partitionrow_{L-\ell} &:= \{ \matindex{\supp(\bS_{\pattern_{1} * \ldots * \pattern_\ell})}{:}{j} \}_j, \\
            \partitionrow_{L-\ell} &:= \{ R_{P} \, | \, P \in \cP(\bS_{1, \ell}, \bS_{\ell+1,L}) \}, \\
            % \partitioncol_{\ell} &:= \{ \matindex{\supp(\bS_{\pattern_{\ell + 1} * \ldots * \pattern_L})}{i}{:} \}_i.
            \partitioncol_{\ell} &:= \{ C_{P} \, | \, P \in \cP(\bS_{1, \ell}, \bS_{\ell+1,L}) \}.
        \end{split}
    \end{equation}
    Assume that $a_1 = d_L = 1$. Then:
    \begin{enumerate}[leftmargin=*]
        \item    For each $\ell \in \intset{L-1}$, $\partitionrow_{\ell}$ is a partition of $\intset{m}$, and $\partitioncol_{\ell}$ is a partition of $\intset{n}$. Moreover
    %   :     % \LZc{on peut enlever la non redondance?}
    %\begin{itemize}
     %   \item $\{ \partitionrow_{\ell}\}_{\ell=1}^{L-1}$ and $\{ \partitioncol_{\ell}\}_{\ell=1}^{L-1}$ are partitions of $\intset{m}$ and $\intset{n}$, respectively;
     %   \item for each $\ell \in \intset{L-2}$, 
     $\partitionrow_{\ell + 1}$ (resp.~$\partitioncol_{\ell + 1}$) is finer than $\partitioncol_{\ell}$ (resp.~than $\partitionrow_{\ell}$) when $\ell \leq L-2$. 
        % and the partition $\partitionrow_{\ell+1}$ is finer than $\partitionrow_{\ell}$.
    %\end{itemize}
           \item Consequently, when completed by $\partitionrow_0 := \intset{m}$ and $\partitioncol_0 := \intset{n}$, the collections $\{ \partitionrow_{\ell} \}_{\ell=0}^{L-1}$ and $\{ \partitioncol_{\ell} \}_{\ell=0}^{L-1}$ yield two cluster trees, denoted $\treerow_\arch$ and $\treecol_\arch$, of depth $L-1$ with root node $\intset{m}$ and $\intset{n}$, respectively.
            \item For each $\ell \in \intset{L-1}$ we have
  %  by definition, 
    \begin{equation}
        \label{eq:BlockPartitionEquality}
        \{(R_P,C_P): P \in \cP(\bS_{1, \ell}, \bS_{\ell+1,L})\} =
    %is the product block partition associated with
    \treelevel{\treerow_\arch}{L - \ell} \times \treelevel{\treecol_\arch}{\ell}.
    \end{equation}
                \end{enumerate}

    %\todo{A-ton $\partitionrow_1 = \intset{m}$ ou faut-il "ajouter" la racine?}
    % Assume that $a_1 = d_L = 1$. Then, there exist two cluster trees $\treerow$ and $\treecol$ of depth $L - 1$ such that:
    % \begin{enumerate}
    %     \item the nodes at level $\ell \in \intset{L-1}$ of the tree $\treerow$ (respectively $\treerow$) are exactly $\partitionrow_\ell$ (respectively $\treecol$);
    %     \item $\cP(\bS_{1, \ell}, \bS_{\ell+1,L}) = \partitionrow_{L - \ell} \times \partitioncol_{\ell}$ for each $\ell \in \intset{L-1}$, where we recall the notations from \Cref{def:classequivalence}.
    % \end{enumerate}    
\end{proposition}

% \begin{remark}
%     Due to the assumption $a_1 = d_L = 1$, the product $\partitionrow_{L-\ell} \times \partitioncol_\ell$ is indeed a partition of $\intset{m} \times \intset{n}$ where $m = b_1 d_1$ and $n = a_L c_L$.
% \end{remark}
We postpone the proof of \Cref{prop:partitions-for-clr} to immediately highlight that under its assumptions
%Therefore, under the same assumptions of \Cref{prop:partitions-for-clr}, 
the general %definition of the 
complementary low-rank property %given in
of \Cref{def:general-clr-new} coincides with the classical one %given in
of \Cref{def:complementary-low-rank}.

% The following proposition show that it is possible

% Under some conditions on the considered butterfly architure $\arch$, the low-rank characterization of butterfly matrices associated with $\arch$ given in \Cref{cor:characterizationofDBmatrix} is precisely the complementary low-rank property in the sense of \Cref{def:complementary-low-rank}, for two specific cluster trees $\treerow$ and $\treecol$ that we describe now.

\begin{corollary}
\label{cor:equivalence-clr}
Consider a chainable architecture $\arch = (\pattern_\ell)_{\ell=1}^L$ where $\pattern_\ell := (a_\ell, b_\ell, c_\ell, d_\ell)$ for $\ell \in \intset{L}$, and assume that $a_1=d_L=1$.
%    Under the %same 
%    assumptions of \Cref{prop:partitions-for-clr} \RG{(in particular, $a_1=d_1=1$)}, f
    For any matrix $\mat{A}$ of appropriate size, the following are equivalent:
    \begin{itemize}
        \item $\mat{A}$ satisfies the {\em generalized} complementary low-rank property  %(\Cref{def:complementary-low-rank}) 
        (\Cref{def:general-clr-new})
        associated with $\arch$;
        \item $\mat{A}$ satisfies %\emph{exactly} 
        the \emph{classical} complementary low-rank property (\Cref{def:complementary-low-rank}) for the cluster trees $(\treerow_\arch, \treecol_\arch)$ defined in \Cref{prop:partitions-for-clr}, 
        with rank constraint $\mathcal{R}$ such that for each $\ell \in \intset{L-1}$ and every
        $(R,C) \in \treelevel{\treerow_\arch}{L-\ell} \times \treelevel{\treecol_\arch}{\ell}$ we have $r_{R,C} = r(\pattern_\ell, \pattern_{\ell + 1})$.
    \end{itemize}
    % Consider an architecture $\arch = (\pattern_\ell)_{\ell=1}^L = (a_\ell, b_\ell, c_\ell, d_\ell)_{\ell=1}^L$ that is non-redundant and chainable.
    % Assuming $a_1 = d_L = 1$, 
     % if, and only if, 
    % $\rank(\bA[R_P,C_P]) \leq r(\pattern_\ell, \pattern_{\ell + 1})$ for each $ P \in \cP(\bS_{1, \ell}, \bS_{\ell+1,L})$ and $\ell \in \intset{L-1}$.
\end{corollary}

\begin{proof}
    Since $\arch$ is chainable and $a_1 = d_L = 1$, by
    \Cref{lem:suppdbfactorprod} we have $\pattern_1 * \ldots * \pattern_L = (1, m, n, 1)$ for some integers $m,n$ (which turn out to be the dimensions of matrix $\bA$) hence $\bS_{\pattern_1 * \ldots * \pattern_L} = \one{m \times n}$. Therefore, by \Cref{def:general-clr-new}, a matrix $\mat{A}$ satisfies the general complementary low-rank property associated with $\arch$ if, and only if, 
    % the low-rank characterization of a butterfly matrix $\mat{A} \in \setButterfly{\arch}$ in \Cref{cor:characterizationofDBmatrix} is that 
    $\rank(\bA[R_P,C_P]) \leq r(\pattern_\ell, \pattern_{\ell + 1})$ for each $ P \in \cP(\bS_{1, \ell}, \bS_{\ell+1,L})$ and $\ell \in \intset{L-1}$ (we use the shorthand \eqref{eq:DefSupportAggregated}).
    % , where we recall the notation from \Cref{def:classequivalence}. 
    By \Cref{prop:partitions-for-clr}, this is precisely a reformulation of the classical complementary low-rank property (\Cref{def:complementary-low-rank}) for the trees $(\treerow_\arch, \treecol_\arch)$.
    
%     , because     for each $\ell \in \intset{L-1}$ we have
%   %  by definition, 
%     \begin{equation}
%         \label{eq:BlockPartitionEquality}
%         \{(R_P,C_P): P \in \cP(\bS_{1, \ell}, \bS_{\ell+1,L})\} =
%     %is the product block partition associated with
%     \treelevel{\treerow}{L - \ell} \times \treelevel{\treecol}{\ell}.
%     \end{equation}
 
%\end{remark}
\end{proof}

\begin{proof}[Proof of \Cref{prop:partitions-for-clr}]
    % \LZ{todo, la preuve pour le moment repose sur un jeu d'écriture en partant de la définition de $\partitioncol_{\ell}$ en revenant sur le produit de kronecker.}
%\todo{Simplifier en ré-utilisant $P_{t,k}$ si possible ? \TL{Simplifie}}
The second claim is an immediate consequence of the first one.
    We only prove the first claim for the column partitions $\{ \partitioncol_{\ell} \}_{\ell=1}^{L-1}$, since the proof is similar for the row partitions $\{ \partitionrow_{\ell} \}_{\ell=1}^{L-1}$. 
        Given $\ell \in \intset{L-1}$, let us first show that $\partitioncol_{\ell}$ is a partition of $\intset{n}$ for any $\ell \in \intset{L-1}$, where $n := a_L c_L d_L$.
    %by definition.
    By \eqref{eq:producttheta}: $\pattern_{\ell+1} * \ldots * \pattern_{L} = \left( a_{\ell+1}, \frac{b_{\ell+1} d_{\ell+1}}{d_L}, \frac{a_L c_L}{a_{\ell+1}}, d_L \right) = \left( a_{\ell+1}, b_{\ell+1} d_{\ell+1}, \frac{a_L c_L}{a_{\ell+1}}, 1 \right)$ since $d_L = 1$. 
    % Denote $a := a_{\ell+1}$ and $c = \frac{a_L c_L}{a}$, so that $\pattern_{\ell+1} * \ldots * \pattern_{L} = (a,b,c,1)$.
    Therefore, by \Cref{def:dbfactorfour}, the elements of $\partitioncol_{\ell}$ 
    % have the form:
    % \begin{equation}
    % \label{eq:partitioncol-indices}
    %     \intset{(j-1)c_\ell + 1, jc_\ell}\, \quad \text{where} \; j \in \intset{a_{\ell+1}} \text{ and } c_\ell := \frac{a_L c_L}{a_{\ell+1}}, %\left\{ \{ k + (j-1) c \}_{k=1}^c \, | \, j \in \intset{a_{\ell+1}} \right \} \quad \text{with} \quad c := \frac{a_L c_L}{a_{\ell+1}},
    % \end{equation}
    % which is indeed easily checked to be 
    partition the integer set $\intset{n}$ (with
    $n := a_L c_L = a_Lc_Ld_L$) into consecutive intervals of length $c_\ell$, where 
    $c_\ell := \frac{a_L c_L}{a_{\ell+1}}$.
        It remains to show that each element of $\partitioncol_{\ell}$ can itself be partitioned into elements %those
    of $\partitioncol_{\ell + 1}$. Since the latter are consecutive intervals of length $c_{\ell+1}$,
    % of the form $\intset{(i-1)c_{\ell+1} + 1, ic_{\ell+1}}$, $i \in \intset{a_{\ell+2}}$ and they also form a partition of $\intset{n}$, it is sufficient to show that for each $i \in \intset{a_{\ell+2}}$ there exists $j \in \intset{a_\ell}$ such that
    % $\intset{(i-1)c_{\ell+1} + 1, ic_{\ell+1}} \subseteq \intset{(j-1)c_{\ell+1} + 1, ic_{\ell+1}}$. This
    this is a direct consequence of the fact that $c_{\ell}$ is an integer multiple of $c_{\ell+1}$: indeed,
    since $(\pattern_{\ell+1}, \pattern_{\ell +2})$ is chainable, we have: $a_{\ell + 1} \mid a_{\ell + 2}$, $\gamma = a_{\ell + 2} / a_{\ell + 1} \in \NN$ and $c_{\ell} = \gamma c_{\ell + 1}$. 

       To prove~\eqref{eq:BlockPartitionEquality} observe that the left hand side is trivially a subset of $\treelevel{\treerow_\arch}{L - \ell} \times \treelevel{\treecol_\arch}{\ell}$, by the definition of $\treelevel{\treerow_\arch}{L - \ell} = \partitionrow_{L-\ell}$ and $\treelevel{\treecol_\arch}{\ell}= \partitioncol_{\ell}$
    (cf~\eqref{eq:DefBlockPartitionTrees}). The equality will follow from the fact that both sides share the same number of elements: this is a direct consequence of the fact that   %By \cref{rmk:cluster-tree-not-dyadic}, we have: 
    $|\treelevel{\treerow_\arch}{L - \ell}| = d_{\ell}$, $|\treelevel{\treecol_\arch}{\ell}| = a_{\ell + 1}$, a property that we prove immediately below.
    Indeed, this property implies that  $|\treelevel{\treerow_\arch}{L - \ell} \times \treelevel{\treecol_\arch}{\ell}| = a_{\ell+1}d_\ell$, and by \cref{lem:seq-qchainability}, the number of equivalence classes $P$ of $\cP(\bS_{1,\ell}, \bS_{\ell+1,L})$ is also $\frac{a_{\ell}c_\ell d_\ell}{r(\pattern_\ell, \pattern_{\ell+1})} = a_{\ell+1}d_\ell$. 
    %Hence the result.
    % \todo{attention cette notation n'est pas correcte}     % \LZc{continue here}
        To conclude, we prove that $|\treelevel{\treerow_\arch}{L - \ell}| = d_{\ell}$, $|\treelevel{\treecol_\arch}{\ell}| = a_{\ell + 1}$.
% \begin{remark}
%     \label{rmk:cluster-tree-not-dyadic}
        %\todo{Est-ce évident? utile ? a insérer dans la proposition ( et à prouver)? a dire plus loin ? \TL{No, it is used in the proof of \cref{prop:partitions-for-clr}.}}
   % \todo{A insérer dans la preuve}
    By~\Cref{lem:suppdbfactorprod}, %for $\ell \in \intset{L-2}$, 
    we have $\pattern_{\ell+1} * \ldots * \pattern_{L} = \left(a_{\ell+1}, b_{\ell+1} d_{\ell+1}, \frac{a_L c_L}{a_{\ell+1}}, 1  \right)$ 
    %by \eqref{eq:producttheta},
    since $d_L=1$
    hence $\bS_{\ell+1,L}$
    % by assumption.
    %Therefore, the factor
 %   \todo{A FINiR: il faut faire plus progressivement la traduction}
    is a block diagonal matrix with $a_{\ell+1}$ dense blocks of the same size. Tracing back all definitions this implies that  $\treelevel{\treecol_\arch}{\ell} = \partitioncol_{\ell}$ is made of exactly $a_{\ell+1}$ consecutive intervals of the same size. The proof is similar for $\treelevel{\treerow_\arch}{L-\ell}$. 
    % \TL{Therefore there are exactly $a_{\ell + 1}$ nodes} of cardinality $\frac{a_L c_L}{a_{\ell + 1}}$ at level $\ell$ in $\treecol_\arch$. Each of them has $\frac{a_{\ell+2}}{a_{\ell+1}}$ children, because $a_{\ell+1} \mid a_{\ell+2}$ by chainability of $(\pattern_{\ell+1}, \pattern_{\ell+2})$. 
    % This means that  $\treecol_\arch$ is not dyadic in general, but the nodes at a same level have the same number of children. Similarly, ... \todo{faire la même pour les lignes}
    % Similarly, there are $d_{\ell}$ nodes of cardinal $\frac{b_1 d_1}{d_{\ell}}$ at level $L - \ell$ in $\treerow_\arch$. Each of them have $\frac{d_{\ell+1}}{d_\ell}$ children.
\end{proof}

\subsection{\revision{Illustration of the generalized complementary low-rank property (\Cref{def:general-clr-new})}}

\revision{We give an illustration of the generalized complementary low-rank property for the square dyadic butterfly architecture $\arch = (\pattern_\ell)_{\ell=1}^L$ for matrices of size $n \times n$ with $n=16$ and $L=\log_2(n)=4$, i.e., $\pattern_\ell = (2^{L - \ell}, 2, 2, 2^{\ell - 1})$ for $\ell \in \intset{L}$. Let $\mat{A}$ be a matrix satisfying such a property.
% One can verify that $\pattern_1 * \ldots * \pattern_L = (1, n, n, 1)$ due to \Cref{lem:suppdbfactorprod}, and $r(\pattern_\ell, \pattern_{\ell+1}) = 1$ for each $\ell \in \intset{L-1}$ by \Cref{def:chainableDBfour}. This means that, 
By the second property of \Cref{def:general-clr-new}, for each $ P \in 
         \cP(
         \bS_{\pattern_1 * \ldots * \pattern_\ell},\bS_{\pattern_{\ell+1} * \ldots * \pattern_L}
         )$ and $\ell \in \intset{L-1}$, the rank of the submatrix $\bA[R_P,C_P]$ is at most $r(\pattern_\ell, \pattern_{\ell+1})$, which is equal to $1$ by \Cref{def:chainableDBfour}.
\Cref{fig:complementary-low-rank} illustrates these rank-one submatrices, by remarking that $(\pattern_1, \pattern_2 * \pattern_3 * \pattern_4) = ((1, 2, 2, 8), (2, 8, 8, 1))$, $(\pattern_1 * \pattern_2, \pattern_3 * \pattern_4) = ((1, 4, 4, 4), (4, 4, 4, 1))$ and $(\pattern_1 * \pattern_2 * \pattern_3, \pattern_4) = ((1, 8, 8, 2), (8, 2, 2, 1))$, due to \Cref{lem:suppdbfactorprod}.}

\begin{figure}[h]
    \centering
		\begin{subfigure}[t]{0.32\textwidth}
			\centering
			\includegraphics[width=\textwidth]{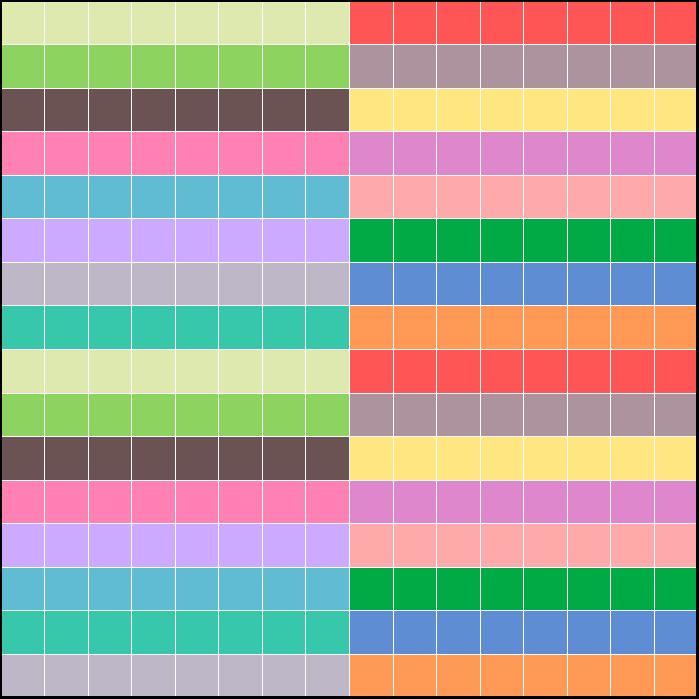}
			\caption{$\ell=1$}
		\end{subfigure}%
		~ 
		\begin{subfigure}[t]{0.32\textwidth}
			\centering
			\includegraphics[width=\textwidth]{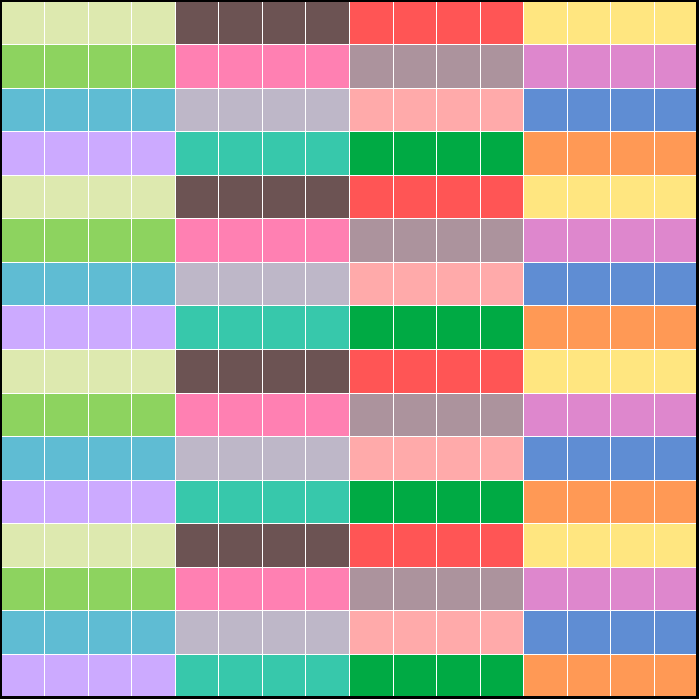}
			\caption{$\ell=2$}
		\end{subfigure}%
		~
		\begin{subfigure}[t]{0.32\textwidth}
			\centering
			\includegraphics[width=\textwidth]{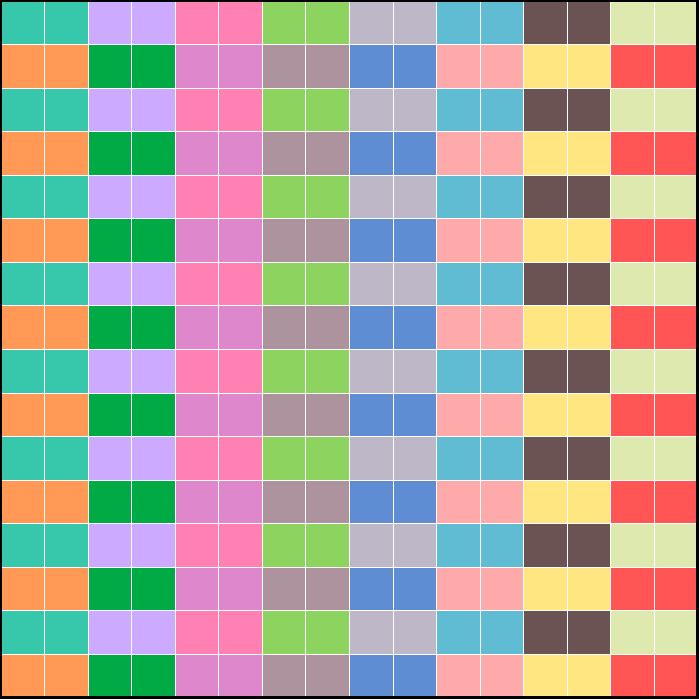}
			\caption{$\ell=3$}
		\end{subfigure}%
		\caption{\revision{Illustration of the complementary low-rank property associated to the square dyadic butterfly architecture $\arch = (\pattern_\ell)_{\ell=1}^L$ for matrices of size $n \times n$ with $n=16$ and $L=\log_2(n)=4$. For each $\ell \in \intset{L-1}$, we represent the partition of the matrix indices $\intset{n} \times \intset{n}$ into $\{ R_P \times C_P, \, P \in \cP(\bS_{\pattern_1 * \ldots * \pattern_\ell},\bS_{\pattern_{\ell+1} * \ldots * \pattern_L} ) \}$ using different colors, where each color represent one element of the partition.}}
    \label{fig:complementary-low-rank}
\end{figure}

\section{Numerical experiments: additionnal details and results}
\label{app:MoreNumerical}
In complement to \Cref{sec:experimentbutterfly}, we include further experimental details and results.
%\todo{Insérer dans section numérique? Dépend de la place prise par les figures}
%
\begin{figure}
    \centering
    \includegraphics[width=\textwidth]{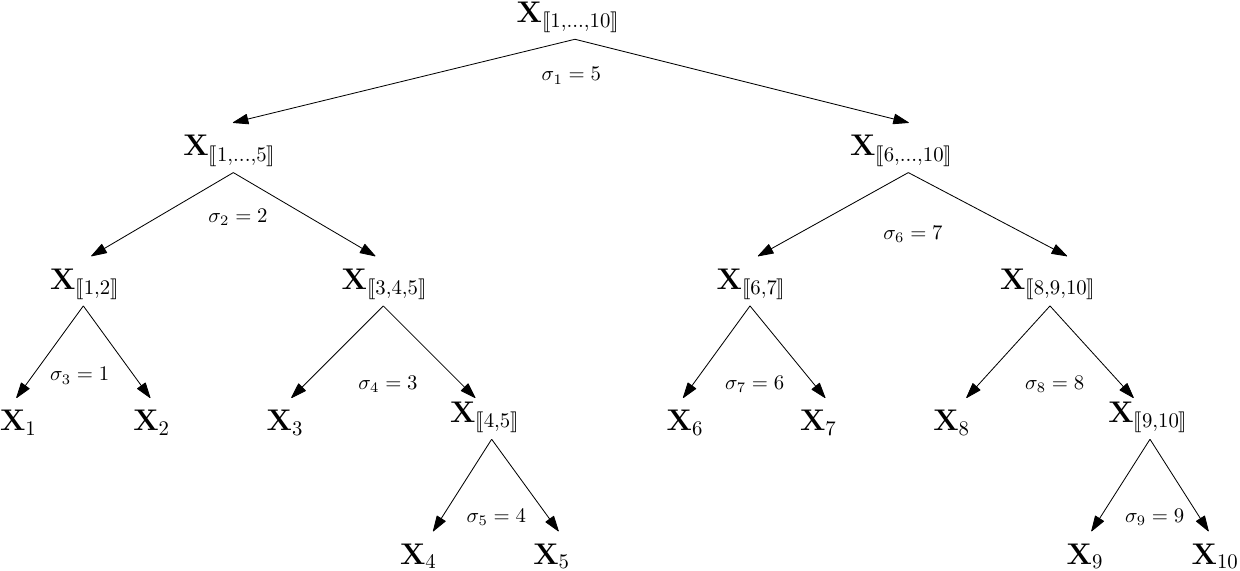}
    \caption{Illustration of the permutation $\sigma = (5,2,1,3,4,7,6,8,9)$ corresponding to the balanced factor-bracketing tree of $\intset{10}$.}
    \label{fig:factor-bracketing-tree-balanced}
\end{figure}
%
% \LZc{todo mettre la figure de l'arbre balanced pour illustrer la permutation.}
%
%
%
%\subsection{Varying the noise level}
%\label{app:more-graph-approximation-butterfly-noise}
% \LZc{TODO}
In \Cref{fig:error-hierarchical}, we observed that the approximation error obtained by the hierarchical algorithm \emph{with} orthonormalization operations (\Cref{algo:modifedbutterflyalgo}) is always smaller than the noise level $\epsilon = 0.1$, as opposed to the one obtained without orthonormalization.
In fact, we have the same observation for any values of $\epsilon \in \{ 0.01, 0.03, 0.1, 0.3 \}$, as shown in \Cref{fig:butterfly-approximation-other-values-epsilon}.
%\ER{if we keep the tree here, I think we can keep fig 8 also}

\begin{figure}[htbp]
		\centering
		\begin{subfigure}[t]{0.48\textwidth}
			\centering
			\includegraphics[width=\textwidth]{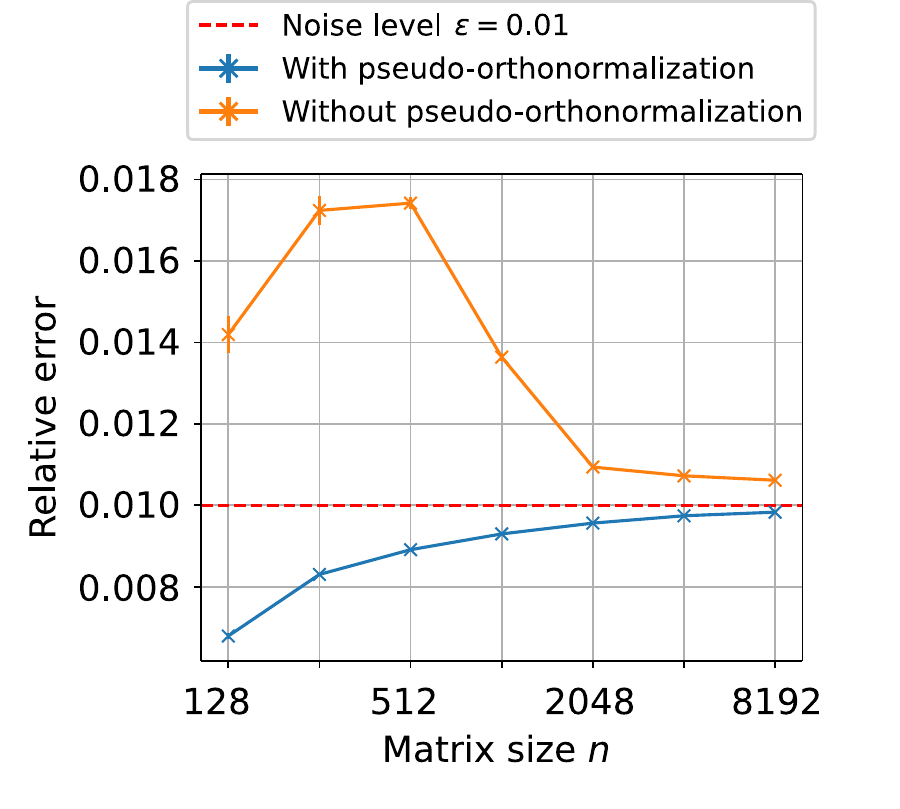}
			\caption{$\epsilon=0.01$}
		\end{subfigure}%
		~ 
		\begin{subfigure}[t]{0.48\textwidth}
			\centering
			\includegraphics[width=\textwidth]{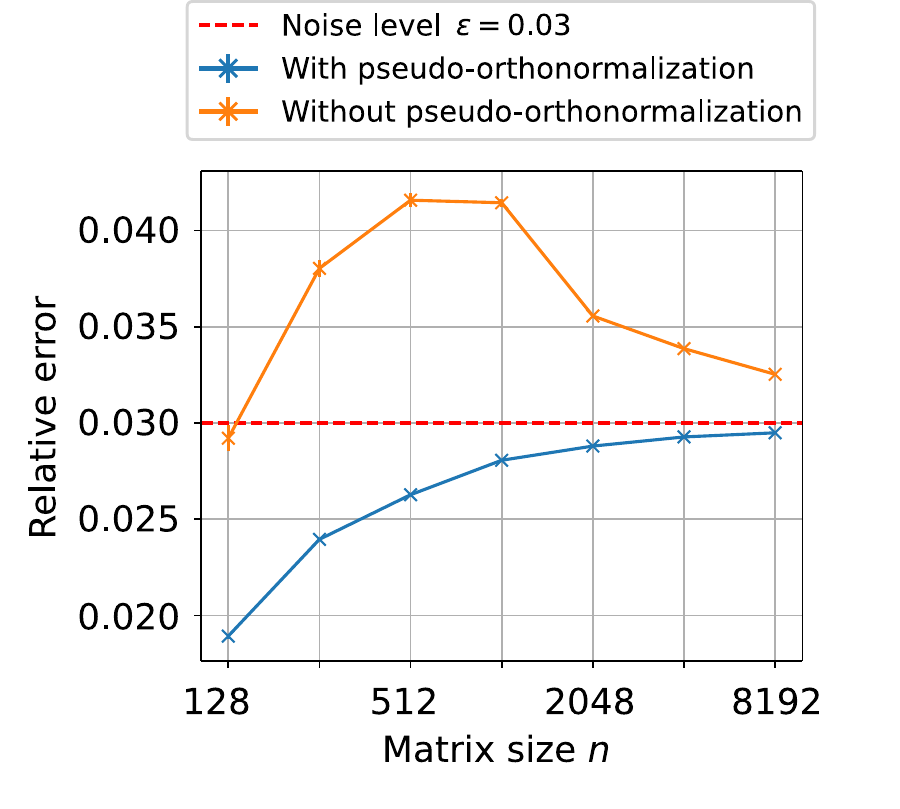}
			\caption{$\epsilon=0.03$}
		\end{subfigure}%
		~
		\newline
		\begin{subfigure}[t]{0.48\textwidth}
			\centering
			\includegraphics[width=\textwidth]{noise_0.1.pdf}
			\caption{$\epsilon=0.1$}
		\end{subfigure}%
		~ 
		\begin{subfigure}[t]{0.48\textwidth}
			\centering
			\includegraphics[width=\textwidth]{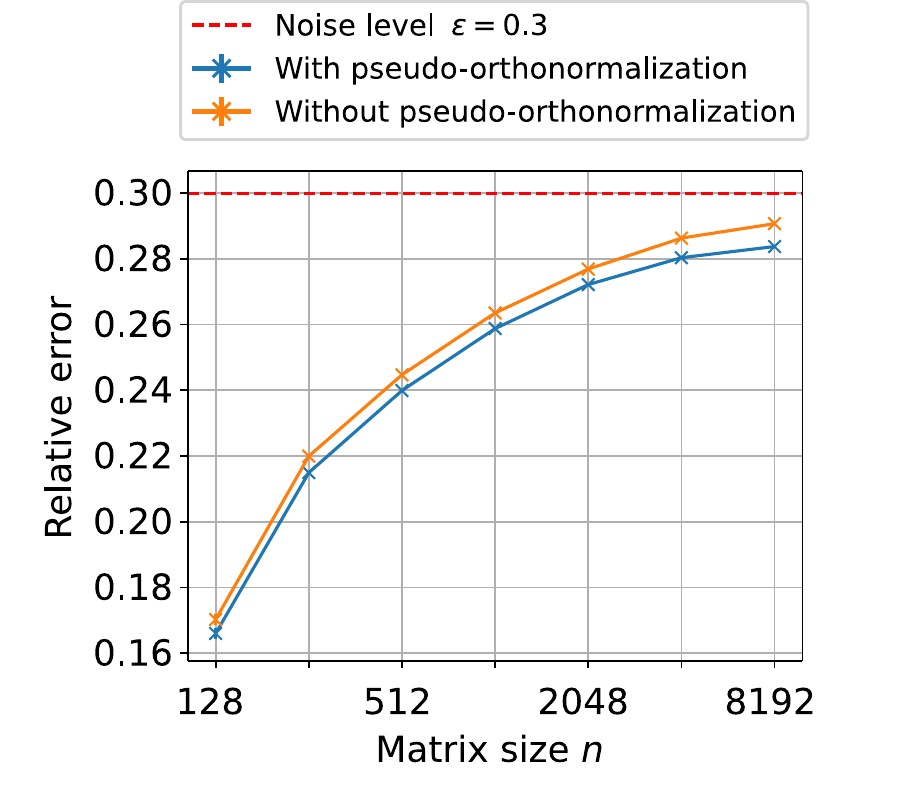}
			\caption{$\epsilon=0.3$}
		\end{subfigure}%
		\caption{Relative approximation errors \emph{vs.} the matrix size $n$, for \Cref{algo:recursivehierarchicalalgo} (without orthonormalization) and \Cref{algo:modifedbutterflyalgo} (with orthonormalization), for the instance of Problem \eqref{eq:butterfly-approximation-pb} described in \Cref{subsec:exp-ortho} {with $r=4$}. We show mean and standard deviation on the error bars over $10$ repetitions of the experiment.}
  % \LZc{notations}
		\label{fig:butterfly-approximation-other-values-epsilon}
	\end{figure}

\end{document}